\documentclass[11pt]{amsart}
\usepackage[margin=0.9in]{geometry}

\def\uvuci{\parindent=0em}

\usepackage[dvipsnames]{xcolor}
\usepackage[centertableaux]{ytableau}
\ytableausetup{smalltableaux}
\usepackage{graphicx}
\usepackage{wrapfig}
\usepackage{eso-pic}
\usepackage{mathtools}
\usepackage{float} 
\usepackage{amsmath, amsthm, amsfonts, amssymb, amscd, bm}
\usepackage{mathrsfs}
\usepackage{multicol}
\usepackage{xcolor}
\definecolor{OxBlue}{HTML}{002147}

\usepackage[colorlinks]{hyperref}
\hypersetup{colorlinks = true, citecolor={OxBlue}, urlcolor  ={OxBlue}, linkcolor={OxBlue} }

\usepackage{perpage}
\MakePerPage{footnote}

\usepackage[normalem]{ulem}

\usepackage{enumerate}
\usepackage[capitalise, noabbrev]{cleveref}

\usepackage{tikz-cd} 
\usepackage{subfigure}

\usepackage[all,cmtip]{xy}
\usepackage[utf8]{inputenc}
\usepackage[T1]{fontenc}

\newtheorem{thm}{Theorem}[section]
\newtheorem{cor}[thm]{Corollary}

\newtheorem{prop}[thm]{Proposition}
\newtheorem{lm}[thm]{Lemma}

\theoremstyle{definition}
\newtheorem{de}[thm]{Definition}
\newtheorem{ex}[thm]{Example}

\newtheorem{nota}[thm]{Notation}
\theoremstyle{remark}
\newtheorem{rmk}[thm]{Remark}

\renewcommand{\qedsymbol}{$\blacksquare$}

\def\det{\text{det}}

\def\char{\text{char}}
\def\Hom{\text{Hom}}
\def\codim{\mathrm{codim}}
\def\rk{\text{rk}}
\def\diag{\text{diag}}
\def\ol{\overline}
\def\ad{\text{ad}}
\def\Ad{\text{Ad}}

\def\z{\zeta}
\def\la{\langle}
\def\ra{\rangle}
 
\def\Q{\mathbb Q}
\def\R{\mathbb R}
\def\C{\mathbb C}
\def\N{\mathbb N}

\def\Z{\mathbb Z}

\def\P{{\mathbb P}}

\def\DD{\mathbb D}

\def\Kh{K\"{a}hler}
\def\hk{hyperk\"{a}hler}

\def\eps{\varepsilon}
\def\a{\alpha}
\def\b{\beta}
\def\c{\beta}

\def\Fi{\varphi}
\def\Ker{\text{Ker}}

\def\W{\mathbb{W}}

\def\fun{\rightarrow}
\newcommand{\fja}[1]{\xrightarrow{#1}}

\def\dejstvo{\curvearrowright}

\newcommand\norm[1]{\left\lVert#1\right\rVert}

\def\M{\mathfrak{M}}
\def\Aff{\text{Aff}}
\def\CM0{\C[\M_0]}

\def\L{\mathfrak{L}}

\def\F{\mathfrak{F}}

\def\Con1{Con_1(\M)}

\newcommand{\PZ}[1]{\mathcal{P}_{\Z/{#1}}}
\newcommand{\TZ}[1]{\mathcal{T}_{\Z/{#1}}}
\def\Hm{\mathcal{H}_m}
\def\Ymc{Y_{m,\beta}}

\def\sl{\mathfrak{sl}}

\def\B{\mathcal{B}}
\def\BP{\mathcal{B}_p}

\def\O{\mathcal{O}}

\def\FF{\mathscr{F}}
\def\FFF{{\mathscr{F}_{\mathbb{B}}}}
\def\val{\mathrm{ini}}

\def\NN{\mathcal{N}}
\def\NNN{\widetilde{\mathcal{N}}}

\def\SS{\mathcal{S}}
\def\SSS{\widetilde{\mathcal{S}}}

\def\x{{\color{white}1}}

\def\MM0{\mathfrak{M}_{\tiny{(\zeta_{\mathbb{R}},0})}(Q,{\normalfont\textbf{v}},{\normalfont\textbf{w}})}

\def\MG0{\mathcal{M}_{0,\z_\C}(Q,{\normalfont\textbf{v}},{\normalfont\textbf{w}})}

\def\iso{\cong}

\def\om{\omega}

\def\k{\mathbb{K}}

\def\Sp{\Sigma_p}

\def\Epqr{E(\Fi)^{pq}_r}

\def\J{I}

\def\ph{pseudoholomorphic}
\def\homotopying{homotoping}

\def\CC{$\C^*$}

\def\Slicep{\Sigma_p}

\def\Core{\mathrm{Core}}
\def\Fil{\mathscr{F}}

\def\PartI{\cite{RZ1}}

\def\Aff{\mathrm{Aff}}
\def\MM{\mathcal{M}}
\def\Spp{\Sigma}

\def\Y{Y}
\def\CP{\mathbb{C}P}
\def\RP{\mathbb{R}P}
\def\omI{\om}

\def\yinfty{y_{0}}

\def\iinfty{0}
\def\f{F}

\def\ell{m}

\def\x{y}

\def\AB{Atiyah--Bott}
\def\MB{Morse--Bott}
\def\MS{Morse--Smale}
\def\PL{Poincaré--Lefschetz}
\def\MBF{Morse--Bott--Floer}
\def\CFHW{Cieliebak--Floer--Hofer--Wysocki}
\def\KvK{Kwon\,--\,van Koert}

\def\RS{Robbin--Salamon}
\def\BO{Bourgeois--Oancea}
\def\contracting{conical}
\def\G{\Gamma}

\makeatletter
\newcommand{\doublewidetilde}[1]{{%
  \mathpalette\double@widetilde{#1}%
}}
\newcommand{\double@widetilde}[2]{%
  \sbox\z@{$\m@th#1\widetilde{#2}$}%
  \ht\z@=.9\ht\z@
  \widetilde{\box\z@}%
}
\makeatother

\begin{document}
\title
[Filtrations on quantum cohomology]
{Filtrations on quantum cohomology\\ via Morse--Bott--Floer Spectral Sequences}
\author{Alexander F. Ritter}
\address{A. F. Ritter, 
Mathematical Institute, University of Oxford, 
OX2 6GG, U.K.}
\email{ritter@maths.ox.ac.uk}
\author{Filip \v{Z}ivanovi\'{c}}
\address{F. T. \v{Z}ivanovi\'{c}, 
Simons Center for Geometry and Physics, 
Stony Brook, NY 11794-3636, U.S.A.}
\email{fzivanovic@scgp.stonybrook.edu}

\begin{abstract} 
Using {\MBF } spectral sequences, we describe a filtration by ideals on quantum cohomology for symplectic manifolds with a Hamiltonian $S^1$-action that extends to a {\ph } $\C^*$-action.
These spaces include all Conical Symplectic Resolutions, in particular all Quiver Varieties. 
Our spectral sequences give explicit descriptions of birth-death phenomena of the barcode of the persistence module associated to the $\C^*$-action. 
This paper contains the foundational work to rigorously construct a filtration on Floer complexes from the $\C^*$-action, announced in our earlier paper. A substantial appendix on {\MBF} theory deals with several of the technical difficulties of the paper.
We compute a plethora of explicit examples, each highlighting various features, for Springer resolutions, ADE resolutions, and several Slodowy varieties of type A. We also consider certain Higgs moduli spaces, for which we compare our filtration with the famous $P=W$ filtration.
\end{abstract}

\maketitle
\setcounter{secnumdepth}{3}
\setcounter{tocdepth}{1}

\tableofcontents  %
\section{Overview of the results}\label{Introduction}

\subsection{The filtration on quantum cohomology of symplectic $\C^*$-manifolds}
We introduced in {\PartI} a large class of symplectic manifolds called symplectic $\C^*$-manifolds. This class includes many
equivariant resolutions of affine singularities, such as Conical Symplectic Resolutions (CSRs)\footnote{A weight-$s$ CSR is a projective $\C^*$-equivariant resolution $\pi:\M\to \M_0$ of a normal affine variety $\M_0$, whose $\C^*$-action contracts $\M_0$ to a point, and the holomorphic symplectic structure  $(\M,\om_\C)$ satisfies $t \cdot \omega_{\C}=t^s \omega_{\C}$.} for example Nakajima quiver varieties and hypertoric varieties,
and crepant resolutions of quotient singularities $\C^n/G$ for finite subgroups $G\subset SL(n,\C)$ (the spaces involved in the generalised McKay correspondence).
It also includes Moduli spaces of Higgs bundles, cotangent bundles of flag varieties,
negative complex vector bundles, and {\Kh} quotients of $\C^n$ including semiprojective 
toric manifolds.

These spaces are often {\hk}, so they can be studied also in terms of exact symplectic forms $\omega_J$, $\omega_K$ making them Liouville manifolds that often admit special exact $\omega_J$-Lagrangian submanifolds \cite{vzivanovic2022exact} relevant to the A-side of Homological Mirror Symmetry.
Our goal however is to study directly the B-side: we consider the complex structure $I$ with its typically non-exact {\Kh} form $\omega_I$. These spaces are rarely convex at infinity, so they have been out of reach of Floer-theoretic study so far; they are very rich in $I$-holomorphic curves, whereas in the {\hk} case the $\omega_J$, $\omega_K$ are exact and therefore prohibit non-constant closed $J$- or $K$-holomorphic curves. 

We will develop the {\MBF}
techniques needed to study that class of spaces. Such {\MB} methods were used in McLean--Ritter \cite{McLR18} to prove the cohomological McKay correspondence via Floer theory when $\C^n/G$ was an isolated singularity. The assumption ensured the crepant resolution was convex at infinity. The tools developed in this paper, for example, bypass that condition. 

A {\bf symplectic $\C^*$-manifold} (over a convex base) is a symplectic manifold $(Y,\omega)$, with a {\ph } $\C^*$-action $\Fi$ with respect to an $\omega$-compatible almost complex structure $I$, 
such that outside of a compact subset $Y^{\mathrm{in}}$, so on $Y^{\mathrm{out}}:=Y\setminus \mathrm{int}(Y^{\mathrm{in}})$, there is 
 a {\ph } proper map%
\begin{equation}\label{Equation intro Psi}
\Psi: Y^{\mathrm{out}} \to B=\Sigma \times [R_0,\infty)
\end{equation}
to the positive symplectisation of a closed contact manifold $\Sigma$. Moreover, the $S^1$-part of the $\C^*$-action $\Fi$ 
is Hamiltonian, and its vector field 
$X_{S^1}$ maps to the Reeb vector field via $\Psi$,
\begin{equation}\label{Equation psi Xs1 is Reeb}
    \Psi_* X_{S^1} = \mathcal{R}_B.
\end{equation}
There are no conditions on the dimension of $B$, and we emphasise that $\Psi$ is typically not symplectic. Indeed $\omega$ on $Y^{\mathrm{out}}$ is usually non-exact, and $I$-holomorphic spheres  often appear in the fibres of $\Psi$. We refer to \cite[Section 1]{RZ1} for a discussion of many classes of examples.

By definition, such spaces therefore come with a moment map for the $1$-periodic $S^1$-action,
\begin{equation}
\label{Equation intro moment map}
H: Y \to \R,
\end{equation}
yielding two commuting vector fields, $X_{\R_+}=\nabla H$ and $X_{S^1}=X_H=IX_{\R_+}$.
The action is automatically contracting: the $-\nabla H$ flow pushes $Y$ arbitrarily close to the core: a compact path-connected subset
$$
\mathrm{Core}(Y):=\{y\in Y: \lim_{t\to \infty}\varphi_t(y) \textrm{ exists}\},
$$
which is usually singular; the $\C^*$-fixed locus is a smooth subset of it, with connected components $\F_\a$,
\begin{equation}\label{Equation intro fixed locus}
\F = \mathrm{Crit}(H) = Y^{S^1} = Y^{\C^*} = \sqcup_\a \F_\alpha \subset \mathrm{Core}(Y).
\end{equation}
There is a unique minimal component $\F_{\min}:=\min H$. For weight-1 CSRs, $\F_{\min}$ is an $\omega_J$-Lagrangian submanifold \cite{vzivanovic2022exact}.\footnote{The core is an $\omega_J$-Lagrangian subvariety, but it is singular when the CSR is not the cotangent bundle of a flag variety.} All $\F_\a$ are $I$-holomorphic and $\omega$-symplectic submanifolds, and the linearised $S^1$-action on the complex vector space $T_y Y$, at $y\in \F_{\a}$, yields a weight decomposition
\begin{equation}\label{Intro weight spaces}
\qquad \qquad\qquad T_{y} Y = \oplus_{k\in \Z} H_k, \qquad \textrm{ where }H_0=T_y \F_{\a}.
\end{equation}
The $\F_\a$ are also  {\MB } submanifolds of the {\MB } function
 $H$. Their {\MB } indices $\mu_\a\geq 0$ are even and they determine the 
cohomology of $Y$ as a vector space (working over a field),\footnote{Here $A[d]$ means we shift a graded group $A$ down by $d$, so $(A[d])_n=A_{n+d}$, and $\check{H}^*$ is \v{C}ech cohomology.}
\begin{equation}\label{EqnFrankel}
H^*(Y)\cong \oplus H^*(\F_\a)[-\mu_\a] \cong \check{H}^*(\mathrm{Core}(Y)).
\end{equation}
The first isomorphism is originally due to Frankel \cite{Frankel59} for closed {\Kh} Hamiltonian $S^1$-manifolds. It can be proved by using Atiyah--Bott's filtration of the space $Y$ in terms of the stable manifolds of the $-\nabla H$ flow \cite{AtB83}. We reviewed this in {\PartI}, to motivate a Floer-theoretic generalisation of \eqref{EqnFrankel}, namely we constructed a $\Fi$-dependent filtration by ideals of the quantum cohomology algebra.

\begin{thm}\cite{RZ1}\label{Cor intro filtration}\label{FiltrationOnSingCohomology}\label{Definition fi-filtration on QH}
There is a filtration by graded ideals of $QH^*(Y)$, ordered by  $p\in \R\cup\{\infty\}$,\begin{equation}\label{DefinitionOfFiltration}
\FF^{\varphi}_p :=\bigcap_{\mathrm{generic}\,\lambda>p} \left(\ker c_{\lambda}^*:QH^*(Y)\to HF^*(H_{\lambda})\right), \qquad \FF_{\infty}^{\Fi}:=QH^*(Y),
\end{equation} 
where $c_\lambda^*$ is a Floer continuation map, a grading-preserving $QH^*(Y)$-module homomorphism, and $H_{\lambda}:Y\to \R$ is a Hamiltonian that at infinity equals $\lambda H+\textrm{constant}$.

The filtration $\Fil_p^{\varphi}(QH^*(Y))$
    is an invariant of  $Y$ up to isomorphism of symplectic $\C^*$-manifolds.\footnote{i.e. {\ph } $\C^*$-equivariant symplectomorphisms, without conditions on the $\Psi$ maps in \eqref{Equation intro Psi}.}
\noindent
    The real parameter $p$ is a part of the invariant; its geometric interpretation is: $x\in \Fil_p^{\varphi}(QH^*(Y)) \Leftrightarrow x$ 
    can be represented as the boundary of a Floer chain involving non-constant $S^1$-orbits of period $\leq p.$

    More generally, this idea is used in \cite{RZ1} to construct a persistence module $(V)_{p\in \R\cup \{\infty\}}$, with $V_0=QH^*(Y)$, giving rise to a barcode associated to the $\C^*$-action $\Fi$, satisfying a $1$-periodicity property.
\end{thm}

In \cref{Remark choice of coefficients} we mention the mild assumptions on $Y$ needed for quantum and Floer cohomology to be well-defined, and we comment on the field of coefficients $\k$. The Novikov field $\k$ is defined over a base field $\mathbb{B}$ and involves Laurent ``series'' in a formal variable $T$. We define the ``specialisation'' map ``$\val$'' to be the initial term of the series, meaning the coefficient of the smallest real power of $T$:
$$\FFF_{\lambda}^{\Fi}(H^*(Y;\mathbb{B}))=\FFF_{\lambda}^{\Fi}:=\val(\FF_{\lambda}^{\Fi}) \subset H^*(Y;\mathbb{B}).$$
This is a filtration by ideals 
with respect to ordinary cup product, satisfying
\begin{equation}\label{Equation intro FB filtration ranks}
    \mathrm{rank}_{\mathbb{B}} \FFF_{\lambda}^{\Fi}
    = 
    \mathrm{rank}_{\k} \FF_{\lambda}^{\Fi}.
\end{equation}
Thus there is a $\k$-linear isomorphism $\FFF_{\lambda}^{\Fi}\otimes_{\mathbb{B}}\k\to \FF_{\lambda}^{\Fi} \subset \k \cdot ( \FFF_{\lambda}^{\Fi} + H^*(Y;\k_{> 0})),$
where $\k_{>0}\subset \k$ denotes the ideal of elements involving only $T^{>0}$-terms.

\vspace{1mm}\indent
The main goal of this paper is to develop spectral sequence methods to effectively describe the filtration $\FF^{\Fi}_{\lambda}$ of $QH^*(Y)$, and $\FFF^{\Fi}_{\lambda}$ of $H^*(Y;\mathbb{B}).$
This is a Floer-theoretic analogue of the
{\MB} spectral sequence for Morse cohomology which can be used to prove the left isomorphism in \eqref{EqnFrankel}.

To give the reader a sense of the structure of the maps in \eqref{DefinitionOfFiltration}, we recall the main tool used in \cite{RZ1} to study the $\Fi$-filtration:
suppose $H^{odd}(Y)=0$, or equivalently, each $H^{odd}(\F_\a)=0$, due to \eqref{EqnFrankel}. Then for generic $\lambda>0$ we have $HF^{odd}(\lambda H)=0$, moreover 
the $\k$-linear continuation map becomes:
\begin{equation}\label{pureHam} 
\textstyle c_{\lambda}^*:QH^*(Y)\cong \bigoplus_{\a} H^*(\F_\a)[-\mu_\a]\longrightarrow \bigoplus_{\a} H^*(\F_\a)[-\mu_\lambda(\F_\a)] \cong HF^*(\lambda H),
\end{equation}
where the even integers $\mu_{\lambda}(\F_\a)$ are Floer-theoretic indices generalising the {\MB} indices $\mu_\a$.
For small $\lambda=\delta>0$, $\mu_\a=\mu_\delta(\F_\a)$, and $c_{\lambda}^*:QH^*(Y)\to HF^*(\lambda H)$ is an isomorphism, reproving the left isomorphism in \eqref{EqnFrankel}. For  larger slopes $\lambda$, analysing \eqref{pureHam} we obtained some constraints, such as
\begin{equation}\label{Cor intro about Fmin surviving}
\textstyle \FFF^{\Fi}_{\lambda}\subset \bigoplus_{\a \neq \min}\ H^*(\F_\a;\mathbb{B})[-\mu_\a] \;\;\textrm{ for }\;\;\lambda <\lambda_{\min}:=\min \{ \tfrac{1}{|k|}: k\neq 0 \textrm{ is a weight of }\F_{\min} \textrm{ in \eqref{Intro weight spaces}}\}.
\end{equation}
Beyond that, we really need the {\MBF} methods from this paper, as we confirm in examples in \cref{ExamplesSpSeq}.
We begin by providing the big picture, which relates quantum cohomology with Floer theory.
This involves symplectic cohomology $SH^*(Y,\Fi)$, whose chain level generators are loosely the $S^1$-orbits, and its positive version $SH^*_+(Y,\Fi)$, which ignores the constant $S^1$-orbits in $\F$.
The direct limit as $\lambda\to \infty$ in \eqref{DefinitionOfFiltration} defines a canonical algebra homomorphism $c^*:QH^*(Y)\to SH^*(Y,\Fi)$, satisfying: 

\begin{thm}\cite{RZ1}\footnote{The technical assumption, needed for Floer theory and to make sense of $QH^*(Y)$, is that $Y$ satisfies a certain weak+ monotonicity property (\cref{Remark choice of coefficients}). This includes for example all non-compact Calabi-Yau and non-compact Fano $Y$.}
\label{Prop vanishing of SH}
The homomorphism $c^*:QH^*(Y)\to SH^*(Y,\Fi)$ is surjective. It equals localisation at a  
Gromov--Witten invariant 
$Q_{\Fi}\in QH^{2\mu}(Y),$
where $\mu$ is the Maslov index of $\Fi$.
$$
SH^*(Y,\Fi) \cong QH^*(Y)/E_0(Q_{\Fi}) \cong QH^*(Y)_{Q_{\Fi}}
$$
where
$
E_0(Q_{\Fi})=\ker c^* \subset QH^*(Y)$ 
is the generalised $0$-eigenspace of quantum product by $Q_{\Fi}$.
As a $\k$-module, $SH^{*-1}_+(Y,\Fi)\cong E_0(Q_{\Fi})$, yielding a Floer-theoretic presentation of $QH^*(Y)$ as a $\k$-module,
\begin{equation}\label{Equation Introduction QH is SH+}
QH^*(Y)\cong SH^{*-1}_+(Y,\Fi) \oplus SH^*(Y,\Fi).
\end{equation}
Moreover, $c^*$ is the direct limit of continuation maps $c^*_{N^+}: QH^*(Y) \to HF^*(H_{N^+})$, for $N^+\to \infty\in \R$, and $c^*_{N^+}$
can be identified with quantum product $N$ times by $Q_{\Fi}$ on $QH^*(Y)$. In particular,
$$
SH^*(Y,\Fi)=0 \;\Longleftrightarrow  \;(\FF_{\lambda}^{\Fi}=QH^*(Y)\textrm{ for some }\lambda<\infty)
 \;\Longleftrightarrow \; 
(Q_{\Fi}\in QH^{2\mu}(Y)\textrm{ is nilpotent}),
 $$
which occurs if $c_1(Y)=0$ (e.g.\;CSRs, crepant resolutions of quotient singularities, Higgs moduli spaces).
\end{thm}

The above vanishing result is in fact beneficial in determining the $\Fi$-filtration by {\MBF} spectral sequences, as it implies that a posteriori all edge differentials must kill ordinary cohomology.

For CSRs, $QH^*(Y)\cong H^*(Y)$ is ordinary cohomology (with suitable coefficients)
so we obtain a $\Fi$-dependent filtration on $H^*(Y)$ by ideals with respect to cup-product, and $H^*(Y)\cong SH^{*-1}_+(Y,\Fi)$.

\subsection{{\MB} manifolds of $1$-orbits}\label{Subsection intro MB mfds of 1 orbits}
Via the $1$-periodic $S^1$-flow, letting $G_{\lambda}:=\langle e^{2\pi i \lambda }\rangle \subset \C^*$,
$$
 (1\textrm{-orbits of }\lambda H) \stackrel{1:1}{\longleftrightarrow}  (\lambda\textrm{-periodic orbits of the }S^1\textrm{-flow})  \stackrel{1:1}{\longleftrightarrow}(G_{\lambda}\textrm{-fixed points in }Y).
$$
We consider Hamiltonians $H_{\lambda}=c(H)$ that are increasing functions of the moment map $H$ of the $S^1$-action. Floer cohomology $HF^*(H_{\lambda})$ at chain level is generated by the $1$-orbits of $H_{\lambda}$.

At infinity, $H_{\lambda}$ has a {\bf generic} slope $c'(H)=\lambda>0$, meaning $\lambda$ is not an $S^1$-period (e.g.\;irrational $\lambda$), so no $1$-orbits appear at infinity. We ensure that $1$-orbits of $c(H)$ near $\mathrm{Core}(Y)$ are precisely the constant orbits at points of $\F$.
The interesting $1$-orbits are the non-constant ones: such $1$-orbits $x$ arise with a positive rational slope $T$, which equals some $S^1$-period,
\begin{equation}\label{Equation intro critical time} 
T=c'(H(x))=\tfrac{k}{m} \quad (\textrm{for }k,m\textrm{ coprime}).
\end{equation}
These $1$-orbits are trapped in the fixed locus $Y_m$ of the cyclic subgroup  $G_{T}\cong \Z/m$ of $\C^*$. That locus $Y_m:=Y^{\Z/m}$ is a $\C^*$-invariant $I$-{\ph} $\omega$-symplectic submanifold, whose connected components are the \textbf{torsion $m$-submanifolds} $\Ymc$.
Locally near $\F_\a$, $Y_m$ is $\exp|_{\F_\a}(\oplus_{b\in \Z} H_{mb})$, using  \eqref{Intro weight spaces}. Thus, there are at most $\#\{\F_\a\}$ such components $\Ymc$. Each $\Ymc$ contains a distinguished subcollection of the $\F_\a$, and has strata that converge to those $\F_\a$ under the contracting $\C^*$-action.
There are also only finitely many such integers $m$, as the weights in \eqref{Intro weight spaces} depend only on $\F_\a$, not on $y$.

The $\Ymc$ are {\MB} manifolds of $1/m$-periodic orbits of $H$. ``{\MB}'' here refers to a non-degeneracy property analogous to the classical {\MB} property of functions, but referred to a certain functional defined on free loops. We are however more interested in the {\MB} manifolds of $1$-orbits of $H_{\lambda}$. The motivation is the same as in the case of the geodesic flow on tangent bundles, where one wants to consider longer and longer closed geodesics. In our setting, using $H_{\lambda}=c(H)$ ensures that, as the slope $c'(H)$ grows, we are considering $S^1$-orbits of higher and higher periods.

Let us blur the distinction between a $1$-orbit $x(t)$ and its initial point $x(0)$. Then $1$-orbits of $H_{\lambda}$ of period $T=\tfrac{k}{m}$ are slices of the torsion $m$-submanifolds, obtained by intersecting with a level set of $H$,
\begin{equation}\label{Equation Bkm slices intro}
    B_{T,\beta}:=Y_{m,\c}\cap \{c'(H)=T\}.
\end{equation}
We choose $c''(H)>0$ when $c'(H)=T$: then $B_{T,\beta}$ is smooth and also a {\MB} property holds.
The $B_{T,\beta}$ are connected smooth odd-dimensional submanifolds
which we call {\bf {\MB } manifolds} (of $1$-orbits) of $H_{\lambda}$. 
Their diffeomorphism type only depends on $m$, not $k$.
The compact torsion manifolds $\Ymc\subset \mathrm{Core}(Y)$ do not contribute to \eqref{Equation Bkm slices intro}:\;the only $1$-orbits of $H_{\lambda}$ near $\mathrm{Core}(Y)$ are constant $1$-orbits in $\F.$ Only \emph{non-compact} $\Ymc$ arise in \eqref{Equation Bkm slices intro}, called {\bf outer torsion manifolds}; we call the $c'(H)=\tfrac{k}{m}$ in \eqref{Equation Bkm slices intro} {\bf outer $S^1$-periods}. To simplify notation, label outer $S^1$-periods in ascending order by $T_p\in \Q$, for $p=-1,-2,-3,\ldots$; abbreviate $B_{p,\beta}:=B_{T_p,\beta}$. These have Floer-theoretic indices $\mu(B_{p,\beta})$ that we compute in terms of weight spaces \eqref{Intro weight spaces}. We abusively use $p=0$, $T_0=0$, to label constant orbits in $\F$.

\subsection{Lyapunov property of $F$-filtrations}
Our motivation for requiring a map \eqref{Equation intro Psi}, in defining symplectic $\C^*$-manifolds, is to project the $Y^{\mathrm{out}}$-part of Floer solutions $u:\R \times S^1 \to Y$ for Hamiltonians $H_{\lambda},$ that on $Y^{\mathrm{out}}$ are of type $H_{\lambda}=c(H)$, to a Floer-like solution in $B$.

By \eqref{Equation psi Xs1 is Reeb}, the Floer equation $\partial_s u + I(\partial_t u - c'(H)X_{S^1})=0$ in $Y^{\mathrm{out}}$ via $\Psi$-projection to $B$ becomes
\begin{equation}\label{Equation projected Floer}
\partial_s v + I_B(\partial_t v - k(s,t)\mathcal{R}_B)=0, \qquad \textrm{where }k(s,t):=c'(H(u(s,t))).
\end{equation}
 We will construct a functional $F:\mathcal{L}Y \to \R$ on the free loop space of $Y$, satisfying the following {\bf Lyapunov property}: $F$ is non-increasing along the above Floer solutions $s\mapsto u(s,\cdot)$. 
This yields a filtration on the Floer chain complex $CF^*(H_{\lambda})$.
We build $H_{\lambda}$ so that this filtration is equivalent (up to reversal) to filtering by $S^1$-periods. In practice, this means that the Floer differential will be non-increasing on the period-value, in particular it must increase or preserve the $p$-labelling of the periods $T_p$. This phrasing has the advantage of showing that the filtration does not depend on the specific choices in the construction, so we call it {\bf period-filtration}.

In \cite{RZ1} we provided a simplified construction of the $F$-filtration, sufficient to construct for example $SH^*_+(Y,\Fi)$ which at chain level quotients out the subcomplex of constant orbits. For the purpose of constructing spectral sequences that are compatible under increasing the slope $\lambda$, the more complicated construction of $F$ from this paper is necessary. 
However, even in that simplified construction there is a catch: to obtain smooth moduli spaces of Floer solutions, either a perturbation of $I$ or a perturbation of $H_{\lambda}$ is required. The first perturbation ruins pseudoholomorphicity, $\Psi_* \circ I = I_B \circ \Psi_*$, so ruins the ``$I_B$'' in \eqref{Equation projected Floer}; the second perturbation ruins $\Psi_* X_{H_{\lambda}}\in \R\mathcal{R}_B$, so ruins the ``$k(s,t)\mathcal{R}_B$'' in \eqref{Equation projected Floer}. Both those properties were crucial for the proof of the Lyapunov property.

Overcoming these transversality issues is carried out in \cref{Section Transversality for Floer solutions}. It involves a delicate Gromov compactness argument; energy bracketing the Novikov ring; and carrying out direct limits over continuation maps as the perturbed $I'\to I$. In particular, the energy bracketing procedure was crucial to be able to apply Gromov-compactness in our setting where $\omega$ is not exact, as there are no a priori energy estimates. The ideas here were inspired by the work of Ono \cite{Ono} on filtered Floer cohomology.

\begin{thm}
We construct Hamiltonians $H_{\lambda}$, for a sequence of values $\lambda\to \infty$, so that the Floer chain complex $CF^*(H_{\lambda})$ is filtered by $S^1$-period values (the Floer differential is non-increasing on periods), and we construct continuation homomorphisms $\psi_{\lambda',\lambda}:CF^*(H_{\lambda})\to CF^*(H_{\lambda'})$ for $\lambda\leq \lambda'$, so that
\begin{enumerate}
    \item $\psi_{\lambda',\lambda}$ is an inclusion of a subcomplex;
    \item $\psi_{\lambda',\lambda}$ preserves the period-filtration.
\end{enumerate}
Any two such choices of sequences are related by period-filtration preserving continuation maps.
\end{thm}

\begin{cor}\label{Corollary period bracketed SH}
{\bf Period-bracketed symplectic cohomology}
$SH^*_{(a,b]}(Y,\Fi):=\varinjlim HF^*_{(a,b]}(H_{\lambda})$
is well-defined: at chain level one only considers $1$-orbits of $H_{\lambda}$ of $S^1$-period $T\leq b$, and quotients out those of period $T\leq a$. 
These groups admit a natural long exact sequence for $a<b<c$ in $ \R\cup \{\pm\infty\}$.
\end{cor}

\begin{rmk}
    Although \cref{Corollary period bracketed SH} will sound familiar to readers who encountered symplectic cohomology for Liouville manifolds \cite{Vi96}, where such a result is immediate from the Floer action functional,  we cannot overstate how fiendishly unfriendly Floer theory has been in complying with that result in our setup where $\omega$ is non-exact, because the action functional is not well-defined.
\end{rmk}

For $SH^*_+(Y,\Fi):=SH^*_{(0,\infty]}(Y,\Fi)$, noting $QH^*(Y)\cong SH^*_{(-\infty,\delta]}(Y,\Fi)$ for small $\delta>0$, we get the LES
\begin{equation}\label{LES for SHplus}
\cdots 
\to
SH^{*-1}(Y,\Fi)
\to 
SH^{*-1}_+(Y,\Fi)
\to 
QH^{*}(Y)
\stackrel{c^*}{\to}
SH^*(Y,\Fi) 
\to 
\ldots
\end{equation}
For example, if $c_1(Y)=0$ then $SH^*(Y,\Fi)=0$, thus $SH^{*-1}_+(Y,\Fi)\cong QH^*(Y)$.
More generally, taking $a<b<c$ to be $-\infty<\delta<\lambda$ gives an equivalent description of the $\Fi$-filtration:

\begin{thm}\label{Theorem intro Flambda is image of SHplus}
For any generic $\lambda>0$, the $\Fi$-filtration $\Fil^{\varphi}_\lambda$ of $QH^*(Y)$ from \eqref{DefinitionOfFiltration} satisfies
\begin{equation}
\label{Equation thm filtration as image of SHplus}
\Fil^{\varphi}_\lambda = \ker \,\left(c_{\lambda}^*:QH^*(Y) \to SH^{*}_{(-\infty,\lambda]}(Y,\Fi)\right) = \mathrm{Image}\,\left(SH^{*-1}_{(0,\lambda]}(Y,\Fi)\to QH^*(Y)\right),
\end{equation}
where the latter map is the connecting map in the relevant LES from \cref{Corollary period bracketed SH}, explicitly given by applying the unfiltered Floer differential, which counts so-called ``Floer spiked-discs''.\footnote{The spike is the Morse half-trajectory of an auxiliary Morse function $f: \F \to \R$ on the locus $\F=\sqcup \F_\a$ of constant orbits; the solution converges to a non-constant Hamiltonian orbit on the boundary of the ``disc'' (cf.\,\cite{PSS}).}
\end{thm}

\subsection{{\MBF} spectral sequence}
Having rigorously built the $F$-filtration on Floer chain complexes, by standard homological algebra arguments we obtain a spectral sequence that converges to $SH^*(Y,\Fi)$, and more generally to period-bracketed $SH^*_{\mathcal{P}}(Y,\Fi)$ for any period interval $\mathcal{P}\subset (-\infty,+\infty]$.

The heart of the matter now becomes identifying the $E_1$-page; discovering its properties such as symmetries and periodicities; and obtaining bounds for the degrees in which it is supported.

Our previously discussed period-labelling $T_p\in \Q$ by $p=0,-1,-2,\ldots$ corresponds to the $p$-index of the $E_1$-page $E_1^{p,q}$. So $E_1^{p,q}=0$ for $p>0$, and $E_1^{0,*}\cong H^*(Y)$ as the constant orbits recover the cohomology of the fixed locus with index shifts (this essentially underlies the {\MB} proof of \eqref{EqnFrankel}).

More interestingly, for $p<0$, the $E_1^{p,*}$ column is computing the Floer cohomology for $S^1$-orbits that have the same period $T_p$, i.e.\,they lie in the same slice $c'(H)=T_p$. In traditional Floer-theoretic settings, this is usually harmless, as everything reduces to Morse cohomology: a zero action functional difference implies Floer solutions have zero energy, so -- as we are working in a {\MB} setting -- the only solutions are those arising from the Morse cohomology of a chosen auxiliary Morse function
\begin{equation}\label{Equation auxiliary morse intro}
f_{p,\c}: B_{p,\c}\to \R.
\end{equation}
In our case, unfortunately, it is a hornets' nest: high-energy Floer solutions could arise, as there are no a priori energy bounds, and in fact the fibres of \eqref{Equation intro Psi} are typically rich in $I$-holomorphic spheres.

\begin{thm}\label{Prop local HF intro}
Local Floer cohomology of the slice $c'(H)=T_p$ is well-defined, and admits a convergent energy-spectral sequence starting with Morse cohomology of \eqref{Equation auxiliary morse intro}, 
\begin{equation}
    \label{Equation energy sp seq in intro}
\oplus_{\beta} MH^*(B_{p,\c};f_{p,\c};\mathcal{L})[-\mu(B_{p,\c})] \Rightarrow \oplus_{\beta} H_{p,\c}^*[-\mu(B_{p,\c})],
\end{equation}
 where $\mathcal{L}$ is a suitable Floer-theoretic local system supported near the Morse trajectories, and $\mu(B_{p,\c})$ are Floer-theoretic indies that are explicitly computable in terms of the weight decomposition \eqref{Intro weight spaces}.
 
When $c_1(Y)=0$, the local system $\mathcal{L}$ becomes trivial, so the left side of \eqref{Equation energy sp seq in intro} is $H^*(B_{p,\c})[-\mu(B_{p,\c})]$.
\end{thm}
\begin{rmk}
The study of local Floer cohomology, when  {\MB} manifolds are circles, goes back to {\CFHW } \cite{CFHW}. 
{\KvK} \cite{KwonvanKoert} extended that construction to Liouville manifolds whose {\MB} manifolds satisfy a ``symplectic triviality'' condition. That triviality condition unfortunately rarely ever holds for symplectic $\C^*$-manifolds. This is the main reason for Appendices \ref{AppendixCascades2}-\ref{Section {\MBF } theory: orientations}, in which we work in full generality.
\end{rmk}

On first reading, it may be easier to ignore the difference between $H^*(B_{p,\c})[-\mu(B_{p,\c})]$ and the limit group $H_{p,\c}^*[-\mu(B_{p,\c})]$ in \eqref{Equation energy sp seq in intro}.
Although the latter needs to be placed in our $E_1$-page, in examples we abusively place $H^*(B_{p,\c})[-\mu(B_{p,\c})]$ in the columns. We describe circumstances where this is legitimate. In general, one must be cautious in our spectral sequence tables that some residual vertical differentials of the $E_0$-page may still need to be accounted for on the $E_1$-page.\footnote{There are two local Floer cohomologies in play: the one occurring in a neighbourhood of $B_{p,\c}$, and the one occurring in a neighbourhood of the slice $c'(H)=T_p$. Low-energy Floer solutions will be local to a single $B_{p,\c}$, but we cannot exclude high-energy Floer solutions, either local to a single $B_{p,\c}$, or travelling between different $B_{p,\c}$ with equal $T_p$-values.} 
The rationale for the abuse of notation is that, a posteriori, one often deduces that no such vertical differentials occur, as all ranks need to be accounted for if $E_*^{p,q}$ is to converge correctly (e.g.\;if $c_1(Y)=0$, we know $SH^*(Y,\Fi)=0$, so all edge-differentials of $E_1$ and later pages must eventually ``cancel out'' all cohomology).

\begin{thm}\label{Theorem intro spectral seq}
Let $Y$ be a symplectic $\C^*$-manifold. 
For any period interval $\mathcal{P}\subset (-\infty,+\infty]$, and working over a Novikov field $\k$,
there is a convergent spectral sequence of $\k$-modules,
	$$
 E_{r}^{pq}(\Fi;\mathcal{P}) \Rightarrow SH_{\mathcal{P}}^*(Y,\Fi), \textrm{   where } E_{1}^{pq}(\Fi;\mathcal{P})=  \begin{cases*}
				H^q(Y) & if $p=0$\textrm{ and }$0 \in \mathcal{P}$\\
				\bigoplus_{\c} H_{p,\c}^*[-\mu(B_{p,\c})] & if $p<0, \textrm{ and } T_p \in \mathcal{P}$   \\
				0 & otherwise.
			\end{cases*}
   $$
\end{thm}

\noindent {\bf Example.}\;For $\mathcal{P}=(0,\infty]$ we use columns $E_1^{p,*}$ for $p<0$, and it converges to $SH^*_+(Y,\Fi)$, see \eqref{LES for SHplus}. 
Edge-differentials hitting the $0$-th column $E_1^{0,*}=H^*(Y)$ arise from ``Floer spiked-discs'' (\cref{Theorem intro Flambda is image of SHplus}).

\begin{rmk}
The proof takes up \cref{SpecSeqMBF} and relies on substantial Appendices \ref{AppendixCascades}-\ref{Section {\MBF } theory: orientations} that construct {\bf {\MB} Floer cohomology}.
In the context of ordinary Morse theory, this is analogous to the cascade model for a {\MB} function developed by Frauenfelder \cite{Frauenfelder-cascades}. When $Y$ is a Liouville manifold and the {\MB} manifolds are circles and points, the {\MBF} model was constructed by Bourgeois--Oancea \cite{Bourgeois-Oancea2}.
Our Appendices construct it in full generality, and we cover folklore properties of {\MB} Floer cohomology that do not appear to be written up elsewhere.

 Our {\MBF} spectral sequences also work when $c_1(Y)\neq 0$, but computing local Floer cohomologies $H_{p,\c}^*[-\mu(B_{p,\c})]$ of {\MB} manifolds $B_{p,\c}$ becomes hard; also care is needed as Floer cohomology may not be $\Z$-graded. For $Y$ monotone/Fano, it is $\Z$-graded if one suitably grades the formal Novikov variable $T$ of the Novikov field $\k$ (cf.\;\cref{Example intro neg vb}). In general, it is just $\Z/2$-graded.
\end{rmk}

\subsection{Deducing the $\Fi$-filtration from {\MBF} spectral sequences}
Although Floer invariants are notoriously difficult to compute, in \cref{SpecSeqMBF-Properties} we prove effective properties of the spectral sequence. \cref{ExamplesSpSeq} includes a plethora of spectral sequence tables, that illustrate various features we discovered.
This introduction will only illustrate the simplest example.
By convention, we abusively draw on the $E_1$-page
also edge-differentials occurring on later pages: some caution is needed as the arrows all together do not constitute a differential; rather they are {\bf iterated differentials} (the $E_{r+1}$ page is the cohomology of the $E_r$ page using only arrows of a specific length depending on $r$).

\begin{prop}\label{Prop intro how to read filtration}
A class $x$ in the $\k$-module 
$H^*(Y)=QH^*(Y)$ lies in $\FF^{\Fi}_{\lambda}$ if the columns of the $E_1$-page given by $\oplus_{T_p\leq \lambda}\oplus_{\beta} H_{p,\c}^*[-\mu(B_{p,\c})]$ kill $x\in E_1^{0,*}=H^*(Y)$ via iterated differentials. 
\end{prop}

\begin{cor}[Stability property]\label{Prop filtration is stable}
$\Fil^{\varphi}_\lambda  = \Fil^{\varphi}_{\lambda'}$ if there are no outer $S^1$-periods $T_p\in (\lambda,\lambda'].$
\end{cor}

This is immediate from \cref{Prop intro how to read filtration}: no new columns appear as we increase the slope from $\lambda^+$ to $(\lambda')^+$, so no new classes in $H^*(Y)$ got killed. This result could not be proved by \cite{RZ1}-methods. 

\begin{ex}[$A_2$-singularity]\label{Example running example of intro 2}
Consider the minimal resolution
\footnote{$\Z/3$ acts by third roots of unity $\zeta$: $(z_1,z_2)\mapsto (\zeta z_1,\zeta^{-1}z_2)$ on $\C^2$. We embed $\C^2/(\Z/3)\hookrightarrow \C^3$, $[z_1,z_2]\mapsto (z_1^3,z_2^3,z_1 z_2)$, then $M$ is the blow up at $0$ of the image variety given by the vanishing locus $V(XY-Z^3)\subset \C^3$.} $\pi:M \fun \C^2/(\Z/3)$ of $\C^2/(\Z/3)\cong V(XY-Z^3)\subset \C^3$.
The $\Core(M)=\pi^{-1}(0)=S^2_1 \cup S^2_2$ consists of two copies of $S^2$ intersecting transversely at a point $p$. Consider two  $\C^*$-actions, by lifting the following actions\footnote{
(a) is a weight-$2$ CSR (classical McKay action); (b) is a weight-$1$ CSR; %
$\F_{\min}$ is a minimal special exact $\omega_J$-Lagrangian \cite{vzivanovic2022exact}.
 The square $(t^2X,t^4Y,t^2Z)$ of (b) is \cref{ExampleZ3Exotic} (weight-$2$ CSR). Squaring just rescales periods in the picture.
 In \eqref{Intro weight spaces} we get $\dim=1$ summands of weights: $(1,1)$ at $\F_{\min}=p$, $(3,-1)$ at $p_i$ for (a); $(0,1)$ at $\F_{\min}$, $(2,-1)$ at $p_i$, for (b).} 
on $V(XY-Z^3)\subset \C^3$:\\
{\small
\strut\quad \qquad (a) %
$(t^3 X, t^3Y, t^2Z)$; \qquad then: \quad $\F=p_1 \sqcup p \sqcup p_2$ (3 points), where $p=\F_{\min}$ and $p_i\in S^2_i$.\\
\indent 
\strut\quad \qquad 
\; \; The map \eqref{Equation intro Psi} is $\Psi=(X^2,Y^2,Z^3): M \to \C^3$, for the weight $6$ diagonal $\C^*$-action on $\C^3$.
\\
\strut\quad \qquad (b) %
$(tX,t^2Y,tZ)$; \qquad\quad then: \quad  $\F=S_1^2 \sqcup p_2$, where $S_1^2=\F_{\min}$.
\\
\indent 
\strut\quad \qquad 
\; \; The map \eqref{Equation intro Psi} is $\Psi=(X^2,Y,Z^2): M \to \C^3$, for the weight $2$ diagonal $\C^*$-action on $\C^3$.
}
\\
Each dot in the $E_1$-page tables indicates one copy of the Novikov field $\k$ of characteristic zero. 
The arrows indicate the possible edge differentials of $E_1$ and later pages. 
\\[1mm]
\begin{minipage}{0.585\textwidth}
\includegraphics[height=20ex,width=60ex]{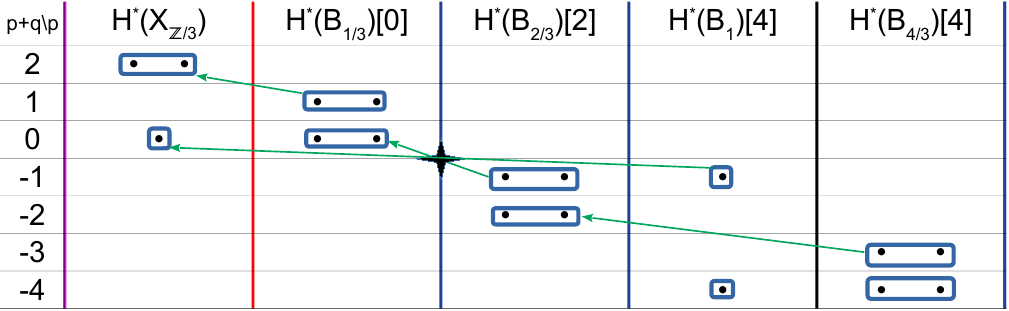}\\[2mm]
\includegraphics[height=20ex,width=60ex]{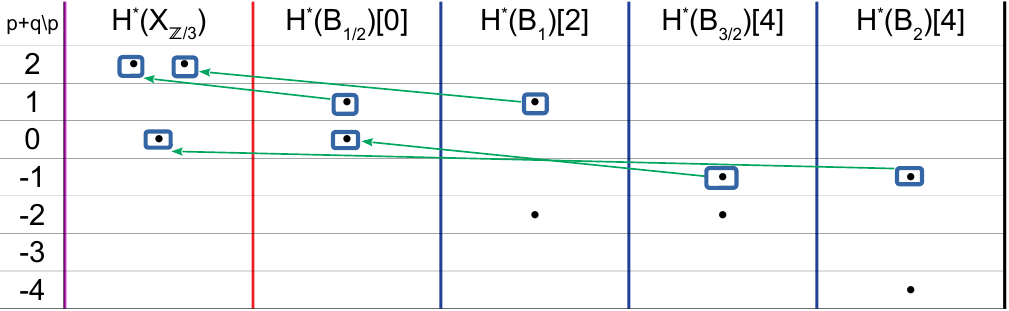}
 \end{minipage}\;
 \begin{minipage}[m]{0.405\textwidth}
\vspace{0mm}
(a) There are $\Z/3$-torsion lines\\ $(*,0,0)\cong \C$ and $(0,*,0)\cong \C$,\\ hitting the two spheres at $p_1,p_2$.\\[1mm]
Slicing each line yields a\\ copy of $B_{\textrm{non-integer}}\cong S^1$.\\ Also,
$B_{\textrm{integer}}\cong S^3/(\Z/3)$.\\[1mm] $\Rightarrow$ Filtration: $0\subset H^2(M )\subset H^*(M ).$\\[1mm]
(b) There is a $\Z/2$-torsion line\\ $(0,*,0)\cong \C$
hitting $S^2_2$ at $p_2$.\\[1mm] Slicing yields $B_{\textrm{non-integer}}\cong S^1$ and \\ $B_{\textrm{integer}}\cong S^3/(\Z/3)$.\\[1mm] $\Rightarrow$ Filtration:\vspace{-1mm} $$0\subset V \subset H^2(M ) \subset H^*(M ),\vspace{-1mm}$$ for a $1$-dimensional subspace $V\!\subset\! H^2(M )$.
 \end{minipage}\\[1mm]
 In general, by \eqref{Equation intro FB filtration ranks}, our rank-wise discussion for $\FF_{\lambda}^{\Fi}$ yields analogous results for the filtration $\FFF_{\lambda}^{\Fi}$.
\end{ex}

\begin{rmk}
In {\PartI} we computed the filtrations $\FF^{\Fi}_{\lambda}$ directly from continuation maps $c^*_{\lambda}$ in \eqref{DefinitionOfFiltration}. This approach requires knowing information about fixed components $\F_\a$ and their indices $\mu_{\lambda}(\F_\a)$, which are computed from the weights in \eqref{Intro weight spaces}.
The spectral sequence method instead uses information about torsion manifolds $\Ymc$, which yield the slices $B_{p,\c}$, and their indices $\mu(B_{p,\c})$ again computable from the weights \eqref{Intro weight spaces}.
Although we were able to 
reprove with the continuation method many results we had discovered using {\MBF} techniques, the spectral sequence method has proven to be much more practical to study the $\Fi$-filtration (e.g.\;compare \cref{Example running example of intro 2} versus \cite[Ex.1.31,1.41]{RZ1}).
\end{rmk}

\subsection{Properties of the $E_1$-page of the spectral sequence}
First, we relate the indices $\mu(B_{p,\c})$ with the indices $\mu_{\lambda}(\F_\a)$ from \cite{RZ1}. Let $|V|:=\dim_{\C}V,$ and let $\mu$ be the Maslov index of the $S^1$-action.  
\begin{prop}\label{Intro Maslov indeces comparison}
Let $T_p^{\pm}$ be generic values just above and %
    below $T_p$. Abbreviate $\F_\a:=\min (H|_{\Ymc})$ and $\mathrm{rk}(Y_{m,\c}):=|\Ymc| - |\F_\a|$. Then $H^*(B_{p,\c})[-\mu(B_{p,\c})]$ is supported precisely between the values
\begin{align*}
\mu(B_{p,\c}) &= \mu_{T_p}(\F_\a) - \mathrm{rk}(Y_{m,\c})=\mu_{T_p^+}(\F_\a),
\\
\mu(B_{p,\c})+\dim_{\R}B_{p,\c} &= 
\mu_{T_p}(\F_\a) + |Y_{m,\c}| + |\F_\a| -1 = \mu_{T_p^-}(\F_\a) + 2|\F_\a| - 1.
\end{align*}
\end{prop}

We provide intuition for these gradings:\;it is related to a {\bf birth-death process} in the barcode and persistence module from \cite[Sec.1.6, Sec.1.13]{RZ1}. Note the difference with the approach in \cite{RZ1}: \eqref{pureHam} is computed at generic times $\lambda$, whereas $B_{p,\beta}$ arises at critical times $T_p$ in \eqref{Equation intro critical time}.
Assume $H^*(Y)$ lies in even degrees (e.g.\;all CSRs).
Columns below time $T_p$ are a spectral sequence converging to $$HF^*(T_p^-H)\cong \oplus H^*(\F_\a)[-\mu_{T_p^-}(\F_\a)].$$ 
If we include the time-$T_p$ column, it converges to $HF^*(T_p^+H)\cong \oplus H^*(\F_\a)[-\mu_{T_p^+}(\F_\a)]$. \cref{Intro Maslov indeces comparison} suggests that top classes of $H^*(B_{p,\beta})[-\mu(B_{p,\beta})]$ are killing (via edge differentials, of total degree $1$) the $H^*(\F_\a)[-\mu_{T_p^-}(\F_\a)]$ summand; the bottom classes are creating the $H^*(\F_\a)[-\mu_{T_p^+}(\F_\a)]$ summand. To see this play out in a simple example, compare \cref{Example running example of intro 2} with %
\cite[Ex.1.31,1.41]{RZ1}.
The intuition in the simplified example above is in fact a consequence of the following more general observation.

\begin{prop}\label{Prop intro adding a column to sp seq}
    Replacing $H_{\lambda}=c(H)$ by $H_{\lambda,\lambda_0}:=c(H)+\lambda_0 H$, for any $\lambda_0>0$, yields a spectral sequence converging to $HF^*(H_{\lambda+\lambda_0})$ whose $E_1$-page has $0$-th column $HF^*(\lambda_0 H)$ and $p$-th columns $\bigoplus_{\c} H_{p,\c}^*[-\mu(B_{p,\c})]$ for $p$ satisfying $T_p \in \mathcal{P}:=(\lambda_0,\lambda+\lambda_0)$.\footnote{Another viewpoint: we have a long exact sequence $\cdots \to HF^*(\lambda_0 H) \to HF^*((\lambda+\lambda_0)H) \to HF^*_{\mathcal{P}}(H_{\lambda,\lambda_0}) \to \cdots$, and if we only consider the columns with $T_p \in \mathcal{P}:=(\lambda_0,\lambda+\lambda_0)$ then the spectral sequence converges to $HF^*_{\mathcal{P}}(H_{\lambda,\lambda_0})$.}
    So, given $p$, taking $\lambda_0:=T_p^-$, $\lambda+\lambda_0:=T_p^+$, the $E_1$-page has two columns, $HF^*(T_p^- H)$ and $\bigoplus_{\c} H_{p,\c}^*[-\mu(B_{p,\c})]$, and converges to $HF^*(T_p^+ H)$.
\end{prop}

Using \cref{Intro Maslov indeces comparison}, we deduce two symmetry properties for the spectral sequence:
\begin{cor}\label{Cor symmetries intro} There is a translation symmetry: $T_p \to T_{p}+1$ shifts rows down by $2\mu$:
$$\mu(B_{p,\c})=\mu(B_{p',\c})-2N\mu \qquad \textrm{ if }\;T_{p}=T_{p'}+N, \textrm{ any }N\in \N.$$
When $\char\ \k=0$, for $T_p\in (0,1)$ there is a reflection symmetry about the point with (column,row) ``coordinates'' $(\frac{1}{2}, |Y| -\frac{1}{2}-\mu)$ (e.g.\;the ``star-point'' $(\tfrac{1}{2},-\tfrac{1}{2})$ in \cref{Example running example of intro 2}). Explicitly, for $T_{p'}:=1-T_p$,
$$
\biggl(H^*(B_{p',\beta})[-\mu_{T_{p'}}(\F_\a)]\biggr)_{d }
\cong 
\biggl(H^*(B_{p,\beta})[-\mu_{T_p}(\F_\a)]\biggr)_{ 2|Y| - 1 - 2\mu -d.}
$$
\end{cor}

This Corollary implies that the $E_1$-page of the spectral sequence is known once we find the columns for $0<T_p\leq \tfrac{1}{2}$ and for $T_p=1$.
Combining \cref{Cor symmetries intro} and \cref{Intro Maslov indeces comparison}, we obtain bounds 
for the support
of the $E_1$-page 
(\cref{Lemma block description spectral sequence}). This implies 
by which page 
the filtration $\FF^{\Fi}_{\lambda}$ is determined:

\begin{cor}
If $\mu_{\lambda}(\F_\a)\leq \mu_\a$ for all $\lambda$ (e.g.\;any CSR), then $H^*(Y)$ can only be hit by columns with 
$$T_p \leq 1+\tfrac{t}{2\mu}\; \ \textrm{ if }\ T_p\notin \N,  \textrm{ and } T_p \leq \tfrac{|Y|}{\mu}\ \textrm{ if }\ T_p\in \N, \quad \textrm{where }t:=\max\{ m\in \N: H^m(Y)\neq 0\}\leq 2|Y|-2.$$
\end{cor}

The $T_p=1$ column (and all integer-time columns, i.e.\,slices $\Spp_N:=\{c'(H)=N\}$) is often understood: 

\begin{prop} 
Let $\Spp$ be the slice $\{c'(H)=1\}$.
For $\char\ \k=0$, suppose $H^*(Y)$ is supported in degrees $[0,n=\dim_\C Y]$ %
(e.g.\;all CSRs and Higgs moduli spaces).
Let $K$ be the kernel of the intersection form $H_{n}(Y) \times H_{n}(Y) \fun \k$ (e.g.\;$K \iso H^n(Y)$ if the form vanishes; $K=0$ if the form is non-degenerate).
		\[H^k(\Spp)\iso  \begin{cases*}
			H^{2n-1-k}(Y) & \ $k\geq n+1,$\\
			H^k(Y) & \ $k\leq n-2$,
		\end{cases*}
        \quad\textrm{ and }\quad
        H^{n}(\Spp) \iso H^{n-1}(\Spp) \iso  
            H^{n-1}(Y)\oplus K.
		\]
        
For weight-$s$ CSRs: $H^k(\Spp)=0$ in mid-degrees $k=n,n-1$; it lies in even degrees for $k\leq n-2$ and odd degrees for $k\geq n+1$; thus \eqref{Equation energy sp seq in intro} collapses for $T_p=N\in \N$ giving $H_{p}^*[-\mu(B_{p})] \cong H^*(\Spp)[snN]$.%
\end{prop}

Suppose $H^*(Y)$ lies in even degrees (e.g.\;all CSRs), equivalently all $H^*(\F_\a)$ lie in even degrees. Then also $H^*(Y_{m,\beta})$ lies in even degrees. Letting $y_{m,\beta}=\dim_{\R} Y_{m,\beta}$, via \eqref{Equation Bkm slices intro} we show that for $q$ even\footnote{equivalently:\;the kernel of the natural map 
$H_q(Y_{m,\beta})\to H_q^{l\!f}(Y_{m,\beta})\cong H^{-q+y_{m,\beta}}(Y_{m,\beta})$.}
$$
H_q(B_{T,\beta}) \cong
(\textrm{Kernel of the intersection form } H_q(Y_{m,\beta}) \otimes H_{-q+y_{m,\beta}}(Y_{m,\beta}) \to \k),
$$
which by Poincar\'{e} duality determines $H_q(B_{T,\beta})\cong H_{-1-q+y_{m,\beta}}(B_{T,\beta})$ for odd $q$.
Moreover, we have:

\begin{prop}
Suppose $H^*(Y$) lies in even degrees.
From the $E_1$-page onward, the spectral sequence splits horizontally into a direct sum of two-row spectral sequences. 
\end{prop}

Combining this with the fact that the right-hand side of \eqref{pureHam} is in even degrees yields:

\begin{cor}\label{Cor intro inj on odd classes}
Suppose $H^*(Y$)  lies in even degrees.
For any given odd degree class $x$ of the $E_1$-page, some edge-differential is eventually injective on $x$, so kills a class in a column strictly to the left of $x$.
\end{cor}

Using \cref{Prop intro adding a column to sp seq}, we can strengthen a result about \eqref{pureHam} obtained in \cite{RZ1}. Suppose $H^*(Y)$ lies in even degrees.
Call $\F_\a$ {\bf $p$-stable} if $\tfrac{1}{p}\Z_{\neq 0}\cap ($weights($\F_\a))=\emptyset$.
When considering an interval $[\lambda,\gamma]$, we call  
$\F_{\a}$ {\bf stable} if 
$[\lambda,\gamma]$ does not contain any $k/|m|\in \Q$ where $m\neq 0$ is a weight of $\F_{\a}$.
Otherwise, call $\F_{\a}$ unstable.
For stable $\F_\a$, the indices $\mu_{\lambda}(\F_\a)=\mu_{\gamma}(\F_\a)$ agree; for unstable $\F_\a$ they could differ.

In \cite{RZ1}, we showed that for stable $\F_\a$, if $\gamma-\lambda$ is sufficiently small, the part  of the continuation map 
$\psi_{\gamma,\lambda}:HF^*(\lambda H)\to HF^*(\gamma H)$
given by
$H^*(\F_\a)[-\mu_{\lambda}(\F_\a)]
\to \oplus_{\beta} H^*(\F_{\beta})[-\mu_{\gamma}(\F_{\beta})]$
equals\footnote{working with bases coming from $H^*(\F_{\a};\mathbb{B})$ and $H^*(\F_{\beta};\mathbb{B})$, and letting $\mathrm{id}_{\alpha}$ be the identity map on $H^*(\F_\a)$-classes.}
\begin{equation}\label{Equation continuation estimate map id plus higher}
T^{(\gamma-\lambda)\cdot H(\F_{\a})}\cdot \left(\mathrm{id}_{\alpha}+T^{>0}\textrm{--terms}\right),
\textrm{ thus: }\;\mathrm{ini}(\ker \psi_{\gamma,\lambda})\subset \oplus_{\beta\neq \a}H^*(\F_{\beta};\mathbb{B})[-\mu_{\lambda}(\F_{\beta})].
\end{equation}
Now consider the continuation map in the final claim of \cref{Prop intro adding a column to sp seq}, using \eqref{pureHam},
\begin{equation}\label{Equation intro HFTpminus to HFTpplus}
HF^*(T_p^- H)\cong \oplus_{\a} H^*(\F_{\a})[-\mu_{T_p^-}(\F_\a)]
\to
\oplus_{\a} H^*(\F_{\a})[-\mu_{T_p^+}(\F_\a)]
\cong HF^*(T_p^+ H).
\end{equation}
\begin{cor}
Suppose $H^*(Y)$ lies in even degrees.
In the final claim of \cref{Prop intro adding a column to sp seq}, let $\mathcal{C}^*$ be the classes in the $E_1$-page's column $HF^*(T_p^- H)$
killed\footnote{by the edge differentials $d_r$ for $r\geq 1$. The continuation map \eqref{Equation intro HFTpminus to HFTpplus} is an inclusion at chain level, and corresponds to including the first column into the two-column $E_1$-page. So $\mathcal{C}^*$ represents the classes that die under \eqref{Equation intro HFTpminus to HFTpplus}.} by the new column
$\bigoplus_{\c} H_{p,\c}^*[-\mu(B_{p,\c})]$.
Then
\begin{equation}\label{Equation iniC introduction}
\val(\mathcal{C}^*) \subset 
\oplus\{ H^*(\F_{\a};\mathbb{B})[-\mu_{T_p^-}(\F_\a)]:\;\F_{\a}\textrm{ is }T_p\textrm{-unstable}\},
\end{equation}
so for $q$ even: $\mathrm{rank}_{\k}\,\left(\bigoplus_{\c} H_{p,\c}^{*}[-\mu(B_{p,\c})]\right)_{q-1}\!\! =\, \mathrm{rank}_{\k}\,\mathcal{C}^q \,=\, \mathrm{rank}_{\mathbb{B}}\,\mathrm{ini}(\mathcal{C}^q) \,\leq \,
\mathrm{rank}_{\mathbb{B}}($RHS in \eqref{Equation iniC introduction}$)_q$.
\end{cor}

\subsection{Illustration of the results in examples}

\begin{ex}[CSRs] 
For all CSRs, $c_1(Y)=0$,\footnote{Since they are holomorphic symplectic.}
 $QH^*(Y)=H^*(Y)$ as rings\footnote{Since by \cite{Nam08}, $(Y,I)$ can be deformed to an affine algebraic variety.
} 
and lies in even degrees in $[0,n]$, $n:=\dim_{\C}Y$. The Maslov index satisfies $2\mu=sn$, where $s$ is the weight of the CSR (in most examples $s=1$ or $2$). 
For weight-1 CSRs, by computing the gradings 
in \cref{Intro Maslov indeces comparison}
we show:
$$
\FF^{\Fi}_{\lambda<1}\subsetneq H^{\geq 2}(Y),
\qquad\quad
\FF^{\Fi}_{1} = H^{\geq 2}(Y),
\qquad\quad \FF^{\Fi}_{2}=H^*(Y),$$
and $\mathrm{codim}_{\k}(\FF^{\Fi}_{1^-}\cap H^n(Y)\subset H^n(Y)) = 1$ (the top class of $H^*(B_1)[n]$ will kill $H^{n}(Y)$, \cref{Prop using gradings to prove QFi result}).
\end{ex}

\begin{ex}[$T^*\C\P^m$ and Springer resolutions]\label{Intro Spr Resln}
Let $Y=T^*\C\P^m$, viewed as a CSR, with the standard $\C^*$-action on fibers.
The {\MB} manifolds $B_N$ of $1$-orbits  %
 for full rotations by $N=1,2,\ldots$ are sphere subbundles $ST^*\C\P^m$ with cohomology 
shifted down by $2N\mu=2Nm$ (cf.\,\cref{ExamplesSpSeqSprRes}, \cref{SS_S11}, for $m=1$);
$H^*(ST^*\C\P^m)$ lies in even degrees $0,2,\ldots,2m-2$ and odd degrees $2m+1,\ldots, 4m-1$, and is free of rank 1 there. In the spectral sequence, odd classes of $B_1$ kill $H^*(Y)\cong H^*(\C\P^m)=\k[x]/\la x^{m+1}\ra $ in degrees $2,4,\ldots,2m$; the top odd class of $B_2$ kills $H^0(Y)$. The filtration is: 
\begin{equation*}%
    0 \quad \subset  \quad \FF^{\Fi}_1= H^{*\geq 2}(Y) = \la x \ra \quad \subset  \quad  \FF^{\Fi}_2= H^*(Y).
\end{equation*}

The same filtration holds for Springer resolutions, i.e.\;cotangent bundles of flag varieties (\cref{ExamplesSpSeqSprRes}).

For twisted actions, so $\C^*$ acting non-trivially on the zero section, the filtration gets refined. E.g.\;$\C^*$ acting on $\C\P^m$ by $[z_0,tz_1,\ldots,t^{m-1}z_{m-1}]$ gives the filtration by cohomological grading (\cref{Example Tstar of flag variety with a twisted action}).
\end{ex}
\begin{ex}[Negative vector bundles]\label{Example intro neg vb} 
For a general negative vector bundle $$Y=\mathrm{Tot}(E\to \C\P^m)$$ %
of rank $r=\mathrm{rk}_{\C}(E)$, 
with the standard \CC-action $\Fi$ on fibres, %
then $B_N$ are $(2m+2r-1)$-dimensional sphere subbundles $SE$ with cohomology shifted down by $2N\mu=2Nr$;
$H^*(B_N)$ has 
even classes in degrees $0,2,\ldots,2r-2$
and
odd classes in degrees $2m+1,\ldots,2m+2r-1$.
Although $c_1(E)\neq 0$, it can happen that $c_1(Y)=0$ 
(e.g.\ the line bundle $\mathcal{O}(-1-m)\to \C\P^m$).

Let us suppose $c_1(Y)=0$.
In the spectral sequence, the odd classes of $B_1$ kill the classes of $H^*(E)\cong H^*(\C \P^m)$ in degrees $2m+2-2r,\ldots,2m$, the odd classes of $B_2$ kill the next $2r$-block of classes in degrees $2m+2-4r, \ldots, 2m-2r$, etc. The unit will be killed by $B_N$ for $N=1+\lfloor \tfrac{m}{r} \rfloor$. This is consistent with \cref{Prop vanishing of SH}: by \cite{R14}, $Q_{\Fi}$ is a non-zero multiple of the pull-back of the Euler class of $E$, in degree $2\mu=2r$, so its $N$-th quantum-product power is zero for degree reasons.

When $c_1(Y)\neq 0$, Floer theory is not $\Z$-graded, and \cite{R14,R16} illustrates many situations where $SH^*(Y)\neq 0$ in the regime where $Y$ is monotone. Let us consider the line bundle 
$$Y=\text{Tot}(\mathcal{O}(-k)\to \C\P^m) \textrm{ in the monotone regime }1\leq k \leq m.
$$ 
Let $x=\pi^*\omega_{\C\P^m}$.
Then $c_1(Y)=(1+m-k)x\neq 0$, and we place the Novikov parameter $T$ in grading $|T|=2(1+m-k)$ to get a $\Z$-grading. Work over the Novikov field
$\k=\Q(\!(T)\!)=\{$formal Laurent series in $T\}$. Abbreviate $y:=x^{1+m-k}-(-k)^kT$.
By \cite{R16}, in \cref{Prop vanishing of SH} we have 
$$QH^*(Y)=\k[x]/(x^k y),\quad Q_{\varphi}=-kx,\quad E_0(Q_{\varphi})=\langle y \rangle \subset QH^*(Y),\quad \textrm{ and }\quad SH^*(Y)\cong \k[x]/(y).$$
The $\varphi$-filtration by ideals on $QH^*(Y)$ is:
$$
0\subset 
\langle x^{k-1}y \rangle 
\subset 
\langle x^{k-2}y \rangle 
\subset
\cdots
\subset
\langle x y \rangle
\subset 
E_0(Q_{\Fi})=\langle y \rangle 
\subset QH^*(Y).
$$
There are three ways to study the {\MBF} spectral sequence: either use a $\Z/2$-grading with $|T|=0$ by only remembering the parity; or  refine this to a $\Z/2(1+m-k)$-grading with $|T|=0$ \cite[Sec.1.8]{R14}; or use a $\Z$-grading at the cost of grading the Novikov ring $|T|=2(1+m-k)$.
Odd classes of $H^*(B_N)[2N]$ must continue to kill even classes as before, by \cref{Cor intro inj on odd classes}, but new phenomena occur.
For $\mathcal{O}_{\C P^m}(-1)$, the degree $-1$ class $y_{1+m}$ generating $H^*(B_{1+m})[2(1+m)]$, which in the CY-case $\mathcal{O}_{\C P^n}(-n-1)$ would have hit the unit in $H^0(Y)$, must no longer do so, as $SH^*(Y)\neq 0$.

Consider $Y=\mathcal{O}_{\C P^1}(-1)$. 
Let $x_m,y_m$ be a choice of generators of $H^0(B_m)[2m],H^3(B_m)[2m]$ in grading 
$|x_m|=-2m$, $|y_m|=-2m+3$,
in column $m$ of the $E_1$ page, noting $B_m\cong S^3$ for $m=1,2,\ldots$. In column $0$ we have the generators  $x_0,y_0\in H^*(\C P^1)$ in gradings $0,2$.
Note $\omega+T\cdot 1 \in \ker c_{1^+}$, as one full rotation corresponds to quantum product by $\mathcal{Q}=-[\omega]$ on $QH^*(Y)\cong \k[\omega]/(\omega^2+T\omega)$.
It follows that on the $E_1$-page, $y_1$ must hit (a non-zero multiple of) $\omega+T\cdot 1$; note $\omega,T\cdot 1$ lie in the same grading even though $\omega,1\in H^*(Y)$ do not.
So, if we consider only up to column 1, then $x_0,x_1$ represent $E_2$-page generators (they survive to $HF^*(H_{1^+})\cong QH^*(Y)[2]$). Thus $y_2$ in column 2 must not survive to $E_2$, otherwise it would have to kill the unit $x_0$. So $y_2$ must kill $Tx_1$ on the $E_1$-page.
Similarly, $y_m$ must kill $Tx_m$ on the $E_1$-page. So only $x_0$ survives, as expected: $SH^*(Y)=\k\cdot [x_0]$. 

Another illustration of a $c_1(Y)\neq 0$ example is when the vector bundle is ``very negative'' \cite[Sec.1.10-1.11]{R13}. In that case, Floer theory is $\Z/2M$-graded for a very large integer $M$. The ambiguity of the grading caused by being allowed to translate degrees by $2M$ has little to no effect on our ability to determine what happens in the spectral sequence below the $N$-th column, where $N=1+\lfloor \tfrac{m}{r} \rfloor$. So the degree $-1$ part of $H^*(B_N)[2rN]$ must kill $H^0(Y)$, which confirms that $SH^*(Y)=0$ \cite{R14}.
\end{ex}

\subsection{Higgs moduli and comparison with the $P=W$ filtration}
A famous class of examples are
moduli spaces $\MM$ of Higgs bundles 
over a Riemannian surface $\Sigma_g$.
Roughly speaking, their elements are
(conjugacy classes) of stable pairs $(V,\Phi),$ where $V$ is a vector bundle over $\Sigma_g$ of fixed coprime rank and degree,
and $\Phi\in \Hom(V,V\otimes K_{\Sigma_g})$ 
is the so-called Higgs field,\footnote{here, $K_{\Sigma}$ is the canonical bundle of $\Sigma_g$.} 
which can potentially have some poles.
The space $\mathcal{M}$ is a 
complete 
{\hk} manifold,
with a natural %
\CC-action given by $t\cdot (V,\Phi)=(V,t \Phi).$ 
It also admits the so-called Hitchin fibration, given as the characteristic polynomial of the Higgs field 
\begin{equation}\label{Intro Hitchin fibration}
\qquad \qquad \Psi: \mathcal{M}\to B\iso \C^{N}   \quad \textrm{ where }N:=\tfrac{1}{2} \dim_\C\MM.
\end{equation}
Interestingly, the torsion manifolds of these spaces, involved in our spectral sequences, consist of so-called \textit{cyclic Higgs bundles}, considered by several authors \cite{baraglia2009a,Ba15,dai2020cyclic,RaSu19}.

Higgs moduli come with the ``$P=W$'' filtration, \cite{de2012topology,hausel2022p,maulik2022p},
which is the Perverse filtration of $\Psi,$ and it 
has a different structure to ours as it filters from the unit towards the higher classes, whereas ours goes the other way around. 
However, it still makes sense to compare these two filtrations on each $H^k(\MM).$
In the simplest case of ordinary Higgs bundles of rank 1, where $\MM\iso T^*T^{2g},$ the filtrations degree-wise agree, both having  trivial filtrations $0\subset \FF_1(H^k(\MM))=H^k(T^{2g}).$

More interestingly, consider \textit{parabolic} Higgs bundles of $\dim_\C=2,$ given as crepant resolutions
\begin{equation}
    \pi: \MM_\Gamma \fun (T^*E)/\Gamma,
\end{equation}
where $E$ is an elliptic curve, and
$\Gamma \in \{{0},\ \Z/2,\ \Z/3,\ \Z/4,\ \Z/6\},$
is a finite group of automorphisms of $E$. %
Their cohomology is supported only in degrees 0 (where both filtrations are trivial) and 2. %
\begin{prop}
    The filtration $\FF(H^2(\MM_\Gamma))$ refines the $P=W$ filtration.
\end{prop}
In addition, our filtration has a representation-theoretic meaning for these examples. 
Namely, the intersection graph $Q_\Gamma$ of the components of $\Core(\MM_\Gamma)$
is known to be an affine Dynkin graph, thus has the imaginary root $\delta=\sum_{i\in {Q_{\Gamma}^0}} n_i \a_i,$ 
where $\a_i$ are the simple roots. The following holds: %
    \begin{prop} Denoting by $\FF_k$ the $k$-th level of the filtration $\FF(H^2(\MM_\Gamma)),$ we have
\par 
        \vspace{\abovedisplayshortskip}
\hfill 
$\displaystyle \dim_{\k}(\FF_k) = \#\{i\in Q_{\Gamma}^0 \mid n_i \leq k\}.$ \qed
\end{prop}
\subsection{{\MB} spectral sequence for $S^1$-equivariant symplectic cohomology}\label{Subsection intro S1 spectral seq}
\strut\\[1mm]
\indent
In \cite{RZ1,RZ3} we discuss $S^1$-equivariant symplectic cohomology $ESH^*(Y,\Fi)$ in detail. 
Briefly, in \cite{McLR18}-conventions, $ESH^*(Y,\Fi)$ is a $\k [\![u]\!]$-module, with a canonical $\k [\![u]\!]$-module homomorphism $$Ec^*:EH^*(Y)\cong H^*(Y)\otimes_{\k}\mathbb{F} \to ESH^*(Y,\Fi).$$ Here $u$ is a degree two formal variable, and at chain level, each $1$-orbit contributes a copy of the $\k [\![u]\!]$-module $\mathbb{F}:=\k(\!(u)\!)/u\k [\![u]\!]\cong H_{-*}(\C\P^{\infty})$ where we identify $u^{-j}=[\C\P^j]$, so that the action of $H^*(\C\P^{\infty})=\k[u]$ corresponds to the nilpotent cap product. 
Above, we used the notation from \cite{McLR18}, 
$$\qquad \qquad EH^*(Y)\cong H_{2\dim_{\C}Y-*}^{l\!f,S^1}(Y) \quad \textrm{(locally finite }S^1\textrm{-equivariant homology)}.$$
In \cite{RZ3} we allow $S^1$ to act also on $Y$, but in the present context $S^1$ only acts by $S^1$-reparametrisation on $1$-orbits, so $EH^*(Y)\cong H^*(Y)\otimes_{\k}\mathbb{F}$, as $EH^*(Y)$ arises from constant $1$-orbits.

When $c_1(Y)=0$, we have $ESH^*(Y,\Fi)=0$ for the same grading reasons that prove $SH^*(Y,\Fi)=0$. Thus $SH^{*-1}_+(Y,\Fi)$ is a free $\mathbb{F}$-module isomorphic to $H^*(Y)\otimes \mathbb{F}$ as a $\k [\![u]\!]$-module.

Equivariant and non-equivariant symplectic cohomology are related by a Gysin sequence
$$
\cdots \to SH^*_{+}(Y,\Fi) \stackrel{\mathrm{in}}{\to} ESH^*_{+}(Y,\Fi) 
\stackrel{u\cdot}{\to} ESH_{+}^{*+2}(Y,\Fi) 
\stackrel{b}{\to} SH_{+}^{*+1}(Y,\Fi) 
\to \cdots
$$
where ``$\mathrm{in}$'' is induced at chain level by the inclusion as the $u^0$-part. For example, when $ESH^*_+(Y,\Fi)$ lives in odd grading, this implies
$SH^*_+(Y)\cong \ker (u:ESH^*_+(Y)\to ESH^*_+(Y)).$
\begin{thm}\cite{RZ1}
    There is an $\R_{\infty}$-ordered filtration by graded ideals of the $\k [\![u]\!]$-module $EH^*(Y)$,
\begin{equation}\label{DefinitionOfFiltration2}
E\FF^{\varphi}_p :=\bigcap_{\mathrm{generic}\,\lambda>p} \left(\ker Ec_\lambda^*:H^*(Y)\otimes_{\k}\mathbb{F}\to EHF^*(H_{\lambda})\right), \qquad E\FF_{\infty}^{\Fi}:=H^*(Y) \otimes_{\k} \mathbb{F},
\end{equation} 
where the equivariant continuation map $Ec_\lambda^*$ is grading-preserving and $\k [\![u]\!]$-linear. Also, \eqref{Equation thm filtration as image of SHplus} becomes:
$$
E\Fil^{\varphi}_\lambda = \ker \,\left(Ec_{\lambda}^*:H^*(Y)\otimes_{\k}\mathbb{F} \to ESH^{*}_{(-\infty,\lambda]}(Y,\Fi)\right) = \mathrm{Image}\,\left(ESH^{*-1}_{(0,\lambda]}(Y,\Fi)\to H^*(Y)\otimes_{\k}\mathbb{F}\right).
$$
In general, $\FF^{\Fi}_\lambda \subset E\FF^{\Fi}_\lambda$. If $H^*(Y)$ lies in even degrees (e.g.\;CSRs), then $\FF^{\Fi}_\lambda  = QH^*(Y)\cap E\FF^{\Fi}_\lambda.$
\end{thm}

\begin{thm}\label{Theorem intro spectral seq 2}
Let $Y$ be a symplectic $\C^*$-manifold. 
For any period interval $\mathcal{P}\subset (-\infty,+\infty]$, and working over a Novikov field $\k$,
there is a convergent spectral sequence of $\k$-modules,
	$$
 E_{r}^{pq}(\Fi;\mathcal{P}) \Rightarrow ESH_{\mathcal{P}}^*(Y,\Fi), \textrm{   where } E_{1}^{pq}(\Fi;\mathcal{P})=  \begin{cases*}
				H^q(Y)\otimes_{\k}\mathbb{F} & if $p=0$\textrm{ and }$0 \in \mathcal{P}$\\
				\bigoplus_{\c} EH_{p,\c}^*[-\mu(B_{p,\c})] & if $p<0, \textrm{ and } T_p \in \mathcal{P}$   \\
				0 & otherwise,
\end{cases*}
   $$
where the $\k$-module $EH_{p,\c}^*$ is the limit of an energy-spectral sequence starting with $S^1$-equivariant Morse cohomology\footnote{in the sense of \cite[Sec.4]{McLR18}.} $\oplus_{\beta} EMH^*(B_{p,\c};f_{p,\c};\mathcal{L})$, 
for the slice $c'(H)=T_p$ with a local system $\mathcal{L}$ as in \cref{Prop local HF intro}.

If $H^*\left(B_{p,\c}/S^1\right)$ lies in even degrees, then the energy-spectral sequence immediately collapses, and\footnote{The final isomorphism is \cite[Thm.4.3]{McLR18}.}
\begin{equation}\label{EH is in odd degree intro}
EH_{p,\c}^*[-\mu(B_{p,\c})] = EH^*(B_{p,\c})[-\mu(B_{p,\c})] \cong
H^*\left(B_{p,\c}/S^1\right)[-1-\mu(B_{p,\c})]
\end{equation}
lies in odd degrees.
Although the spectral sequence generally forgets the $u$-action, within that column the $u$-action is cap product by the negative Euler class associated with the map
$B_{p,\c}\to B_{p,\c}/S^1$.
\end{thm}

This $S^1$-equivariant version of the {\MB} spectral sequence was particularly helpful in \cite{McLR18} for crepant resolutions of isolated quotient singularities, because %
the columns $EH^*(B_{p,\c})[-\mu(B_{p,\c})]$ 
were in odd grading.
The same phenomenon occurs for many examples of symplectic $\C^*$-manifolds.

\begin{cor}\label{Cor eq sp seq collapses intro}
    If $H^*\left(B_{p,\c}/S^1\right)$ lies in even degrees for all $p,\beta$, then \eqref{EH is in odd degree intro} lies in odd degrees and determines the spectral sequence for $ESH^*_+(Y,\Fi),$ which has collapsed on the $E_1$-page:
$$ESH^*_+(Y,\Fi)=E_{1}^{**}(\Fi;(0,\infty]).$$ 
    In the $SH^*(Y,\Fi)=0$ case (e.g.\,when $c_1(Y)=0$) this also forces $H^*(Y)$ to lie in even grading, so we can read off $H^*(Y)\otimes \mathbb{F}$ from that $E_1$-page, since $H^*(Y)\otimes_{\k}\mathbb{F} \cong ESH^{*}_+(Y,\Fi)[-1]$.
\end{cor}
In \cite{McLR18} this Corollary was the key to proving the cohomological McKay correspondence.
Below we will illustrate a large class of examples where that argument applies, but first a simple example:

\begin{ex}
 Continue \cref{Example running example of intro 2}. Let $p_i$ abusively denote $H^2(S^2_i)$; let $\k_m(\lambda)$ be a copy of $\k$ in grading $m$ arising in the slope $\lambda$ column; and identify the $\k$-modules $\mathbb{F}=\k[u^{-1}]=\k_0 \oplus \k_{-2} \oplus \k_{-4}\oplus \cdots$. 
  For both actions: $B_{\textrm{non-integer}}/S^1=\mathrm{(point)}$, $B_{\textrm{integer}}/S^1\cong \C\P^1$. E.g.\;in the first table of \cref{Example running example of intro 2} $H^*(B_1)[4]$ gets replaced by $H^{*}(B_1/S^1)[-1+4]=H^*(\C\P^1)[3]=\k_{-3}(1)\oplus \k_{-1}(1)$, and after the very last $[-1]$ shift in \cref{Cor eq sp seq collapses intro} contributes $\k_{-2}(1)\oplus \k_{0}(1)$.
  Doing the same for all columns, we get:
 \begin{align*}
 \textrm{(a)} 
 & \strut &
p[u^{-1}]\oplus p_1[u^{-1}] \oplus p_2[u^{-1}]
& \cong
\k_2(\tfrac{1}{3})^{\oplus 2} \oplus \k_{0}(\tfrac{2}{3})^{\oplus 2} \oplus \k_{0}(1) \oplus \k_{-2}(1)\oplus
\k_{-2}(\tfrac{4}{3})^{\oplus 2}\oplus \cdots
\\
\textrm{(b)} 
 & \strut &
H^*(S^2_1)[u^{-1}]\oplus p_2[u^{-1}]
& \cong
\k_2(\tfrac{1}{2}) \oplus \k_{2}(1) \oplus \k_{0}(1)\oplus
\k_{0}(\tfrac{3}{2})\oplus \k_{0}(2)\oplus \k_{-2}(2) \oplus \cdots
 \end{align*}
So in (a):\;$\FF_{1/3}^{\Fi}=p_1\oplus p_2$, $\FF_{1}^{\Fi}=H^*(M)$. In (b):\;$\mathrm{rk}(\FF_{1/2}^{\Fi}\cap H^2(Y))=1$, $\FF_{1}^{\Fi}=H^2(M)$, $\FF_{2}^{\Fi}=H^*(M).$
For (b) we improve $\FF_{1}^{\Fi}\subset H^2(M)$ to an equality as $\FF^{\Fi}_1\neq H^*(M)$: $u$ acts non-trivially $\k_0(1)\to \k_2(1)$ by \cref{Theorem intro spectral seq 2}, whereas $u$ acts trivially on $1\cdot u^0\in H^0(S^2_1)$.
In (a) we cannot determine $\FF_{2/3}^{\Fi}$: we would need additional information about the $u$-action which the $E_1$-page has ``forgotten''.
\end{ex}

{\bf Assumption.} We make the following simplifying assumption, which holds for all examples in \cref{ExamplesSpSeq}. Assume that the outer torsion submanifolds are complex vector bundles over their cores, 
\begin{equation}\label{Ymc_a_bundle_above_the_core}
\Ymc=:E_{m}\fun \mathrm{Core}(\Ymc), \;\textrm{ and let }r:= \mathrm{rk}(\Ymc):=\mathrm{rk}_\C(E_m).
\end{equation}
Let $B_{p,\beta}$ denote the slice arising with slope $T_p=\tfrac{k}{m}$ for coprime $k,m$. We
get a complex projectivisation 
$$
B_{p,\beta}/S^1 \cong 
(E_m \setminus (\textrm{zero section}))/\C^* = \P(E_m).
$$
Omitting the even index shift $[-\mu(B_{p,\c})]$ for readability,%
\begin{equation}\label{EH_Bpc_formula}
   EH^*(B_{p,\beta}) \!\cong\! H^{*}(B_{p,\beta}/S^1)[-1] \!=\! H^{*}(\P(E_m))[-1] \!\cong\! H^{*}(\mathrm{Core}(\Ymc))
   \la 1,c_1(L),\ldots,c_1(L)^{r-1}\ra [-1]
\end{equation}
where the first isomorphism is \cite[Thm.4.3]{McLR18}, and the last is the Leray--Hirsch theorem that yields\footnote{$B_{p,\beta}/S^1$ is an orbifold with cyclic stabilisers, but over a field of characteristic $0$ this is not an issue \cite[Sec.4.7]{McLR18}.}
the \emph{free} module over the cohomology of the base $\mathrm{Core}(\Ymc)$ of $E_m$, generated by powers of the first Chern class 
of the tautological line bundle $L:E_m \fun \P(E_m)$. 
The cohomology of $\mathrm{Core}(\Ymc)$ is determined by \eqref{EqnFrankel},
under mild assumptions\footnote{e.g.\;this holds by \cite{RZ1} if $\mathrm{Core}(\Ymc)$ has the homotopy type of a CW complex (this holds for example when $\mathrm{Core}(\Ymc)$ is cut out by analytic equations, e.g.\;it always holds when $Y$ is a CSR).} that ensure the first isomorphism below,
$$
H^*(\mathrm{Core}(\Ymc))\iso H^*(\Ymc) \iso \bigoplus_{\F_\a\subset \Ymc} H^*(\F_\a)[-\mu(\F_\a;\Ymc)],
$$
where $\mu(\F_\a;\Ymc)$ is the (even) {\MB} index of $\F_\a \subset \Ymc$ for the restricted moment map $H|_{\Ymc}$. 

Assume $c_1(Y)=0$. Suppose $H^*(Y)$ lies in even degrees. Then by \eqref{EqnFrankel} the same holds for $H^*(\F_\a),$ and thus for 
$H^*(\mathrm{Core}(\Ymc)).$ By \eqref{EH_Bpc_formula}, we have \eqref{EH is in odd degree intro}
lies in odd degrees, and \cref{Cor eq sp seq collapses intro} applies. 
An explicit example of this, for the Slodowy variety $\mathcal{S}_{32}$, is shown in \cref{Example Slodowy variety S32 equivariant}.
Summarising:
\begin{cor}\label{Cor big S1 equi formula intro}
Assume:\;$c_1(Y)=0$, $H^{\mathrm{odd}}(Y)=0$, the outer $\Ymc\fun \mathrm{Core}(\Ymc)$ are complex vector bundles, $\mathrm{Core}(\Ymc)$ are CW complexes up to homotopy. Then the $\k$-module $H^*(Y)\otimes_{\k}\mathbb{F}$ is:\footnote{the first isomorphism uses \eqref{EqnFrankel}, and $B_N=\{c'(H)=N\}\cong \Sigma:=B_1$ has $\mu(B_N)=-2N\mu$, $\mu=$ Maslov index of $\Fi|_{S^1}$.}
\begin{equation} \label{The Big Formula}    \small
\begin{aligned}
& \bigoplus_{i=0}^{\infty} H^*(Y)[2i]  \cong
\bigoplus_{\a}\bigoplus_{i=0}^{\infty} H^*(\F_\a)[-\mu_\a+2i]\cong
ESH^{*}_+(Y,\Fi)[-1]\cong 
\bigoplus_{p,\beta} H^*(B_{p,\beta}/S^1)[-2-\mu(B_{p,\beta})] \cong
\\
& \cong  \bigoplus_{N=1}^{\infty} EH^*(\Sigma)[-1+2N\mu]\oplus\!\! \bigoplus_{m\geq 2,\,\beta}\; \bigoplus_{(k,m)=1} 
\!\left( \bigoplus_{j=1}^{\mathrm{rk}(\Ymc)}\bigoplus_{\F_\a\subset \Ymc} H^*(\F_\a)[-2j-\mu(\F_\a;\Ymc)-\mu(B_{k/m,\c})]\right)\!. 
\end{aligned}
\end{equation}
\end{cor}
For CSRs, 
$
EH^*(\Sigma) 
$
can be inductively computed from Betti numbers of the $\F_\a$, by the Gysin sequence relating
$EH^*(\Sigma)$, $H^*(\Sigma).$ %
All CSR examples in \cref{ExamplesSpSeq} satisfy the conditions of \cref{Cor big S1 equi formula intro}.

We remark that \eqref{The Big Formula} is a highly non-trivial identification, in general: 
\begin{enumerate}
    \item The $\F_\a$ not belonging to any outer $\Y_{m,\beta}$ will not appear in \eqref{The Big Formula}.\footnote{E.g.\;for the minimal resolution of an $A_4$-singularity, $M\to \C^2/(\Z/5)$, the three fixed points lying on the inner two spheres of the core do not belong to outer torsion submanifolds.
The copies of the middle point get identified with the $EH^*(\Sigma)$-summands. The other two inner points get identified with their adjacent 
(connected via a sphere) 
outer points.} 
 Even when all appear, some may appear more often, 
causing identifications,
e.g.\;identifications between $\min$ $\&$ $\max$ components in an 
outer $\Ymc$: this occurs for a twisted $T^*\CP^3$ and for the Slodowy variety $\SS_{32}$.
\item Even when the $\F_\a$ match perfectly via \eqref{The Big Formula}, classes in $EH^*(\Sigma)$ rephrased as Betti numbers of $\F_\a$ 
can match non-trivially with Betti numbers of $\F_\a$ on the left in \eqref{The Big Formula}, 
causing
identifications, e.g.\;between Poincaré-dual classes for some $\F_\a$: this occurs for the Slodowy variety $\SS_{211}$.
\item Both types of identifications in (1) and (2) can occur simultaneously (e.g.\;Slodowy variety $\SS_{33}$).
\item Finally, \eqref{The Big Formula} describes a free $\mathbb{F}$-module $ESH^{*}_+(Y,\Fi)[-1]$ that is isomorphic to $H^*(Y)\otimes \mathbb{F}$ as a $\k [\![u]\!]$-module. So the ranks over $\k$ of 
\eqref{The Big Formula} must arrange themselves into copies of $\mathbb{F}$. This arrangement is not always clear because the spectral sequence forgot the $u$-action. 
In any case, \eqref{The Big Formula} often yields the precise $\Fi$-filtration ranks
$\mathrm{rk}_{\k}\FF^{\Fi}_{\lambda}$, and otherwise yields good lower bounds.
\end{enumerate}

\begin{rmk}[Stacking the $\k$ to form copies of $\mathbb{F}$]
Regarding (4) above, a similar situation arose in the naturality comments in \cite[Sec.1.5]{McLR18}.\footnote{there one would like to stack $\k$-copies to form $\mathbb{F}$-summands in the way that the ranks in \cite[Cor.2.13]{McLR18} are suggesting, but not enough information about the $u$-action was retained by the spectral sequence.} We expect that $EH^*(B_{k/m,\beta})$ should stack with $EH^*(B_{(k+1)/m,\beta})$ to build $\mathbb{F}$, suitably interpreted when $k,m$ or $k+1,m$ are not coprime.\footnote{For example, if $(k,m)=1$ but $(k+1)/m=k'/m'$ after reducing to coprime $k',m'$, then there is some $B_{k'/m',\gamma}$ containing a stratum diffeomorphic to $B_{k/m,\beta}$ such that the diffeomorphism is given by a rescaling of the $X_{\R_+}$-flow.} 
More precisely, in the coprime case, let $C:=B_{k/m,\beta}/S^1$, $C':=B_{(k+1)/m,\beta}/S^1$, and $\lambda=((k+1)/m)^+$.
In the \cite[Sec.4.2]{McLR18}-notation, and referring to maxima/minima of auxiliary Morse functions on $C,C'$, one might expect that the maximum $x$ of $C$, which yields a class $[x u^0]$ in the $E_1$-page, is represented in $EHF^*(H_{\lambda})$ by $z=xu^0+x'u^{-j-1},$  $j=\dim_{\C}C'$, for $x'$ a chain in the local Floer complex for $C'$ not appearing in $E_1$ due to a non-trivial vertical differential $\delta_{j+1}(x')\neq 0$. As $[z]$ is a cycle, so $(\delta_0 + u \delta_1 + \cdots)z=0$, one expects $\delta_{j+1}(x')u^0=-\delta_0(x)u^0$, with the $u$-action $u\cdot [z]=[x'u^{-j}]$ yielding the minimum of $C$ on the $E_1$-page. The question is what might be a canonical Floer trajectory, $\partial_s u + I(\partial_t u - k(s,t) X_{S^1})=0$ for $k(s,t)=c'(H(u(s,t)))$, counted by $\delta_0(x)$ that ``links'' $C$ to $C'$?\\
\uvuci {\bf Conjecture.}
It is $u(s,t)=\varphi_{\exp(2\pi i m/k)}(v(s))$ with $v'(s)=- \lambda_s \nabla H$, so $v$ is a Morse continuation solution for a homotopy $\lambda_s H$ where $\lambda_s= k(s,t)-\tfrac{m}{k}$ goes from $0$ to $1/k$. The Floer equation decouples as $\partial_s u = - \lambda_s \nabla H$ and $\partial_t u = \tfrac{m}{k} X_{S^1}$, so a family of $S^1$-orbits $u(s,\cdot)$ in $Y_{m,\beta}$ moving by the $\R_+$-action.
\end{rmk}

\noindent \textbf{Acknowledgements.} 
We thank Gabriele Benedetti, Johanna Bimmermann,
Fr\'{e}d\'{e}ric Bourgeois, Kai Cieliebak,
André Henriques, Mark McLean, Alexandre Minets, Alexandru Oancea, Paul Seidel, Nick Sheridan, Otto van Koert, and Chris Wendl for helpful conversations. 
The first author is grateful to the Mathematics Department of Stanford University for their hospitality during the author's sabbatical year, where the paper was completed.
Part of this work is contained in the second author’s DPhil thesis \cite{FZ20}, and he acknowledges the support of Oxford University, St Catherine's College, and the University of Edinburgh where he was supported by ERC Starting Grant 850713 – HMS.

\section{Symplectic $\C^*$-manifolds}\label{Section preliminaries}\label{Subsection Kahler mfds admitting C*action}\label{Subsec overview}\label{Subsection weight spaces}
 \label{Section Torsion sbmfds attraction graph}

We will summarise some definitions and results from \cite{RZ1} to facilitate referencing them.  

\subsection{The definition of symplectic $\C^*$-manifolds}\label{Def:CstarManifold}\label{Def:KahlerMfdWithProjection}

In this paper, symplectic $\C^*$-manifolds are always over a convex base, which we now recall.
We have a connected symplectic manifold $(Y,\omega)$, a choice of $\omega$-compatible almost complex structure $\J$ on $Y$, and a {\ph }
$\C^*$-action $\Fi$ on $(Y,\J)$ whose $S^1$-part is Hamiltonian, where $S^1:=\R/\Z \hookrightarrow \C^*$, $s\mapsto e^{2\pi i s}\subset \C^*$. Let $H:Y \to \R$ be the moment map of the $S^1$-action. The induced $S^1$-invariant Riemannian metric is $g(\cdot, \cdot):=\omega(\cdot,\J\cdot)$.
Call $\R_+$ the subgroup $\R\hookrightarrow \C^*$, $s\mapsto e^{2\pi s}$. The flows of the {\bf $\R_+$-part} and {\bf $S^1$-part} of $\Fi$ determine vector fields
\begin{equation}\label{Lemma:XR+ is nabla h}
    X_{\R_+}=\nabla H, \qquad X_{S^1}=X_H=\J\nabla H = \J X_{\R_+}.
\end{equation}

\noindent
We require that there is a compact subdomain $Y^{\mathrm{in}}\subset Y$ and an $(I,I_B)$-{\ph } proper map 
$$
\Psi: Y^{\mathrm{out}} = Y \setminus \mathrm{int}(Y^{\mathrm{in}}) \to B^{\mathrm{out}}=\Sigma \times [R_0,\infty),
$$
where $R_0 \in \R$ is any constant; $(\Sigma,\alpha)$ is a closed contact manifold; $I_B$ is a $d(R\alpha)$-compatible almost complex structure on $B^{\mathrm{out}}$ of contact type such that 
\begin{equation}\label{Equation psi XS1 is Reeb}
\Psi_* X_{S^1} = w\mathcal{R}_B
\end{equation}
where $\mathcal{R}_B$ is the Reeb vector field for $\Sigma$ and $w>0$ is a constant. After a rescaling trick, $w=1$. 
Above, $I$ and $I_B$ are \textit{almost} complex structures.
We assume $B^{\mathrm{out}}$ is geometrically bounded at infinity.

We may abusively write $\Psi: Y \to B$ even though $\Psi$ is only defined at infinity, and $B^{\mathrm{out}}$ is not required to have a ``filling'' $B$.
We call $Y$ %
\textbf{globally defined over a convex base} if $\Psi$ actually extends to a {\ph } proper $S^1$-equivariant map
$\Psi:(Y,I) \to (B,I_B)$
defined on all of $Y$, whose target is a symplectic manifold
$(B,\omega_B,\J_B)$ convex at infinity, with a Hamiltonian $S^1$-action, whose Reeb flow at infinity agrees with the $S^1$-action. Again, $\J_B$ is an $\omega_B$-compatible almost complex structure, of contact type at infinity.
Such $\Psi$ are in fact $\C^*$-equivariant, by integrating $\Psi_*X_{\R_+}$ and $\Psi_*X_{S^1}$ on $\mathrm{Im}(\Psi)\subset B$.

\begin{rmk}\label{Rmk about S1 actions}
If we only had an $\J$-{\ph }\footnote{meaning $(\psi_t)_*\circ \J = \J \circ (\psi_t)_*$ for $t\in \R$.} $S^1$-action, 
$\psi_t:=\varphi_{e^{2\pi it}}:Y \to Y$, then this locally extends to a $\C^*$-action. The Lie derivative of its vector field $X_{S^1}$ satisfies $\mathcal{L}_{X_{S^1}}(\J)=0$, so $X_{S^1}$ and $X_{\R_+}:=-\J X_{S^1}$ commute.\footnote{$[X_{S^1},\J X_{S^1}]=\mathcal{L}_{X_{S^1}}(\J X_{S^1})=\J \mathcal{L}_{X_{S^1}}(X_{S^1})=0$.} So we get a partially defined {\ph } map $\varphi:\C^*\times Y \to Y$, $\varphi_{e^{2\pi(s+it)}} = \mathrm{Flow}_{X_{\R_+}}^s\circ \psi_t$. If  $X_{\R_+}$ is integrable then this $\varphi$ becomes a globally defined $\C^*$-action.
\end{rmk}

\begin{rmk}\label{Rmk weight w of Reeb}
$B$ being \textbf{convex at infinity} means there is a compact subdomain $B^{\mathrm{in}}\subset B$ outside of which we have a \textbf{conical end} $B^{\mathrm{out}}:=B\setminus\mathrm{int}(B^{\mathrm{in}}) \cong \Sigma \times [R_0,\infty)$ such that the symplectic form becomes $\omega_B=d(R\alpha)$. 
The radial coordinate $R\in [R_0,\infty)$ yields the Reeb vector field $\mathcal{R}_B$
for the contact hypersurface $(\Sigma,\alpha)$, $\Sigma:=\{R=R_0\}$ (defined by $d\alpha(\mathcal{R}_B,\cdot)=0$ and $\alpha(\mathcal{R}_B)=1$). So $\mathcal{R}_B=X_R$ is the Hamiltonian vector field for the function $R$. 
After increasing $R_0$ if necessary, we can always assume that $I_B$ is $\omega_B$-compatible and of {\bf contact type} on $B^{\mathrm{out}}$, meaning%
\footnote{Equivalently $dR= R\alpha\circ I_B$, so $I_B$ preserves the contact distribution $\xi=\ker \alpha \subset T\Sigma$. The $\omega_B$-compatibility condition ensures that $d\alpha$ is an $I_B$-compatible symplectic form on $\xi$.
By \cite[Lemma C.9]{R16}, it suffices to assume $a(R)dR= R\alpha\circ I_B$ for a positive smooth function $a$, equivalently
$I_B Z_B = a(R)\mathcal{R}_B$. 
} 
$I_B Z_B = \mathcal{R}_B$, where $Z_B = R \partial_R$ is the {\bf Liouville vector field} defined by $\omega_B(Z_B,\cdot)=R\alpha$ on $B^{\mathrm{out}}$. If $I_B$ does not depend on $R$,
then clearly $B^{\mathrm{out}}$ is geometrically bounded at infinity due to the radial symmetry. 
If $I_B$ depends on $R$ (on the $\xi=\ker d\alpha$ orthogonal summand of $TB^{\mathrm{out}}=\xi\oplus \R Z_B\oplus \R \mathcal{R}_B$),
then it is desirable to require
that $B^{\mathrm{out}}$ is geometrically bounded at infinity. This assumption is needed to prove that  Floer solutions ``consume $F$-filtration'' if they go far out at infinity on a long region on which $c'(H)$ is linear (\cref{LemmaDeltaiFiltration}). This property is needed in \cref{Prop filtration is stable}, in the construction of the $Q_{\Fi}$ class  (\cref{Prop vanishing of SH}), and we use it to ensure a certain consistency between {\MBF } spectral sequences so that we can take the direct limit over slopes $\lambda$, specifically in \cref{LemmaDeltaiFiltration}.\footnote{The geometric boundedness assumption on $B$ is not needed for results preceding that Lemma, e.g. one can prove the maximum principle and define $SH^*(Y,\varphi)$ without it.} 
The rescaling trick to make $w=1$ in \eqref{Equation psi XS1 is Reeb}
is: rescale
$R,\alpha,R_0$ to $wR,\alpha/w,wR_0$ (leaving $\omega_B=d(R\alpha)$ and the Liouville form $\theta=R\alpha$ unchanged on $B^{\mathrm{out}}$). 
In \cite{RZ1} we show $\Psi: Y^{\mathrm{out}} \to \mathrm{Im}(\Psi)$ is $\C^*$-equivariant for a partially defined $\C^*$-action on $\mathrm{Im}(\Psi)\subset B^{\mathrm{out}}$
that integrates $\Psi_*X_{\R_+}=\nabla R = Z_B$ and $\Psi_*X_{S^1}=\mathcal{R}_B$ on $\mathrm{Im}(\Psi)$. 
\end{rmk} 
\begin{lm}\label{Lemma H bdd below}\label{Def:Contracting CstarManifold}
The $\C^*$-action $\Fi$ on a symplectic $\C^*$-manifold $(Y,\omega,\J,\Fi)$ is {\bf contracting}, meaning there is a compact subdomain $Y^{\mathrm{in}}\subset Y$ such that the $-X_{\R_+}$ flow starting from any point $y\in Y$ will eventually enter and stay in $Y^{\mathrm{in}}$. In particular, $H$ is bounded below, the $\Fi$-fixed locus $\F=\sqcup_\a \F_\a$ in \eqref{Equation intro fixed locus} is compact, and all points $y\in Y$ have well-defined {\bf convergence points} $y_0:=\lim_{\C^* \ni t\to 0} t\cdot y \in \F$.
\end{lm}

We always assume that the moment map $H$ on a symplectic $\C^*$-manifold is proper, by tweaking $\omega$:

\begin{lm}\label{FibersMomentMapConnected} 
\label{Lemma making H proper}
Let $\Psi: Y^{\mathrm{out}} \fun B^{\mathrm{out}}=\Sigma \times [R_0,\infty)$ be a symplectic $\C^*$-manifold, and $\phi:[R_0,\infty)\to [0,\infty)$ a non-decreasing unbounded smooth function vanishing near $R_0$. Then $\om_{\phi}:=\om+d(\Psi^*(\phi(R)\alpha))$ 
is symplectic, cohomologous to $\om$, and the $S^1$-action is Hamiltonian for $\om_{\phi}$, with proper moment map.

If $\Psi: Y \fun B$ is globally defined over a convex base, $\om+\Psi^*\om_B$ is a symplectic form for which the $S^1$-action is Hamiltonian and has a proper moment map.

Once $H$ is proper, it is also exhausting and has connected level sets (using that $Y$ is connected).
\end{lm}

\subsection{Torsion points. Attraction graph.}
\begin{lm}\label{TorsionPtsTorsionSubmflds}
\label{Lemma torsion submanifolds}\label{TorsionMfdIsSymplectic}
\label{Lemma Umin is dense}\label{Cor Ymc hits core in path connected subset}
\label{Lemma Torsion bundle is a bundle}   \label{Defn torsion bundle}\label{CorSec2OrbitsOfellH}
The fixed locus of the subgroup 
$\Z/m\hookrightarrow S^1\subset \C^*$, $k\mapsto e^{2\pi i k/m}$, for $m\geq 2$, defines a $\C^*$-invariant $I$-{\ph} $\omega$-symplectic submanifold $Y_{\ell}\subset Y$, whose points are called
{\bf $\Z/\ell$-torsion points}, and whose connected components $Y_{\ell,\beta}$ are called {\bf torsion submanifolds}. Moreover:
\begin{enumerate}
\item 
$Y_{\ell}$ is the image of all $1$-periodic Hamiltonian orbits of $\tfrac{1}{m} H:Y\to \R$.
    \item $Y_m$ contains all $Y_{mb}$ for integers $b\geq 1$.
    \item $Y_m\subset Y$ is a closed subset, with a relatively open dense stratum $Y_m\setminus \cup_{b\geq 2} Y_{mb}$.
    \item Each $\Ymc$ is a symplectic $\C^*$-submanifold of $Y$, and its $\C^*$-action admits an $m$-th root.
    \item At $p\in \F_\a$ the tangent space $T_p Y_{\ell}$ is the $\Z/\ell$-fixed locus of the linearised action, so
\begin{equation}\label{TangentSpaceOfTorsionMfd}
T_p Y_{\ell}=\oplus_{b\in \Z}%
H_{\ell b}\subset T_p Y.	
\end{equation}
\item In a sufficiently small neighbourhood of $\F_\a$, $Y_{\ell}$ is the image
$Y_{\a,\ell}^\mathrm{loc}$ via $\exp_{\F_\a}$ %
of a small neighbourhood of the zero section in $\oplus_{b}%
H_{\ell b}\subset T_{\F_\a} Y$. Globally $Y_{\ell}=\cup_{\a} (\C^*\cdot Y_{\a,\ell}^\mathrm{loc})$.
\item There is a finite number of $Y_{m,\b},$ each of which is a union of various $\C^*\cdot Y_{\a,\ell}^{\mathrm{loc}}$.
\item $\Ymc$ has a unique minimal component $\min (H|_{\Ymc}:\Ymc\to \R)$, which is the $\F_{\alpha}\subset \Ymc$ for which all weights $mb$ in \eqref{TangentSpaceOfTorsionMfd} are non-negative. 
\item 
  The intersection $Y_{m,\b}\cap \mathrm{Core}(Y)=\mathrm{Core}(\Ymc)$ is path-connected.
  \item If $Y_{m,\b}$ contains a single $\F_\a$, it is called a {\bf torsion bundle} $\Hm\fun\F_\a$, and 
  $\Ymc$ is diffeomorphic to the weight $m$-part 
$\oplus_{b} H_{mb}=\oplus_{b\geq 0} %
H_{mb}$ of the normal bundle of $\F_\a$, and $\mathrm{Core}(\Ymc)=\F_\a$.
\end{enumerate}
\end{lm}

\begin{de}\label{Definition generic points of Ymc}\label{Lemma generic points of Ymc}
$\F_\a$ is {\bf $m$-minimal} if it is the minimal component of some $Y_{m,\b}$,
equivalently $\F_\a$ has a non-zero weight divisible by $m$ 
and all such weights are positive ($T_{p}\Ymc=T_p\F_\a\oplus \oplus_{b\geq 1} H_{mb}$ in \eqref{TangentSpaceOfTorsionMfd}). 
\end{de}

\begin{lm}\label{LemmaCanonicalFillingDiscs}\label{CorollaryGradientTrajBecomeSpheres}
Any $y\in Y$ yields a {\ph } disc
$\psi_y: \DD \to Y$, $ \psi_\x(z)=\varphi_z(y)$, $\psi_y(0)=\yinfty$.
A unitary basis $v_i$ for $T_{\yinfty}Y$ induces a canonical unitary (so symplectic) trivialisation $v_i(z)$ of $\psi_y^*TY$ with $v_i(z)=v_i$. The trivialisation is $S^1$-equivariant in $y$ in the sense that $v_i(tz)$ is the canonical trivialisation of $\psi_{ty}^*TY$ induced by $v_i$, for any $t\in S^1$. 

If the $S^1$-orbit of $y$ has minimal period $1/m$, for $m\in \N$, then $\psi_\x(z)$ is an $m$-fold cover of a {\ph } disc $\hat{\psi}_\x:\DD \to Y$, and $v_i(z)$ is induced by a canonical trivialisation of  $\hat{\psi}_{y}^*TY$, in particular it is $\Z/m$-equivariant: $v_i(\zeta z)=v_i(z)$ whenever $\zeta^m=1$.

$\mathrm{Core}(Y)$ is covered by copies of $\C\P^1$ arising as the closures of $\C^*$-orbits. The non-constant spheres are embedded except possibly at the two points where they meet $\F$ (where several different $\C\P^1$ may meet). 
The $\C^*$-orbit closure of any $y\in \mathrm{Core}(Y)$ determines a {\ph } sphere 
$$
u_y: \C\P^1 \to Y, \;\; u_y([1:t]) = \varphi_t(y),\;\; u_y([1:0])=y_0,\;\; u_y([0:1])=y_{\infty},
$$
where $y_0$ is the convergence point of $y$, and $y_{\infty}:=\lim_{\C^* \ni t\fun\infty} t \cdot y$ 
In particular,
$$
\C\cong \mathrm{im}\, d u_y|_{[1:0]} \subset T_{y_0}Y\;\;\mathrm{and}\;\; \C\cong \mathrm{im}\, d u_y|_{[0:1]} \subset T_{y_{\infty}} Y
$$
are weight subspaces of opposite weights $k$ and $-k$ respectively, for some integer $k\geq 0$ (with $k=0$ precisely if $u_y\equiv y \in \F$ is constantly equal to a fixed point).
\end{lm}

\begin{de}\label{Subsection Attraction graphs}\label{defnExtendedAttractionGraph} The \textbf{attraction graph} $\G_{\Fi}$ %
is a directed graph:\;the $\F_\a$ label the vertices; the number of edges from 
	 $\a_1$ to $\a_2$ is the number of connected components of the space of $\C^*$-flowlines from $\F_{\a_1}$ to $\F_{a_2}$ (as $t\to \infty$, so $H$ increases).
	\textbf{Leaves} of $\G_{\Fi}$ are vertices with no outgoing edges, i.e.\,if $\F_\a$ is a local maximum of $H|_{\mathrm{Core}(Y)}$.
A leaf is {\bf $m$-minimal} if the corresponding $\F_\a$ is $m$-minimal.
The \textbf{extended attraction graph}
 $\widetilde{\G}_{\Fi}$
 decorates ${\G_{\Fi}}$ %
 with an outward arrow from $\a$ for each torsion bundle $\Hm\fun \F_\a$. %
\end{de}

\begin{lm}\label{ThereAreLeaves}\label{THERE_ARE_ISOTROPIES}
The attraction graph satisfies the following properties, where $V:=\#\textrm{Vertices}(\G_{\Fi})$,
\begin{enumerate}
    \item $\G_{\Fi}$ is a connected directed acyclic graph.
    \item Any $m$-minimal leaf for $m\geq 2$ has an $m$-torsion bundle, so $\Fi$ is not free on $Y\setminus \mathrm{Core}(Y)$.
\item For CSRs with $V\geq 2$, 
every leaf $\a$ is $m_\a$-minimal for the largest weight $m_\a\geq 2$ of $\F_\a$.
    \end{enumerate}
\end{lm}

\subsection{Technical remarks:\;coefficients and monotonicity assumptions}\label{Rmk about coeffs Novikov}\label{Rmk technical symplectic assumptions on Y}\label{Remark choice of coefficients}

We do not review in detail the chain-level construction of $HF^*(F)$ (e.g.\;cf.\;\cite{R10}). 
The Floer chain complex $CF^*(F)$ is a module over a certain Novikov field $\mathbb{K}.$ 
When $c_1(Y)=0$, we will in fact work over the {\bf Novikov field}
\begin{equation}\label{EqnSec2NovikovField}
\textstyle \mathbb{K} = \{\sum n_j T^{a_j}: a_j\in \R, a_j\to \infty, n_j\in \mathbb{B} \},
\end{equation}
where $T$ is a formal variable in grading zero, and $\mathbb{B}$ is any choice of base field. In the monotone case, the same Novikov field can be used but $T$ lies in non-zero grading \cite[Sec.2A]{R16}. In other situations, e.g.\,the weakly-monotone setup, the Novikov field is more complicated \cite[Sec.5B]{R16}.

The Hamiltonian $F:Y \to \R$ needs to be {\bf admissible}: outside of a compact subset of $Y$, it is linear in $H$ with generic slope. The chain complex $CF^*(F)$ is a free $\mathbb{K}$-module generated by the $1$-periodic orbits of a $C^2$-small compactly-supported time-dependent generic perturbation of $F$. The Appendices explain the {\MBF} model for $CF^*(F)$ that avoids perturbations:\;the free generators over $\k$ are then critical points of auxiliary Morse functions on the {\MB} manifolds of $1$-orbits of $F$.

The field $\k$ in \eqref{EqnSec2NovikovField} is a $\mathbb{B}$-vector space, 
flat over $\mathbb{B}$, so $H^*(Y;\k)\cong H^*(Y;\mathbb{B})\otimes_{\mathbb{B}} \k$. \cref{Cor intro filtration} yields an ($\omega$- and $\Fi$-dependent) filtration on $H^*(Y;\mathbb{B})$ by $\mathbb{B}$-vector subspaces for any field $\mathbb{B}$. Floer/quantum cohomology is also defined over a Novikov ring $\k$ using any underlying ring $\mathbb{B}$ \cite{HS95} (by \cref{Rmk technical symplectic assumptions on Y}), and \cref{Cor intro filtration} still holds.
If one forgoes the multiplicative structure by ideals, one can more generally replace $\mathbb{B}$ by any abelian group, yielding a filtration on $H^*(Y;\mathbb{B})$ by abelian subgroups.

We always tacitly assume $(Y,\omega)$ is {\bf weakly-monotone}
so that transversality arguments in Floer theory can be dealt with by methods in Hofer--Salamon \cite{HS95}. It means one of the following holds: \begin{enumerate} 
\item $c_1(Y)(A)=0$ when we evaluate on any spherical class $A\in \pi_2(Y)$, or 
\item $\omega(Y)(A)=0$ when we evaluate on any spherical class $A\in \pi_2(Y)$, or 
\item for some $k>0$ we have $c_1(Y)(A)=k \cdot \omega(A)$ for all $A\in \pi_2(Y)$, or 
\item the smallest positive value $c_1(Y)(A)\geq n-2$ where $\dim_{\R}Y=2n$.\end{enumerate} 

Case (2) holds if $\omega$ is exact; (3) is the {\bf monotone} case.
In \cref{Prop vanishing of SH} we use {\bf weak+ monotonicity} \cite[Sec.2.2]{R14} which means the same as above except $n-2$ in (4) becomes $n-1$.

We do not use the Floer cohomology of the base $B$, so we do not require that $B$ is weakly-monotone when $Y$ is globally defined over a convex base (of course $B^{\mathrm{out}}$ is exact).

\subsection{Floer-theoretic indices}
\label{Section RS indeces}\label{Subsection discussion of ma and muFalapha}\label{GradingConventions}\label{AppendixFloer}\label{RSIndeces}
Let $Y$  be a
symplectic $\C^*$-manifold with $c_1(Y)=0$. After choosing a trivialisation of the canonical bundle %
$\Lambda^{\mathrm{top}}_{\C}T^*Y$, Hamiltonian Floer cohomology (and thus $SH^*(Y,\Fi)$) can be $\Z$-graded. 
The Hamiltonian $S^1$-action $\varphi$ admits a \textbf{Maslov index} $\mu$, discussed in \cite{RZ1}.

\begin{lm}\label{LemmaSec2MaslovIndex}
Let $T_x Y = \oplus \C_{w_i}$ be the weight decomposition for the linearised $S^1$-action at a fixed point $x\in \F$, where $\C_{w_i}$ is a copy of $\C$ with the standard $S^1$-action of weight $w_i\in \Z$. Then
$\mu=\sum w_i>0$.
\end{lm}

In \cite[Appendix A]{RZ1} we discussed conventions about {\RS} indices $\mu_{RS}$, and we now recall its properties. It is a $\tfrac{1}{2}\Z$-valued index defined for any continuous path $[0,1] \rightarrow Sp(\C^{n},\Omega_0)$ of real symplectic matrices in $\C^{n}$ with the standard real symplectic structure $\Omega_0$. Define $\W :\R \to \Z$ by
\begin{equation}\label{WfunctionForRSAppendix}
	\qquad \W(x) := \left\{
	\begin{array}{ll}
	2 \lfloor x \rfloor + 1 & \text{if} \ x \notin \Z \\
	2x & \text{if}  \ x \in  \Z.
	\end{array}
	\right.\qquad
    \textrm{Thus: }\;
    \begin{array}{ll}
    2x\geq \W(x)-1\geq 2x-2,\; \W(-x)=-\W(x)
	\\
\W(x)\textrm{ is odd except at }\Z, \textrm{ and }\W(0)=0. 
	\end{array}
	\end{equation}

\begin{thm}\label{RSProperties} The {\RS } index satisfies the following properties:
\begin{enumerate}[(1)]
	\item \label{RShomotopiesproperty}$\mu_{RS}$ is invariant under homotopies with fixed endpoints.
	\item \label{catenation}$\mu_{RS}$ is additive under concatenation of paths.
	\item \label{RSofSUM} For a direct sum $\psi_1\diamond\psi_2:\C^{n} \oplus \C^{m} \fun \C^{n} \oplus \C^{m}$, $\mu_{RS}(\psi_1\diamond\psi_2) = \mu_{RS}(\psi_1)+\mu_{RS}(\psi_2).$ 
	\item \label{prop:invariance}
	For $\psi , \phi$ two continuous paths of symplectic matrices, $\mu_{RS}(\phi\psi\phi^{-1})=\mu_{RS}(\psi).$
	\item \label{RSofRotationInC}
	$\mu_{RS}((e^{ 2\pi is})_{s \in [0,x]})=\W(x)$.
	\item \label{shear}The {\RS } index of the \textit{symplectic shear}
    $\Big(\begin{smallmatrix} 1&0\\
		b(t)&1 \end{smallmatrix} \Big)$
    is equal to $\tfrac{1}{2}(\mathrm{sign}(b(1))-\mathrm{sign}(b(0))).$
\end{enumerate}	
\end{thm}

For any 1-orbit $x$ of a Hamiltonian $F$ in a symplectic manifold $M$ of dimension $2n$, let $\phi_t$ be the Hamiltonian flow and consider its linearisation 
$(\phi_t)_*: T_{x(0)}M \fun T_{x(t)}M$.
We pick a symplectic trivialisation $\Phi: x^* TM \fun \C^n \times S^1$ 
of the tangent bundle above the orbit $x$ to get a path of symplectic matrices $\psi(t)=\Phi_t \circ (\phi_t)_* \circ \Phi^{-1}_0: \C^n \fun \C^n.$
Then define $RS(x,F)$ and the {\bf grading}\footnote{That convention ensures that for a $C^2$-small Morse Hamiltonian, $|x|$ equals the Morse index.} $|x|$ of $x$ by
\begin{equation}\label{RSForAnOrbit} 
	RS(x,F):=\mu_{RS}(\psi) \qquad \textrm{ and } \qquad 
 |x|:=n - RS(x,F).
\end{equation}
For the Hamiltonian $\lambda H,$ the \textbf{grading} of a connected component $\F_\a$ of $\F$ is defined as
$$\mu_\lambda(\F_\a):=\dim_\C\, Y - \dim_\C\, \F_\a- RS(x,\lambda H), $$ which is independent of the choice of $x\in\F_\a$. This $\mu_\lambda(\F_\a)$ is the Floer grading of $\F_\a$ seen as a {\MB} manifold of Hamiltonian $1$-orbits of $\lambda H$ (\cref{AppendixCascades}). 

\begin{prop}\label{PropSec2RSIndices}\label{indecesgotoinfinity}
For a generic slope $\lambda$, the only $1-$periodic orbits of the Hamiltonian $\lambda H$ are the constant orbits, i.e.\,the fixed points $x\in \F$, and $\lim_{\lambda\fun +\infty} RS(x,\lambda H)=+\infty$ uniformly in $x$.
\end{prop}

 \begin{lm}\label{LemmaSec2GradingOfFa}
 $\mu_\lambda(\F_\a)$
are even integers for positive $\lambda \notin \cup_i  \Z\cdot \frac{1}{w_i}\subset \Q,$ 
where $w_i$ are weights of $\F_\a.$ 
\end{lm}

We abbreviate $|V|:=\dim_{\C}\,V$.
In the weight decomposition \eqref{Intro weight spaces}, let $h_k^{\alpha}=|H_k|=\#\{i: w_i=k\}$.
\begin{cor}\label{mu_lambda go to -infty}\label{Lemma difference between muFalpha and mua}\label{For Lambda small Floer Shift Equal To Morse}
 The {\MB } index $\mu_\a$ of $\F_\a,$ the Maslov index $\mu$, and the $\mu_{\lambda}(\F_\a)$ above, satisfy: 
\begin{equation}\label{mu_lambda_calculation}
\begin{split}
&\mu_\a=2 \sum_{k<0} h_k^{\alpha}
\qquad\qquad\qquad
\mu = \sum k\cdot h_k^{\alpha} = \sum_{k>0} k(h_k^{\alpha}-h_{-k}^\a)
\\
&\mu_\lambda(\F_\a) = |Y| - |\F_\a| - \sum_k \mathbb{W}(\lambda k)\,h_k^{\alpha} 
 = \sum_{k\neq 0} (1-\mathbb{W}(\lambda k))\,h_k^{\alpha}
 = \sum_{w_i\neq 0} (1-\mathbb{W}(\lambda w_i))
 \end{split}
\end{equation}
\begin{enumerate}
\item 
$\mu_\lambda(\F_\a)\leq 2|Y|- |\F_\a| - 
2\lambda \mu$ for any $\lambda>0$.
\item $\mu_\lambda(\F_\a)\to-\infty$ as $\lambda\to \infty$.
 \item  $\mu_\a - \mu_{\lambda}(\F_\a) = \sum_{k>0} (\mathbb{W}(\lambda k)-1)(h_{k}^{\alpha} - h_{-k}^{\alpha})$.
\item  $\mu_\lambda(\F_\a)=\mu_\a$ for all $0<\lambda < \lambda_{\alpha}:=\min\{\tfrac{1}{|k|}: h_k^{\alpha}\neq 0 \textrm{ for }k\in \Z\setminus \{0\}\}$.
\item $\mu_\lambda(\F_\a)=\mu_\a-2N\mu$ for all $ N <\lambda < N+\lambda_{\alpha}$, where $N\in \N$.
\end{enumerate}
\end{cor}

\begin{de}\label{Remark local rank}\label{Definition critical times}
\label{Definition compatibly weighted}
Let $k_1>k_2>\cdots > 0$ denote the absolute values of the non-zero weights $w_i$ at $\F_\a$ in descending order without repetitions.
We call $Y$ {\bf compatibly-weighted} if $\delta_r^\a:=\sum_{j=1}^r (h_{k_j}^\a - h_{-k_j}^\a)\geq 0$ for all $r$, $\a$.
E.g.\;all CSRs are compatibly-weighted, due to \eqref{NonDegenPairingByOmC}.

Let $\tau_p:=\tfrac{k_0}{m}$, for $k_0,m$ coprime positive integers, so $\tau_p k \in  \Z$ precisely when $k=mb$ for some $b\in \Z$.  
Call $\tau_p$ a {\bf critical time} if some $h_{mb}^{\alpha}\neq 0$ for some $b\in \Z$. 
Call $\tau_p$ an {\bf $\alpha$-critical time} if 
some $mb$ is a weight of $\F_\a$ (so $\lambda_\a$ is the first $\alpha$-critical time).
Note that $\mu_{\lambda}(\F_\a)$
can only change when $\lambda$ is an $\alpha$-critical time (that is when $\mathbb{W}(\lambda k)$ can jump/drop).
We abusively call $0<\lambda<\min_{\alpha}\lambda_\a$ all {\bf 0-th critical times}, and $0<\lambda<\lambda_\a$ all 0-th critical $\alpha$-times (for those, $\mu_\a=\mu_{\lambda}(\F_\a)$).

Note: if both $\F_\a,\F_\gamma\subset Y_{m,\c}$ then $\mu_{\lambda}(\F_\a)=\mu_{\lambda}(\F_\gamma)$ for the critical time $\lambda=\tfrac{k_0}{m}$, as those indices equal the (homotopy-invariant) Floer grading of the $1$-orbits of $\lambda H$ whose initial points lie in $Y_{m,\c}$.

For a torsion submanifold $\Ymc$, with minimal component $\F_\a$ (\cref{Definition generic points of Ymc}),
call
$$\textstyle  \mathrm{rk}(Y_{m,\c}):=|Y_{m,\c}|-|\F_\a|=\sum_{b\geq 1} h_{mb}^{\alpha}$$
 the {\bf rank} of $\Ymc$. When $Y_{m,\c}$ is a torsion bundle, this is its complex rank. %
\end{de}

\begin{prop}\label{Lemma muFa is smaller than MB index mua}
If $Y$ is compatibly-weighted (\cref{Definition compatibly weighted}) then
    $\mu_{\lambda}(\F_\a)\leq \mu_\a$ for all $\lambda>0.$
\end{prop}

\begin{nota}
We write $\lambda^+$ to mean a non-critical time just above a given critical time $\lambda=\tfrac{k_0}{m}$ ($k_0,m$ coprime), with no critical times between $\lambda$ and $\lambda^+$. Similarly for $\lambda^-$ below $\lambda$. 
\end{nota}
\begin{lm}\label{Lemma Fma is jump in grading}
\begin{enumerate}
    \item \label{mu lambda before and after crit value difference}
     At a critical time $\lambda=\tfrac{k_0}{m}$, $\gcd(k_0,m)=1$
     we have
     $$\textstyle\sum_{b\geq 1}(h_{mb}^\a-h_{-mb}^\a) = 
\mu_{\lambda^-}(\F_\a) - \mu_{\lambda}(\F_\a)= \mu_{\lambda}(\F_\a) - \mu_{\lambda^+}(\F_\a)=\frac{1}{2} (\mu_{\lambda^-}(\F_\a)-\mu_{\lambda^+}(\F_\a));$$
\item \label{m-minimal Fma equal to rk Ymb} if $\F_\a$ is $m$-minimal (\cref{Definition generic points of Ymc}), then the above sum equals $\mathrm{rk}(Y_{m,\c})$, so:
$$
\mu_{\lambda^+}(\F_\a) = \mu_{\lambda}(\F_\a) -
\mathrm{rk}(Y_{m,\c})\quad
\textrm{ and } \quad
\mu_{\lambda^-}(\F_\a) = \mu_{\lambda}(\F_\a) +
\mathrm{rk}(Y_{m,\c}).
$$
\item \label{Shifts just before the integer time} In particular, denoting $H_+=\oplus_{k>0} H_k$ and $H_-=\oplus_{k<0} H_k$ the subspaces of \eqref{Intro weight spaces}, $$\mu_{N^-}(\F_\a)=\mu_\a-2N\mu+2(|H_+|-|H_-|)\;\; \textrm{ (for weight-1 CSRs: }\mu_{1^-}(\F_\a)=0).$$
\end{enumerate}
\end{lm}

\subsection{Review of Conical Symplectic Resolutions} \label{CSRs}
	A \textbf{weight-$s$ CSR} is a projective resolution\footnote{Meaning: $\M$ is a smooth variety, and $\pi$ is an isomorphism over the smooth locus of $\M_0$.} $\pi:\M \fun \M_0$ of a normal\footnote{The normality assumption on $\M_0$ is %
equivalent to the condition that $\pi:\M\fun \M_0$ is the affinisation 
map, $\M\fun \Aff(\M):=Spec(H^0(\M,\O_{\M})), \  p \mapsto \{f \mid f(p)=0\},$ see \cite[Lem.3.16]{vzivanovic2022exact}.}
	affine variety $\M_0$, where $(\M,\om_\C)$ is a holomorphic symplectic manifold and $\pi$ is equivariant with respect to $\C^*$-actions on $\M$ and $\M_0$ (both denoted by $\Fi$). We require:
	\begin{enumerate}
		\item[(1)] The complex symplectic form $\om_\C$ has a \textbf{weight} $s\in \N$, so 
  $\Fi_t^* \om_\C = t^s \om_\C$ for all $t\in \C^*$.
		\item[(2)]\label{contracts} The action $\Fi$ contracts $\M_0$ to a single fixed point $x_0$, so $\forall x \in \M_0, \displaystyle \lim_{t\fun 0} t \cdot x=x_0.$
		Algebraically, $\C[\M_0]=\bigoplus_{n\geq 0}\C[\M_0]^n$ {and} $\C[\M_0]^0=\C,$ where $\C[\M_0]^n$ denotes the $n$-weight space.\footnote{explicitly
			$\C[\M_0]^n:=\{f \in \C[\M_0] \mid (t \cdot f)(x)=f(t\cdot x)=t^n f(x)\}.$} 
	\end{enumerate}
Actions satisfying (1)-(2) are called \textbf{weight-$s$ conical actions.}

We summarise below the properties of CSRs and refer to \cite{RZ1} for detailed references.

\begin{thm}\label{CSRproperties}\label{LemmaCalabiYau}\label{CohomologyOfACSRProperties}\label{CSR_definite_intersection_form}\label{quantumequaltocup}\label{CorFiltrForCSR}
\label{LemmaMomentMapMorseBott}
\label{GradientIsRPlus}\label{MomentMapExhausting} 
\label{PropMapPsi}

Let $\pi:\M\fun\M_0$ be a weight-s CSR.
	\begin{enumerate}
 \item
	$(\M,\Fi)$ is a symplectic $\C^*$-manifold with 
        an $S^1$-invariant $I$-compatible {\Kh} structure $(g,I,\om_I),$ 
        whose $S^1$-action is Hamiltonian, and its moment map $H$ is proper.
        \item $(\M,\Fi)$ is globally defined over a convex base via a map $\Psi: \M \to \C^N$ (\cref{Def:KahlerMfdWithProjection}).
\item $c_1(T\M,I)=0,$ where $I$ is its complex structure.
\item \label{Thm_CSR_Maslov_Index_formula}
The Maslov index of the $S^1$-action is
	$\mu=\frac{1}{2}s\cdot\dim_{\C} \M.$

	    \item \label{centralfibrecore} Its core $\L:=\{p\in \M \mid \lim_{\C^* \ni t\fun \infty} t\cdot p \text{ exists}\}$ is the central fibre, $\L=\pi^{-1}(0).$
\item \label{isotropicCore}
 $\L$ is an $\om_\C$-isotropic and when $s=1,$
 an $\om_\C$-Lagrangian subvariety (usually singular).
     \item \label{Thm_CSR_Sing_Coh_Over_Char_Zero} Singular cohomology 
     $H^*(\M)$ over characteristic zero %
    fields
	lies in even degrees $\leq \dim_\C \M.$ 
    \item \label{Thm_CSR_definite_intersection_form}
    The intersection form $H_{\dim_\C \M}(\M) \times H_{\dim_\C \M}(\M) \fun \Q$ is definite.
 \item $QH^*(\M)\iso H^*(\M,\k)$ as rings.
 \item The $\Fi$-filtration in \cref{Cor intro filtration}, ordered by $p\in \R$, becomes
$$\textstyle
\Fil^{\varphi}_{p}:=\bigcap_{\textrm{generic }\lambda>p} \;\,\ker(c_{\lambda}^*:H^*(\M,\k) \fun HF^*(H_{\lambda })).
$$
It is a filtration on singular cohomology $H^*(\M,\k)$
by ideals with respect to the cup product.
	    \item \label{WeightDecompCSR}\label{NonDegClaimCSR} For the
     weight decomposition
	$T_p \M = \oplus_{k\in\Z} H_k$ at $p\in\F_\a,$ 
     the following pairing is non-degenerate:
	\begin{equation}\label{NonDegenPairingByOmC}
	 \qquad \qquad\om_\C: H_k \oplus H_{s-k} \fun \C.   \qquad (\omega_{\C}\textbf{-duality})
	\end{equation}
  \item \label{Thm_CSR_weight_1_number_of_Fa_equals_Htop}When $s=1$, we have\;\; $2\dim_{\C}\F_\a + \mu_\a = \dim_{\C} \M$\;\; and \;\;$ \mathrm{rk}(H^{\dim_{\C}\M}(\M))=\# \{\F_\a\} \neq 0$.
\end{enumerate}
\end{thm}

\begin{rmk}\label{RmkFunctionPsi}
The proper $\C^*$-equivariant holomorphic map $\Psi$ can be constructed as follows: $$\Psi=\Theta \circ j \circ \pi=(\pi^*(f_1)^{w/w_1},\ldots,\pi^*(f_N)^{w/w_N}): \M \to \C^N,$$
where $f_1,\dots,f_N$ is any choice of non-constant homogeneous polynomial generators of $\C[\M_0]$, with weights $w_i>0$.
Thus $\C^*$ acts diagonally on $\C^N$ with weight $w:=\text{lcm}(w_1,\dots,w_N)$, but we can use the rescaling trick in \cref{Rmk weight w of Reeb} to get rid of $w$ in \eqref{Equation psi XS1 is Reeb}. The map $\Theta \circ j:\M_0\to \C^N$ is
a $\C^*$-equivariant holomorphic map, and a local embedding except at $0\in \M_0$.
Above, $j$ is the embedding $j:\M_0 \to \C^N$, $p\mapsto (f_1(p),\dots,f_N(p)),$ with $j(0)=0$, and $\Theta:\C^N \to \C^N$, $\Theta(z_1,\dots,z_N)=(z_1^{w/w_1},\dots, z_N^{w/w_N}).$

Let $\hat{\pi}:=3.1415\ldots$. The $S^1$-action $e^{2\hat{\pi} it}$ on $\C^N$ has Hamiltonian $\hat{\pi} (|z_1|^2+\cdots + |z_N|^2)$, and we define
\begin{equation}\label{Equation Phi function}
    \Phi:=\hat{\pi}\,\Psi^*(|z_1|^2+\cdots + |z_N|^2)=\hat{\pi}\sum \pi^*(|f_i|^{2w/w_i}):\M\to \R,
\end{equation}
so $\L=\Phi^{-1}(0)$. The function $\Phi$ is typically not related to the moment map $H:\M\to \R$. 
\end{rmk}

CSRs are rarely ever convex at infinity, because any choice of $I$-compatible (real) symplectic form $\omega$ on $\M$ is non-exact at infinity unless $0\in\M_0$ is an isolated singularity; indeed a CSR is convex at infinity only if it is isomorphic to $\C^{2n}$ for some $n$, or it is a minimal resolution of an ADE singularity.

\section{Filtration on the Floer chain complex}\label{filtrationFloer}

In this Section, $(Y,\omega,\J,\Fi)$ is a symplectic $\C^*$-manifold over a convex base $B$. Recall that we can always tweak $\omega$ by \cref{Lemma making H proper} to make $H$ proper in \eqref{Equation intro moment map}, which we assume from now on.
\subsection{Overview}
Our goal is to construct positive symplectic cohomology $SH^*_+(Y)$ and a {\MBF} spectral sequence converging to $H^*(Y),$ from which to read-off the filtration in \cref{FiltrationOnSingCohomology}.
We construct a specific Hamiltonian
$H_\lambda$ with generic slope $\lambda$ which filters $CF^*(H_\lambda)$ by the value of $H$.
For simplicity, in this overview, suppose $\Psi: Y \to B$ is globally defined (Definition \ref{Def:KahlerMfdWithProjection}).
We use the $2$-form
\begin{equation}\label{Equation eta expanded}
\eta := d(\phi(R)\alpha) = \phi(R)\, d\alpha + \phi'(R)\, dR\wedge \alpha
\end{equation}
on the base $B$, for a non-decreasing cut-off function $\phi$, with $\phi=0$ on $B^{\mathrm{in}}$. This form is non-negative on Floer solutions in $B$, and this in turn leads to a functional that filters the Floer complexes for $B$ (a trick introduced in \cite[Sec.6]{McLR18} for symplectic manifolds convex at infinity). We cannot apply this directly to $Y$ as it is not convex at infinity, so we try to exploit the projection $\Psi$ to $B$ where it does apply. Unfortunately, the projection of a Floer solution in $Y$ may not be a Floer solution in $B$. Even for a Hamiltonian $F=c(H): Y \to \R$ which is just a function of the moment map $H$, the Hamiltonian vector field $X_F=c'(H)\,X_{S^1}$ projects via $\Psi$ to a vector field $c'(H)\, \mathcal{R}_B$ on $B$ which still depends\footnote{$H$ is in general unrelated to the pulled back Hamiltonian $\Psi^*H_B=H_B\circ \Psi$ since 
$\Psi^*\om_B$ and $\om$ typically differ on $Y$.} on the coordinates of $Y$ via $H$. For a Floer solution $u=u(s,t)$ in $Y$, letting 
\begin{equation}\label{Equation kst definition}
k(s,t)=c'(H\circ u), 
\end{equation}
this means that $c'(H)\, \mathcal{R}_B=X_{k(s,t)\cdot H_B}$ is a domain-dependent Hamiltonian vector field. 
So the second key idea is to create regions in $Y$ where $c'(H)$ is constant, thus $k(s,t)$ is domain-independent there, and to pick $\phi$ suitably on these regions so that $\eta$ is non-negative on projected Floer solutions. Then if $u$ crosses these regions it will cause a strict decrease in the filtration functional.

If $\Psi: Y \to B$ is not globally defined, we only have $\Psi:Y^{\mathrm{out}}\to \Sigma \times [R_0,\infty)$ at our disposal (see Definition \ref{Def:KahlerMfdWithProjection}) so the projection $v=\Psi\circ u$ is not defined at $(s,t)$ if $u(s,t)\notin Y^{\mathrm{out}}$.
However, in all our computations in $B$ the terms involving $v(s,t)$ will contain a factor $\phi$ or $\phi'$, and so these terms can be extended to be zero when $u(s,t)\in Y^{\mathrm{in}}$. More intrinsically on $Y$, our computations can be rephrased using the form $\Psi^*\eta$, and this form is well-defined on all of $Y$ by extending it to be zero on $Y^{\mathrm{in}}$, since we assume $\phi=0$ near $R=R_0$ (compare \cref{Lemma making H proper}).

To avoid over-complicating this Section further, we will postpone the issue of transversality for Floer solutions to \cref{Section Transversality for Floer solutions}.
Although classical methods involving perturbations of $I$ (respectively $c$) achieve transversality, such perturbations affect the holomorphicity of $\Psi$ (respectively the equation $\Psi_*X_F=k(x,t)\,\mathcal{R}_B$), which are used throughout this Section.

\subsection{Construction of a specific  Hamiltonian $H_\lambda$}\label{SectionConstructionOfHLambda}
Let $(Y,\omega,\J,\Fi)$ be a symplectic $\C^*$-manifold over a convex base.
We have an $S^1$-equivariant proper holomorphic map $\Psi: Y^{\mathrm{out}} \to \Sigma \times [R_0,\infty)$, but we abusively write
$
\Psi:Y \fun B
$
and it is understood that all constructions involving $\Psi$ are only defined on $Y^{\mathrm{out}}$. For example, the pull-back of the Reeb flow on the base $B$,
$$
\Phi := \Psi^*H_B=H_B\circ \Psi: Y \to \R,
$$
is abusively written: $\Phi$ is only defined on $Y^{\mathrm{out}}$, and  $\Phi=R\circ \Psi$ there. The abusive notation is to preserve the intuition coming from the setup where $\Psi$ is globally defined, such as CSRs (in which case $B=\C^N$ and $\Phi$ is as in \cref{RmkFunctionPsi}).

Although $\Psi_*X_{S^1}=\mathcal{R}_B$ we emphasise that $\Phi$ and $H$ are usually unrelated as $\Psi$ is not symplectic, and their level sets are in general unrelated as $\|X_{\R_+}\|$ is typically not constant on a level set of $H$.

\begin{figure}[ht]
	\centering
	{
  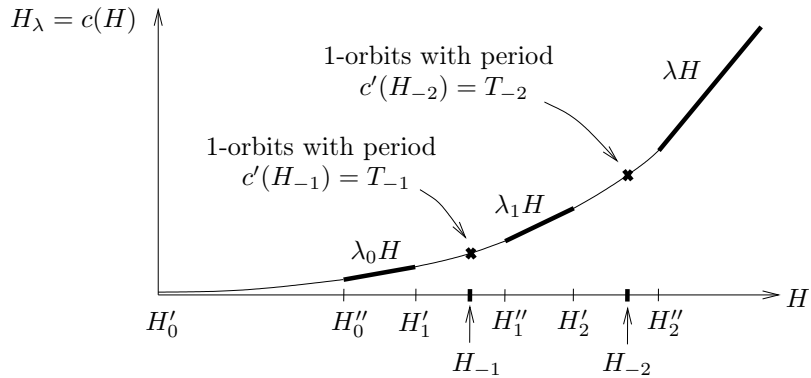
		\caption{Graph of $H_\lambda$ for $r=2$}
		\label{H_lambdagrapj}
	}
\end{figure}

The Hamiltonian $H_\lambda$ will be constructed as in Figure \ref{H_lambdagrapj} in terms of a function $c$ of $H$,
$$ H_\lambda:=c \circ H.$$
This ensures that $X_{H_{\lambda}}=c'(H)\cdot X_H$, so $1$-periodic Hamiltonian orbits of $H_{\lambda}$ corresponds precisely to orbits of period $T=c'(H)$ of the flow of $X_{S^1}=X_H$. 
The construction of $c:[\min H,+\infty)\fun \R$ involves a choice of values $\{R_0',R_0'',\dots,R_r',R_r''\}$ for $\Phi$ and $\{H_0',H_0'',\dots,H_r',H_r''\}$ for $H$, which we describe later, as well as a choice of generic slopes $\lambda_i$ lying in between the periods $T_{-i}<\lambda$ of the $S^1$-action, %
so 
$$0=T_0<\lambda_0<T_{-1}<\lambda_1<T_{-2}<\dots<T_{-r}<\lambda_r:=\lambda.$$
The period $T_0=0$ refers to the constant orbits, i.e. the fixed loci $\F_\a.$
As illustrated in \cref{H_lambdagrapj}, over the intervals $[H_i',H_i'']$ the slope $c'(H)$ can increase and in particular within this interval we ensure $H$ attains precisely one value $H_{-i}\in  (H_i',H_i'')$ (for $i\geq 1$) at which the slope is a period,
\begin{equation}\label{Eqn c' is T}
c'(H_{-i})=T_{-i},
\end{equation}
so the non-constant $1$-periodic orbits of $H_\lambda$ arise precisely in the regions $H=H_{-i}$.

We construct $c:[\min H,+\infty)\fun \R$ so that
\begin{enumerate}
	\item[(1)]  $c'\geq 0$.
	\item[(2)]  $c''\geq 0$.
	\item[(3)]  $c''(H)>0$ whenever $c'(H)=T_{-i}$, for all $i=1,\dots,r$. \label{Condition3}
	\item[(4)]  $c(H)=\lambda_i H,$ when $H\in [H_i'',H_{i+1}'],$ for all $i=0,\dots,r-1$.
	\item[(5)]  $c(H)=\lambda_r H$ for $H\in[H_r'',+\infty).$
\end{enumerate}

We also assume that $c'$ is sufficiently small on $Y^{\mathrm{in}}$ so that there are no non-constant $1$-periodic orbits of $H_{\lambda}$ in $Y^{\mathrm{in}}$ since the potential period values $c'$ are smaller than $T_{-1}$, thus in $Y^{\mathrm{in}}$ the only $1$-orbits are constants at points of the $\F_\a$ submanifolds. More precisely, by properness of $H$ we can assume that $$Y^{\mathrm{in}}:=\{H\leq \ell\}$$ is a sublevel set, where we choose $m$ large enough so that 
$Y^{\mathrm{in}} \supset H^{-1}(H(\mathrm{Core}(Y)) \supset \mathrm{Core}(Y).$ %
We construct $c'$ to be small on $[\mathrm{min }(H),\ell]$, and ensure that $c'\neq 0$ except possibly at  
$\mathrm{Crit}(H)$, so that $\mathrm{Crit}(H)=\mathrm{Crit}(H_{\lambda})$, and so that $H_{\lambda}$ has the same (constant) $1$-orbits as $H$ in that region.
In particular, (1) and (2) above are really only needed for $Y^{\mathrm{out}}=\{H\geq \ell\}$. 

When $Y=\M$ is a CSR, $\Psi$ is globally defined over $B=\C^N$, and $H_B=\pi w\|z\|^2$ is defined everywhere (the $S^1$-action has weight $w$, see \cref{RmkFunctionPsi}), and one can pick $\M^{\mathrm{in}}=\{H\leq \ell\}$ to be any sublevel set containing the core $\L=\Psi^{-1}(0)$.
We recall that 
by the rescaling trick in Remark \ref{Rmk weight w of Reeb} we can get rid of the factor of $w$ that would appear in \eqref{Equation psi Xs1 is Reeb} for CSRs.

\begin{lm}
The level sets $\Phi^{-1}(q)=\Psi^{-1}(\{H_B=q\})\subset Y$ 
for $q\geq R_0$ are closed submanifolds of $Y.$
\end{lm}
\begin{proof}
 $d\Phi=dH_B \circ \Psi_*$ is non-zero on $\nabla H=X_{\R_+}$ as $\Psi_*X_{\R_+}=\nabla H_B$ (by \cref{Lemma:XR+ is nabla h}). 
\end{proof}

As in Figure \ref{h(R)graph} we pick a function $h(R)$ of $R$ which is linear of slope $\lambda_i>0$ over the interval
$$
L_i := [R_i'',R_{i+1}'],
$$
where $L$ stands for ``linear'', and $N_i := [R_i',R_i'']$ are the intervals where $h$ can be non-linear.
	\begin{figure}[ht]
	\centering
	{
 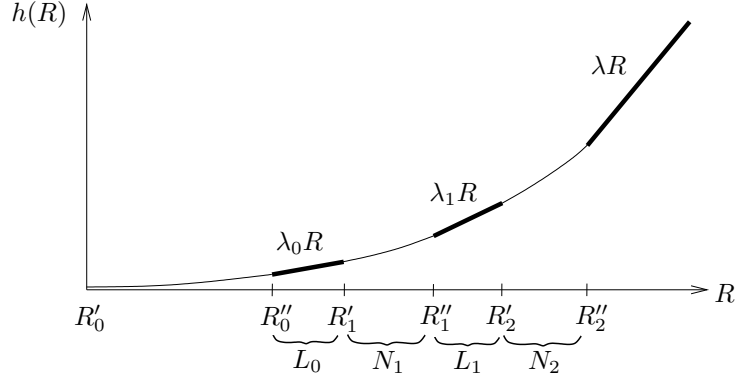
		\caption{Graph of $h(R)$ for $r=2$}
		\label{h(R)graph}
	}
\end{figure}

Explicitly, $h:[R_0,\infty)\fun [0,+\infty)$ is required to satisfy
\begin{enumerate}
	\item[(1)] $h'\geq 0$.
	\item[(2)] $h'' \geq 0$, and $h''(R)>0$ if there is a $1$-periodic orbit of $X_h$ in $\Sigma\times \{R\}$\\ (equivalently, if there is a Reeb orbit in $\Sigma$ of period $h'(R)$).
	\item[(3)] $h'(R)=\lambda_i$ 
	on $L_i$ for $i=0,\ldots,r-1$.
	\item[(4)] $h'(R)=\lambda_r$
	for $R\geq R''_r$.
\end{enumerate}

We also pick $h$ to be $C^2$-small and Morse near $R=R_0$. The intuition here is that if $\Psi$ were globally defined, then we would extend $h$ to a Hamiltonian $$h: B \to \R$$ which is $C^2$-small and Morse on $B^{\mathrm{in}}$, and which on the conical end $B^{\mathrm{out}} \cong \Sigma \times [R_0,\infty)$ becomes the function $h(R)$ that depends only on the radial coordinate $R\in [R_0,\infty)$.

We will choose $R_i', H_i', H_i'', R_i''$ so that the sublevel sets of the functions $H$ and $\Phi$ are ``nested'':
$$
\{H\leq H_i''\} \subset
\{\Phi\leq R_i''\} \subset
\{\Phi\leq R_{i+1}'\} \subset
\{H\leq H_{i+1}'\} \subset
\{H\leq H_{i+1}''\},
$$
starting with $R_0'=R_0$ and $H_0'=\ell,$ choosing them so that $\{\Phi \leq R_0 \} \subset \{H \leq m\}.$ %
The nesting condition ensures that $H_\lambda=c(H)=\lambda_i H + \mathrm{constant}$ 
in a region of $Y$ that covers (via $\Psi$) the region in $B$ where $R\in L_i$, so where $h(R)=\lambda_i R + \mathrm{constant}$.
That this nesting can be achieved follows from $\Phi$ and $H$ being proper. The precise choices are not so important, but we need to carry out this choice. 

We fix constants $\varepsilon_0>0$ and $\varepsilon_i>0$ (we choose these in Lemma \ref{LemmaDeltaiFiltration}), where $\varepsilon_i$ is the width of $L_i$, $$\varepsilon_i=\textrm{width}(L_i)=R_{i+1}'-R_{i}''.$$
 Define $H_0'':=H_0'+\varepsilon_0$ and $R_0'':=\max \Phi(H^{-1}(H_0''))+\varepsilon_0,$ and for $i=1, \dots ,r$:
\begin{align*}
	R_i':=&R_{i-1}''+\varepsilon_{i-1}\\
	H_i':=&\max H(\Phi^{-1}(R_i'))+\textrm{(any non-negative constant)}\\
	H_i'':=&H_i'+\textrm{(any positive constant)}\\
	R_i'':=&\max \Phi(H^{-1}(H_i''))+\textrm{(any non-negative constant)}.
\end{align*}
Also a choice of value $H_{-i}$ with $H_i'<H_{-i}<H_i''$ is involved, where \eqref{Eqn c' is T} holds. The choices of the bracketed constants above are irrelevant, so we avoid further notation. One can choose the bracketed constants to be arbitrarily small, and thus make $N_i=[R_i',R_i'']$ arbitrarily small, while still being able to make $\varepsilon_i=\textrm{width}(L_i)$ arbitrarily large. Such a choice will begin to matter in Lemma \ref{LemmaDeltaiFiltration}.

\subsection{The cut-off functions $\phi$}

We now choose a smooth cut-off function $\phi:[0,+\infty)\fun \R$ that is zero for small $R$, grows on the ``linear intervals'' $L_i=[R_i'',R_{i+1}']$, and is constant on the ``non-linear intervals'' $N_i = [R_i',R_i'']$, in particular $\phi'= 0$ outside of the $L_i$. In summary:
\begin{enumerate}
	\item[(1)] $\phi=0$ on $[0,R_0'']$.
	\item[(2)] $\phi'\geq 0$ everywhere.
    \item[(3)] $\phi'=0$ on $N_i$ for $i=1,\ldots,r-1$ (so $\phi$ is constant there).
\end{enumerate}
For the purposes of constructing a filtration on $CF^*(H_{\lambda})$ the choice of $\phi$ on $R\geq R_r'$ is irrelevant, as Floer solutions do not enter the region above $R>R_r'$ by the maximum principle. 
We emphasise that the choice of $\phi$ depends on the construction of $H_{\lambda}$, or more precisely on the choices of $R_i',R_i''$.
We let $\phi_i\geq 0$ denote the amount by which $\phi$ increases when crossing the interval $L_i$, so $$\phi=\phi_0+\phi_1+\cdots +\phi_{i-1} \quad \textrm{on}\;N_i.$$ 
We will always assume $\phi_0>0$. Later, starting from Lemma \ref{LemmaDeltaiFiltration}, it will be important to choose $\phi_i\geq 0$ to be sufficiently large, and we will also make $\varepsilon_i=\textrm{width}(L_i)$ large and $\phi'\geq 1$ on a large subinterval of $L_i$.
But for now, for everything we carry out before Lemma \ref{LemmaDeltaiFiltration}, one could just choose $\phi$ to be constant on $R\geq R_1'$ (so $\phi_i=0$ for $i\geq 1$). In particular, this simple choice of $\phi$ works in \cref{PositiveSH}.

\subsection{Filtration functional on $B$}
\label{SubsectionFiltrationFunctional} 
The cut-off function $\phi$ defines the exact 2-form $\eta$ from \eqref{Equation eta expanded} on $B$ (where $B$ is abusive notation meaning $B^{\mathrm{out}}=\Sigma \times [R_0,\infty)$ when $\Psi$ is not globally defined),
and an associated $1$-form $\Omega_{\eta}$ on the free loop space $\mathcal{L}B=C^{\infty}(S^1,B)$ given by
\begin{equation}
	\Omega_{\eta}: T_x\mathcal{L}B = C^{\infty}(S^1,x^*TB)\fun \R, \ \  \xi \mapsto -{\int} \eta(\xi,\partial_t x - X_{h})\, dt.
\end{equation}
Further, let $f: \R \to [0,\infty)$ be the smooth function defined by 
$\textstyle f(R):= \int_0^R \phi'(\tau)\, h'(\tau)\, d\tau.$
Define the {\bf filtration functional} $F: \mathcal{L}B \to \R$ on the free loop space by
$$
F(x):=-\int_{S^1} x^*(\phi \alpha) + \int_{S^1} f(R(x(t)))\, dt.
$$

\begin{lm}{\normalfont\cite[Thm.6.2(1)]{McLR18}} \label{Fprimitive} $F$ is a primitive of $\Omega_{\eta}.$ That is, $dF(x)(\xi)=\Omega_{\eta}(x)(\xi).$ \qed
\end{lm}

When $\Psi$ is not globally defined, for a loop $y\in \mathcal{L}Y$ the loop $x=\Psi\circ y$ can be ill-defined, if some $y(t)\in Y^{\mathrm{in}}$, as $\Psi: \mathcal{L}Y \to \mathcal{L}B$ is only defined for $y$ contained in $Y^{\mathrm{out}}$. Nevertheless the partially defined pull-backs $F^Y=\Psi^*F$ and $\Omega_{\eta}^Y=\Psi^*\Omega_{\eta}$ can be extended naturally to $\Omega_{\eta}: T_y\mathcal{L}Y \fun \R$ and $F:\mathcal{L}Y \to \R$:
$$
\textstyle
\Omega_{\eta}^Y(\xi)=-\int_{S^1} \Psi^*\eta(\xi,\partial_t y - X_{h\circ \Phi})\, dt, \quad\quad F^Y(y) := -\int_{S^1} y^*(\Psi^*(\phi \alpha)) + \int_{S^1} f(\Phi(y(t)))\, dt,
$$
where $\Psi^*\eta$ and $\Psi^*(\phi \alpha)$ both extend by zero to $Y$ (as $\phi=\phi'=0$ near $R=R_0$), and $\Phi=\Psi\circ R$ on $Y^{\mathrm{out}}$ extends to $Y$ by making $\Phi$ decay%
\footnote{We can extend $\Psi$ above a collar neighbourhood $\Sigma \times (R_0-\epsilon,R_0]$ of $\partial B^{\mathrm{out}}=\Sigma \times \{R_0\}$ since $X_{\R_+}$ is strictly outward pointing on $\Psi^{-1}(\partial B^{\mathrm{out}})$. Let $n:(R_0-\epsilon,R_0]\to [0,1]$ be a non-decreasing cut-off function with $n=0$ near $R_0-\epsilon$, and $n=1$ near $R_0$. Then set $\Phi:=n(R\circ \Psi)$ over the collar, and $\Phi=0$ on the rest of $Y^{\mathrm{in}}.$}
to zero on $Y^{\mathrm{in}}$ over a collar neighbourhood of $\partial B^{\mathrm{out}}$ (the above data is independent of the choice of extension of $h(R)$ to $R\leq R_0$ as $\eta=0$ there). The calculation \cite[Thm.6.2(1)]{McLR18} of $dF-\Omega_{\eta}$ in \cref{Fprimitive} relied on showing that the relevant integrands cancelled out \emph{at each time} $t$ (not just after integrating over $t$). Those integrands all vanish when $R\circ x(t)$ approaches the boundary $R=R_0$ of $B^{\mathrm{out}}$ since $\phi,\phi',\eta$ vanish for $R$ close to $R_0$. The same holds for the integrands involved in $dF^Y-\Omega_{\eta}^Y$ when $y(t)\in Y^{\mathrm{in}}$ since we extended the data so that $\Psi^*\eta,\Psi^*(\phi\alpha),f(\Phi(y))$ vanish for $y(t)\in Y^{\mathrm{in}}$.
Thus $F^Y$ is a primitive of $\Omega_{\eta}^Y$ on all of $\mathcal{L}Y$. A similar line of argument holds for all our calculations in the next sections: we will abusively write $y(t)=\Psi \circ x(t)$ or $v(s,t)=\Psi \circ u(s,t)$ even though $\Psi$ is not globally defined, because at those values of $s,t$ the terms we consider will involve factors of $\phi$ or $\phi'$, so those terms vanish for $R$ close to $R_0$ and they are extended by zero for $R\leq R_0$.

\subsection{The filtration inequality for $CF^*(H_\lambda)$}\label{FiltrationOnCFLambda}

For a 1-orbit $x$ of $H_\lambda,$ we define by abuse of notation:
$$F(x):=F(\Psi(x)).$$ 

\begin{thm}\label{H_lambdaIsOneDirected}
	The Floer chain complex $CF^*(H_\lambda)$ has a filtration given by the value of $F.$ That is, given two $1$-periodic orbits $x_-,x_+$ of $H_\lambda$ and a Floer cylinder for $(H_\lambda,I)$ from $x_-$ to $x_+,$
	\begin{equation}\label{FiltrationByF}
		F(x_-) \geq F(x_+).  %
	\end{equation}
\end{thm}

\begin{proof}
	The Floer cylinder $u:\R \times S^1 \fun Y$ for $H_\lambda$ satisfies
	$$\partial_s u + I(\partial_t u -X_{H_\lambda})=0.$$ 
 Using $\Psi_*X_{S^1}=\mathcal{R}_B$, we get
	\begin{equation}\label{EqnPhiProjectionNice}
		\Psi_*(X_{H_\lambda})=\Psi_*(c'(H)X_{H})=c'(H)\Psi_*(X_{H})=c'(H) \mathcal{R}_B,
	\end{equation}
	noting that $c'(H)$ depends on the original coordinates in $Y$.
	Projecting $u$ via $\Psi$ defines a map 
	\begin{equation} \label{projectFloerinM}
		v:=\Psi \circ u:\R \times S^1 \fun B, \ \ \ \partial_s v + I_B(\partial_t v -k(s,t)\mathcal{R}_B)=0
	\end{equation}
	that converges to $y_-=\Psi(x_-)$, $y_+=\Psi(x_+)$ at $s=-\infty$, $+\infty,$ respectively, where $$k(s,t):=  c'(H(u(s,t)))$$
	is a domain-dependent function.
	Now define the \textbf{$s$-parametrised 1-form }
	\begin{equation}
		\Omega_{\eta}^v(s): T_x\mathcal{L}B = C^{\infty}(S^1,x^*TB)\fun \R, \ \ \xi \mapsto -{\textstyle\int} \eta(\xi,\partial_t x - k(s,t)\mathcal{R}_B)\, dt.
	\end{equation}
	Using 
 \eqref{Equation eta expanded},\
	$d\a(\cdot, \mathcal{R}_B)=0$, \ $dR(\mathcal{R}_B)=0$ and $\alpha(\mathcal{R}_B)=1$:
	\begin{equation}\label{omegaV}
		\begin{aligned}
			\Omega_\eta^v(s)(x)(\xi)= &
            \textstyle
            -\int \eta(\xi,\partial_t x)\ dt + \int k(s,t)\,\eta(\xi,\mathcal{R}_B)\ dt\\
			=&
            \textstyle
            -\!\int\! \eta(\xi,\partial_t x) dt+\!\int\! k(s,t)\phi(R)d\a(\xi,\mathcal{R}_B) dt
			+\!\int\! k(s,t)\,\phi'(R)(dR\wedge \a)(\xi,\mathcal{R}_B) dt\\
			=&
            \textstyle
            -\int \eta(\xi,\partial_t x)\ dt + \int \phi'(R) dR(\xi) k(s,t) \ dt.
		\end{aligned}
	\end{equation}
	
	Similarly, using that $X_h=h'(R)\mathcal{R}_B$ in the region where $\phi'\neq 0$,
	\begin{equation}
		\label{OmegaOrdinary}
        \textstyle
		\Omega_\eta(x)(\xi)= -\int \eta(\xi,\partial_t x)\ dt + \int \phi'(R) dR(\xi) h'(R(x(t))) \ dt.
	\end{equation}
	
By the construction of $H_\lambda$ and $h(R),$ for $R(v(s,t))\in L_i$ we have
	\begin{equation}\label{EquationkstIsHprime}
		k(s,t)= c'(H(u(s,t)))= \lambda_i = h'(R(v(s,t))),
	\end{equation}
	whereas outside that region we have $\phi'(R)=0.$
 Therefore for all $s,t$ we have
	\begin{equation}\label{EquationPhi'kIsPhih'}
		\phi'(R(v(s,t)))\ k(s,t)=\phi'(R(v(s,t)))\ h'(R(v(s,t))).
	\end{equation}
	Combining (\ref{omegaV}) and (\ref{OmegaOrdinary}) yields
	\begin{equation*}\label{OmegaStdAndOmegaH_lambdaAreEqual}
		\Omega_\eta(v(s,t))(\partial_s v)=\Omega_\eta^v(s)(v(s,t))(\partial_s v).
	\end{equation*}
	Combining this with Lemma \ref{Fprimitive}, we deduce that $F(x_-) - F(x_+)$, meaning $F(y_-)-F(y_+)$, equals
	\begin{equation}\label{F_DifferenceInTermsOfOmega}
            \textstyle
            - \int_{-\infty}^{+\infty} dF(v(s,t))(\partial_s v) ds\\
            \textstyle
			= -\int_{-\infty}^{+\infty} \Omega_\eta(v(s,t))(\partial_s v)ds  =
            - \int_{-\infty}^{+\infty} \Omega^v_\eta(s)(v(s,t))(\partial_s v) ds. 
	\end{equation}
	Hence, it is sufficient to prove that $$\Omega^v_\eta(s)(v(s,t))(\partial_s v)\leq 0$$ holds for $v$ satisfying equation (\ref{projectFloerinM}). This reduces us to the same computation as in the convex setting \cite[Lem.6.1]{McLR18}: we use the Floer equation $\partial_t v - k(s,t)\mathcal{R}_B = I_B\partial_s v$, and abbreviate $\rho=R\circ v$:
		\begin{equation} \label{nonpositivityOfOmegaV}
			\begin{array}{rcl}
				\eta(\partial_s v, \partial_t v - k(s,t)\mathcal{R}_B)
				& = &
				\eta(\partial_s v,I_B\partial_s v)
				\\
				& = &
				\phi(\rho) \cdot d\alpha(\partial_s v,I_B\partial_s v) + \phi'(\rho) \cdot (dR \wedge \alpha)(\partial_s v,I_B\partial_s v)
				\\
				& = &
				\textrm{positive}\cdot \textrm{positive} + \textrm{positive}\cdot (dR \wedge \alpha)(\partial_s v,I_B\partial_s v),
			\end{array}
		\end{equation}
	where ``positive'' here means ``non-negative''.
	To estimate the last term, we may assume that $R\geq R_0''$ since $\phi'=0$ otherwise. Thus, 
	we decompose $\partial_s v$ according to an orthogonal decomposition of $T B$:
	\begin{equation}\label{Eqn C in xi part}
	\partial_s v = C\oplus y \mathcal{R}_B \oplus zZ \in \xi \oplus \R \mathcal{R}_B \oplus \R Z,
	\end{equation}
	where $Z=-I_B\mathcal{R}_B=R\partial_R$ is the Liouville vector field and $\xi=\ker \alpha|_{R=1}$. Notice that $\ker \a = \xi \oplus \R Z.$ Thus: $dR(\partial_s v) = Rz$ and $\alpha(I_B\partial_s v) = \alpha(I_BzZ) = \alpha(z \mathcal{R}_B)=z$. So,
	\begin{equation}\label{Eqn non xi part}
	(dR \wedge \alpha)(\partial_s v,I_B\partial_s v) =
	dR(\partial_s v) \alpha(I_B\partial_s v) - \alpha(\partial_s v) dR(I_B \partial_s v)
	= Rz^2 + R y^2 \geq 0.
	\end{equation}
	Hence $\Omega^v_\eta(s)(v(s,t))(\partial_s v)\leq 0,$ thus $F(\Psi(x_-))-F(\Psi(x_+))\geq 0.$
\end{proof}

\subsection{The $F$-filtration values on $1$-orbits}

\begin{cor} \label{FiltrValue} The $F$-filtration values satisfy the following properties:
\begin{enumerate}
\item $F=0$ at the constant orbits, so at each point of  $\F=\sqcup_\a \F_\a$;
\item $F(y)=F(\Psi(x))<0$ for every non-constant $1$-periodic orbit $x$ of $H_{\lambda}$;
\item $F(y)$ only depends on the Reeb period $T_{-i}=c'(H(x))$ of the projected orbit $y$, see \eqref{Equation filtration value of orbit};
\item for non-constant orbits, $F(y)$ decreases as $H(y)$ increases;
\item on non-constant orbits, the $F$-filtration is equivalent to filtering by $-H$, or equivalently: filtering by negative $S^1$-period values $-c'(H)=-T_{-i}$. 
\end{enumerate}

\end{cor}
\begin{proof} Let us calculate the value of the functional $F(y)$ explicitly for the projection $y:=\Psi(x(t))$ of a 1-periodic orbit $x$ of $H_\lambda.$ If $x$ is a fixed point, $y$ lies in the region where $\phi=0$ so $F(y)=0.$ Otherwise, $c'(H(x))=T_{-i}$ for some $0<i\leq r.$ Thus $y$ is a 1-periodic orbit in $B$ of the Hamiltonian $T_{-i}R,$ so $R\in N_i$ lies between $L_{i-1}$ and $L_i$, thus $\phi(y)=\phi_0+\cdots+\phi_{i-1}$. Using that $\phi'=0$ except possibly on the intervals $L_k$ where $h'(R)=\lambda_k$, we deduce:
\begin{align}
\begin{split}
\label{Equation filtration value of orbit}
		F(y(t))=&-\int_{S^1} y^*(\phi \alpha) + \int_{S^1} f(R(y(t))) dt\\
		=& -\phi(y)  T_{-i} + \int_{L_0} \phi'(R) h'(R)dR + \cdots + \int_{L_{i-1}} \phi'(R) h'(R)dR \\
		=& -(\phi_0+\cdots+\phi_{i-1})  T_{-i} +   \phi_0 \lambda_0+\cdots+  \phi_{i-1} \lambda_{i-1},
\end{split}
\end{align}
which is negative as $T_{-i}$ is strictly larger than $\lambda_0,\ldots,\lambda_{i-1}$, and $\phi_0>0$.	
The drop in filtration value for $1-$orbits $y_{-i},y_{-(i+1)}$ arising for successive slopes $T_{-i}<T_{-(i+1)}$ is:
\begin{equation}\label{Equation drop in filtration value}
F(y_{-i}) - F(y_{-(i+1)}) = 
(\phi_0+\cdots+\phi_{i-1}) (T_{-(i+1)}-T_{-i})+\phi_i (T_{-(i+1)}-\lambda_i) > 0.
\end{equation}
Thus $F(x(t))<0$ strictly decreases when $c'(H(x))=T_{-i}$, hence $H(x),$ increases.
\end{proof}

\subsection{Dependence of the filtration on $\phi$ and $H_{\lambda}$}\label{Subsection Dependence of the filtration on phi}\label{PositiveSH}

Our convention is that $x_-$ appears in the output of the chain differential $\partial(x_+)$ if a Floer trajectory $u$ flows from $x_-$ to $x_+$. As $F(x_-)\geq F(x_+)$, 
this means $\partial$ ``increases the $F$-filtration'', so it decreases $H$, and decreases the period $c'(H_{-i})=T_{-i}$. Thus, restricting $1-$orbits by the condition $F\geq A$ defines a subcomplex, $CF^*_{[A,\infty)]}(H_{\lambda})$.
So the filtration separates the Floer chain complex into a finite collection of distinct nested subcomplexes
$$
CF^*_{[0,\infty)}(H_{\lambda}) \subset CF^*_{[-a_1,\infty)}(H_{\lambda})\subset CF^*_{[-a_2,\infty)}(H_{\lambda})\subset \cdots \subset CF^*(H_{\lambda}),
$$
where $a_1,a_2,\ldots \in (0,\infty)$ are increasing.
The number of these subcomplexes  increases with $\lambda$. The $[0,\infty)$-piece corresponds to the subcomplex generated by constant orbits at $\F$; for those $F=0$.
By \cref{FiltrValue}, these lists of distinct subcomplexes (after forgetting the labelling by $F$-values $a_i$) do not depend on the choice of $\phi$, since for non-constant $1$-orbits we can instead label the subcomplexes by $(-H)$-values, or by negative $S^1$-period values $-c'(H)=-T_{-i}$. 

\begin{de}\label{Definition filtration meaning}
``{\bf Filtration}'' on a Floer chain complex refers to the above distinct subcomplexes, without labelling the steps of the filtration by $F$-values. A chain map is {\bf filtration preserving} if a $1$-orbit with associated period value $p$ is sent to a $\k$-linear combination of $1$-orbits with associated period values $\leq p$ (we avoid the issue that $F$ depends on $h,\phi$ which partly depend on $H_{\lambda}$).
\end{de}

To prove that certain maps counting Floer solutions $u$ are filtration preserving, we typically seek an inequality of type $\partial_s (F\circ v)\leq 0$ where $v=\Psi\circ u$, as we did for Floer cylinders in \cref{H_lambdaIsOneDirected}. For Floer continuation solutions this is problematic: {\homotopying} by $H_s$ from $H_{\lambda'}$ to $H_{\lambda}$, we need to decide how to pick $\phi,h$ when defining an $s$-dependent filtration $F_s$. Making $\phi$ $s-$dependent causes problematic terms in $\partial_s F_s$. In \cite[Sec.6.5]{McLR18} it is shown that in a convex symplectic manifold $B$, Floer continuation solutions $v$ for a homotopy $h_s$ of the Hamiltonians still satisfy $\partial_s (F_s \circ v)\leq 0$ when using an $s$-dependent $f_s(R)=\int_0^R \phi'(\tau)h_s'(\tau)\,d\tau$ in the definition of the filtration $F=F_s$, provided that 
$$\partial_s h_s'\leq 0 \textrm{ where }\phi'\neq 0.$$

\begin{lm}\label{Lemma Floer continuation maps are ok}
Floer continuation maps $CF^*(H_{\lambda}) \to CF^*(H_{\lambda'})$ for $\lambda \leq \lambda'$ respect the filtration if the homotopy $H_s = c_s\circ H$ of the Hamiltonians is $s$-independent in the region $\phi'\neq 0$.

More generally, this holds if the linearity intervals $L_i$ 
which are used in the construction of $\phi$
are such that $H_{\lambda}$, $H_{\lambda'}$ and $H_s$ are linear in $H$ over $L_i$ and
the slope $c_s'=\lambda_{i,s}$ on $L_i$ satisfies $\partial_s \lambda_{i,s}\leq 0$.
\end{lm}
\begin{proof}
Under the assumption, \eqref{EquationPhi'kIsPhih'} holds also for projections of continuation solutions, so the filtration is non-increasing along $v$ (and the output $x_-$ of the continuation operator at the chain level has filtration value at least that of $x_+$). For the more general assumption, abbreviating $\rho=R(v(s,t))$, \eqref{EquationPhi'kIsPhih'} becomes $\phi'(\rho)k(s,t)=\phi'(\rho) \lambda_{i,s} = \phi'(\rho)h_s'(\rho)$ where $h_s$ is built to satisfy $h_s'(R)=\lambda_{i,s}$ for $R\in L_i$,
so that $\partial_s h_s' \leq 0$ for $R\geq R_0'=R_0$ ensuring that the argument in \cite[Sec.6.5]{McLR18} applies.
\end{proof}

\begin{de}\label{Definition monotone hpy}
A homotopy $H_s$ of Hamiltonians is {\bf monotone} if $\partial_s H_s\leq 0$. The weight $t^{E_0(u)}$ with which Floer continuation solutions $u$ are counted then involves a non-negative quantity $E_0(u)\geq 0$ related to energy via $E(u)=E_0(u)+\int_{\R \times S^1} \partial_s H_s\, ds \wedge dt$ (cf.\,\cite[Sec.3.3]{R13}), so $E_0(u)\geq E(u)\geq 0$. If $H_s=\lambda_s H + \textrm{constant}(s)$ at infinity, $H_s$ is an {\bf admissible homotopy} (we allow non-generic slopes $\lambda_s$).
\end{de}

\begin{cor}\label{Cor continuation is always possible}
Given admissible $H_{\lambda}$, $H_{\lambda'}$ with $\lambda'\geq \lambda$, there exists a filtration preserving continuation map $CF^*(H_{\lambda}) \to CF^*(H_{\lambda'})$ (cf.\,\cref{Definition filtration meaning}). This can be achieved by composing a filtration-preserving admissible monotone continuation map $CF^*(H_{\lambda}) \to CF^*(H_{\lambda'}+C)$ for some large enough constant $C>0$, with the natural filtration-preserving identification $CF^*(H_{\lambda'}+C)=CF^*(H_{\lambda'})$.
\end{cor}
\begin{proof}
The statement $CF^*(H_{\lambda'}+C)=CF^*(H_{\lambda'})$ holds because adding a constant to a Hamiltonian does not change the Hamiltonian vector field, so the Floer theory and the filtration values do not change (we do not need to change $h$; as $F$ only depends on $h'$, changing $h$ to $h+C$ does not affect $F$). By picking $C=\max (H_{\lambda}-H_{\lambda'})$ (finite since $\lambda'\geq \lambda$) we may assume $H_{\lambda'}\geq H_{\lambda}$ everywhere.

For this proof, we will use $\phi$ with $\phi'=0$ except over $L_0$ (so $\phi_0>0$ but $\phi_i=0$ for $i>0$).
If $H_{\lambda},H_{\lambda'}$ are linear with the same slope $\lambda_0$ over the linearity interval $L_0$, then any admissible monotone homotopy $H_s$ from $H_{\lambda'}$ to $H_{\lambda}$ which preserves the slope $\lambda_0$ over $L_0$ satisfies the claim by \cref{Lemma Floer continuation maps are ok}.

To reduce to the above case, we perform admissible monotone homotopies supported in a compact region, so their Floer continuation maps induce quasi-isomorphisms  (ensuring they preserve the filtration, \cref{Definition filtration meaning}).
Call $L_0$, $L_0'$ the respective linearity intervals of the two Hamiltonians, for slopes $\lambda_0$, $\lambda_0'$. 
If $\lambda_0>\lambda_0'$, we first use a monotone admissible homotopy of $H_{\lambda}$ supported near $R\leq R_1'$ which is slope-decreasing on $L_0$, so \cref{Lemma Floer continuation maps are ok} applies. 
We also do a monotone admissible homotopy of $H_{\lambda}$ supported near $R\leq R_1'$ to increase the size of $L_0$ to the left (decreasing $R_0''$).

For $H_{\lambda'}$ we use a different $\phi$ to define the $F$-filtration: $\phi'=0$ except over $L_0'$. Arguing as above, if $\lambda_0'>\lambda_0$ we homotope $\lambda_0'$ down to $\lambda_0$, and we extend $L_0'$ to the left. That these two filtration-preserving continuation maps were built using a different $F$ does not matter, by \cref{FiltrValue} and \cref{Definition filtration meaning}.

We may now assume that $H_{\lambda'}, H_{\lambda}$ have the same slope over $L_0\cap L_0'$ and that this overlap contains a non-trivial interval. By redefining $L_0$ to be that interval we have reduced to the initial case.
\end{proof}

\begin{de}
Abbreviate $CF_0^*(H_\lambda):=CF^*_{[0,\infty)}(H_{\lambda}) \subset CF^*(H_\lambda)$ the subcomplex generated by the fixed locus $\F=\sqcup_\a \F_\a$ (constant orbits have filtration $F=0$).
The {\bf positive Floer cohomology} $HF_+^*(H_\lambda)=H_*(CF_+(H_\lambda))$ is the homology of the 
quotient complex $CF_+(H_\lambda):=CF^*(H_\lambda)/CF_0^*(H_\lambda)$.
\end{de}

Now pick a cut-off function $\phi$ which is constant for $R\geq R_1'$, and consider those $H_{\lambda}$ of \cref{SectionConstructionOfHLambda} which are linear in $H$ of slope $\lambda_0$ on $[H_0'',H_1']$ (which covers via $R\circ \Psi$ the interval $L_0=[R_0'',R_1']$ where $\phi'\neq 0$). Then \cref{Lemma Floer continuation maps are ok} yields filtration-preserving slope-increasing continuation maps, so the direct limit of the $HF^*_+(H_\lambda)$ is well-defined and is called \textbf{positive symplectic cohomology},
$$SH_+(Y,\Fi,\omI):=\lim_{\lambda\fun \infty} HF_+^*(H_\lambda).$$
Up to a filtration-preserving isomorphism, $SH_+(Y,\Fi,\omI)$ does not depend on the choices of $\lambda_0,H_0'',H_1'$ by \cref{Cor continuation is always possible} and \cref{Definition filtration meaning}.
We deduce the following (see {\PartI} for details).

\begin{prop}\label{Corollary LES for H SH and SHplus}
	For any symplectic $\C^*$-manifold over a convex base,
	there is a long exact sequence 
	$$
	\cdots \to QH^*(Y,I) \stackrel{}{\to} SH^*(Y,\Fi,\omI) \to SH^*_+(Y,\Fi,\omI) \to QH^{*+1}(Y,I) \to \cdots
	$$
\end{prop}

\subsection{Description of Floer solutions for which the filtration value remains constant}
\begin{lm}\label{Lemma Floer solutions for constant filtration value}\label{Theorem no vertical solutions in Hyperkahler case} 
If at least one among $x_-,x_+$ is non-constant, and $F(x_-) = F(x_+)$, then
\begin{itemize}
    \item $H(x_-)=H(x_+)$, and both $x_\pm$ are non-constant;

\item $x_{\pm}$ correspond to orbits of the $S^1$-flow of equal period $T:=c'(H(x_-))=c'(H(x_+))$;
\item $v=\Psi(u)$ lies entirely in a region where $\phi'=0$;  \item $y_{\pm}=\Psi(x_{\pm})$ correspond to Reeb orbits of period $T$ (but may not have the same $R$-value); 
\item $\mathrm{Im}(v) \subset \C^*\cdot p$ lies in the $\C^*$-orbit of some point $p\in B$;
\item $\mathrm{Im}(v)$ lies in some $m$-torsion submanifold of $B$ (\cref{TorsionPtsTorsionSubmflds}), for $k,m$ coprime, $T=\tfrac{k}{m}\in \Q$;
\item $\mathrm{Im}(du)\subset \C\cdot X_{S^1}\oplus \ker \Psi_*$ where $\C\cdot X_{S^1}=\R\cdot X_{\R_+}\oplus \R \cdot X_{S^1}$.
\end{itemize}

\end{lm}
\begin{proof}
As one among $x_-,x_+$ is non-constant, $v$ does not lie entirely in the region where $\phi=0$. As \eqref{FiltrationByF} is not strict, we must have $C=0$ in \eqref{Eqn C in xi part} otherwise the first positive term
in \eqref{nonpositivityOfOmegaV} is strictly positive. Thus $\mathrm{Im}(v)\cap B^{\mathrm{out}} \subset \C^*\cdot p$  lies in the $\C^*$-orbit of $p=y_-(0)\in B$. 
Also, $v$ cannot enter the region $\phi'\neq 0$, otherwise the last term in \eqref{nonpositivityOfOmegaV} becomes strictly\footnote{
If \eqref{Eqn non xi part}$\,\equiv 0$ then $\partial_s v\equiv 0$, so $y_{\pm}(t)=v(s,t)$ is $s$-independent, so $u$ lies in $\Psi^{-1}(y_{\pm})=\Psi^{-1}(\Psi(x_{\pm}))$ and there $\phi'=0$. 
}
positive by \eqref{Eqn non xi part}, so $v$ stays in a connected subset where $\phi$ is constant. 
By \cref{FiltrValue}, the condition $F(x_-) = F(x_+)$ implies $c'(H(x_-))=c'(H(x_+))$ is the same period $T:=T_{-i}$ say, and thus $H(x_-)=H(x_+)$ by the construction of $c$. Thus $y_{\pm}$ are $1$-periodic $T\cdot \mathcal{R}_B$-orbits. 
As $v\subset \C^*\cdot p$, it must lie in the same connected component of torsion points as $p$.
\end{proof}

\subsection{Floer solutions ``consume filtration'' when crossing linearity regions}
\label{Subsection Floer solutions consume filtration when crossing linearity regions}

So far we have not needed to specify the choices of constants $\phi_i$ and $\varepsilon_i$ in the construction of $\phi$ and $L_i$. We will choose the $\phi_i$ inductively so that Floer solutions must consume an a priori lower bound of filtration-energy when crossing the ``linearity'' region above $L_i$ (so:\,$R\circ \Psi\in L_i$). Then we get a continuation $CF^*(H_{\lambda})\to CF^*(H_{\lambda'})$ that is a subcomplex when $H_{\lambda'}$ is a modification of $H_{\lambda}$ at infinity by increasing the slope. For convex symplectic manifolds this easily follows from the maximum principle for radial Hamiltonians; but in our setup the maximum principle only holds over linearity regions, so we need a filtration argument. This trick will also be used to ensure that certain spectral sequences are compatible with increasing slopes of Hamiltonians, so that one can take a direct limit of the spectral sequences.

We refine the choices of $\varepsilon_i$, $N_i$, so that now $\phi$ belongs to the class of cut-off functions satisfying:
\begin{enumerate}
	\item[(4)] $\phi\geq R$ for $R \geq R_1'$.
	\item[(5)] $\phi'\geq 1$ on a subinterval of $L_i$ of width at least $\varepsilon_i/2$.
\end{enumerate}
This can be achieved by choosing $N_0,\ldots,N_i$ to be small, $\varepsilon_i=\textrm{width}(L_i)$ to be large, and by making $\phi'$ very large on a large subinterval of $L_i$, in particular making $\phi_i$ large enough so that the constant value of $\phi$ on $N_{i+1}$ is larger than $R_{i+1}''=\max_{N_{i+1}} R$.

\begin{lm}\label{LemmaDeltaiFiltration}
	Given $\delta_i>0$, and any choice of $\phi$ defined on $R\leq R_i''$ within the above class, there is a way to extend $\phi$ within that class over $L_i=[R_i'',R_{i+1}']$ so that any Floer trajectory $u:\R\times S^1 \to Y$ from $x_-$ to $x_+$ for $H_{\lambda}$ which crosses the region above $L_i$ satisfies $F(x_-)-F(x_+)>\delta_i$.
	
	More precisely, if there are values $s_0<s_1\in \R \cup \{\pm \infty\}$ for which the loops $u(s_0,\cdot)$ and $u(s_1,\cdot)$ project via $R\circ \Psi$ to different components of $\R \setminus L_i$,
	then 
	\begin{equation}\label{Equation filtration jump of Floer when crossing region}
	    F(x_-)-F(x_+)\geq F(u(s_0,\cdot))-F(u(s_1,\cdot)) > \delta_i.
	\end{equation}
\end{lm}
\begin{proof}
	As \eqref{nonpositivityOfOmegaV} is pointwise non-negative, it suffices that we bound the integral 
	$\int \eta(\partial_s v, \partial_t v - \lambda_i \mathcal{R}_B)$
	from below by $\delta_i$ over the subset of all $(s,t)\in \R \times S^1$ for which $R(v(s,t))$ lies in the subinterval of $L_i$ where we ensured that $\phi'\geq 1$ (and $h'=\lambda_i$). For such $(s,t)$, using 
	$\phi\geq R$ on $R\geq R_1'$,
	$$
	\eta(\partial_s v, \partial_t v - \lambda_i \mathcal{R}_B)
	\geq R \cdot d\alpha(\partial_s v,\J_B\partial_s v) + (dR \wedge \alpha)(\partial_s v,\J_B\partial_s v) = d(R\alpha)(\partial_s v, \partial_t v - \lambda_i \mathcal{R}_B).
	$$
	The latter expression is precisely the integrand of the energy $\int \|\partial_s v\|^2 \, ds\, dt$ for a Floer solution $v$ in $B$ for the Hamiltonian $\lambda_i R$. So the problem reduces to showing that a Floer solution in $B$ for a radial Hamiltonian consumes a lot of energy if it crosses a long radial stretch in which the slope of the Hamiltonian is constant, as we can choose the width $\varepsilon_i=\textrm{width}(L_i)$ to be arbitrarily large. This is a standard ``monotonicity lemma argument'': one uses Gromov's trick of turning a Floer solution in a symplectic manifold $M$ into a {\ph } section of a bundle with fibre $M$ (e.g.\,see \cite[Sec.5.3]{R14}). One then applies the monotonicity lemma to obtain lower bounds on the energy (see  \cite[Lem.33]{R14}). This is the step that uses the assumption in \cref{Def:KahlerMfdWithProjection} that $(B,I_B)$ is geometrically bounded at infinity, and the observation that in the auxiliary almost complex structure in Gromov's trick \cite[Def.26]{R14} the off-diagonal terms involve $\lambda_i \mathcal{R}_B$ and $\lambda_i Z_B$ which are radially invariant.
\end{proof}

\begin{rmk}\label{Remark montone hpy is ok}
Lemma \ref{LemmaDeltaiFiltration} also holds for Floer continuation solutions for homotopies $H_s$ as in Lemma \ref{Lemma Floer continuation maps are ok}. In the proof, on the interval $L_i$ we have $h_s' = \lambda_{i,s}$, and the condition $\partial_s \lambda_{i,s}\leq 0$ ensures that the monotonicity argument still applies
(see the discussion of monotone homotopies in \cite[Sec.5.3]{R14}).
\end{rmk}

\subsection{Cofinal family of Hamiltonians $H_{\lambda_i}$, and a filtration on $SH^*$}
\label{Subsection cofinal Hams}

To compute $SH^*(Y,\varphi)$ it suffices to take a direct limit over a cofinal family of Hamiltonians $H_{\lambda_i}$, meaning $\lambda_i \to \infty$ as $i \to \infty$. Such $H_{\lambda_i}$ can be constructed so that the Floer continuation maps are particularly well-behaved.

\begin{cor}\label{Cor Subcomplex Trick}\label{CorSubcomplexTrick}
	There is a cofinal sequence of Hamiltonians $H_{\lambda_i}$ and a cut-off function $\phi$, such that $CF^*(H_{\lambda_i})\subset CF^*(H_{\lambda_{i+1}})$ is a subcomplex. There are filtration-preserving Floer continuation maps $CF^*(H_{\lambda_i})\to CF^*(H_{\lambda_{j}})$ for $i<j$ which equal the inclusion maps of subcomplexes.
	The union $$SC^*(Y,\varphi):=\cup\, CF^*(H_{\lambda_i})$$ of these nested complexes is a filtered complex, and $SH^*(Y,\varphi)$ is (isomorphic to) its cohomology.
Analogously $SC^*_+(Y,\varphi):=\cup\, CF^*_+(H_{\lambda_i})$ is a filtered complex with cohomology $SH^*_+(Y,\varphi)$.
\end{cor}
\begin{proof}
We build $H_{\lambda_i}$ inductively as in \cref{SectionConstructionOfHLambda} so that $H_{\lambda_i}=\lambda_i H + \mathrm{constant}$ for $H\geq H_i''$, and $H_{\lambda_{i+1}}$ is obtained from $H_{\lambda_i}$ by increasing its (generic) slope at infinity to $\lambda_{i+1}>\lambda_i$.
Inductively, we ensure that $H_{\lambda_{i+1}}=H_{\lambda_i}$ for $H\leq H_i''$. The construction of $H_{\lambda_{i+1}}$ involves a choice of new parameters $H_{i+1}',H_{i+1}'',R_{i+1}',R_{i+1}''$ in \cref{SectionConstructionOfHLambda}, in particular a choice of width $\varepsilon_i=\mathrm{width}(L_i)$.
We build the Floer continuation map by considering a homotopy $H_s$ from $H_{\lambda_{i+1}}$ to $H_{\lambda_i}$ that decreases the slope of $H_{\lambda_{i+1}}$ back down to $\lambda_i$ on the region $H \geq H_i''$, in particular over $L_i$ we ensure that  $H_s=\lambda_s H +\mathrm{constant}(s)$ with $\partial_s \lambda_s \leq 0$. 
The continuation map then respects the filtration, by Lemma \ref{Lemma Floer continuation maps are ok}.
If the continuation map were not the inclusion, there would exist a Floer continuation solution $u$ from a $1$-orbit $x_-\in CF^*(H_{\lambda_{i+1}})$ with $S^1$-period value $T_{-(i+1)}$, to a $1$-orbit 
$x_+\in CF^*(H_{\lambda_i})$; we need to prevent that.
Such a Floer solution $u$ crosses the linearity region $L_i$ in the construction of $H_{\lambda_{i+1}}$. By \cref{Remark montone hpy is ok}, picking $\delta_i$ to be larger than the possible values $F(x_-)-F(x_+)$ for $1$-orbits $x_-$ of $H_{\lambda_{i+1}}$, with $S^1$-period value $T_{-(i+1)}$, and $1$-orbits $x_+$ for $H_{\lambda_i}$, a Floer continuation solution that crosses the region over $L_i$ would contradict  inequality \eqref{Equation filtration jump of Floer when crossing region}.
So, such Floer continuation solutions are trapped in the region $H\leq H_i''$ where $H_s$ is $s$-independent. But such Floer continuation solutions are constant as they cannot be isolated, due to the $\R$-reparametrisation in $s$. Thus $u$ does not cross $L_i$ after all. %
\end{proof}

\begin{rmk}
The Floer continuation solution may achieve a maximum $(R\circ \Psi)$-value in the region where $c'$ is not locally constant, where the maximum principle for the $\Psi$-projection to $B$ may not hold.
\end{rmk}

\section{Transversality for Floer solutions}
\label{Section Transversality for Floer solutions}

\subsection{Overview of the technical difficulty}

As mentioned in \cref{Rmk technical symplectic assumptions on Y}, transversality for Floer solutions can be achieved by the standard methods of \cite{FloerHoferSalamon, Salamon-Zehnder} for $Y$, by $C^{\infty}$-small perturbations $I+\varepsilon$ of $I$ (where $\varepsilon$ is an operator of a small norm, such that $I+\varepsilon$ is an $\omega$-compatible almost complex structure). However, such a perturbation will in general ruin the holomorphicity of the map $\Psi$, and the projection $v=\Psi\circ u$ of a Floer solution $u\subset Y$ would only satisfy an equation in $B$ of type
$$\partial_t v - k(s,t)X_{S^1,B} = I_B\partial_s v+w(s,t)$$
where $w(s,t)=\Psi_*\circ \varepsilon_{u(s,t)} \circ \partial_s u$.
This may ruin the estimate in \cref{nonpositivityOfOmegaV} due to the term $d\alpha(\partial_s v, w(s,t))$. Even in a region where $\phi'\neq 0$, so where $\eta$ is symplectic, the term $\eta(\partial_s v, I_B\partial_s v)\geq 0$ may be too small to off-set a negative term $d\alpha(\partial_s v, w(s,t))$. Specifically, a large $\|\partial_s u\|$ has small $\|\partial_s v\|$ when $\partial_s u$ lies close to $\ker \Psi_*$, so $\|w\|$ cannot a priori be bounded by $\|\partial_s v\|$ by making $\|\varepsilon\|$ small, unless $\varepsilon$ preserves $\ker \Psi_*$. Once we demand that $\varepsilon$ preserves $\ker \Psi_*$ we run into two problems. Firstly, $\ker \Psi_*$ is not a locally trivial fibration, e.g.\,$\dim \ker \Psi_*|_p$ is only known to be upper semi-continuous in $p\in Y$, so it is not always possible to extend a vector $v\in \ker \Psi_*|_p$ to a local section of $\ker \Psi_*$. Secondly, the standard construction of the perturbation of $I$ in \cite[pp.1345-1346]{Salamon-Zehnder} and \cite[p.268]{FloerHoferSalamon} seems to run into problems when forcing $\varepsilon$ to preserve $\ker \Psi_*$, for instance if the adjoint $L^2$-section along $u$ in that argument were entirely contained in $\ker \Psi_*\subset TY$.

We will resolve these issues in an indirect way, by using a Gromov compactness argument combined with the use of a construction similar to the filtered Floer homology introduced by Ono \cite[Sec.3]{Ono}.
Consider what happens to a sequence of Floer solutions $u=u_{\varepsilon}$ of $(H_{\lambda},I+\varepsilon)$ for a sequence of perturbations $\varepsilon\to 0$ converging to zero. Assume for simplicity that the Hamiltonian $1$-orbits $x_{\pm}$ at the ends of $u_{\varepsilon}$ are fixed.
If we knew that the $u_{\varepsilon}$ converged, in the sense of Gromov compactness, to a broken Floer solution for $(H_{\lambda},I)$, then we would deduce that the filtration satisfies the correct inequality $F(x_-)\geq F(x_+)$ in \eqref{FiltrationByF} by \cref{H_lambdaIsOneDirected}.
Gromov compactness applies to Floer solutions with bounded energy 
so, as explained in \cref{Subsection Gromov compactness when the energy is bounded}, 
one achieves the correct filtration inequality for all Floer solutions $u$ of $(H_{\lambda},I+\varepsilon)$ of energy $E(u)\leq E_0$
for any given perturbation $\varepsilon$ that is sufficiently small. However, how small $\varepsilon$ needs to be will a priori depend on the choice of energy bound $E_0$, so it does not rule out the possibility of a sequence of Floer solutions $u_{E_0}$ of energy $E(u_{E_0})>E_0$ that defy the filtration inequality, and which do not have a subsequence with a meaningful Gromov limit.
We construct a filtered version of the Floer chain complex, which artificially creates the necessary energy bounds so that Floer cohomology is obtained by a limiting procedure from Floer cohomology groups that only count Floer trajectories below a certain energy bound.

\subsection{Gromov compactness when the energy is bounded}\label{Subsection Gromov compactness when the energy is bounded}

From now on, we continue with the assumptions on $Y$ from the start of \cref{filtrationFloer}.
By the maximum principle from \cite{RZ1}, Floer solutions for $(H_{\lambda}, I)$ remain in a compact region $K\subset Y$, so regularity of Floer solutions can be achieved by perturbing $I$ in $K$. We want to do this perturbation compatibly with the filtration:

\begin{lm}\label{Lemma perturb I subject to energy bound}
Given an admissible Hamiltonian $H_\lambda$, a
constant $E_0>0$, and a compact region $K\subset Y$, any sufficiently $C^{\infty}$-small perturbation $I'$ supported on $K$ of the $\omega$-compatible almost complex structure $I$ satisfies the filtration inequality \eqref{FiltrationByF}, for all Floer trajectories $u$ of $(H_{\lambda},I')$ of energy $E(u)\leq E_0$. 

If $K$ contains all $1$-periodic orbits of $H_{\lambda}$ in its interior, then any generic perturbation $I'$ will also satisfy transversality for all Floer solutions.

Moreover, there is a $\delta_E>0$ such that for any $s$-dependent perturbation $I_s'$ (compactly supported in $s$) {\homotopying} between two given perturbations $I'_-,I'_+$, with $\|I_s - I\|<\delta_E$, all Floer continuation solutions for $(H_{\lambda},I_s)$ of energy at most $E$ will satisfy the filtration inequality.
\end{lm}
\begin{proof}
Suppose by contradiction that this fails.
Then there is a sequence of $\omega$-compatible almost complex structures $I_n$ agreeing with $I$ outside of the compact subset $K$, and a sequence $u_n$ of Floer trajectories for $(H_{\lambda},I_n)$, with energy $E(u_n)\leq E_0$, with $F(x_{n,-})<F(x_{n,+})$ (here $x_{n,\pm}$ are the $1$-orbits of $H_\lambda$ to which $u_n$ exponentially converges, as it has finite energy).
As there are finitely many filtration values for $1$-orbits of $H_{\lambda}$, we deduce that $F(x_{n,-})-F(x_{n,+})\leq e_0$ for some $e_0<0$.
As the initial points $x_{n,\pm}(0)$ lie in a compact subset of $Y$, by passing to a subsequence we may assume $x_{n,\pm}\to x_{\pm}$ converge to $1$-orbits $x_{\pm}$ of $H_{\lambda}$. 
By Gromov compactness, the Floer trajectories $u_n$ converge in $C^{\infty}_{\mathrm{loc}}$ to a broken Floer trajectory $u$ of $(H_{\lambda},I)$ with ends $x_{\pm}$, possibly with finitely many $I$-holomorphic spheres attached, and with $E(u)\leq E_0$. So $F(x_{n,-})-F(x_{n,+})\leq e_0<0$ converges to the inequality $F(x_{-})-F(x_{+}) \leq e_0 <0$ contradicting $F(x_-)\geq F(x_+)$ in \eqref{FiltrationByF} for the unperturbed Floer solution $u$ of $(H_{\lambda},I)$.

The second claim is a consequence of the standard transversality argument \cite{Salamon-Zehnder, FloerHoferSalamon}, independently of the issue of filtrations. 
The final claim is proved analogously: now $u_n$ are continuation solutions for $(H_{\lambda},I_{s,n})$ with $I_{s,n}\to I$ as $n\to \infty$, but the Gromov limit $u$ is still a broken bubbled trajectory $u$ of $(H_{\lambda},I)$ with $I$-bubbles.
\end{proof}

\subsection{Energy-bracketed Floer cohomology}\label{Subsection Filtered Floer cohomology}

The energy bound $E_0$ is essential in \cref{Lemma perturb I subject to energy bound}. A priori, the size of the perturbation $I$ needed to achieve the filtration estimate may decay to zero as we let $E_0$ grow. There is the danger of arbitrarily high-energy conspiring Floer solutions which ``disappear'' in the limit, meaning: there may exist Floer solutions $u_n$ of $(H_{\lambda}, I_n)$, with $E(u_n)\to \infty$ and $F(x_{n,-})-F(x_{n,+})<0$, which fail to have a subsequence admitting a Gromov limit. 
We circumvent this issue by imposing an artificial bound on the energy by the following algebraic trick. The construction we carry out is essentially that of ``filtered Floer cohomology'' introduced by Ono\footnote{Ono's paper uses capped $1$-periodic orbits, and works with the Floer action functional on a cover of the free loop space. We could have done the same (for a non-contractible $1$-orbit $x: S^1 \to Y$, the analogue of a capping is then a choice of homotopy between $x$ and a chosen representative loop in the free homotopy class $[x]\in [S^1, Y]$). We opted for bypassing these matters since we are using a simpler less refined version of the Novikov ring compared to \cite{Ono}.} in \cite[Sec.3]{Ono}.

Initially, let us ignore the almost complex structure. 
To simplify the discussion, we assume that we use the Novikov field $\k$ from \cref{EqnSec2NovikovField} (e.g.\,when $c_1(Y)=0$).
Define a modified Novikov ring 
$$
\k_{E_0} = \k_{\geq 0}/T^{E_0}\k_{\geq 0}
$$
where $\k_{\geq 0} = \{ \sum n_j T^{a_j}: a_j\geq 0\in\R, a_j\to \infty,n_j\in\mathbb{B}  \}$. The Floer complex 
$$C^*_{E_0}:=CF^*(H_{\lambda},\k_{E_0})$$ 
is the $\k_{E_0}$-module generated by certain $1$-orbits of $H_{\lambda}$ (if one uses a cascade model, the auxiliary Morse functions chosen on the {\MB } manifolds of $1$-orbits select a subset of $1$-orbits, see \cref{AppendixCascades}). The Floer differential acts as a matrix with $\k_{\geq 0}$ entries, since the Novikov weights $T^{E(u)}$ in the counts involve non-negative energies $E(u)\geq 0$. The matrix is therefore well-defined
also over $\k_{E_0}$ (and still squares to zero), which is tantamount\footnote{Indeed even in the proof that the differential squares to zero we ignore higher energy solutions, because the proof involves looking at broken solutions arising as Gromov limits when compatifying  $1$-dimensional moduli spaces of Floer solutions of a specified energy value, and the (non-negative) energy of pieces of the broken solution adds up to that value.}
to ignoring Floer trajectories of energy greater than $E_0$. The natural quotient map $\k_{E_0}\leftarrow \k_{E_1}$ for $E_0\leq E_1$ induces a natural chain map
\begin{equation}\label{Eqn increase energy bracket}
C_{E_0}^*\leftarrow C_{E_1}^*.
\end{equation}
The inverse limit $C^*:=\varprojlim C^*_{E}$ over these maps as $E\to \infty$ can be identified with $CF^*(H_{\lambda},\k_{\geq 0})$.
Observe that $\k$ is the localisation of $\k_{\geq 0}$ at $T$, and the localisation $C^*\otimes_{\k_{\geq 0}} \k$ of $C^*$ at $T$ can be identified with $CF^*(H_{\lambda},\k)$. The exact functor of localisation induces an isomorphism on cohomology,
$$
HF^*(H_{\lambda},\k) \cong H^*(C^*,\k_{\geq 0})\otimes_{\k_{\geq 0}} \k.
$$
We need to refine the above argument, by taking into account the almost complex structure.
For now, we assume that a choice has been made of some admissible Hamiltonians $H_{\lambda}$ labelled by a sequence of the generic slopes $\lambda=\lambda_i>0$ which diverge to infinity (later we will compare two such choices). We assume that $H_{\lambda},\phi$ are constructed as in \cref{Cor Subcomplex Trick}, so that for $\lambda_i< \lambda_{i+1}$ the Hamiltonians $H_{\lambda_i},H_{\lambda_{i+1}}$ agree up to the linearity interval $L_i$. 

In particular, this means that an admissible monotone homotopy $H_s$ from $H_{\lambda_{i+1}}$ to $H_{\lambda_i}$ exists (see \cref{Definition monotone hpy}), satisfying the more general condition in \cref{Lemma Floer continuation maps are ok}.

Let $\mathcal{I}_{\lambda}$ denote the space (with the $C^{\infty}$-topology) of $\omega$-compatible almost complex structures on $Y$ equal to $I$ outside of a chosen compact subset $K_{\lambda}\subset Y$ containing all $1$-periodic orbits of $H_{\lambda}$ in its interior, and equal to $I$ over all\footnote{It suffices that this holds on specified subintervals of the $L_i$, and that in arguments where we need $L_i$ to be large (in monotonicity lemma arguments on $B$) we also ensure that such subintervals of $L_i$ are large. We recall that to ensure transversality for Floer solutions, the genericity of $I'$ just needs to be ensured in some region that the solution crosses through. For example, one might wish to perturb $I$ near the boundary of the linearity interval, so that transversality is ensured for any Floer solution that tries to enter the linearity region.} of the linearity intervals $L_i$.

Let $\mathcal{I}_{E,\lambda}\subset \mathcal{I}_{\lambda}$ denote a path-connected neighbourhood of $I$ consisting of $I'$ for which all Floer solutions of $(H_{\lambda},I')$ of energy at most $E$ satisfy the filtration inequality, and all Floer continuation solutions of $(H_{\lambda},I_s)$ for any homotopy $I_s\subset \mathcal{I}_{E,\lambda}$ of energy at most $E$ satisfy the filtration inequality (that such a neighbourhood exists is implied by \cref{Lemma perturb I subject to energy bound}). When we refer to ``generic'' $I'$ or $I_s$ we are referring to the Baire second category subspaces consisting respectively of those $I'$ for which transversality holds for Floer trajectories of $(H_{\lambda},I')$ and those $I_s$ for which transversality holds for continuation maps of $(H_{\lambda},I_s)$ (here $I_s$ need not be generic for each $s$, but we assume that $I_s=I_-$ for $s\ll 0$ and $I_s=I_+$ for $s\gg 0$ for generic $I_\pm$). Denote by $\mathcal{I}_{\epsilon}\subset \mathcal{I}_{\lambda}$ all (or any collection of) neighbourhoods of $I$ labelled by partially ordered parameters $\epsilon$ (e.g.\,$\epsilon$ can be taken to equal the neighbourhood itself, partially ordered by inclusion). We abusively write ``$\epsilon\to 0$'' to mean that we are taking the directed limit over inclusions of the $\mathcal{I}_{\epsilon}$ as the $\mathcal{I}_{\epsilon}$ converge to $I$ in the $C^{\infty}$-topology.

\begin{cor}\label{Corollary defining energy filtered Floer cx}
For generic $I'\in \mathcal{I}_{E,\lambda}$, the Floer complex 
$$CF^*_{E,\lambda}(I'):=CF^*(H_{\lambda},I',\k_E)$$
is well-defined and the differential respects the filtration. The quotient $\k_{E} \leftarrow \k_{E'}$ for $E\leq E'$ induces an ``energy-restriction map'' for any $I'\in  \mathcal{I}_{E',\lambda}$,
$$
\psi_{E,E'}:CF^*_{E,\lambda}(I') \leftarrow CF^*_{E',\lambda}(I').
$$
Any generic homotopy $I_s\subset \mathcal{I}_{E,\lambda}$ from $I_-$ to $I_+$ induces a filtration preserving continuation map
$$
\varphi_{I_s}: CF^*_{E,\lambda}(I_+) \to CF^*_{E,\lambda}(I_-),
$$
and is a quasi-isomorphism. 
For 
generic $I_s\subset \mathcal{I}_{E',\lambda}$, and $E\leq E'$, these fit into a commutative diagram,
$$
\xymatrix@C=45pt@R=20pt{
 CF^*_{E,\lambda}(I_+) 
 \ar@{<-}_-{\psi_{E,E'}}[d]
 \ar@{->}^-{\varphi_{I_s}}[r] &
  CF^*_{E,\lambda}(I_-)
  \ar@{<-}^-{\psi_{E,E'}}[d]
 \\
 CF^*_{E',\lambda}(I_+) 
  \ar@{->}^-{\varphi_{I_s}}[r] &
  CF^*_{E',\lambda}(I_-)
}
$$
Given an admissible monotone homotopy $H_s$ from $H_{\lambda_-}$ to $H_{\lambda_+}$, there is a neighbourhood $\mathcal{I}_{E,\lambda,\lambda'}\subset \mathcal{I}_{E,\lambda'}$ of $I$ for which the continuation map
$$
\varphi_{H_s}: CF^*_{E,\lambda_+}(I') \to CF^*_{E,\lambda_-}(I').
$$
is filtration-preserving. By making the linearity intervals sufficiently large in the construction of the $H_{\lambda}$, the $\varphi_{H_s}$ will be inclusions of subcomplexes. The $\varphi_{H_s}$ maps commute with $\psi_{E,E'}$, and up to quasi-isomorphism they commute with $\varphi_{I_s}$.

In particular, for any given slopes and energies, any sufficiently small $\epsilon$ will ensure that the above statements hold for all $I'$, respectively $I_s$, inside $\mathcal{I}_{\epsilon}$. 
\end{cor}
\begin{proof}
These claims follow essentially immediately from \cref{Lemma perturb I subject to energy bound}, so we just make some comments.

Regarding the energies, once we impose an energy bound $E_0$, then the proof that the Floer continuation map is a chain map over $\k_{E_0}$
only requires looking at $1$-dimensional moduli spaces of Floer continuation solutions of total energy up to $E_0$, and when such solutions break they can only give rise to a Floer trajectory of energy at most $E_0$ since energy is additive.

The commutativity of the diagram follows from the fact that both horizontal maps are induced by the natural quotient map $\k_{\geq 0} \to \k_E=\k_{\geq 0}/T^E\k_{\geq 0}$ applied to the continuation map $\varphi_{I_s}: CF^*(H_{\lambda},I_+,\k) \to CF^*(H_{\lambda},I_-,\k)$ (this map is well-defined but may not preserve the filtration). The same reasoning implies that $\varphi_{H_s}$ commutes with $\psi_{E,E'}$. That $\varphi_{H_s}$ and $\varphi_{I_s}$ commute up to quasi-isomorphism is a standard fact about Floer continuation maps.

Regarding $\varphi_{H_s}$: by \cref{Lemma perturb I subject to energy bound}, for $I'$ close to $I$, all continuation solutions respect the filtration. We then run the same proof as in \cref{Cor Subcomplex Trick} (in particular see \cref{Remark montone hpy is ok}), using that $I_s=I$ on the linearity intervals, so the monotonicity lemma for $B$ holds there. 

The final statement is the observation that given $E,\lambda$ we have $I\in \mathcal{I}_{\epsilon}\subset \mathcal{I}_{E,\lambda}$ for all ``small $\epsilon$''.
\end{proof}

\subsection{Period-bracketed Floer cohomology}\label{Subsection Period-filtered Floer cohomology}

\begin{de}
Let $CF^*_{E,\lambda,[0,b]}(I')\subset CF^*_{E,\lambda}(I')$ be the subcomplex generated by all $1$-orbits whose associated $S^1$-periods lie in $[0,b]$. By \cref{Corollary defining energy filtered Floer cx} this is well-defined (see also \cref{Subsection Dependence of the filtration on phi}, and notice here we filter by period, rather than by $F$, so imposing ``period$\,\leq b$'' yields a subcomplex).
Its cohomology 
$HF^*_{E,\lambda,[0,b]}(I')$ does not depend up to isomorphism on the choice of $I'\in \mathcal{I}_{E,\lambda}$. Similarly, one can define the quotient complex $$CF^*_{E,\lambda,(a,b]}(I'):=CF^*_{E,\lambda,[0,b]}(I')/CF^*_{E,\lambda,[0,a]}(I'),$$ and $HF^*_{E,\lambda,(a,b]}(I')$ will be independent of such $I'$.
We call $(a,b]$ the {\bf period filtration bracket}. If in the quotient we used $[0,a)$ instead of $[0,a]$, we would get period bracket $[a,b]$, etc. 

Any filtration preserving map (see \cref{Definition filtration meaning})
induces a map with period brackets imposed, e.g.\,the commutative diagram in \cref{Corollary defining energy filtered Floer cx} holds with any period bracket.
\end{de}

\begin{rmk}
It is more intrinsic that the labels $[a,b]$ refer to $S^1$-period values $T_{-i}$ (and period zero for the constant orbits), rather than using $F$-filtration value brackets. This is because the $F$-values change if we choose a different construction of the $H_{\lambda}$ (see Section \ref{Subsection Dependence of the filtration on phi} and \cref{FiltrValue}).
\end{rmk}

\begin{cor}\label{Cor HF and LES of filtered HF}
Let $\mathcal{P}$ be any period bracket, e.g.\,$[a,b]$ or $(a,b]$ for $a,b\in [0,\infty]$.
Up to filtration preserving Floer continuation isomorphisms, the $\k_E$-modules
$$
HF^*_{E,\lambda,\mathcal{P}}:= \varinjlim_{I'\to I} HF^*_{E,\lambda,\mathcal{P}}(I'):= \varinjlim_{\epsilon\to 0} \{ HF^*_{E,\lambda,\mathcal{P}}(I'): I'\in \mathcal{I}_{\epsilon}\}
$$
are independent of the choices of $\mathcal{I}_{E,\lambda}$ or the homotopies $I_s$. These fit into a commutative diagram,
$$
\xymatrix@C=45pt@R=20pt{
 HF^*_{E,\lambda,\mathcal{P}}
 \ar@{<-}_-{\psi_{E,E'}}[d]
 \ar@{->}^-{\varphi_{H_s}}[r] &
  HF^*_{E,\lambda',\mathcal{P}}
  \ar@{<-}^-{\psi_{E,E'}}[d]
 \\
 HF^*_{E',\lambda,\mathcal{P}}
  \ar@{->}^-{\varphi_{H_s}}[r] &
  HF^*_{E',\lambda',\mathcal{P}}
}
$$
thus the following $\k$-module is well-defined
$$
HF^*_{\lambda,\mathcal{P}}:=\k\otimes_{\k\geq 0}\varprojlim_{E\to \infty} HF^*_{E,\lambda,\mathcal{P}}.
$$
Up to a filtration-preserving isomorphism induced by Floer continuation maps, this $\k$-module does not depend on the choices of $H_{\lambda}$.
\end{cor}
\begin{proof}
This mostly follows by \cref{Corollary defining energy filtered Floer cx}, up to some additional comments.

Continuation maps only depend on the choice of homotopies up to chain homotopy, so the maps $\varphi_{I_s}$ and $\varphi_{H_s}$ on cohomology are independent of the choice of generic homotopies $I_s$ and $H_s$. By virtue of the colimit definition, $HF^*_{E,\lambda,\mathcal{P}}$ is  independent of the choices of $\mathcal{I}_{E,\lambda}$.

Recall that continuation maps compose correctly at the cohomology level, indeed one can work with homotopies of the pairs $(H_s,I_s)$; and composites of such maps up to chain homotopy agree with the continuation map induced by concatenating the homotopies of those pairs. Thus the maps $\varphi_{H_s}$ and $\varphi_{I_s}$ (and, as discussed above, also $\psi_{E,E'}$) all fit into filtration-preserving commutative diagrams on cohomology whenever they are defined.

Consider now two different choices of constructions of admissible Hamiltonians $H_{\lambda_1}^{(1)}$, $H_{\lambda_2}^{(2)}$ (where $\lambda_1$ ranges over some cofinal collection of values, and so does $\lambda_2$). To build an isomorphism between the two colimits that these two choices induce, it suffices to construct filtration-preserving continuation maps%
\footnote{The inverse map on direct limits will be constructed similarly via maps $HF^*(H_{\lambda_1}^{(1)},I',\k_E) \to HF^*(H_{\lambda_2'}^{(2)},I',\k_E)$ when $\lambda_1\leq \lambda_2'$. The fact that the composites of these two maps is the identity map follows because the composites $HF^*(H_{\lambda_2}^{(2)},I',\k_E) \to HF^*(H_{\lambda_1}^{(1)},I',\k_E) \to  HF^*(H_{\lambda_2'}^{(2)},I',\k_E)$ are continuation maps $HF^*(H_{\lambda_2}^{(2)},I',\k_E) \to  HF^*(H_{\lambda_2'}^{(1)},I',\k_E)$ which are part of the colimit definition, so they actually define the identity map on the colimit.} 
$HF^*(H_{\lambda_2}^{(2)},I',\k_E) \to HF^*(H_{\lambda_1}^{(1)},I',\k_E)$ when $\lambda_2\leq \lambda_1$ and $I'$ is sufficiently close to $I$. As $\lambda_2\leq \lambda_1$, we can always build such a filtration-preserving admissible monotone continuation map by \cref{Cor continuation is always possible}, and being continuation maps they will automatically be compatible with the $\varphi_{I_s}$-continuation maps on cohomology, and admissibility ensures they count Floer continuation solutions with non-negative $T$ powers thus they are also compatible with the $\psi_{E,E'}$-maps.
\end{proof}

\subsection{Period-bracketed symplectic cohomology}\label{Subsection Period-filtered symplectic cohomology}
Morally, we have filtered the Floer cohomology of $(Y,\varphi)$ by ``action brackets'' induced by the action functional on the projected space $B$ via $\Psi$. This is not entirely accurate, as we use a cut-off function $\phi$ to ignore the contributions of the action functional in certain regions. The natural inclusion/truncation of action brackets induces a long exact sequence:

\begin{cor}\label{Cor long exact sequence period brack}
For $a\leq b\leq c$ in $[0,\infty]$, there is a long exact sequence of $\k$-modules
$$
\cdots 
\to 
HF^*_{\lambda,[a,b]}
\to 
HF^*_{\lambda,[a,c]}
\to 
HF^*_{\lambda,(b,c]}
\to 
HF^{*+1}_{\lambda,[a,b]}
\to
\cdots
$$
and its $\lambda$-colimit induces a long exact sequence on period-bracketed symplectic cohomology groups,
$$
\cdots 
\to 
SH^*_{[a,b]}(Y,\varphi)
\to 
SH^*_{[a,c]}(Y,\varphi)
\to 
SH^*_{(b,c]}(Y,\varphi)
\to 
SH^{*+1}_{[a,b]}(Y,\varphi)
\to
\cdots
$$
In the case $a=0$, $b=0$, $c=+\infty$, we get the usual long exact sequence relating the $\k$-modules $QH^*(Y)$, $SH^*(Y,\Fi)$, and $SH_+^*(Y,\Fi):=SH^*_{(0,\infty]}(Y,\Fi).$ \qed
\end{cor}

\begin{de} The Corollary ensures that the following $\k$-module is well-defined,
$$
SH^*_{[0,b]}(Y,\Fi):=\varinjlim_{\lambda\to \infty} HF^*_{\lambda,[0,b]}=
\varinjlim_{\lambda\to \infty} \left(\k\otimes_{\k\geq 0}
\varprojlim_{E\to \infty} 
\varinjlim_{I'\to I} HF^*_{E,\lambda,[0,b]}(I')\right).
$$
\end{de}

\begin{cor}
$SH^*_{[0,b]}(Y,\Fi)$ is a graded-commutative unital $\k$-algebra admitting a natural unital $\k$-algebra homomorphism
$QH^*(Y)\to SH^*_{[0,b]}(Y,\Fi)$.
\end{cor}
\begin{proof}
That the filtration is compatible with (a careful construction of) the pair-of-pants product was shown in \cite{RZ1}.
For $\lambda''\geq 2 \,\mathrm{max}(\lambda,\lambda')$, the pair-of-pants product partly respects the period filtration: 
$$HF^*_{\lambda,(a,b]}\otimes HF^*_{\lambda',(a',b']} \to HF^*_{\lambda'',(\mathrm{max}(a+b',a'+b),\;b+b']} \quad \textrm{ for any }a,b,a',b'\in \R\cup \{\pm \infty\}$$
(by convention $a+b'=-\infty$ if $a=-\infty$, and similarly for $a'+b$). For $a=a'=-\infty$, we obtain the claimed product structure on $SH^*_{[0,b]}(Y,\Fi)$.
The natural map from $QH^*(Y)$ arises from the PSS-isomorphism $QH^*(Y)\cong HF^*_{\lambda,[0,0]}\cong HF^*_{\textrm{small slope}}$ (the latter's chain level generators are the $1$-periodic orbits of $H_{\lambda}$ with filtration value $F=0$), followed by the natural map into the direct limit.
\end{proof}

Letting $b=\lambda$ be an admissible slope for which we chose $H_{\lambda}$ in the construction, we deduce that
$$
SH^*_{[0,\lambda]}(Y,\Fi)\cong HF^*_{\lambda}.
$$
More generally, if $b=\mu$ equals an $S^1$-period value $T_{-i}=c'(H)=\mu\leq \lambda$, then
$$
HF^*_{\lambda,[0,b]} \cong HF^*_{\mu^+} \cong SH^*_{[0,\mu^+]}(Y,\Fi)
$$
where $\mu^+$ is any admissible slope\footnote{or explicitly use $H_{\mu^+}$ constructed from $H_{\lambda}$ by extending with slope $\mu^+$ in the region at infinity where $H_{\lambda}$ has slope$\,\geq\mu^+$. The filtration bracket $[0,b]$ means we ignore the new $1$-orbits that $H_{\lambda}$ has which $H_{\mu^+}$ does not, and that $CF^*_{\mu^+}(I')\subset CF^*_{\lambda}(I')$ is a subcomplex for $I'$ sufficiently close to $I$.} in the interval $(T_{-i},T_{-(i+1)},)$.
Under these identifications, a continuation map $HF^*_{\lambda}\to HF^*_{\lambda'}$ for $\lambda\leq \lambda'$ corresponds to the natural map that enlarges the right-bracket,
$$
SH^*_{[0,\lambda]}(Y,\Fi) \to SH^*_{[0,\lambda']}(Y,\Fi).
$$
\subsection{The $\Fi$-filtration revisited}\label{Subsection The fi filtration revisited}
We can now rewrite \cref{Definition fi-filtration on QH} as follows.

\begin{cor}The $\Fi$-filtration ideals of $QH^*(Y)$\,are
$
\FF^{\varphi}_\lambda :=\bigcap_{\mathrm{generic}\,\mu>\lambda} \ker (QH^*(Y) \to SH^*_{[0,\mu]}(Y,\Fi)).
$%
\end{cor}

\begin{rmk}\label{Remark limits and colimits}
If one wished to carry out a more refined construction at the chain level, one could try to apply the telescope model methods by Varolgunes \cite{Varolgunes}.

We comment on the order of the limits/colimits.
The colimit over $I'$ must be done first, otherwise the filtration estimates may fail. In general, the order of the $\lambda$-colimit and the $E$-limit in the definition of $SH^*(Y,\Fi)$ is important.
Having the $\lambda$-colimit on the outside means that any given class can be realised for some large enough slope $\lambda$, and a representative is described by following a ``zig-zag'' diagram
$$
\xymatrix@R=12pt{
 CF^*_{E,\lambda}(I) 
 \ar@{->}^-{\varphi_{I_s}}[r] 
 &
 CF^*_{E,\lambda}(I') 
 \ar@{->}^-{\varphi_{I_s}}[r]  
  &
  \cdots 
  \ar@{->}^-{\varphi_{I_s}}[r] 
  & 
  CF^*_{E,\lambda}(I^{(N)})
  \ar@{<-}^-{\psi_{E,E'}}[d]
  &
 \\
 &
 &
 &
 CF^*_{E',\lambda}(I^{(N)})
 \ar@{->}^-{\varphi_{I_s}}[r] 
 &
 \cdots
}
$$
where one can keep lifting the cycle to higher-energy brackets ($E'\geq E$ above) at the cost of applying enough quasi-isomorphisms to the cycle to reach an almost complex structure close enough to $I$ ($I^{(N)}$ above). If however the $E$-limit were on the outside, it may happen that a ``zig-zag'' requires repeated increases of $\lambda$, as it may otherwise not be possible to find a lifted \emph{cycle}. Indeed, a cycle $c$ for $(H_{\lambda},\k_E)$ will lift to a cycle $c+T^E c'$ for $(H_{\lambda},\k_{E'})$ precisely if the new boundary terms of order $T^e$, with $e\in [E,E']$, appearing in $\partial c$ can be cancelled by $T^E \partial c'$ for some $\k_{\geq 0}$-combination of $1$-orbits of $H_{\lambda}$. Increasing $\lambda$ to $\lambda'$ means more $1$-orbits are at our disposal to build $c'$.
A simple situation, where lifting cycles is not a problem, occurs when no new $1$-orbits appear in degree $\deg (c)$ for large enough $\lambda$.
\end{rmk}

\begin{cor}
For all symplectic $\C^*$-manifolds with $c_1(Y)=0$ (e.g.\,all CSRs), for any period bracket $\mathcal{P}$ we may interchange the following (co)limits,
$$
SH^*_{\mathcal{P}}(Y,\Fi)\cong 
\k\otimes_{\k\geq 0}
\varinjlim_{\lambda\to \infty} 
\varprojlim_{E\to \infty} 
\varinjlim_{I'\to I} HF^*_{E,\lambda,\mathcal{P}}(I')
\cong 
\k\otimes_{\k\geq 0}
\varprojlim_{E\to \infty} 
\varinjlim_{\lambda\to \infty} 
\varinjlim_{I'\to I} HF^*_{E,\lambda,\mathcal{P}}(I').
$$
Given a degree $k$, for $\lambda$ sufficiently large, and any $a,b\in \R\cup \{\pm \infty\}$, we can identify
$$
SH^k_{\mathcal{P}}(Y,\Fi)
\cong 
HF^k_{\lambda,\mathcal{P}} := \k\otimes_{\k\geq 0}
\varprojlim_{E\to \infty}  
\varinjlim_{I'\to I} HF^*_{E,\lambda,\mathcal{P}}(I').
$$
In particular, picking $\lambda$ large enough so that the above holds in all degrees $k\in [-1,2\dim_{\C} Y)$, then for $I'$ sufficiently close to $I$ we have a $\k$-isomorphism
$$
\partial: HF^{k-1}_{\lambda,(0,\infty]}(I') \stackrel{\cong}{\longrightarrow} QH^k(Y;\k)
$$
induced by the Floer differential,\footnote{The image of $\partial(c)$ on a class $[c]\in HF^{k-1}_{\lambda,(0,\infty]}(I')$ lies in filtration zero, since $[\partial c]=0\in HF^{k-1}_{\lambda,(0,\infty]}(I')$.} for $k\in [0,2\dim_{\C}Y)$ (i.e.\,where $QH^k(Y;\k)$ may be supported). 

The $\Fi$-filtration ideals defined by \cref{Definition fi-filtration on QH} on $QH^*(Y;\k)$ are then precisely the images 
$$
\partial: HF^{*-1}_{\lambda,(0,b)}(I') \longrightarrow QH^*(Y;\k)
$$
as $b>0$ varies (this image can change only for outer $S^1$-period values $b$, which are non-admissible slopes $\lambda$, where the $\Fi$-filtration of \cref{Definition fi-filtration on QH} can jump).
\end{cor}
\begin{proof}
That localisation $\k \otimes_{\k_{\geq 0}}$ commutes with the $\lambda$-colimit
is the categorical fact that two colimits always commute (localisation at $T$ of a $\k_{\geq 0}$-module $M$ can be viewed as a colimit over inclusions of the modules $\k_{\geq  a}\otimes_{\k_{\geq 0}} M$ as $a \to -\infty$).
By \cref{PropSec2RSIndices}, 
in any given degree there are only finitely many Floer chain level generators ($1$-orbits of $H_{\lambda}$ corresponding to critical points of auxiliary Morse functions on {\MB } manifolds of $1$-orbits of $H_{\lambda}$, for $\lambda$ sufficiently large, cf.\,\cref{AppendixCascades} and \cref{Rmk perturbns of Ham issue}).
So the condition at the end of \cref{Remark limits and colimits} ensures the $\lambda$-colimit and $E$-limit commute.

As the chain complex in a given degree $k$ is finitely generated, the $\lambda$-colimit is taken over linear maps $\varphi_j$ of finite dimensional $\k$-vector spaces $V_j$ of uniformly bounded dimension. In such a general situation, $V_1 \stackrel{\varphi_1}{\to} V_2 \stackrel{\varphi_2}{\to} V_3 \to \cdots$, by considering the dimension of the direct limit, there must be a finite $n$ such that all classes are represented in $V_n$, and the direct limit is naturally identifiable with $V_n/\ker \varphi_n^m$ where we pick any large enough $m$ since the kernel of $\varphi_n^m:=\varphi_{n+m}\circ \varphi_{n+m-1}\circ \cdots \circ \varphi_{n}$ stabilises for dimension reasons.
In our case, the $\varphi_i$ are the cohomology maps of inclusions of subcomplexes, and $V_n$ is the Floer cohomology for $H_{\lambda}$ for large enough $\lambda$. We may assume that increasing $\lambda$ further does not produce new $1$-orbits in degrees $k$ and $k-1$, thus $\varphi_n^{m}$ is the identity map on cohomology.

The final claim follows from the long exact sequence in \cref{Cor long exact sequence period brack}, since a class in $QH^*(Y)$ vanishes via the natural map $\widetilde{c}_b:QH^*(Y)\cong HF^*_{\lambda,[0,0]} \to HF^*_{\lambda,[0,b]}$ if and only if it lies in the image of the connecting map $HF^{*-1}_{\lambda,(0,b]}\to HF^*_{\lambda,[0,0]}$ of the LES (and by chasing the definition of the connecting map of a cohomology LES induced by a short exact sequence one sees that the connecting map is the Floer differential before quotienting out any high-filtration terms). By the final claim in \cref{Cor HF and LES of filtered HF}, if $b=T_{-i}=c'(H)=\mu$, then $\widetilde{c}_b$ can be identified with the map $c_{\mu^+}$ used in \cref{Definition fi-filtration on QH}.
\end{proof}

\begin{rmk}[Perturbations of the Hamiltonian]
\label{Rmk perturbns of Ham issue}
In the above discussion, we used the {\MB } (i.e.\,cascade) model for Floer cohomology (\cref{AppendixCascades}), as we did not perturb $H_{\lambda}$. One could carry out a perturbation $H_{\lambda}'$ of $H_{\lambda}$ near the {\MB } manifolds to obtain non-degenerate orbits, keeping in mind this ruins the filtration%
\footnote{\eqref{EqnPhiProjectionNice} fails if we perturb $H_{\lambda},$ one gets a non-vanishing term $d\a(\xi,W)$ in equation \eqref{omegaV}, where $W:=\Psi_*(X_{\widetilde{H}_\lambda}-X_{H_\lambda})$ is the projection of the difference in Hamiltonian vector fields in $Y$ that occurred by perturbation. The change in the difference of the filtration values \eqref{F_DifferenceInTermsOfOmega} is caused by integration $\int_{-\infty}^{+\infty} d\a(\partial_s v,W)ds$ of this term, which we cannot estimate.}
analogously to perturbations of $I$. Using the perturbation procedure described in \cref{Appendix {\MBF } theory: perturbations}, one can ensure that non-constant $1$-orbits are a subset of the original $1$-orbits and so project via $\Psi$ to $S^1$-orbits in $B$, and their filtration values are unchanged. The same Gromov compactness arguments used for perturbations $I'$ of $I$ apply to perturbations $H_{\lambda}'$ of $H_{\lambda}$, and similarly for continuation maps we use perturbations $I_s',H_s'$ of $I,H_s$. Thus the same arguments as above hold after replacing the innermost $\varinjlim$ over $I'\to I$ by a $\varinjlim$ over pairs $(I',H') \to (I,H)$.
\end{rmk}

\subsection{Adapting energy-bracketing to the $S^1$-equivariant setup}\label{Subsection Filtered Floer cohomology equiv setup}

In \cref{Subsection intro S1 spectral seq}
we work with $S^1$-equivariant Floer cohomology in the sense of \cite{McLR18}, where a $1$-orbit contributes a copy of the $\k [\![u]\!]$-module $\mathbb{F}$, rather than $\k$, where 
$\mathbb{F}=\k(\!(u)\!)/u\k [\![u]\!]\cong H_{-*}(\C\P^{\infty})$.
In \cref{Subsection Filtered Floer cohomology} we explained energy-bracketing using $
\k_{E_0} = \k_{\geq 0}/T^{E_0}\k_{\geq 0}
$, so that each $1$-orbit contributed to $C^*_{E_0}:=CF^*(H_{\lambda},\k_{E_0})$ a copy of  $\k_{E_0}$. In the $S^1$-equivariant setup, we instead get a contribution of $\mathbb{F}_{E_0}$ for each $1$-orbit, where
$$
\mathbb{F}_{E_0}:=
\k_{E_0}(\!(u)\!)/u\k_{E_0} [\![u]\!].
$$
The non-equivariant methods then adapt very easily to the $S^1$-equivariant case, thus will be omitted. 
\section{Morse--Bott--Floer spectral sequences}\label{SpecSeqMBF}
\newcommand{\RZ}[1]{Y_{#1}}

\subsection{Morse--Bott manifolds of Hamiltonian 1-orbits of $H_\lambda$}\label{SubsectionHam1-orbitsOfHlambda}
We continue with the initial assumption from \cref{filtrationFloer}. %
Call {\bf 1-orbits} the 1-periodic orbits of the Hamiltonian $H_\lambda$ from \cref{filtrationFloer}. As $X_{H_\lambda}=c'(H)X_H,$ the 1-orbits %
are either constant $1$-orbits at points in $\F=Crit(H_\lambda)=Crit(H),$ or, for some $p=-1,\dots,-r,$ are non-constant $1$-orbits in a {\bf slice} 
\begin{equation}\label{DefinionSlice}
	\Slicep:=\{y \in Y \mid H(x)=H_p\}\subset  Y,
\end{equation}
where the values $H_p$ are defined by $c'(H_p)=T_p.$ 
Denote by $\mathcal{O}_p:=\mathcal{O}_{p,H_{\lambda}}$ the moduli space of parametrised $1$-orbits of $H_\lambda$ in $\Slicep,$ and by $B_p$ the set of their initial points, $$B_p:=\{x(0)\mid x\in \mathcal{O}_p \}\subset \Slicep.$$  
It will be convenient to allow the notation in the case $p=0$, for the constant $1$-orbits $B_0=\mathcal{O}_0=\F$, so $T_0=0$ (also $T_p=c'(T_p)$ will be an abuse of notation when $p=0$, as $H_0$ is not well-defined since $H$ may take different values on the connected components $\F_\a$ of $\F$, and we do not require $c'=0$ there).
By \cref{FiltrValue}, the filtration value $F$ on $1$-orbits only depends on $c'(H_p)=T_p$, so let
$$
F_p:=F(\mathcal{O}_p) \textrm{ for } p\leq 0, \textrm{ and }F_p=p \textrm{ for } p \geq 0.
$$
Recall $X_{S^1}=X_H$ is the vector field of the $1$-periodic $S^1$-flow $\varphi_{e^{2\pi it}}$, so $t\in \R/\Z \cong S^1$. 
Recall $Y_m$ is the set of $\Z/m$-torsion points in $Y$ (\cref{TorsionPtsTorsionSubmflds}), so fixed points of the $S^1$-flow for time $1/m$.

\begin{lm}\label{BpcAreTorsionIntersectedBySlice}
	$\displaystyle B_p= \Sp \cap \RZ{m}$ %
	where $T_p=k/m$ for coprime $k,m\in\N.$
\end{lm}
\begin{proof}
	$B_p$ is the union of the  period $T_p$ orbits of the $S^1$-flow on $\Sp$, as 
$X_{H_\lambda}=c'(H_p)X_H = T_p X_{S^1}$ on $\Sp.$
As all proper closed subgroups of $S^1$ are discrete, 
	$T_p$ has to be a rational number, $T_p=k/m$ for some coprime $k,m\in\N.$
	The latter condition, $(k,m)=1,$ implies $\a k + \b m=1$ for some $\a,\b\in\Z$.  Thus, points in $B_p$ are also fixed by the  $S^1$-flow for time
	$\a k/m = (1/m)- \b$, and thus also for time $1/m$, so they lie in $\RZ{m}.$ Thus, $B_p \subset \Sp \cap \RZ{m}$.
	Conversely, a point in $\Sp \cap \RZ{m}$ is fixed by the  $S^1$-flow for time $1/m$, and thus also for time $T_p$, so it lies in $B_p.$ 
\end{proof}

\begin{cor}\label{Cor dim of Bpbeta}
The connected components of $B_p=\Sp \cap Y_m$ are precisely the smooth submanifolds
\begin{equation}\label{EqConnCompOfBp}
	B_{p,\c}=\Sp \cap Y_{m,\c},
\end{equation}
where the $\beta$-labelling runs over the non-compact connected components $Y_{m,\b}$ of $Y_m$.

\end{cor}
\begin{proof}
    By \cref{BpcAreTorsionIntersectedBySlice}, $B_p$ is the fixed locus of the $\Fi$-action by $\Z/m\leq \C^*$ on $\Sp,$ so it is a smooth submanifold of $\Sp$.\footnote{being a fixed locus of a compact Lie group action on a smooth manifold, \cite[p.108]{DK00}.} 
    By \cref{TorsionMfdIsSymplectic}, $Y_{m,\c}\subset Y$ is a symplectic $\C^*$-submanifold so \cref{FibersMomentMapConnected} applied to $H|_{Y_{m,\c}}: Y_{m,\c} \to \R$ shows that 
    $H|_{Y_{m,\c}}^{-1}(H_p)=\Sp \cap Y_{m,\c}$ is connected.
Observe that only non-compact components $Y_{m,\c}$ intersect $\Sp$. For $\beta\neq \beta'$, the $Y_{m,\b},Y_{m,\b'}$ are disjoint, so the connected closed submanifolds $B_{p,\b},B_{p,\b'}$ are disjoint and are thus distinct connected components of $B_p$.
\end{proof}

\begin{lm}\label{Only the primitive Ones matter}
	The $\beta$-labelling depends on $m,$ not $k$: if $T_{p_1}=k_1/m$, $T_{p_2}=k_2/m$, where $(k_1,m)=(k_2,m)=1$, then $B_{p_1}\cong B_{p_2}$ are diffeomorphic (indeed diffeomorphic to the $B_{p}$ with $T_{p}=1/m$). 
\end{lm}
\begin{proof}
	As
	$B_{p_1,\c} = \Sigma_{p_1} \cap \Ymc$ and $B_{p_2,\c} = \Sigma_{p_2} \cap \Ymc,$ we have
	$B_{p_1,\c} \iso B_{p_2,\c}$
	via the normalised gradient flow $\nabla H/\norm{\nabla H}^2$ that flows from one to the other. This holds since $H$ is constant on $\Sp$, and $\Ymc$ is preserved by the gradient flow of $H$ (which is the action by $\R_+\subset \C^*$).%
\end{proof}
		
Consider\footnote{The orthogonal $\perp$ is calculated with respect to the metric $g(\cdot,\cdot)=\om(\cdot,I\cdot).$} the complex codimension-one distribution $\xi:=(\la X_{\R_+} \ra \oplus \la X_{H_\lambda} \ra)^{\perp}$ on $ Y\setminus \F.$ Thus, for any point $x\in  Y \setminus \F,$ we have orthogonal splittings of tangent spaces 
\begin{equation}\label{splittingTangentSpaceContact}
	T_x  Y= \la X_{\R_+} \ra \oplus \la X_{H_\lambda} \ra \oplus \xi,\ \qquad  \ T_x \Sp =\la X_{H_\lambda} \ra \oplus \xi.
\end{equation}

\begin{de}
	Denote by $\phi_{\tau}^{X}$ the flow of a vector field $X$ for time $\tau\in \R.$ When $X=X_F$ is a Hamiltonian vector-field we
	will abbreviate this flow by $\phi_{\tau}^{F}$.
\end{de}

\begin{lm}\label{LemmaForTheShear}
	$(\phi_{\tau}^{{H_\lambda}})_* X_{\R_{+}}=X_{\R_+}+\tau c''(H)\norm{\nabla(H)}^2 X_{H}.$
\end{lm}
\begin{proof}
	Let $\x\in Y.$ Abbreviate $\gamma(s)=\phi_s^{X_{\R_+}}(\x)=
 e^{2\pi s} \cdot \x$, so $\gamma'(s)=X_{\R_+}(\gamma(s))=(\nabla H)(\gamma(s))$, and
	$$(\phi_{\tau}^{{H_\lambda}})_*(\x) X_{\R_{+}}={\phi_{\tau}^{{H_\lambda}}}_*\left(\left.{\tfrac{d}{ds}}\right|_{s=0}\gamma(s)\right)=
	\left.{\tfrac{d}{ds}}\right|_{s=0} (\phi_{\tau}^{{H_\lambda}}(\gamma(s)))=\left.{\tfrac{d}{ds}}\right|_{s=0}(\phi_{\tau c'(H(\gamma(s)))}^{{H}}(\gamma(s))),$$ 
	using $X_{H_\lambda}=c'(H)X_H.$ 
 Let $u(s,t):=\phi_t^{H}(\gamma(s))$ (which is $e^{2\pi (s+it)} \cdot x$).
The last derivative becomes: 	
$$
		\left.{\tfrac{d}{ds}}\right|_{s=0} u(s,\tau c'(H(\gamma(s))))
        = {\tfrac{\partial u}{\partial s} } +\left(\left.{\tfrac{d}{ds}}\right|_{s=0}\tau c'(H(\gamma(s)))\right) \tfrac{\partial u}{\partial t } 
        = X_{\R_{+}} +\tau c''(H(x))\norm{\nabla H(x)}^2 X_{H}.
		\qedhere
$$
\end{proof}

The $B_{p,\c}$ satisfy a non-degeneracy property analogous to {\MB} critical submanifolds:\footnote{This analogy holds if we think of $B_{p,\c}$ as critical submanifolds of the action $1-$form $d\mathcal{A}_{H_\lambda}(x)(\xi)=\int \om(\dot{x}-X_{H_\lambda},\xi).$}

\begin{prop}\label{TorsionAreMorseBottSubmanifolds}
	Given any point $\x\in B_{p,\c},$ the 1-eigenspace of the linearised return map of the flow for $H_\lambda$
	is precisely the tangent space of $B_{p,\c},$ so
	$\Ker\,(d_\x \phi_1^{{H_\lambda}}-Id)=T_\x B_{p,\c}.$
\end{prop}
\begin{proof}
	As in Lemma \ref{BpcAreTorsionIntersectedBySlice}, let $T_p=k/m,$ for coprime $k,m\in\N.$ 
	Then, on $\Sp$, $$\phi_1^{{H_\lambda}}=\phi_{T_p}^{{H}}=(\phi_{1/m}^{H})^k.$$
	So, given $v\in T_\x \Sp,$ the condition $(\phi_1^{{H_\lambda}})_*(v)=v$ is equivalent to $(\phi_{1/m}^{H})_*^k(v)=v,$
	and thus to $(\phi_{1/m}^{H})_*(v)=v,$ due to $(k,m)=1.$
	Note $\phi_{1/m}^{H}$ generates a $\Z/m$-action on $\Sp$ by isometries (since the metric is $S^1$-invariant), and its fixed
	locus is $B_p$ by Lemma \ref{BpcAreTorsionIntersectedBySlice}. 
	Thus, by \cite[Ch.VI, Thm.2.2]{Bre72}, there is a $\Z/m$-invariant tubular neighbourhood $\mathcal{N}B_{p,\c}\subset \Sp$ of $B_{p,\c}$ 
	arising as the image of the $\Z/m$-equivariant exponential map on a neighbourhood $NB_{p,\c}$ 
	of the zero section of the normal bundle of $B_{p,\c}$.
This identifies $NB_{p,\c}$ with $\mathcal{N}B_{p,\c}$, making $\Z/m$ act on $NB_{p,\c}$ by the linearisation $(\phi_{1/m}^{H})_*$.
	Hence, if a vector $v\in T_x \Sp$ satisfies $(\phi_{1/m}^{H})_*(v)=v,$
	then $v\in (N B_{p,\c})_x$ is fixed by the $\Z/m$-action and so is the curve $exp(vt) \subset \mathcal{N}B_{p,\c}.$ As $B_{p,\c}$ is an isolated fixed locus in $\mathcal{N}B_{p,\c}$, this curve has to lie in $B_{p,\c},$ so $v\in T_\x B_{p,\c}.$ Thus $\Ker((d_\x \phi_1^{{H_\lambda}}-Id)|_{T_\x \Sp})=T_\x B_{p,\c}.$ 
	Note $\R X_{\R_+}$ is a complementary subspace to $T_\x \Sp\subset T_\x Y$, so the claim follows by Lemma \ref{LemmaForTheShear} using that $c''(H)>0$ for $c'(H)=T_p.$
\end{proof}

\begin{lm}\label{LinearizationIsComplexLinear}
	The linearisation of the $S^1$-flow is complex linear with respect to a unitary trivialisation of $\xi$ along any 1-orbit in $B_{p,\c}.$ 
\end{lm}
\begin{proof}
	Follows immediately as the $S^1$-flow is $I$-linear and a unitary trivialisation is complex-linear.
\end{proof}

\subsection{The {\MB } spectral sequence for $HF^*(H_\lambda$)}\label{SpectralSequenceConstruction}
We now build a spectral sequence analogously to \cite[Sec.7]{McLR18}. There, %
the space at infinity was a positive symplectisation, involving a radial coordinate $R$, an exact symplectic form, contact hypersurfaces $R=\mathrm{constant}$ and a Reeb flow (the Hamiltonian flow for $R$). In our setting, on $Y^{\mathrm{out}}$, we instead have the coordinate $H$, a typically non-exact symplectic form, level sets $H=\mathrm{constant}$ and the $S^1$-flow (the Hamiltonian flow for $H$).

As the Hamiltonian $H_\lambda$ from \cref{SectionConstructionOfHLambda} is autonomous, its 1-orbits are not isolated: they arise in {\MB } manifolds by \cref{SubsectionHam1-orbitsOfHlambda}, labelled $B_{p,\c}$ and $\F_\a$.  We will therefore use a \textbf{{\MB } Floer complex} $BCF^*(H_{\lambda})$ as described in \cref{AppendixCascades}, involving a choice of auxiliary Morse functions $f_{p,c}:B_{p,\c}\to \R$ and $f_{\a}:\F_\a\to \R$. To unify the notation, we write $f_i:M_i \to \R$ for these auxiliary Morse functions, denoting the {\MB } manifolds by $M_i$. As explained in \cref{AppendixCascades} we will blur the distinction between the initial points of the $1$-orbits $x_0\in M_i$ and their associated $1$-orbits $x(t)$, $x(0)=x_0$ (in \cref{SubsectionHam1-orbitsOfHlambda} we distinguished them via the notation $B_p$ and $\mathcal{O}_p$). We call {\bf critical $1$-orbits} the $1$-orbits associated to the critical points $x_0\in \mathrm{Crit}(f_i)\subset M_i$. These are the free generators of $BCF^*(H_{\lambda})$. By \cref{Lemma MorseBottFloer is perturbed Floer}, 
there is an isomorphism 
$$BHF^*(H_{\lambda}) \cong HF^*(\widetilde{H}_{\lambda})$$ 
between the {\MB } Floer cohomology (for the unperturbed $H_{\lambda}$) and the Floer cohomology for any generic time-dependent compactly supported perturbation $\widetilde{H}_{\lambda}$ of $H_{\lambda}$; moreover these isomorphisms are compatible with continuation maps. That compatibility ensures that the isomorphisms behave well in the limit constructions used when we energy-bracket as in \cref{Subsection Filtered Floer cohomology}.
So for all intents and purposes we may simply write $HF^*(H_{\lambda})$ instead of $BHF^*(H_{\lambda})$. Also, up to continuation isomorphisms, there is no ambiguity in the meaning of $SH^*(Y,\varphi)$, and we may period-filter as in \cref{Subsection Period-filtered symplectic cohomology}. 

We will be using the terminology introduced in \cref{AppendixCascades} (ledge, drop, crest, base, cascade, etc.). A {\bf simple cascade} is a cascade with no drops, so it is a Morse trajectory inside some $M_i$ for $f_i$. 

\begin{de}
A {\bf vertical drop} is a drop whose crest and base lie in {\MB } submanifolds with the same period value $c'(H)$, i.e.\,the asymptotics lie in the same slice $\Sigma_p$ (see \cref{SubsectionHam1-orbitsOfHlambda}). A {\bf self-drop} is a vertical drop whose crest and base lie in the same {\MB} submanifold.
\end{de}

\begin{lm}\label{Lemma filtration for cascades}
\label{Lemma excluding vertical drops between certain Bott manifolds}
 If there is a cascade solution joining critical $1$-orbits $x_-\in M_i$ to $x_+\in M_j$ for the Hamiltonian $H_{\lambda}$ and the unperturbed almost complex structure $I$, then the filtration satisfies
$$
\hspace{20ex} F(x_-)\geq F(x_+) \qquad (\Leftrightarrow\,\textrm{the periods satisfy }c'(H(x_-))\leq c'(H(x_+)),
$$
with equality if and only if the cascade is simple or all of its drops are vertical drops.
In particular, $F$ remains constant on ledges and vertical drops, and vertical drops for $F\neq 0$ are as in \cref{Lemma Floer solutions for constant filtration value}.

The cascades for $I$ with $F(x_-)=F(x_+)=F_p$ are trapped in the neighbourhood of the slice $\Sigma_p$ where $\phi$ is constant, which only contains $1$-orbits from $\mathcal{O}_p$. For $F_p \neq 0$, the $\Psi$-image of such cascades lies in an $m$-torsion submanifold of $B$ (\cref{TorsionPtsTorsionSubmflds}) where $T_p=c'(H(x_{\pm}))=k/m\in \Q$ for $k,m$ coprime.
If a vertical drop $u$ has asymptotics $u_{\pm}$ in $B_{k/m,\beta_{\pm}}$,
then the $Y_{m,\beta_{\pm}}$ have overlapping $\Psi$-projections in $B^{\mathrm{out}}=\Sigma \times [R_0,\infty)$, as
$\mathrm{Image}\,(\Psi(u))\subset \Psi(Y_{m,\beta_-})\cap \Psi(Y_{m,\beta_+})\neq \emptyset$.

Given $E>0$, for any perturbation $I'$ sufficiently close to $I$ the cascades for $I'$ of energy at most $E$ are close to those for $I$ (in $C^0$ and in $C^{\infty}_{\mathrm{loc}}$) and any vertical drop $u$ from $u_{-}$ to $u_+$ for $F\neq 0$ satisfies
\begin{itemize}
\item $H(u_-)=H(u_+)$, and $u_{\pm}$ are both non-constant $1$-orbits;
\item $u_{\pm}$ correspond to orbits of the $S^1$-flow of equal period $T:=c'(H(u_\pm))$; 
\item $v=\Psi(u)$ lies entirely in a region where $\phi$ is constant; 
\item $v_{\pm}=\Psi(u_{\pm})$ correspond to Reeb orbits of period $T$ (but may not have the same $R$-value); 
\item $\mathrm{Im}(v) $ is close to the $\C^*$-orbit of some point $p\in B$;
\item $\mathrm{Im}(v)$ is close to some $m$-torsion submanifold of $B$, where $T=k/m\in \Q$ for $k,m$ coprime;
\item $\mathrm{Im}(du)$ is close to $\C\cdot X_{S^1}\oplus \ker \Psi_*$.
\end{itemize}
Thus, if that vertical drop $u$ for $I',E$ goes from $B_{k/m,\beta_-}$ to $B_{k/m,\beta_+}$,
then the 
torsion submanifolds $Y_{m,\beta_{\pm}}$ have overlapping $\Psi$-projections in $B^{\mathrm{out}}=\Sigma \times [R_0,\infty)$, so
$\Psi(Y_{m,\beta_-})\cap \Psi(Y_{m,\beta_+})\neq \emptyset$.
\end{lm}
\begin{proof}
For the first part, we use \cref{H_lambdaIsOneDirected} and \cref{Lemma Floer solutions for constant filtration value}. Recall that a ledge lies in a {\MB } submanifold, and $F$ is constant there. When $F(x_-)=F(x_+)$ the ledges involve {\MB } manifolds with the same period $T_p=c'(H_p)$, so they lie in the same torsion submanifold of $Y$.
Cascades for $I$ with only vertical drops cannot exit the $\phi=\textrm{constant}$ neighbourhood of $\Sigma_p$, as it would enter the region where the filtration $1$-form is strictly negative on a drop, contradicting $F(x_-)=F(x_+)$.

The claim about $E_0,I'$ follows by the Gromov-compactness argument from \cref{Subsection Filtered Floer cohomology}, and suitably rephrasing \cref{Lemma Floer solutions for constant filtration value} after perturbation.
In more detail, by the energy-bracketing in \cref{Section Transversality for Floer solutions}, the only vertical drops for the perturbed $I'$ counted in the Floer complex must be close to (possibly broken and bubbled) vertical drops for the unperturbed $I$. Now $I$-holomorphic bubbles must lie entirely in fibres of $\Psi$ as $B$ only admits constant $I_B$-holomorphic spheres. The filtration cannot detect those spheres. So a broken and bubbled vertical drop for $I$ will project to a broken Floer trajectory in $B$. By the zero filtration difference, all components of that Floer trajectory in $B$ are trapped in a $\C^*$-orbit. 

The first claim about the $B_{k/m,\beta_{\pm}}\subset Y_{m,\beta_{\pm}}$, follows from the fact that $\Psi(Y_{m,\beta_{\pm}})$ are $\C^*$-invariant, and $\mathrm{Image}\,(v)$ lies in a $\C^*$-orbit. The second claim follows as a perturbed $v$ is close to such a solution.
\end{proof}

As the almost complex structure needs to be perturbed to achieve transversality for cascade solutions, it is understood that we run the energy-bracketing construction of \cref{Section Transversality for Floer solutions} to ensure that the filtration inequality persists. Abbreviate our {\bf period-filtered energy-bracketed Floer complex} by
$$
C^*=BCF^*_{E,\lambda,\mathcal{P}}(I')=BCF^*_{\mathcal{P}}(H_{\lambda},I';\k_{E}),
$$
using notation from \cref{Subsection Filtered Floer cohomology} and \eqref{Subsection Period-filtered Floer cohomology}, where $I'$ is a perturbation of $I$; the constant $E>0$ prescribes the energy-bracketing, thus the choice of Novikov ring $\k_{E}=\k_{\geq 0}/T^{E}\k_{\geq 0}$; and $\mathcal{P}\subset [0,\infty]$ is a period bracket forcing generators to have the chosen period values $T_p=c'(H_p)$. For example, $\mathcal{P}=(0,\infty]$ gives rise to the positive Floer complex where we quotient out the subcomplex of constant orbits.

Recall that $1$-orbits in $M_i$ have the same filtration value $F_p$ dependent only on the period value $T_p=c'(H_p)$ associated with $M_i$ (so $M_i\subset \mathcal{O}_p$), by \cref{SubsectionHam1-orbitsOfHlambda}. 

\begin{cor}
The {\MBF } differential non-strictly increases the $F$-filtration $F_p=F(\mathcal{O}_p)$, equivalently it non-strictly decreases the period-filtration $T_p=c'(H_p)$.
\end{cor}
\begin{proof}
Follows by Gromov-compactness (\cref{Subsection Filtered Floer cohomology}), and the filtration inequality in \cref{Lemma filtration for cascades}. 
\end{proof}
Thus $C^*$ admits a {\bf filtration}, using the conventions of \cite[Sec.7]{McLR18}:
$$F^p(C^k):=\{x\in C^k \mid F(x) \geq F_p \},$$ 
where $\mathbf{k=p+q}$ denotes the {\bf total degree}.
In particular, $F^p(C^k)=0$ for $p>0$ since $F\leq 0$ on all $1$-orbits. 
The filtration determines an associated spectral sequence $E_r^{pq}$ such that
$$E_0^{pq}=F^p(C^k)/F^{p+1}(C^k).$$
\begin{de}
By \textbf{slope} of the $p$-th column of $E_r^{pq}$ we mean $T_p=c'(H_p).$
 So $E_r^{pq}=0$ if $T_p\notin \mathcal{P}$.
\end{de}

\begin{prop}\label{PropSpectralSeqForHFHlambda}
The spectral sequence converges,
 $$
 E_r^{pq} \Rightarrow H^*(C) = HF^*_{E,\lambda,\mathcal{P}}(I'),
 $$
where $*$ will agree with the total degree $k=p+q$. It satisfies $E_r^{pq}=0$ for $p>0$. For $p=0$ and $r=1$ it recovers the quantum cohomology if $0\in \mathcal{P}$:
$$
E_1^{0q}\cong QH^q(Y) \textrm{ if }0\in \mathcal{P} \qquad (E_1^{0q}=0 \textrm{ if }0\notin \mathcal{P}).
$$
\end{prop}
\begin{proof}
By construction, the filtration is exhaustive and bounded below, thus it converges.
The claim follows immediately, except for the final two statements. 
The condition $0\in \mathcal{P}$ means we do not quotient out by the generators with filtration value $F=0$.
These generators are precisely the constant $1$-orbits arising as the {\MB } submanifolds $\F_\a$.
Recall that the filtration ensures that these generate a subcomplex (see \cref{PositiveSH}). In particular, because the filtration difference at the ends of a cascade is zero, by \cref{Lemma filtration for cascades} and the construction of $H_{\lambda}$ in \cref{SectionConstructionOfHLambda}, the cascades are trapped in a compact region of $Y^{\mathrm{in}}$ where $c'$ was chosen to be small. It follows that $E_1^{0q}$ agrees with the Floer cohomology of a Hamiltonian of arbitrarily small slope (i.e.\,the case when $\lambda$ is smaller than any non-zero period), and recall %
this is isomorphic to $QH^*(Y)$ via a PSS-isomorphism.
\end{proof}

We now compute $E_1^{pq}$ for $p<0$. 
We will now use the notation $B_p,\mathcal{O}_p$ from \cref{SubsectionHam1-orbitsOfHlambda}, and the construction of {\bf local {\MB } Floer cohomology} from
 \cref{Subsection Local Floer cohomology and the Energy spectral sequence}.

\begin{lm}\label{Lemma small energy cascades near MB mfd}
Let $N_i$ be a neighbourhood of $M_i\subset Y$ disjoint from the other {\MB } submanifolds $M_j\neq M_i$.
Then cascades in $N_i$ only involve self-drops, and only involve $1$-orbits in $M_i$.

If $\omega$ is orbit-atoroidal on $M_i$ (\cref{Definition simple Bott mfds atoroidal Mi}), these cascades are all simple cascades, i.e.\,Morse trajectories in $M_i$. This holds if $\omega$ is exact on $M_i$, or if $\omega$ is aspherical on $\overline{\cup_{t\in (0,1]} \varphi_t(M_i)}$ (\cref{Remark orbit atoroidal examples}).

There is a constant $E_0>0$ such that every cascade for $H_{\lambda}$ of energy at most $E_0$ is a simple cascade.
\end{lm}
\begin{proof}
This follows by \cref{Theorem self-drop exclusion theorem} and \cref{Lemma small energy implies cascades are near Bott mfds}.
\end{proof}

\begin{cor}\label{Cor energy spectral seq}
For $p\leq 0$, and $T_p\in \mathcal{P}$, 
$$
E_1^{pq} =  BHF^*_{\mathrm{loc}}(\mathcal{O}_p,H_\lambda,I';\k_{E}).
$$
Assume $c_1(Y)=0$. For $p\leq 0$, $T_p\in \mathcal{P}$, and for sufficiently small energy $E=E_0>0$, 
$$
E_1^{pq} =  BHF^*_{\mathrm{loc}}(\mathcal{O}_p,H_\lambda,I';\k_{E_0})
\cong \bigoplus_{i\textrm{ with } F(M_i)=F_p}
H^{*-\mu_{H_{\lambda}}(M_i)}(M_i)\otimes \k_{E_0}
$$
is the ordinary cohomology of the components $M_i$ of $B_p$ computed as Morse cohomology for the auxiliary Morse functions, with a grading shift by
\begin{equation}\label{Equation Grading shift spectral seq}
\mu_{H_{\lambda}}(M_i) := \dim_\C\, Y - \tfrac{1}{2}\dim_{\R} M_i - RS(M_i, H_{\lambda}) 
\end{equation}
where $RS(M_i,H_{\lambda})$ is the {\RS} index of one (hence any) $1$-orbit in $M_i$ for $H_{\lambda}$.

If $T_p\notin \mathcal{P}$, then $E_1^{pq}=0$ by definition.
\end{cor}
\begin{proof}
The induced differential on $E_0^{pq}$ (whose cohomology is $E_1^{pq}$) only depends on cascades whose ends $x_{\pm}$ are $1$-orbits with equal filtration value $F=F_p$. By \cref{Lemma filtration for cascades}, their drops must all be vertical drops, and the cascades are trapped in a neighbourhood of $B_p$. So the local {\MB} Floer cohomology of $B_p$ is well-defined by \cref{Subsection Local Floer cohomology and the Energy spectral sequence}, yielding the first claim. The second claim follows by the last part of \cref{Lemma small energy cascades near MB mfd}, Theorem \ref{Prop appendix spectral sequence and E1 page} and \cref{Theorem Construction of the local system}.
\end{proof}
 
Consider now how the above $E_1$-pages $E_1^{pq}:=E_{1}(H_{\lambda})$ depend on the slope $\lambda$. By \cref{Cor Subcomplex Trick}, as we increase the slope $\mu$, the $E_1$-columns associated with slope values up to $\lambda$, and their edge-differentials in the spectral sequence, will not vary for slopes $\mu \geq \lambda$. Also, there are continuation maps $E_{1}(H_{\lambda})\to E_{1}(H_{\mu})$ which induce the identity map on the columns with slopes $\leq \lambda$. Thus, one can take the various (co)limits explained in \cref{Section Transversality for Floer solutions}, so that $I'\to I$, $E\to \infty$ and $\lambda \to \infty$,
implying the following corollary. Denote $H_{p,\c}^*$ the $E_{\infty}$-page of the energy-spectral sequence, whose $E_1$-page is the local Floer cohomology of $\mathcal{O}_p$, in \cref{Cor energy spectral seq} (see also \cref{Prop appendix spectral sequence and E1 page}), after taking those same (co)limits.
From now on, we label the components $M_i$ of $B_p$ by $B_{p,\beta}$. Ordinary cohomology means with coefficients in $\k$, e.g.\,$H^*(Y):=H^*(Y,\k)$ (see \cref{Rmk about coeffs Novikov}). We deduce:

\begin{cor}\label{CorSpectralSeqForSH} 
There is a convergent spectral sequence of $\k$-modules,
	$$
 E_{r}^{pq}(\Fi;\mathcal{P}) \Rightarrow SH_{\mathcal{P}}^*(Y,\Fi), \textrm{   where } E_{1}^{pq}(\Fi;\mathcal{P})=  \begin{cases*}
				H^q(Y) & if $p=0, \textrm{ and } 0 \in \mathcal{P}$\\
				\bigoplus_{\c} H_{p,\c}^*[-\mu(B_{p,\c})] & if $p<0, \textrm{ and } T_p \in \mathcal{P}$   \\
				0 & otherwise,
			\end{cases*}
   $$
Taking $\mathcal{P}=(0,\infty]$ and $\mathcal{P}=[0,\infty]$,  defines two convergent spectral sequences:
	{\small
		\begin{equation} E_+(\Fi)_r^{pq} \Rightarrow SH_+^*(Y,\Fi), \textrm{   where } E(\Fi)_1^{pq}=  \begin{cases*}
				\bigoplus_{\c} H_{p,\c}^*[-\mu(B_{p,\c})] & if $p<0,$    \\
				0 & otherwise,
			\end{cases*}
		\end{equation}
		\begin{equation}
			E(\Fi)_r^{pq} \Rightarrow SH^*(Y,\Fi), \textrm{   where } E(\Fi)_1^{pq}=  \begin{cases*}
				H^q(Y) & if $p=0,$\\
				\bigoplus_{\c} H_{p,\c}^*[-\mu(B_{p,\c})] & if $p<0,$    \\
				0 & otherwise.
			\end{cases*}
		\end{equation}	
	}
\end{cor}
When $c_1(Y)=0$, as $SH^*(Y,\Fi)=0$ by \cref{Prop vanishing of SH}, the spectral sequence $E(\Fi)^{p,q}_r$
in fact converges to zero; also, Proposition \ref{SpSeqIs1periodic} implies that for sufficiently large $r$ we have $E(\Fi)^{p,q}_r=0$ for all $p,q$.

\begin{cor}\label{CorSpectralSeqAgree}
 For columns $p$ of slope $T_p\leq \lambda$, $E(\varphi)_r^{pq}=E_r^{pq}(\Fi;[0,\lambda])$, and $$E_r^{pq}(\Fi;[0,\lambda])\Rightarrow SH^*_{[0,\lambda]}(Y;\Fi) \cong HF^*(H_{\lambda}).$$ Moreover, $HF^*(H_{\lambda})\cong HF^*(\lambda H)$, so when $H^*(Y)$ lies in even degrees also $HF^*(H_{\lambda})$ is supported in even degrees and agrees with the right-hand side of \cref{pureHam}.
\end{cor}
\begin{proof}
	This follows by Corollary \ref{CorSubcomplexTrick}, and the fact that Floer cohomology only depends on the choice of slope at infinity for the Hamiltonian, not on the specific construction of the Hamiltonian.
\end{proof}

\begin{rmk}[{\bf Conventions about the spectral sequence Figures}]\label{Remark Conventions about the spectral sequence Figures}
For illustrations of the previous Corollaries, see the figures in \cref{ExamplesSpSeq}. In those figures, each dot indicates a rank one summand. The arrows indicate the edge homomorphisms from the $E_1$ page onward that are necessary for the spectral sequence to converge to $SH^*(Y,\Fi)=0$ (as those examples all have $c_1(Y)=0$).  We call these the {\bf iterated differentials}. They need to be suitably interpreted: one must take cohomology iteratively, by first considering arrows that only move one column to the left to determine the cohomology $E_2$, then the arrows that move two columns to the left define a differential on the resulting $E_2$ and it computes $E_3$, etc.\;(cf.\;the proof of \cref{PropTwoRowSplitSpectralSeq}).
We abusively place $H^*(B_{p,\c})[-\mu(B_{p,\c})]$ in the $E_1$-columns rather than the correct $H_{p,\c}^*[-\mu(B_{p,\c})]$. Thus, within each column, there is a further energy-spectral sequence from $H^*(B_{p,\c})[-\mu(B_{p,\c})]$ that converges to $H_{p,\c}^*[-\mu(B_{p,\c})]$. Loosely, it means there may be {\bf vertical differentials} (of degree $1$) within the same column that have not yet been taken into account, which are not indicated in the Figures. This explains why certain boxes (such as the arrow from row 1 to row 2 in \cref{S32_sp_seq}) seem to allow vertical differentials. Vertical differentials can often be excluded a posteriori when all classes are needed to kill $H^*(Y)$, or by using \cref{Lemma excluding vertical drops between certain Bott manifolds} or \cref{Lemma small energy cascades near MB mfd}. However, for $\mathcal{S}_{32}$ in \cref{S32_sp_seq}, in the $B_{1/3}$ column we have seven components $B_{1/3,\beta}$ whose cohomology is that of: $S^3$, two $S^2\times S^1$, four $S^1$. Here we cannot exclude a vertical differential from (one or two) top degree classes of the $S^1$'s, which might kill middle-top degree classes of the $S^2\times S^1$'s.
There can also be vertical differentials in the local Floer cohomology of some $S^2\times S^1$, whereas there cannot be a vertical differential for the four $S^1$ because $\omega$ is exact on $S^1$ (using \cref{Lemma small energy cascades near MB mfd}).
\end{rmk}

Finally, we can relate the spectral sequence $\Epqr$ with
the filtration $\FF^{\varphi}_\lambda$ from \cref{FiltrationOnSingCohomology}.
\begin{prop}\label{FiltrationOnCohomologyBySpectralSequence} $\FF^{\varphi}_\lambda $ is the subspace of the 0-th column $E(\Fi)_1^{0*}\cong H^*(Y)$ that gets killed by images of iterated
	differentials for $E(\Fi)_r^{pq}$ defined on columns $p$ of slope $T_p\leq \lambda.$
\end{prop}
\begin{proof}
$\FF^{\varphi}_\lambda $ is defined as the kernel of the continuation map $HF^*(H_{\lambda_0}) \fun HF^*(H_\lambda)$ (where $\lambda_0>0$ is a small slope), and we ensured that at the chain level $CF^*(H_{\lambda_0}) \fun CF^*(H_\lambda)$ is an inclusion of subcomplexes, so the cycles in $CF^*(H_{\lambda_0})$ that get iteratively killed precisely yield that kernel.
\end{proof}
\section{Properties of Morse--Bott--Floer spectral sequences}\label{SpecSeqMBF-Properties}

\subsection{Conventions}\label{NotationForSpSeq}

We continue with the initial assumption from \cref{filtrationFloer}, and that $c_1(Y)=0$.

We will work with $\Epqr,$ but analogous results hold for $E_+(\Fi)^{p,q}_r$ and $E_r^{p,q}(\Fi;\mathcal{P})$ after specifying cases about $p$, e.g.\,whether or not $T_p$ lies in the period-bracket $\mathcal{P}$.

Ordinary cohomology is computed over $\k$ coefficients (see \cref{Remark choice of coefficients}), so $H^*(Y):=H^*(Y,\k)$.

	\textbf{Morse--Bott submanifolds} will refer to the connected manifolds $B_{p,\c}$ of 1-orbits of $H_\lambda$. 
  Recall by \cref{Cor dim of Bpbeta} that
these are precisely the smooth submanifolds $$B_{p,\c}=\{c'(H)=T_p=\tfrac{k}{m}\} \cap Y_{m,\c}$$ obtained by slicing the non-compact $m$-torsion submanifolds,
where $k,m$ are coprime positive integers.

 The column $p$ of the spectral sequence is called \textbf{the time-$T_p$ column}.
	For example, the time-$1$ column of the $E_1$-page is the degree-shifted cohomology $H^*(\Sigma)$ of the whole slice $\Sigma:=\{c'(H)=1\}$.

 We \textbf{number the rows of the spectral sequence by the total degree} $p+q,$ rather than by $q$.

 The gradings will be written in terms of the function $\W(x)$, from \cref{AppendixFloer}.

 It is helpful to read the following results while comparing them with the examples in \cref{ExamplesSpSeq}, keeping in mind the conventions explained in \cref{Remark Conventions about the spectral sequence Figures}.

\subsection{Grading of the Morse--Bott manifolds of $1$-orbits $B_{p,\c}$}

\begin{prop}\label{CZindecesOfMorseBottSubmanifolds} The grading of a 
{\MB} submanifold $B_{p,\c}$ is an even integer equal to 
	\begin{equation}\label{CZindecesOfMorseBottSumbanifolds}\textstyle
		\begin{split}
  \mu(B_{p,\c}) & = \textstyle \dim_{\C} Y - \frac{1}{2}\dim_{\R} B_{p,\c}-\frac{1}{2}- \sum_i \W(T_p w_i)\\ & \textstyle = \codim_{\C}\, Y_{m,\c} - \sum_i \W(T_p w_i),
  \end{split}
	\end{equation}
 where $w_i$ are the weights 
	of the convergence point $\x_{\iinfty}$ of an arbitrary point $\x\in B_{p,\c}.$ 

For the hypersurface $\Sigma=\{ c'(H)=1\}$ corresponding to the period $T_p=1$:
\begin{equation}\label{Equation shift down for integer time hypersurfaces}
 \qquad  \qquad  \qquad  \qquad  \qquad \mu(\Sigma)=-2\mu, \qquad \textrm{and }\qquad \mu(\{ T_p=N\})=-2N\mu \quad\textrm{ for }N\in \N,
\end{equation}
where $\mu=\sum w_i$ is the Maslov index of the $S^1$-action (\cref{Subsection discussion of ma and muFalapha}).
\end{prop}
\begin{proof} We are refining the computation from \cref{Thm Floer traj near Bi new perspective}, and the proof idea is very similar to \cite[Sec.3.3]{OanceaEnsaios}. In short: the difference between the Hamiltonian flows of $H_\lambda$ and $T_p H$ causes a symplectic shear which gives the $-\frac{1}{2}$ term in \eqref{CZindecesOfMorseBottSumbanifolds}, and the {\RS} index of $B_{p,\c}$ with respect to the Hamiltonian $T_p H$ will yield the $\sum_i \W(T_p w_i)$-part, like in the proof of \cite[Prop.5.2]{RZ1}.

	For an arbitrary point $y\in B_{p,\c},$ %
	recall  the orthogonal splittings $T_y Y= \la X_{\R_+} \ra \oplus \la X_{H_\lambda} \ra \oplus \xi$ and $T_\x \Sp =X_{H_\lambda}\oplus \xi$ from \eqref{splittingTangentSpaceContact}. %
Using this splitting, define the matrices
	\[ \small
	F_t=
	\begin{bmatrix}
		1&0&\\
		{\tau c''(H) \norm{\nabla H}^2}&1&\\
		&& {(\phi_t^{H_\lambda})_*|_\xi}
	\end{bmatrix},
 \qquad \qquad 
 \chi(t)=
	\begin{bmatrix}
		1&0&\\
		{\tau c''(H) \norm{\nabla H}^2}&1&\\
		&& Id_{2d-2}
	\end{bmatrix}.
	\]
 We claim that $F_t$ is the matrix for the linearised flow $(\phi_t^{H_\lambda})_*$ of $H_\lambda$. Indeed the first column is due to Lemma \ref{LemmaForTheShear}; the second column is immediate; and the rest is due to the fact that the flow $\phi_t^{H_\lambda}$ restricted to $\Sp$ is equal to $\phi_t^{T_p H},$ hence preserves the metric and thus the orthogonals.
 
 As $\phi_t^{H_\lambda}$  preserves $X_{H_\lambda},$ it preserves its orthogonal $\xi.$
Thus, we have that $(\phi_t^{H_\lambda})_*=  \chi(t) \circ (\phi_t^{T_p H})_*.$ 
	Letting $\Psi(t)=(\phi_t^{T_p H})_*$ as in the proof of \cite[Prop.2.2]{CFHW}, there is a homotopy 
	
	$$K(s, \, t) = \left \{ \begin{array}{ll} \chi(st) \Psi(\frac {2t} {s+1}) 
		, & t \le \frac {s+1} 2 \\ 
		\chi\big( (s+2)t - (s+1) \big) \Psi(1) , & \frac{s+1} 2 \le t  
	\end{array} \right. 
	$$ 
of paths with fixed ends between the path $(\phi_t^{H_\lambda})_*$ and the concatenation of
	$(\phi_t^{T_p H})_*$ and $\chi(t) \circ (\phi_1^{T_p H})_*.$
	
	Hence, by Theorem \ref{RSProperties}.(\ref{catenation}),(\ref{shear}) we have $RS(x,H_\lambda)= RS(x, T_p H) + 
 \frac{1}{2},$ where $x=x(t)$ denotes 
	the 1-orbit of $H_\lambda$ and $T_p H,$ starting at $y=x(0).$
	Thus, it remains to compute the index $RS(x,T_p H).$
	We have to pick a symplectic trivialisation of $x^*TY$ first. 
 By Lemma \ref{LemmaCanonicalFillingDiscs},
 a choice of unitary basis $v_i$ of tangent vectors at $y_0$ with weights $w_i$ yields a symplectic trivialisation $v_i(z)$ of $\psi_y^*TY$, where 
 $\psi_y$ is the canonical holomorphic capping disc from \cref{LemmaCanonicalFillingDiscs}. 
 As $T_p=k/m$ for coprime $k,m$, the $S^1$-orbit $S^1\cdot y$ is an $m$-fold cover of a closed $S^1$-orbit $\hat{y}$ of minimal period $1/m$, and $v_i(z)$ is $\Z/m$-equivariant by \cref{LemmaCanonicalFillingDiscs}. In particular, $v_i(z)$ determines a unitary trivialisation of $\hat{y}^*TY$. As $T_p=k/m$, $x(t)$ is a $k$-fold cover of $\hat{y}$, so $v_i(z)$ determines a unitary trivialisation of $x^*TY$.
	Thus we obtain a canonical unitary and symplectic trivialisation for $x^*TY$ determined by the choice of $v_i$. 
 
 By continuity of {\RS} indices, Theorem \ref{RSProperties}, $RS(x,T_pH)=RS(y_0,T_pH)$ where $y_0=\lim_{t\to 0} t \cdot y$ is the convergence point, viewing $y_0$ as a constant $1$-orbit for the Hamiltonian $T_p H$. 
 On the weight decomposition $T_{y_0} Y = \oplus \C_{w_i}$ the linearisation of the flow of $T_p H$ is rotation on $\C_{w_i}\cong \C$ with speed $T_p w_i$, which contributes $\W(T_p w_i)$ to the RS-index by \cref{RSProperties}(\ref{RSofRotationInC}). In summary:
$$\textstyle
RS(x,H_\lambda)= RS(x, T_p H) + 
 \frac{1}{2} \qquad\textrm{ and }\qquad RS(x,T_pH) = RS(y_0,T_p H)=\sum_i  \W(T_p w_i).
$$
	
	The grading of the {\MB} manifold $B_{p,\c}$ corresponds to the grading of the orbit that represents the minimum of the auxiliary Morse function $f_{p,\c}:B_{p,\c}\to \R.$
		By the argument in \cite[Sec.3.3]{OanceaEnsaios}, the Morse term in the {\RS} index formula for an orbit is equal to half of the signature of the Hessian $\nabla f_{p,\c}$ at the corresponding critical point; thus at a minimum it is equal to $\frac{1}{2}\dim_{\R} B_{p,\c}.$ By our grading conventions from \cref{GradingConventions}, the formula \eqref{CZindecesOfMorseBottSumbanifolds} follows.

To prove $\mu(B_{p,\c})$ is even, it suffices to show $\sum_i \W(T_p w_i) \equiv \codim_{\C}\, Y_{m,\c} \;(\textrm{mod 2})$. 
Recall $T_p= k/m$, for coprime $k,m$, so
$T_p w_i\in  \Z \Leftrightarrow m|w_i.$ %
By \cref{Lemma torsion submanifolds}, $T_{y_0}Y_{m,\c}=\oplus H_{ma}\subset T_{y_0}Y$ 
is summing over the same set of weights.
By \eqref{WfunctionForRSAppendix}, $\W(T_p w_i)$ is even precisely when $T_p w_i\in  \Z,$ or equivalently $m|w_i,$ %
so $\sum \W(T_p w_i)$ modulo $2$ is the complex dimension of the vector space complement of $T_{y_0}Y_{m,\c}\subset T_{y_0}Y$.

  Using \eqref{WfunctionForRSAppendix}, %
$\sum  \W(2\pi w_i)=\sum 2 w_i=%
2\mu$ and $\tfrac{1}{2}(\dim_{\R} \Sigma + 1) = \dim_{\C} Y$, and the final claim follows.
	\end{proof}

\begin{cor}\label{Rmk index of Bott manifolds in terms of muFa}
Let $T_p= \tfrac{k}{m}$ with $k,m$ coprime, and $T_p^{\pm}$ non-critical times just above/below $T_p$. Let $\F_\a$ be the minimal component of $\Ymc$, and $\mathrm{rk}(Y_{m,\c})$ the rank from \cref{Remark local rank}. Then
\begin{equation}\label{relation mu Bpc and mu Fa}
    \mu(B_{p,\c})= \mu_{T_p}(\F_\a) - \mathrm{rk}(Y_{m,\c})=\mu_{T_p^+}(\F_\a),
\end{equation}
which is the lowest degree in which $H^*(B_{p,\c})[-\mu(B_{p,\c})]$ is supported, and the top degree is:
$$
\mu(B_{p,\c})+\dim_{\R}B_{p,\c} = 
\mu_{T_p}(\F_\a) + \dim_{\C} Y_{m,\c} + \dim_{\C} \F_\a -1 = \mu_{T_p^-}(\F_\a) + 2\dim_{\C} \F_\a - 1.
$$
For $T_p<1$ it is supported in degrees $\geq 2\dim_{\C}Y-2\mu-t$.

If $Y$ is compatibly-weighted (\cref{Definition compatibly weighted}, e.g.\,all CSRs), then
 $H^*(B_{p,\c})[-\mu(B_{p,\c})]$ is supported strictly below degree $t:=\max \{n\in \N: H^n(Y)\neq 0\}$.

For a CSR of weight $s=1$, for $T_p<1$ it is supported in $[0,t-1]$, and for $s=2$ in $[-t,t-1]$.
\end{cor}
\begin{proof}
Abbreviate $|V|:=\dim_{\C} V$.
The first displayed equality
combines Equations \eqref{CZindecesOfMorseBottSumbanifolds}
and \eqref{mu_lambda_calculation}, using $\dim_{\R}B_{p,\c}=2|Y_{m,\c}|-1$;
the second displayed equality uses \cref{Lemma Fma is jump in grading}.\eqref{m-minimal Fma equal to rk Ymb} so 
$\mu_{T_p^+}(\F_\a)=\mu_{T_p}(\F_\a)- \rk(Y_{m,\c})$ and 
$\mu_{T_p^-}(\F_\a) = \mu_{T_p}(\F_\a)+F_m^\a$, and we use 
 $\rk(Y_{m,\c}) = |Y_{m,\c}|-|\F_\a|$. 
  So
$$
\mu_{T_p}(\F_\a) + |Y_{m,\c}| + |\F_\a| -1
= \mu_{T_p^-}(\F_\a) + 2|\F_\a| - 1.
$$
When $Y$ is compatibly-weighted: $\mu_{T_p^-}(\F_\a) 
\leq \mu_\a$ by \cref{Lemma muFa is smaller than MB index mua},
and $\mu_\a + 2|\F_\a|\leq t$ by \eqref{EqnFrankel}, so $\mu_{T_p}(\F_\a) + |Y_{m,\c}| + |\F_\a| -1\leq t-1.$ 
If $T_p<1$: by \cref{CentralSymmetrySpecSeq}, $t-1$ is reflected to $2(|Y|-\tfrac{1}{2}-\mu)-(t-1)=2|Y|-2\mu-t$. For CSRs: $2\mu=s\cdot |Y|$, $t\leq |Y|$, and $t=|Y|$ when $s=1$ by \cref{CohomologyOfACSRProperties}(\ref{Thm_CSR_Maslov_Index_formula}).
\end{proof}

 \subsection{The columns of the $E_1$-page are 1-periodic up to downward $2\mu$-shifts}

	\begin{cor}\label{CorMuForTauTau'}
		If $T_{p}=T_{p'}+N,$ for $N\in \N$, then
			$\mu(B_{p,\c})=\mu(B_{p',\c})-2N\mu.$
	\end{cor}
	\begin{proof}
 We use  %
		\cref{LemmaSec2MaslovIndex} ($\mu=\sum w_i$), 
		\cref{CZindecesOfMorseBottSumbanifolds}, and the property $\W(x+N)=\W(x)+2N$:
		$$\textstyle
			\mu(B_{p',\c})\!-\!\mu(B_{p,\c})=\textstyle \sum [\W(T_{p'} w_i)-\W(T_p w_i)] %
                =\sum [\W((T_{p}+N) w_i)-\W(T_p w_i)]= \sum 2Nw_i=2N\mu. \eqno \qed
		$$ \let\qed\relax
		
	\end{proof}
	
	\begin{cor}(Periodicity)\label{SpSeqIs1periodic}
		The $E_1$-page of the spectral sequence $\Epqr$ is time-1 periodic up to shifting the rows downwards by $2\mu$ rows. 
 For a weight-$s$ CSR, $2\mu = s \cdot \dim_\C Y$ by \cref{GradientIsRPlus}(\ref{Thm_CSR_Maslov_Index_formula}).
	\end{cor}
	Thus, in order to compute the $E_1$-page of $\Epqr,$ it is enough to compute the columns $p$ with period $T_p\leq 1$. We show now that it suffices to consider $T_p \leq \frac{1}{2}$ due to a local central symmetry.

 \subsection{Local central symmetry in the $E_1$-page}

 In Figure \ref{SS_S41} the 4-star symbol lies at the point ``$(\frac{1}{2},-\frac{1}{2})$''. In general, the \textbf{point ``$(\frac{1}{2},n-\frac{1}{2})$''} in the table for the spectral sequence lies on the column with time $T_p=\frac{1}{2},$ and on the line between rows $n$ and $n-1.$ If a column with $T_p=\frac{1}{2}$ does not exist, then it lies on the line between the two columns whose periods are closest to $\frac{1}{2}$.
	
	\begin{cor}[Local central symmetry] \label{CentralSymmetrySpecSeq}
		There is a central symmetry of the sub-block of the $E_1$-page of $\Epqr$ consisting of columns $p$ with $0<T_p<1,$ about the point 
		$$\textstyle \left(\frac{1}{2}, \dim_\C Y - \frac{1}{2} - \mu  \right).$$
		If $Y$ is a weight-1 CSR it is the point $(\frac{1}{2}, \frac{1}{2}\dim_\C Y -\frac{1}{2})$, for a weight-2 CSR it is $(\frac{1}{2},-\frac{1}{2})$.
	\end{cor}
	\begin{proof}
  By (\ref{CZindecesOfMorseBottSumbanifolds}), the top-dimensional class $H^{top}(B_{p,\c})$ lies in the row
		$$\textstyle r_1:=\dim_{\R} B_{p,\c}+\dim_{\C} Y - \frac{1}{2}\dim_{\R} B_{p,\c}-\frac{1}{2}- \sum_i \W(T_p w_i).$$ For $T_p<1$, in the column $p'$ such that $T_{p'}=1-T_p<1$, there is a {\MB} manifold $B_{p',\c}$ diffeomorphic to $B_{p,\c}$ by Lemma \ref{Only the primitive Ones matter}. 
		By (\ref{CZindecesOfMorseBottSumbanifolds}), its bottom-dimensional class $H^{0}(B_{p',c})$ lies in the row
		$$\textstyle r_2:=\dim_{\C} Y - \frac{1}{2}\dim_{\R} B_{p',\c}-\frac{1}{2}- \sum_i \W(T_p' w_i).$$
		Thus $\frac{r_1+r_2}{2}=\frac{1}{2}(2 \dim_\C Y -1 - \sum_i (\W(T_p w_i)+\W(T_{p'} w_i))=\frac{1}{2}(2 \dim_\C Y -1- \sum_i 2w_i)=\dim_\C Y - \frac{1}{2} - \mu,$ where
		we used $\W(-a)=-\W(a)$ and $\W(N +a)=2N+\W(a)$.
		So $H^{top}(B_{p,\c})$ and $H^{0}(B_{p',\c})$ match-up via the
		central symmetry about the point $(\frac{1}{2},\dim_\C Y - \frac{1}{2}- \mu ).$ In the other degrees, the ranks match-up due to Poincaré duality and the fact that $B_{p,\c}\iso B_{p',\c}.$
	\end{proof}
\begin{cor}[Symmetries revisited] \label{CentralSymmetrySpecSeq2}
The periodicity in \cref{SpSeqIs1periodic} and local central symmetry in \cref{CentralSymmetrySpecSeq} also hold for the groups $H_{p,\c}^*[-\mu(B_{p,\c})]$ from \cref{CorSpectralSeqForSH}.
  \end{cor}
  \begin{proof}
 Let $\Sigma_p:=\{H=H_p\}=\{c'(H)=T_p\}$.
The periodicity is a consequence of the method from \cref{Subsection A comment about Hamiltonian perturbations} for $G:=c(H)$, $P:=H$, $p:=S^1$-flow (the flow for $H$). We get $p^*I=I$ (as the $S^1$-flow is {\ph}) and $p^*G=\widetilde{c}(H)-H$, and a bijection of the local Floer complex for $(\widetilde{c}(H)-H,I)$  to the local Floer complex for $(\widetilde{c}(H),I)$, via $x\mapsto p\circ x$ that decreases the grading by $2\mu.$ More precisely, we work with {\MB} Floer complexes (see \cref{AppendixCascades}), we use the same auxiliary {\MB} functions for $B_{T_p,\beta}$ and $B_{1+T_p,\beta}$, and the above bijection ensures the required bijections of the drops in the cascades. So the generators and {\MB} Floer differentials are identified for the local {\MB} Floer complexes for $\Sigma_p$ and $\Sigma_{p'}$ where $T_{p'}=1+T_p$.

We now compare the local Floer cohomologies of a neighbourhood $N$ of $c'(H_p)=T_p$ and a neighbourhood $N'$ of $c'(H_{p'})=T_{p'}:=1-T_p$.
Using the flow $\psi_t$ of $\tfrac{\nabla H}{\|\nabla H\|^2}$ we can identify $N \to N'$.
The flow yields a family $\psi_t^*I,\psi_t^*\omega$ of data on $N$, telling us how to deform the data $I,\omega$ so that local Floer cohomology of $N$ can be identified with that of $N'$ if we were using the same Hamiltonian under the identification $N\to N'$ (which we are not). On $N$ we are using the Hamiltonian $c(H)$, and using a {\MBF} model involving auxiliary Morse functions $f_{\beta}:B_{p,\beta}\to \R$ (\cref{AppendixCascades}).
 The local {\MB} Floer cohomology of $(N,c(H))$ is invariant up to isomorphism under changing the precise choice of $c$ provided that we keep, say, our {\MB} manifolds of $1$-orbits fixed and ensure the slope $c'(H_p)=T_p$ and $c''(H_p)>0$ there, call this ``$T_p$-data''.
If we instead have $c''(H_p)<0$, call it ``$T_p$-antidata''.
Observe that the local {\MB} Floer cohomology for $T_p$-data $c(H)$ and for $T_p$-antidata $\widetilde{c}(H)$ are isomorphic\footnote{This is because if we locally concatenate $c(H)$ with $\widetilde{c}(H)$, one can homotope the concatenation to a Hamiltonian with a slope less than $T_p$ whose local Floer cohomology therefore vanishes. Thus the (degree one) {\MB} Floer differential from the complex for the $T_p$-data must map quasi-isomorphically to the complex for the $T_p$-antidata.} via a grading shift $d\mapsto d+1$. This shift by $+1$ can also be explained in terms of the shear in \eqref{Grading equation in appendix} (see also the proof of \cref{CZindecesOfMorseBottSubmanifolds}): for the $T_p$-data we get a contribution of $-\tfrac{1}{2}$ to the grading, whereas for the $T_p$-antidata we get $+\tfrac{1}{2}.$

 It remains to compare slope $(1-T_p)$-data on $N$ and slope $T_p$-data on $N$. By the shifting argument at the beginning of the proof, slope $(1-T_p)$-data has the same local Floer cohomology as $-T_p$-data after changing degrees by $d\to d+2\mu$.
 Suppose $\gamma(H)$ is the $-T_p$-data obtained.

 We have a canonical duality isomorphism 
 of local {\MB} Floer complexes:
$$
BHF^*_{loc}(\Sigma_p,\gamma(H),f_{\beta})
\cong
BHF^{2\dim_{\C}Y-*}_{loc}(\Sigma_p,-\gamma(H),-f_{\beta}),
$$
which sends $x\in \mathrm{Crit}(f_{\beta})$ to $x\in \mathrm{Crit}(-f_{\beta})$, and the {\MB} Floer differentials coincide.\footnote{We reverse the flow-direction of ledges, and we reverse both coordinates for drops, $u(s,t)\to u(-s,-t)$.}
Note that $-\gamma(H)$ is $-T_p$-antidata, since $-\gamma''(H_p)<0$. As explained previously, we can identify this with the local {\MB} Floer cohomology of $T_p$-data after index shifting $d\to d-1.$

In conclusion, going from $(1-T_p)$-data to $T_p$-data arose from quasi-isomorphisms at local {\MB} Floer chain level subject to shifting degrees by 
$$d\to d+2\mu \to 2\,\dim_{\C}Y-(d+2\mu) \to 2\,\dim_{\C}Y-(d+2\mu)-1 = 2\,\dim_{\C}Y-d-2\mu-1,$$
which is a reflection of $d$ about $\dim_{\C} Y - \tfrac{1}{2}-\mu$ of $d$ in $\R$, as required in \cref{CentralSymmetrySpecSeq}.
  \end{proof}

\subsection{Bounds on where the $E_1$-page is supported}

 \begin{de}
The {\bf block number} is the integer part $N:=\lfloor T_p \rfloor$ of the period. 
{\bf Pillars}
are columns of $E(\Fi)_1^{pq}$ for integer periods $T_p$: the $0$-pillar is $H^*(Y)$, the $N$-pillar is $H^*(\Sigma)[2N\mu]$ for $N\geq 1$. 

The {\bf $N$-block} is the sub-block of $E(\Fi)_1^{pq}$ of classes with block number $N$ excluding the pillars, so $N<T_p<N+1$.
 A pillar/block is {\bf supported} in $[a,b]$ if it vanishes for grading values outside of $[a,b]$.
\end{de}

\begin{lm}\label{Lemma block description spectral sequence}
Suppose all $\mu_{\lambda}(\F_\a)\leq \mu_\a$ %
(e.g.\,any compatibly-weighted $Y$, \cref{Definition compatibly weighted}, e.g.\,CSRs).

Let $y:=\dim_{\C} Y$ and $t:=\max \{n\in \N: H^n(Y)\neq 0\}.$
Then $t\leq 2y-2$ is even, and we have:
    \begin{center}
    \begin{tabular}{c|c|c}%
        \strut & lower bound on support & upper bound on support \\%
         \hline
         $0$-pillar & $0$ & $t$ \\%
         $0$-block & $2y-2\mu-t$ & $t-1$ \\%
          $1$-pillar & $-2\mu$ & $2y-1-2\mu$ \\%
           $1$-block & $2y-4\mu-t$ & $t-1-2\mu$ \\%
          $N$-pillar & $-2N\mu$ & $2y-1-2N\mu$ \\%
         $N$-block & $2y-2(N+1)\mu-t$ & $t-1-2N\mu$ \\%
    \end{tabular}
    \end{center}
The $0$-block (resp.$1$-pillar) after adding $-2N\mu$ to the grading is the $N$-block (resp.$(N+1)$-pillar).\\
The blocks have a local central symmetry analogous to the $0$-block (\cref{CentralSymmetrySpecSeq}).

For a weight-$s$ CSR we have $2\mu=sy$ and $t\leq y$ (with $t=y$ for $s=1$), so:\\ %
For $s=2$: the $0$-block is supported in $[-t,t-1]$; the $1$-pillar in $[-2y,-1]$.\\
For $s=1$: $0$-block $\subset [0,y-1]$;  $1$-pillar $\subset [-y,y-1]$; $1$-block $\subset [-y,-1]$; $2$-pillar $\subset [-2y,-1]$.
\end{lm}
\begin{proof}
In \cref{EqnFrankel}, $\F_\a$ is a closed $I$-{\ph} submanifold, so it is even-dimensional, and $\mu_\a$ is even, so $t$ is even (and $t\neq 2y$ as $Y$ is non-compact), so $t\leq 2y-2$.
For $N\geq 1$, the $N$-th pillar arises from an energy-spectral sequence starting with 
$H^*(\Sigma)[2N\mu]$, so is supported in $[-2N\mu,-2N\mu+\dim_{\R}\Sigma]$, and $\dim_{\R}\Sigma = 2y-1$.
The rest follows by \cref{Rmk index of Bott manifolds in terms of muFa} and \cref{SpSeqIs1periodic}.
\end{proof}

\subsection{What classes in the $E_{\geq 1}$-pages can hit $H^*(Y)$}
	
 By \cref{SpSeqIs1periodic} and \cref{FiltrationOnCohomologyBySpectralSequence}, for the purposes of computing the filtration $F^{\varphi}_\lambda$ we only need to consider columns up to a certain time: %
 
 \begin{cor}\label{HitingTimeOfColumns}
 Suppose all $\mu_{\lambda}(\F_\a)\leq \mu_\a$ %
(e.g.\,any compatibly-weighted $Y$, \cref{Definition compatibly weighted}, e.g.\,CSRs).

 	The only $N$-blocks that can potentially hit the 0-pillar via iterated differentials satisfy  $N\leq t/2\mu$.
  The only $N$-pillars that can potentially hit the 0-pillar via iterated differentials satisfy  $N\leq y/\mu$.
 For a weight-$s$ CSR, those two conditions become 
  $N \leq t/sy\leq 1/s$ and $N\leq 2/s$ 
 respectively.   
 \end{cor}
\begin{proof}
This will follow by \cref{Lemma block description spectral sequence}. For the $N$-block, the highest possible grading $t-1-2N\mu$ is strictly less than $-1$ if $N>t/2\mu$, so the iterated differentials (which increase the total grading by $1$) from such an $N$-block cannot hit the $0$-pillar. Higher numbered blocks lie in lower gradings, so the same argument applies. Similarly, for the $N$-pillar we have $2y-1-2N\mu< -1$ if $N>y/\mu$.
\end{proof}
When $H^*(Y)$ lies in even degrees, the only non-zero iterated differentials go from odd-numbered rows to even rows (as illustrated by the figures in \cref{ExamplesSpSeq}):
  
\begin{prop}\label{PropTwoRowSplitSpectralSeq} 
Suppose $H^*(Y$) lies in even degrees ($\Leftrightarrow$ all $H^*(\F_\a)$ lie in even degrees), e.g.\,any CSR.
Then from the $E_1$-page onwards, the spectral sequence $\Epqr$ splits horizontally into a direct sum of two-row spectral sequences. Also, on any given odd degree class $x$ of the $E_1$-page, some iterated differential is eventually injective on $x$, so kills some class in a column strictly to the left of $x$.
	\end{prop}
	\begin{proof}
 Write $E_r(H_{\lambda})$ to mean the $E_r$-page for the spectral sequence for $H_{\lambda}$.
 By induction on the number of the column to the right of the $0$-pillar, we prove: all iterated differentials $d_j$ vanish on even classes, and
 for any odd degree class $x$ we have $d_1(x)=0,$ $d_2(x)=0$, $\ldots$, $d_{r-1}(x)=0$, $d_r(x)\neq 0$ for some $r\geq 1$ possibly equal to $1$, so $x$ kills some
 non-trivial class in a column strictly to the left of $x$ via an iterated differential.
The initial step of the induction holds:
 the $0$-pillar is supported in even degrees, and columns to the left of the $0$-pillar are zero.  
 Suppose induction holds up to columns with $T_p< b$, 
 and consider $b<\lambda$ with %
 $E_1(H_{b})$ and $E_1(H_{\lambda})$ differing in just one column, so a column arising for time $b<T_p<\lambda$. Call ``new'' that column and its classes, and ``old'' the columns to the left and their classes.
 If there is an odd new class $x\neq 0$ with $d_j(x)=0$ for all $j\geq 1$, 
then $[x]$ would be a non-trivial class in all $E_j(H_{\lambda})$ (all columns in $E_1(H_{\lambda})$ to the right of the new column are zero, so $x$ is not in the image of any $d_j$). This yields a non-trivial odd class in  $HF^*(H_{\lambda})$, contradicting
\cref{pureHam}.
Now let $y$ be an even new class; we want $d_j(y)=0$ for all $j\geq 1$.
By contradiction, consider the smallest $j\geq 1$ for which that fails. Then $d_j(y)=x$ kills an odd old class $[x]$ that survived to $E_j(H_{b})$.
By the inductive hypothesis for $H_{b}$, $d_{j-1}(x)=\cdots = d_{r-1}(x)=0$, $d_r(x)\neq 0$ (with $r\geq j$ possibly equal to $j$).
The iterated differentials of $x$ for $H_{b}$ agree with those for $H_{\lambda}$, indeed $(E_i(H_{b}),d_i)\subset (E_i(H_{\lambda}),d_i)$ is a subcomplex.\footnote{by the $F$-filtration, for $x\in E_*(H_{b})$, $d_i(x)$ only involves Floer trajectories trapped in the region where $H_{\lambda}=H_{b}$.}
The exact class
$d_j(y)\in E_r(H_{\lambda})$ is zero so $d_r(d_j(y))=d_r(x)\in E_r(H_{\lambda})$ must be zero. As $d_r(x)$ is not zero in $E_r(H_{b})$, there is a new class $w$ surviving to $E_r(H_{\lambda})$ with $d_r(w)=d_r(x)$. This is impossible: $d_r(w)$ cannot reach the column of $d_r(x)$ ($w$ lies in a column strictly to the right of $x$). 
\end{proof}
\subsection{Hitting the unit}
Using \cref{FiltrationOnCohomologyBySpectralSequence} and results from our previous paper, we can deduce which class in the spectral sequence kills the unit in 
certain special cases of CSRs.

\begin{prop}\label{Killing the unit in sp seq CSR} The unit $1\in H^0(Y)$ in the 0-pillar gets killed by 
\begin{enumerate}
    \item \label{killing the unit wt1 case}The top class of the 2-pillar, when $Y$ is a weight-1 CSR.
    \item \label{killing the unit wt2 Fmin=pt case} The top class of the 1-pillar, when $Y$ is a weight-2 CSR and its minimum\footnote{The minimum of the $S^1$-moment map.} 
    is $\F_{\min}=\{point\}.$
\end{enumerate}    
\end{prop}
\begin{proof}
    Follows directly from \cref{FiltrationOnCohomologyBySpectralSequence} and
    $\FF_{2^-} = H^{\geq 2}(Y),$ $\FF_2(Y)=H^*(Y),$ \cite[Ex.1.45]{RZ1} for the \eqref{killing the unit wt1 case} part; 
    $\FF_1=H^*(Y)$, $1\notin \FF_{1^-}$ \cite[Cor.7.22]{RZ1} for the \eqref{killing the unit wt2 Fmin=pt case} part.
\end{proof}
\subsection{Computing the $E_1$-columns arising from hypersurface sections of vector bundles}
 \label{Subsection Techniques for computing the $E_1$-page}
	
	\begin{lm}\label{CohomologicalIndependenceOnHypersurface}
		Let $M$ be an open manifold, and $H: M\to \R$ an exhausting function with compact critical locus.
		For large $c\in \R$, the cohomology $H^*(\Sigma_c)$ of the level set 
        $\Sigma_c=H^{-1}(c)$
        is independent of $c$ and $H$. If $M$ is a vector bundle with sphere subbundle $S(M)$, $H^*(\Sigma_c)\cong H^*(S(M))$ for large $c$.
	\end{lm}
	\begin{proof} 
		For large $c$, Crit($H$) is contained in $H<c$ (as $H$ is exhausting).
		Thus, for large $c,c'$ we have a diffeomorphism $\Sigma_c\cong \Sigma_{c'}$ via the flow of $\nabla H/\|\nabla H\|^2$. 
		Let $\Sigma_{c,\infty}=H^{-1}[c,\infty)\subset M$ be the subset at infinity diffeomorphic to $\Sigma_c\times [0,\infty)$ obtained by flowing for positive time with  $\nabla H/\|\nabla H\|^2$. Thus $H^*(\Sigma_{c,\infty}) \cong H^*(\Sigma_c)$ (for large $c$).
		Given two exhausting functions $H'$ and $H''$, with level sets $\Sigma_c'$, $\Sigma_c''$, we can choose $c$-values to get nested sets 
		$
		\Sigma_{c',\infty}'
		\subset
		\Sigma_{c'',\infty}''
		\subset 
		\Sigma_{C',\infty}'
		\subset
		\Sigma_{C'',\infty}'',
		$
		where the $c$-values are large so that the critical loci of $H',H''$ lie in lower sublevel sets.
		The composite of the first two inclusions is isotopic to a diffeomorphism (flow by $\nabla H'/\|\nabla H'\|^2$). Similarly, for the composite of the last two inclusions (using $H''$). So those two composite maps induce isomorphisms on cohomology. Thus $H^*(\Sigma_{c'',\infty}'') \leftarrow H^*(
		\Sigma_{C',\infty}')$ is injective and surjective. By the first part, up to isomorphism the choices of large $c'',C'$ do not matter.
		The second claim follows: let $g$ be the fibre-wise metric on $M$ defining $S(M)$, then $H:M\fun \R$, $H(b,\xi)=g(\xi,\xi)$ is exhausting 
		and $S(M)$ is a level set.
	\end{proof}
	Applying \cref{CohomologicalIndependenceOnHypersurface} to $M=\Ymc$ and $H|_{\Ymc}^{-1}(H_p)=\Sigma_p\cap \Ymc =
 B_{p,\c}\subset \Ymc,$ we get:
	
	\begin{enumerate}[(1)]
		\item If $Y_{m,\c}=\Hm$ is a torsion bundle (\cref{Defn torsion bundle})
  over $\F_\a,$ then 
		$$H^*(B_{p,\c})\iso H^*(S(\Hm)),$$
		which is calculable by considering the Gysin sequence for the sphere bundle $S(\Hm)\to \F_\a$.
		\item More generally, if $\mathrm{Core}(\Ymc)=\Ymc\cap \mathrm{Core}(Y)$ is a smooth manifold (not necessarily fixed under the $\C^*$-action), and $\Ymc$ is a smooth vector bundle over it, 
		$$H^*(B_{p,\c})\iso H^*(S(Y_{m,\c})),$$
		which is calculable by considering the Gysin sequence for the sphere bundle $S(\Ymc)\to \mathrm{Core}(\Ymc)$.
	\end{enumerate}

\begin{rmk}
In all the examples we consider in \cref{ExamplesSpSeq}, only (1) and (2) occurred for all $Y_{m,\c}$.
\end{rmk}

By the proof of \cref{CohomologicalIndependenceOnHypersurface}, we may replace $M$ by $M:=\{H\leq c\}$ for large enough $c$, without affecting $H^*(M)$, since they are homotopy equivalent via the flow of $\nabla H/\|\nabla H\|$. In this new notation, $\Sigma=\partial M$. The long exact sequence for the pair $(M,\partial M)$ is
		\begin{equation}\label{LESforPair}
		    	\dots \rightarrow H_{k+1}(M,\partial M) %
		    	\rightarrow H_k(\partial M) \rightarrow H_k(M) \rightarrow H_{k}(M, \partial M) \rightarrow \cdots
		\end{equation}
Assume $\dim_{\R} M = 2n$ is even. Then $H_k^{lf}(M)\cong H_k(M,\partial M)\cong H^{2n-k}(M)$ by
 {\PL} duality 
   over the field $\k$ (here, `lf' refers to locally finite homology), and we have a relation between the intersection pairing $\langle \cdot ,\cdot\rangle$ and Poincar\'{e} duality:
\begin{equation}\label{Equation intersection pairing} 
\begin{tikzcd}[column sep=small]
	\langle \ ,\cdot\rangle:\;H_{n}(M) \arrow{r}{\psi} 
	& H_n^{lf}(M)\cong H_n(M,\partial M)\;\; \arrow[r,shorten=-2mm,"{\iso}","\textrm{P.D.}"'] 
 & \;\;H^n(M)\arrow{r}{\iso}
 & \Hom(H_n(M),\k). 
\end{tikzcd} \qedhere
\end{equation}
The map $\psi$ arises from the natural ``inclusion'' of (compact) chains into lf chains.

\begin{prop}
Let $Y$ be a symplectic $\C^*$-manifold with $H^*(Y)$ lying in even degrees (e.g.\;all CSRs), equivalently all $H^*(\F_\a)$ lie in even degrees. Then the torsion submanifold $Y_{m,\beta}\subset Y$ from \cref{Subsection intro MB mfds of 1 orbits} has $H^*(Y_{m,\beta})$ lying in even degrees. Let $y_{m,\beta}=\dim_{\R} Y_{m,\beta}$. For even $q$, \eqref{Equation Bkm slices intro} implies
\begin{align*}
H_q(B_{T,\beta}) & \cong
(\textrm{Kernel of the intersection form } H_q(Y_{m,\beta}) \otimes H_{-q+y_{m,\beta}}(Y_{m,\beta}) \to \k)\\
& \cong 
\ker\,\big( \psi: H_q(Y_{m,\beta})\to H_q^{lf}(Y_{m,\beta})\cong H^{-q+y_{m,\beta}}(Y_{m,\beta})\big).
\end{align*}
Poincar\'{e} duality then determines $H_q(B_{T,\beta})\cong H_{-1-q+y_{m,\beta}}(B_{T,\beta})$ for odd $q$.
\end{prop}
\begin{proof}
$Y_{m,\beta}$ is a symplectic $\C^*$-submanifold of $Y$ by \cite[Lemma 4.5]{RZ1} and its $\C^*$-fixed locus is a union of a subcollection of the fixed components $\F_\a$ of $Y$ by \cite[Lemma 4.6]{RZ1}.
By \eqref{EqnFrankel} applied to $Y_{m,\beta}$ in place of $Y$, the assumption on $Y$ implies that $H^*(Y_{m,\beta})$ also lies in even degrees. 
Now apply \cref{CohomologicalIndependenceOnHypersurface}
to $M=Y_{m,\beta}$. The claim now follows from  \eqref{Equation intersection pairing} and \eqref{LESforPair}, which simplifies to:
\begin{equation*}
0 
\rightarrow H_{\textrm{even}}(B_{T,q}) \rightarrow H_{\textrm{even}}(Y_{m,\beta}) \stackrel{\psi}{\rightarrow} H_{\textrm{even}}^{lf}(Y_{m,\beta})  \rightarrow 
H_{\textrm{odd}}(B_{T,q}) 
\rightarrow
0. \qedhere
\end{equation*}
\end{proof}

\subsection{Computing the integer-time columns of the $E_1$-page}

When $H^*(Y)$ is supported in degrees $[0,\dim_{\C}Y]$, we can compute the time-1 column of the $E_1$-page, $H^*(\Sigma)[2\mu]$.
By \cite{RZ1}, this assumption %
holds if $\mathrm{Core}(Y)\subset Y$ has the homotopy type of a CW complex of dimension $\leq \dim_{\C}Y$. 
This holds for example for all CSRs and 
Moduli spaces of Higgs bundles.
\begin{prop} \label{CalculationOfTheSliceCohomology} %

Assume that the Novikov field $\k$ is defined over a base field $\mathbb{B}$ of characteristic $0$. Let $n=\dim_{\C} Y$, so $\dim_{\R} \Sigma=2n-1$.
		Suppose that $H^*(Y)$ is supported in degrees $0\leq * \leq n$. Then
		\[H^k(\Spp)\iso  \begin{cases*}
			H^{2n-1-k}(Y) & \ $k\geq n+1,$\\
			H^k(Y) & \ $k\leq n-2$,
		\end{cases*}
        \quad\textrm{ and }\quad
        H^{n}(\Spp) \iso H^{n-1}(\Spp) \iso  
            H^{n-1}(Y)\oplus \ker \psi.
		\]
  where $\psi: H_{n}(Y) \fun H_n^{lf}(Y)\cong H_{n}(Y,Y_{\infty})$ 
  is the natural map for the pair $(Y,Y_{\infty})$,
		where $Y_{\infty}:=H^{-1}[H_p,\infty)$ is the outside of $\Sigma_p=\partial Y_{\infty}$. 

  Moreover, $\ker \psi$ equals the kernel of the intersection form $H_{n}(Y) \times H_{n}(Y) \fun \k$. In particular, if that form vanishes then $\ker \psi =H_n(Y) \iso H^n(Y)$,
  whereas if it is non-degenerate then $\ker \psi=0.$
	\end{prop}
\begin{proof} 
We apply \eqref{LESforPair} to $M:=\{H\leq H_p\}\subset Y$, so $\partial M=\Sigma_p\cong\Sigma$.
		By {\PL} duality 
   over the field $\k$,  $H_k(M,\partial M)\cong H^{2n-k}(M),$ %
		hence
		for $k\leq n-2$ 
		 the outside terms in \eqref{LESforPair} vanish by the assumption on $H^*(Y)$. Thus $H^k(\partial M)\cong H^k(M)$ (by universal coefficients).  %
		By Poincar\'{e} duality for the closed orientable $(2n-1)$-manifold $\partial M=\Sigma$ and universal coefficients, we have $H^k(\Sigma)\cong H_{2n-1-k}(\Sigma)\cong H^{2n-1-k}(\Sigma)$. Thus, for $k\geq n+1$, $H^k(\Sigma) \iso H^{2n-1-k}({Y})$. 

We observe that $H_n(\partial M)/H_{n+1}(M,\partial M)\cong \ker \psi$ arises from \eqref{LESforPair} for $k=n$ and $k=n-1$:
		\begin{equation*}%
  H_{n+1}(M)=0\to H_{n+1}(M,\partial M) \fun H_n(\partial M) \rightarrow H_{n}(M) \fja{\psi} H_n(M,\partial M) \rightarrow H_{n-1}(\partial M)\fun H_{n-1}(M).
  \end{equation*} 
  The claim about $\psi$ follows by using universal coefficients, and the Poincar\'{e} duality isomorphisms $H_{n+1}(M,\partial M)\cong H^{n-1}(M)$ and 
   $H^{n-1}(\partial M)\cong H_n(\partial M)$ ($\cong H^{n}(\partial M)$). 
The last claim uses \eqref{Equation intersection pairing}.        
\end{proof}
\begin{cor}\label{Mid Coh of Hypersurface Vanish}
Let $\M$ be a weight-$s$ CSR, $n:=dim_\C \M$. The cohomology of its hypersurface $\Sigma$ is 
\[H^k(\Spp)\iso  \begin{cases*}
			H^{2n-1-k}(\M) & \ $\textrm{odd }k\geq n+1,$\\
			H^k(\M) & \ $\textrm{even }k\leq n-2$,
		\end{cases*}
        \quad\textrm{otherwise } 
        H^{k}(\Spp) = 0, \textrm{ so } H^n(\Spp)= H^{n-1}(\Spp) =0.
		\]
In terms of the $H_{p,\c}^*[-\mu(B_{p,\c})]$ from \cref{CorSpectralSeqForSH},
$H_{N}^*[-\mu(B_{N})] \cong H^*(\Spp)[Ns\dim_{\C}\M]$ for $N\in \N$.
\end{cor}
\begin{proof}
    
By \cref{CohomologyOfACSRProperties}(\ref{Thm_CSR_Sing_Coh_Over_Char_Zero})(\ref{Thm_CSR_definite_intersection_form}),
$H^*(\M)$ lies in even degrees only, $0\leq * \leq \dim_{\C} \M$, and its intersection form is non-degenerate.
So \cref{CalculationOfTheSliceCohomology} 
applies.
The last claim follows as the energy-spectral sequence collapses on the $E_1$-page: no vertical differentials occur, for grading reasons, as there is a $0$-gap between any two degrees in which $H^*(\Sigma)$ is supported.
We used $2\mu=s\dim_{\C}\M$ (\cref{CSRproperties}(\ref{Thm_CSR_Maslov_Index_formula})).
\end{proof}

\begin{prop}\label{Prop using gradings to prove QFi result}
Let $(\M,\Fi)$ be a weight-1 CSR. %
Then the $\Fi$-filtration satisfies
$$
\FF^{\Fi}_{\lambda<1}\subsetneq H^{\geq 2}(\M),
\qquad\quad
\FF^{\Fi}_{1} = H^{\geq 2}(\M),
\qquad\quad \FF^{\Fi}_{2}=H^*(\M),$$
and $H^{top}(\Sigma)[\dim_{\C}\M]$ kills the last surviving $H^{\dim_{\C}\M}(\M)$ class.
\end{prop}
\begin{proof}
By \cref{HitingTimeOfColumns}, only the $0$- and $1$-blocks, and the $1$- and $2$-pillars can hit $H^*(\M)$. So $\FF^{\Fi}_{2}=H^*(\M)$. 
The $1$-block and $2$-pillar lie in grading $\leq -1$, so 
$H^{\geq 2}(\M)\subset \FF^{\Fi}_{1}$.
It remains to show that $H^0(\M)$ was neither killed by the $0$-block nor by the $1$-pillar. For the $0$-block, this is immediate since it lies in grading $\geq 0$.
The $1$-pillar cannot kill the unit by \cref{Mid Coh of Hypersurface Vanish}, as
$H^*(\Sigma)[\dim_\C \M]$ in degree $-1$ 
equals
$H^{\dim_\C \M-1}(\Sigma),$ which vanishes
by \cref{Mid Coh of Hypersurface Vanish}.
For the final claim, the top class of the $1$-pillar is in grading $\dim_\C \M-1$, and it must kill a class to the left by \cref{PropTwoRowSplitSpectralSeq}.
    By \cref{Lemma block description spectral sequence}, the $0$-block lies in grading $\leq \dim_\C \M-1$. So it must kill a class in the $0$-pillar in $H^{\dim_{\C}\M}(\M)$.
\end{proof}
\section{Explicit examples of {\MBF} spectral sequences}\label{ExamplesSpSeq}
We will illustrate the {\MBF } spectral sequences $\Epqr$ from \cref{SpecSeqMBF} for three classes of CSRs, and for Higgs moduli spaces.
Via \cref{FiltrationOnCohomologyBySpectralSequence}, we deduce the rank-wise filtrations $\FF_{\lambda}^{\Fi}$ of their cohomologies, and then if needed use \cref{Cor intro about Fmin surviving} to get the actual cohomology classes. %
\begin{rmk}{\bf Conventions.} All cohomologies are computed over the Novikov field $\k$, over a base field $\mathbb{B}$ (\cref{Rmk about coeffs Novikov}) of \textbf{characteristic zero}, so \cref{CohomologyOfACSRProperties}(\ref{Thm_CSR_Sing_Coh_Over_Char_Zero}) and  
\cref{CalculationOfTheSliceCohomology} 
apply. Interesting torsion phenomena arise in non-zero characteristics, but it is outside our scope.
We use conventions from  \cref{NotationForSpSeq} and \cref{Remark Conventions about the spectral sequence Figures}, but in Figures we relabel $B_{p}$ to $B_{T_p}$, as $T_p$ indicates the torsion subgroup of the {\MB } submanifold. 
In spectral sequence pictures, we usually only show the columns that could hit non-zero classes in the $0$-column via iterated differentials, as these determine $\Fil^{\varphi}_\lambda.$ We frequently use $\omega_{\C}$-duality, \eqref{NonDegenPairingByOmC}.
We abbreviate a field summand $\k$ in degree $m$ by
$$
\k_m:= \k[-m].
$$
\end{rmk}

\subsection{Springer resolutions }\label{ExamplesSpSeqSprRes}

A decomposition $p=(p_1,\dots,p_n)$ of $n,$\footnote{Meaning integers $p_i \geq 0$ with $p_1+\dots + p_n =n.$} 
defines a $p$-partial flag variety 
$$\BP=\{\f:=(0=F_0 \subset F_1 \subset \dots \subset F_{n-1} \subset F_n=\C^n) \mid \dim F_i/F_{i-1}=p_i\}.$$
If a decomposition $p$ is also a partition,\footnote{meaning all $p_i>0$ and they are weakly-descending ordered ($p_i\geq p_{i+1}$).} one also defines the adjoint orbit $\O_p \subset \sl_n$ of nilpotent matrices whose Jordan normal form decomposes into blocks according to $p.$\footnote{A nilpotent matrix has zero eigenvalues, so $p$ completely determines a Jordan normal form with block sizes $p_i$.}
Given a decomposition $p$, its associated \textbf{generalised Springer resolution} is the map
\begin{equation}\label{partialSpringer}
\nu_p: \NNN_p \fun \ol{\mathcal{O}_{p_+^*}},\qquad \{(e,F) \mid  F\in \B_p, \ e F_{i} \subset F_{i-1} (\forall i=1\dots n) \}\mapsto e.
\end{equation}
Here, $\mathcal{O}_{p_+^*}\subset \sl_n$ is the nilpotent adjoint orbit whose matrices have Jordan partition $p_+^*,$ where $p_+$ is a weakly-descending permutation of $p$ with $p_i=0$ removed and $p_+^*$ is its dual partition.\footnote{defined e.g. by flipping the corresponding Young diagrams.} 

It is known that $\NNN_p \iso T^*\B_p$ (with projection to $\B_p$ given by $(e,F)\mapsto F$). %
Moreover, the map \eqref{partialSpringer} is known to be a \textbf{weight-1 CSR}, with the $\C^*$-action $\Fi$ that contracts the fibres on $T^* \B_p$ and
dilates matrices in $\ol{\mathcal{O}_{p_+^*}}.$
Thus, the action on $T^* \B_p$ is free outside the fixed locus $\B_p,$ %
hence there are no torsion points, so the spectral sequence $\Epqr$ is very simple.
The $0$-th column is the cohomology $H^*(\B_p)$ without a shift, whereas the $1$st and $2$nd columns are the cohomologies of the sphere bundle 
$H^*(S(T^* \B_p))$ with shifts of $d:=\dim_{\R} \B_p$ and $2d$ rows downward, by \eqref{Equation shift down for integer time hypersurfaces} (using \cref{CohomologicalIndependenceOnHypersurface} and \cref{SpSeqIs1periodic}).
Recall $H^*(S(T^* \B_p))$ can be computed by %
\cref{Mid Coh of Hypersurface Vanish}.

\begin{ex}\label{BabyBabyExample} $T^*\C P^1\fun \overline{\mathcal{O}_{11}}.$ (See \cref{SS_S11})
	
	The first three columns\footnote{We number the columns starting from $0$, so we mean ``the first three columns after the 0-th column.''} are $H^*(S(T^* \CP^1))=H^*(\RP^3)=\k_0\oplus\k_{3}$ (recall $\mathrm{char}(\mathbb{B})=0$), %
 with a shift down by $2$, $4$ and $6.$ 
 The filtration $\Fil^{\varphi}_\lambda $ is %
 equal to:
	$$0\subset \k_{2} \subset \k_0\oplus \k_{2} = H^*(T^*\C P^1).$$  
 So, writing $H^*(T^*\C P^1)=\k[x]/x^2,\ |x|=2,$ the filtration is $0 \subset \la x \ra \subset H^*(T^*\C P^1).$
	\begin{figure}[h]
		\centering
		{
			\includegraphics[scale=0.6]{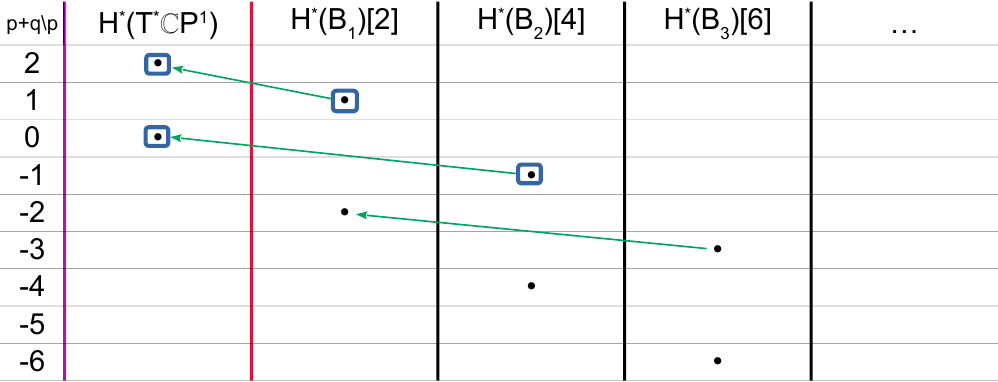}
			\caption{Spectral sequence for $T^*\C P^1$}
			\label{SS_S11}
		}
	\end{figure}
\end{ex}

\begin{ex}\label{Example Tstar of flag variety}$T^* \B_p \fun \ol{\mathcal{O}_{p_+^*}}.$
	
	The first three columns are $H^*(S(T^* \B_p))$ shifted down by $d=\dim_{\R} \B_p$, $2d$ and $3d$. Only the first two can hit the 0-th column via iterated differentials. The second column can hit only $H^0(T^* \B_p )$. As $H^*(S(T^* \B_p))$ has no mid-degree classes, rows $0$ and $-1$ of the first column are zero, so no class in it can hit $1\in H^0(T^* \B_p ),$ thus it gets killed by the second column. 
Altogether, the filtration is
\begin{equation}\label{filtr_flag_var}
    0 \subset H^{\geq 2}(T^* \B_p) \subset H^*(T^* \B_p).
\end{equation}

For example, for the decomposition $p=(1,n)$ of $n+1$, we have $\B_p=\C\mathbb{P}^n$, $p_+^*=(2,1,1,\ldots,1)$ (with $n-1$ copies of $1$), and \eqref{partialSpringer} becomes $ T^*\C\mathbb{P}^n \fun \ol{\mathcal{O}_{21\dots 1}}=\{A\in \mathfrak{sl}_{n+1} \mid A^2=0, \ \rk(A)=1 \}$, 
and the filtration is $0\subset H^{\geq 2}(\C\mathbb{P}^n) \subset H^*(\C\mathbb{P}^n)\cong H^*(T^*\C\mathbb{P}^n)$.
\end{ex}

\begin{rmk}\label{Remark all Springer resolutions}
    The same argument gives the same filtration \eqref{filtr_flag_var} for Springer resolutions of \textit{any} type, so $T^*(G/P),$ for $G$ reductive and $P$ parabolic. We discussed type A here, as we later consider
    Slodowy varieties of type A, in \cref{Examples Slodowy varieties}.
    The filtration \eqref{filtr_flag_var} is also a consequence of \cref{Prop using gradings to prove QFi result}. 
\end{rmk}

We now refine \eqref{filtr_flag_var} by composing the action $\Fi$ that contracts fibres of $T^*\B_p$ with a 1-parameter subgroup $G=(G_t)_{t\in\C^*}$ of the maximal torus $T\leq GL(n,\C)$
(which induces a holomorphic $\C^*$-action on $T^*\B_p$ that commutes with $\Fi$).
The subgroups $G\leq T$ form a lattice isomorphic to $\Z^n,$ via the map
$$h=(h_1,\dots,h_n) \in \Z^n \mapsto t^h:=\diag(t^{h_1},\dots,t^{h_n}) \leq T.$$
However, to keep the composition of $\Fi$ and $G$ contracting, it reduces the choice to a finite convex subset of this lattice. 
More precisely, fixing an integer $s>0$ there is a finite convex subset 
$$K_s:=\{h \mid t^{h}\Fi^s \text{ is a contracting action on } T^*\B_p\} \subset \Z^n.$$
It is not hard to prove that 
$K_1=\emptyset$, $K_s \subset K_{s'}$ for $s<s'$ and $\cup_s K_s = \Z^n.$
Moreover, we obtain:
\begin{prop} Any $s\in \Z_{\geq 1}$ defines a map 
$K_s \fun \{\text{filtrations of } H^*(\B_p) \text{ by cup-ideals} \}, \ h \mapsto \FF_\lambda^{t^h \Fi^s}.$
\end{prop}

\begin{ex}\label{Example Tstar of flag variety with a twisted action} 
\textbf{$T^* \CP^{n-1} \fun \ol{\mathcal{O}_{21 \dots 1}}$, with a twisted action.} (See \cref{SS_TCP3tw})
    
Let %
    $h=(0,1,\dots,n-1).$  
The action $t^h \Fi^n|_{\CP^{n-1}} =t^h \dejstvo \CP^{n-1}$ is $[z_0,t z_1, \dots , t^{n-1} z_{n-1}].$ %
The fixed points are $\F_i:=[0,\dots,1,\dots,0],\ i=0, \dots, n-1,$ where $1$ is in the $i$-th position.
There is a $(t\fun \infty)$ \CC-flowline from $\F_i$ to $\F_j$ if and only if $i<j,$ and it is $\Z/(j-i)$-torsion.
This way we get the weight decompositions of $\F_i$ in $\CP^{n-1}.$ %
For example, the weights of $\F_{n-1}$ are $(-1,-2,\dots,-(n-1)).$
To get the full decompositions in $T^*\CP^{n-1}$, we use the $\om_\C$-pairing 
$H_k \leftrightarrow H_{n-k}$ (\eqref{NonDegenPairingByOmC}, here the action has weight $s=n$). 
This yields additional weights
$(n+1,n+2,\dots,2n-1)$ to $\F_{n-1}.$
Similarly,%
\footnote{One can also compute the weights in coordinates. Near $\F_i\subset \C\P^{n-1}$ we have local coordinates $w_j=z_j/z_i$, for $j=0,1,\ldots,\widehat{i},\ldots,n-1$, whose weights are %
$(-i,-i+1,\ldots,-1,\widehat{0},1,\ldots,n-1-i)$.
Those induce opposite weights on the cotangent fibre basis $dw_j$, but we also need to add weight $n$ coming from $\Fi^n$, so: $(n+i,n+i-1,\ldots,n+1,\widehat{n},n-1,\ldots,i+1)$. The latter are ``outer'' weights of $\F_i$: they yield torsion submanifolds that leave the core. For example, for $n=4$, writing ((core-weights),(outer-weights)), we have for $\F_0$: $((1,2,3),(3,2,1))$, $\F_1$: $((-1,1,2),(5,3,2))$, $\F_2$: $((-2,-1,1),(6,5,3))$, $\F_3$: $((-3,-2,-1),(7,6,5))$. For example, $\F_1$ is contained in a rank one $3$-torsion bundle, $\F_0$ is contained in a rank two $3$-torsion bundle,  
and $\F_2$ is contained in a rank two $3$-torsion bundle caused by the weights $3,6$ (using that $\F_0,\F_1,\F_2$ are all $3$-minimal). This causes the $B_{1/3}$ column to contain the cohomology of a copy of $S^1$ and two copies of $S^3$. 
}  
we deduce the weights for the other $\F_i,$ and we observe that %
for each integer $k\in[n+1,2n-1]$ and $i\in [k-n,n-1]$, there is a 1-dimensional $k$-weight space in $\F_{i}$. 
Each one of them yields a $\Z/k$-torsion line bundle $\mathcal{H}_k \fun \F_i$
in $T^*\CP^{n-1},$ whose corresponding {\MB } submanifold is $B_{1/k,i}\iso S^1.$
By \eqref{CZindecesOfMorseBottSumbanifolds},
$$\mu(B_{1/k,i})=2(k-n-1).$$
Thus, for $c=1,\dots,n-1,$ the $c$-th column or equivalently, time$-\frac{1}{2n-c}$ column of the spectral sequence consists 
of $c$ copies of $H^*(S^1)[2n-2-2c],$ 
and a certain linear combination of its $H^1(S^1)$-classes is killing the $H^{2n-2c}(\CP^{n-1})$ class in the $0$-th column.
There are no other columns between these, as $\frac{1}{n+1}<\frac{2}{2n-1}.$
Thus, these columns kill all the cohomology except for the unit $1=H^0(\F_0)=H^0(\CP^{n-1}),$ which due to \eqref{Cor intro about Fmin surviving}
gets killed by
time $t\geq \frac{1}{n-1}$ columns.\footnote{as the weights of $\F_0$ are $(1,2,\dots,n-1),(n-1,\dots,2,1)$.}
In fact, a careful estimate of indices 
implies that it has to be killed precisely at time-$\frac{1}{n-1}.$\footnote{By \cite[Lemma 5.17]{RZ1}, for any weight-$s$ CSR, by time-$1/(s-1)$ all $\F_\a$ go below grading 0 if $h_1^{\a}+\sum_{j\geq 2} h_{(s-1)j}^{\a}\neq 0$: in our case $s=n$, and for all $\F_\a=\F_i$ except $i=n-1$ we have $h_1^{\a}\neq 0$, and for $\F_\a=\F_{n-1}$ we have $h_{(n-1)2}^{\a}\neq 0$.
}
Altogether, the filtration on $H^*(T^*\C P^{n-1})=\k[x]/x^n,\ |x|=2$ is
$$0 \subset \FF_{\frac{1}{2n-1}}=\la x^{n-1} \ra \subset \FF_{\frac{1}{2n-2}}=\la x^{n-2} \ra \subset \cdots \subset \FF_{\frac{1}{n+1}}=\la x \ra \subset \FF_{\frac{1}{n-1}}=H^*(T^*\C P^{n-1}),$$
so the Schubert cells in 
$\CP^n$ are
killed by certain (linear combinations of) torsion $S^1$-orbits.
\cref{SS_TCP3tw} is the case $n=4.$ 
We match by \textbf{blue boxes} the sets of classes that are killed one from the other by the green edge-differentials (in general we do not know whether these maps are diagonalisable). 
The \textbf{4-star symbol} is the point of local central symmetry (\cref{CentralSymmetrySpecSeq}). \textbf{Asterisk signs $*$ in the shifts} of some columns indicate that shifts of their connected components \textbf{are different.}
	\begin{figure}[h]
		\centering
		{
			\includegraphics[scale=0.6]{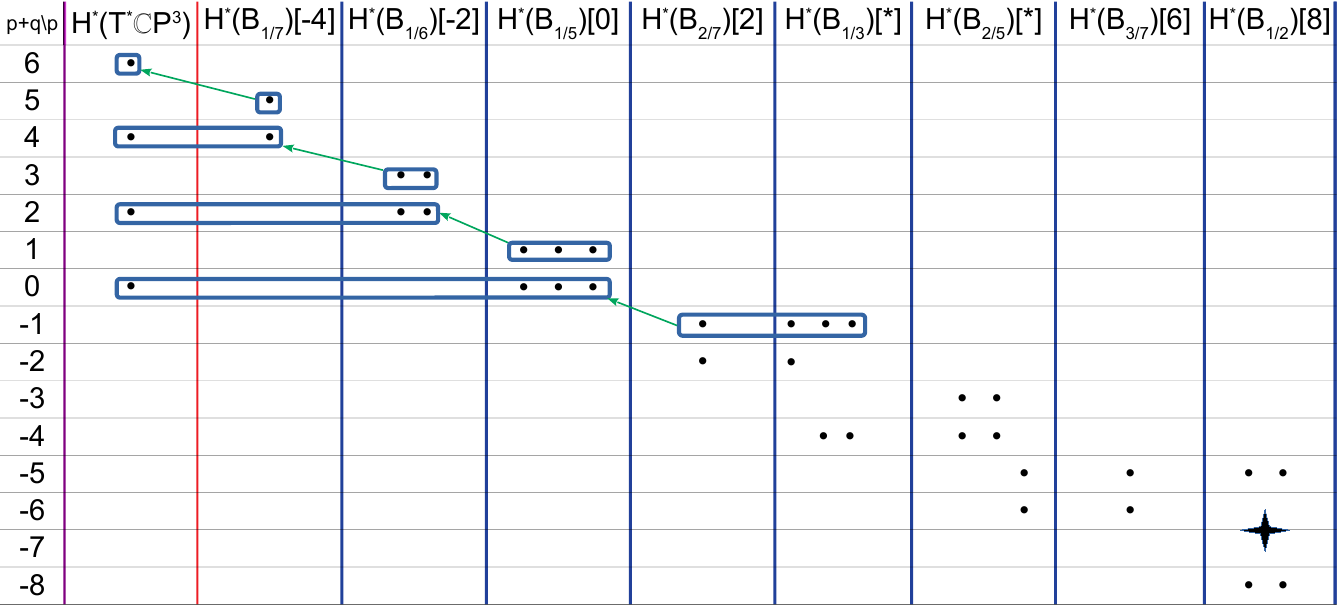}
			\caption{Spectral sequence for $T^*\C P^3$ with a twisted action}
			\label{SS_TCP3tw}
		}
	\end{figure}
\end{ex}

\subsection{Resolutions of ADE singularities}\label{DuValSingularitiesSpSeq}

CSRs of lowest dimension are resolutions of ADE singularities.\footnote{also known as simple surface singularities, ADE singularities, Kleinian singularities or rational double points.} They are the minimal resolutions of quotient singularities for finite subgroups $\Gamma\leq SL(2,\C)$:
$$\pi_{\Gamma}:X_\Gamma\fun \C^2/\Gamma.$$
Such groups $\Gamma$ (up to conjugation) are in bijection with ADE Dynkin graphs $Q_\Gamma$. This \textit{McKay graph} $Q_{\Gamma}$\footnote{more precisely, the quiver obtained by doubling the edges of $Q_\Gamma,$ and directing them in both ways.} encodes representation-theoretic properties of $\Gamma$. 
The core $\pi_{\Gamma}^{-1}(0)$ is a union of 2-spheres intersecting transversely according to  
$Q_\Gamma$.\footnote{The vertices of $Q_\Gamma$ correspond to the spheres; an edge in $Q_\Gamma$ corresponds to the intersection between the spheres.} Any $\C^*$-action on $X_\Gamma$ preserving the core will induce a %
$\C^*$-action on each sphere, so the intersection points of those spheres must lie in the fixed locus $\F$.
The Dynkin graph $A_n$ corresponds to the cyclic group $\Z/(n+1)$; the graph $D_n$ corresponds to the binary dihedral group $BD_{4(n-2)}$; the graphs $E_6, E_7$ and $E_8$ correspond to the binary tetrahedral, octahedral and icosahedral groups.
Klein showed that the coordinate ring $\C[\C^2/\Gamma]=\C[\C^2]^{\Gamma}$ is generated by three polynomials, so these define a polynomial map $(z_1,z_2)\mapsto (f,g,h)\subset \C^3$ that embeds the variety $\C^2/\Gamma$
into $\C^3$ (cf.\,\cite[Sec.1.2]{Dol07}). Using $(X,Y,Z)$ as coordinates on $\C^3$, the embeddings for types A and D are:
\begin{equation}\label{ADEsingularities}
	\begin{aligned}
		& \Gamma=\Z/n: \ \ \ (z_1,z_2)\mapsto (z_1^{n},z_2^{n},z_1 z_2),\text{the image is } XY - Z^{n}=0,\\
		& \Gamma=BD_{4n}:\ (z_1,z_2)\mapsto ((z_1^{2n}-z_2^{2n})z_1 z_2, z_1^{2n}+z_2^{2n} ,z_1^2 z_2^2), %
		\text{the image is } X^2 - Y^2 Z + 4 Z^{n+1}=0.\\
	\end{aligned}
\end{equation}

The \textbf{standard action} on $X_\Gamma\fun \C^2/\Gamma$ is induced from %
the dilation action 
 $t\cdot (z_1,z_2)=(t z_1,t z_2)$ on $\C^2.$
It is a \textbf{weight-2} conical\footnote{i.e. contracting and acting with weight-2 on the holomorphic-symplectic form, see \cref{CSRs}.} 
action.
\cref{BabyBabyExample} is the $A_1$-singularity, i.e. $\Gamma=\Z/2$ case.

We first consider type $A$. %
In the $\C^3$ coordinates of (\ref{ADEsingularities}), the standard action is
	$$t\cdot X=t^{k} X,\ \ \text{  }t\cdot Y=t^{k} Y,\ \ \text{  }t\cdot Z=t^2Z.$$
	The torsion subvarieties are the (complex) lines $l_1=(0,Y,0)$ and $l_2=(X,0,0),$
	which have $\Z/k$-isotropy. 
	The torsion submanifolds in $X_{\Z/k}$ are the two $\Z/k$-torsion lines, $\widetilde{l_1}$ and
	$\widetilde{l_2}$ that converge 
 to the two outer fixed points of the core.
  For $p=-1,\dots,-(k-1)$, the intersections $B_{p,i}=\widetilde{l}_i\cap \Sp$ with each slice $\Sp$  
	yields two {\MB } submanifolds%
 $$B_{p,1}\iso B_{p,2}\iso S^1.$$
 By Lemma \ref{CohomologicalIndependenceOnHypersurface}, the hypersurface $B_k$ has $H^*(B_k)\cong H^*(S^3/(\Z/k))\cong H^*(S^3)\cong \k_0\oplus\k_{3}.$
	As the core of $X_{\Gamma}$ is an $A_{k-1}$-tree of 2-spheres, 
 $H^*(X_{\Gamma})=\k_{0}\oplus \k_2^{k-1}.$ 
 To find the shifts $\mu(B_{p,i})$, consider the weight decomposition at the 
  two outer fixed points.
 The weight decomposition at a fixed point is $\C_k \oplus \C_{2-k}$ for some $k$, by the duality $H_k \xleftrightarrow{\om_\C} H_{2-k}$, due to the weight-2 action. Next, consider the parity of $k$:

\begin{ex}\label{DuVal_odd}\textbf{$\Gamma=\Z/k,$ for odd  $k\geq 3$.\;($A_{k-1}$-singularity)}
 (\cref{SS_S41}:\;spectral sequence for $k=5.$)%

 For this parity, $\tfrac{k-1}{2}$ is an integer, and by (\ref{CZindecesOfMorseBottSumbanifolds}) the shifts are:\vspace{-2.5mm}
\begin{equation*}
    \mu(B_{p,1})=\mu(B_{p,2})=
  1-\W(\tfrac{-p}{k} k)-\W(\tfrac{-p}{k}(2-k))
		=1+\W(\tfrac{2p}{k})=
		\begin{cases*}
			-2, & $p<-\frac{k-1}{2},$\\
			0, & $p\geq -\frac{k-1}{2}.$ 
		\end{cases*}
\end{equation*}
Hence, the first $\frac{k-1}{2}$ columns of $E_1$ consist of two copies of $H^*(S^1)$; the next $\frac{k-1}{2}$ columns consist of two copies of $H^*(S^1)$ shifted down by 2; the next column is $H^*(S^3/(\Z/k))\cong H^*(S^3)$ shifted down by $2\mu=s\cdot \dim_{\C} X_{\Gamma}=2 \cdot 2 = 4$. The classes $H^2(X_{\Gamma})$ in the 0-th column will vanish two by two at each page, and by \cref{Killing the unit in sp seq CSR}(\ref{killing the unit wt2 Fmin=pt case}) the unit $1\in H^0(X_{\Gamma})$ 
gets killed precisely by time-1 column, thus vanishes on page $E_{k+1}$.
 The filtration %
 is rank-wise
\begin{equation}\label{Filtration for Z/k for odd k} 
    0\ \subset \ \k_2^2 \ \subset \ \k_2^4 \ \subset \cdots\ \subset \ \k_2^{k-1}\ \subset\  \k_{0}\oplus \k_2^{k-1}=H^*(X_{\Z/k}).
\end{equation}	
 \begin{figure}[H]
		\centering
		{
			\includegraphics[scale=0.6]{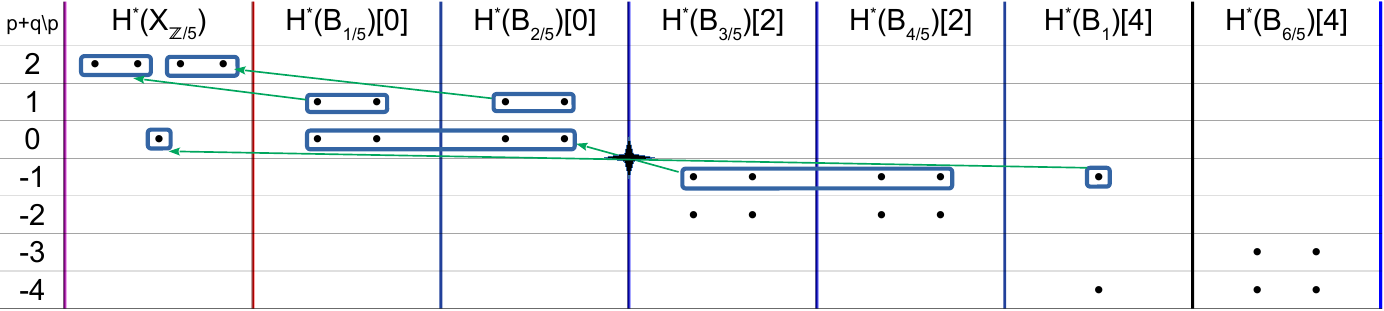}
			\caption{Spectral sequence for $X_{\Z/5}$}
			\label{SS_S41}
		}
\end{figure}

\end{ex}

\begin{ex}\textbf{$\Gamma=\Z/k,$ for even $k\geq 2$.\;($A_{k-1}$-singularity)} (\cref{SS_S51}:\;spectral sequence for $k=6.$)%

  This time the action is even. By considering its square-root of weight-1, the time-$1$ column of the $E_1$-page is $H^*(\Sp)[2],$ and later columns repeat with shift $2\cdot 1=2$. %
  The classes in $H^2(X_{\Gamma})$ vanish two by two until the last one is hit by the top class of the time-$1$ column, and $1\in H^0(X_{\Gamma})$ is killed 
  by time-$2$ column (thus vanishes on page $E_{k+1}$) due to \cref{Killing the unit in sp seq CSR}(\ref{killing the unit wt1 case}).
The filtration is rank-wise %
	$$0\ \subset \ \k_2^2 \ \subset  \ \k_2^4 \ \subset  \cdots \ \subset \ \k_2^{k-2} \ \subset \ \k_2^{k-1}\  \subset \ \k_{0}\oplus \k_2^{k-1}=H^*(X_{\Z/k}).$$ 
	
	\begin{figure}[H]%
		\centering
		{
			\includegraphics[scale=0.6]{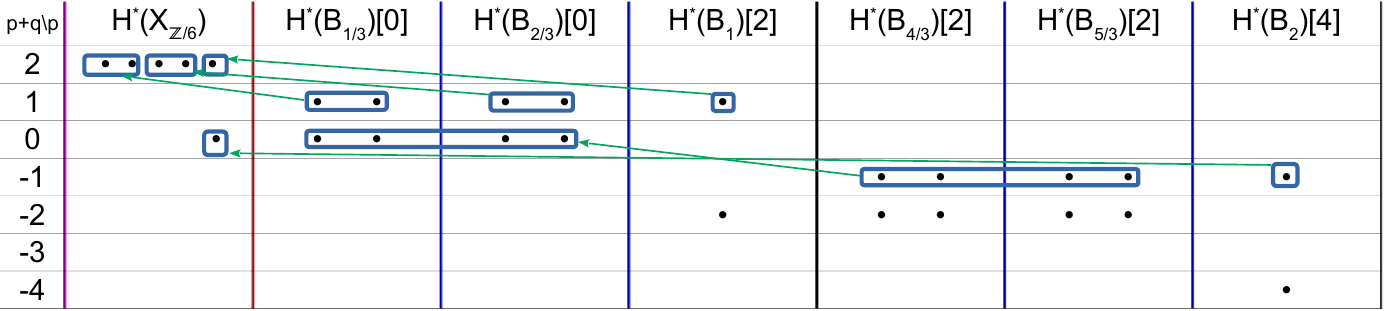}
			\caption{Spectral sequence for $X_{\Z/6}$}
			\label{SS_S51}
		}
	\end{figure}
\end{ex}

\begin{ex}\label{ExampleZ3Exotic}\textbf{$\Gamma=\Z/3,$ with a non-standard action. ($A_2$-singularity)} (See \cref{SS_S21Exotic}) %
 
 We consider the non-standard weight-$2$ action induced by
	$$t\cdot X=t X,\text{  }t\cdot Y=t^{2} Y,\text{  }t\cdot Z=t Z.$$
 Here, there is only a $\Z/2$-torsion line $l_1=(0,Y,0).$ Its resolution $\widetilde{l_1}$ is a 
	$\Z/2$-torsion line that converges to a fixed point $\F_1$ with weight decomposition $\C_2\oplus \C_{-1}.$ By (\ref{CZindecesOfMorseBottSumbanifolds}) the shift of the {\MB } submanifold $B_{1/2}=\widetilde{l_1} \cap \Sigma_1$ is $\mu(B_{1/2})=0.$ 
 The time-1 column has shift %
 $\mu(B_1)=-2.$ The later columns we get by periodicity (\cref{SpSeqIs1periodic}).
As before, the unit is killed by $H^{top}(B_2)$. The filtration 
(rank-wise) is:
	$$0\ \subset \ \FF_{1/2}=\k_{2} \ \subset \ \F_{1}=\k_2^2 \ \subset \ \F_{2}=\k_0\oplus \k_2^2 = H^*(X_{\Z/3}).$$
    	\begin{figure}[H]
		\centering
		{
			\includegraphics[scale=0.6]{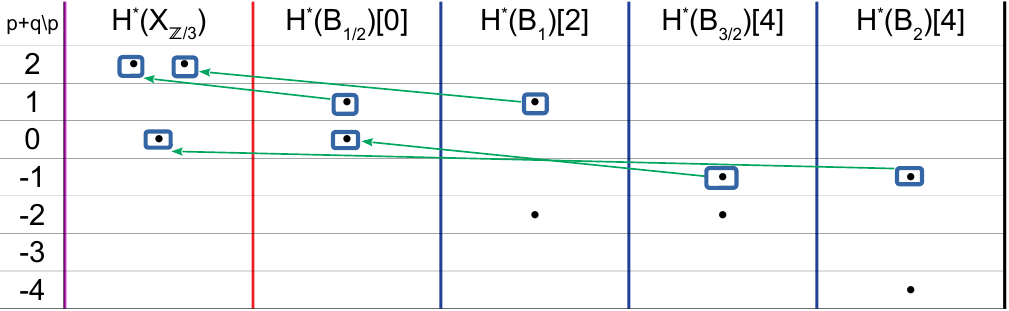}
			\caption{Spectral sequence for $X_{\Z/3}$ with a non-standard action.}
			\label{SS_S21Exotic}
		}
	\end{figure}
By \eqref{Cor intro about Fmin surviving}, we get $\FF_{1/2}=H^0(\F_1)[2]=\k\cdot [S_2^2]$ which determines the filtration
completely,
	and it is different from the filtration  obtained by the standard action in Example \ref{DuVal_odd}, for $k=3$:
	$$0\ \subset \ \FF_{1/3}=\k_2^2 \ \subset \ \FF_{1}=\k_0\oplus \k_2^2 = H^*(X_{\Z/3}).$$
Thus, two different Floer-theoretic presentations of $H^*(X_{\Z/3})$ with two different filtrations by ideals.
\end{ex}

For type D, we consider the weight-1 square root
of the standard action %
on a $D_{n+2}$-singularity, so:
\begin{equation}\label{SquareRootStandard}
	\qquad \qquad t\cdot X=t^{n+1} X,\ \text{  }t\cdot Y=t^{n} Y,\ \text{  }t\cdot Z=t^2 Z.  \qquad (\textrm{in the }\C^3\textrm{ coordinates from \eqref{ADEsingularities}}.)
\end{equation}

\begin{ex} \textbf{$\Gamma=BD_{8k}$ ($D_{2k+2}$-singularity), $k\geq 1.$}\label{ExampleD_nSingEven}  
(See \cref{D6_sp_seq})\vspace{1mm}
	
	For $n=2k$, (\ref{SquareRootStandard}) is
	$t\cdot X=t^{2k+1} X$, $t\cdot Y=t^{2k} Y$, $t\cdot Z=t^2 Z.$
	There are three torsion lines 
	$$\a=(0,Y,0),\ \b_1=(0,2Z^k,Z),\ \b_2=(0,-2Z^k,Z),$$ 
	with respective torsion groups $\Z/2k,\Z/2,\Z/2$.
	As it is a weight-1 action, it has a (unique) fixed core component \cite{vzivanovic2022exact}: the sphere corresponding to the trivalent vertex of the graph $D_{2k+2}.$\footnote{Because that sphere carries three intersection points fixed by the action, so the $\C^*$-action must fix the whole sphere.}
	The resolutions $\widetilde{\a},\widetilde{\b_1},\widetilde{\b_2}$ of  $\a,\b_1,\b_2$ are lines that converge to the fixed points of the spheres corresponding to the leaves of $D_{2k+2}$ (each leaf of the attraction graph yields a torsion bundle converging to it by Proposition \ref{THERE_ARE_ISOTROPIES}).
    The intersections of $\widetilde{\a},\widetilde{\b_1},\widetilde{\b_2}$ with the slices $\Sp$ are the Morse Bott submanifolds $B_{p,c}\cong S^1$. By \eqref{CZindecesOfMorseBottSumbanifolds}, as before, we get the degree shifts $\mu(B_{p,c})$: they are zero for $T_p<1,$ and  $B_1=\Sigma_1$ has a shift down by 2 (as the action has weight 1). The other columns follow by Corollary \ref{SpSeqIs1periodic}.
	The filtration %
 is rank-wise:
	\begin{align*} 
		&\text{for } k=1: \ 0\subset \k_2^3 \subset \k_2^4 \subset \k_{0}\oplus \k_2^{4}=H^*(X_{\Gamma}),\\
		&\text{for } k\geq 2: \ 0\subset \k_{2} \subset \dots \subset \k_2^{k-1} \subset \k_2^{k+2} \subset \k_2^{k+3} %
    \dots \subset \k_2^{2k+1} \subset \k_2^{2k+2} \subset H^*(X_{\Gamma}).
	\end{align*}
	In both cases, the jump in rank by $3$ occurs at the time-$\frac{1}{2}$ column, as all three lines are $\Z/2$-torsion. %
	\begin{figure}[H] 
		\centering
		{
			\includegraphics[scale=0.60]{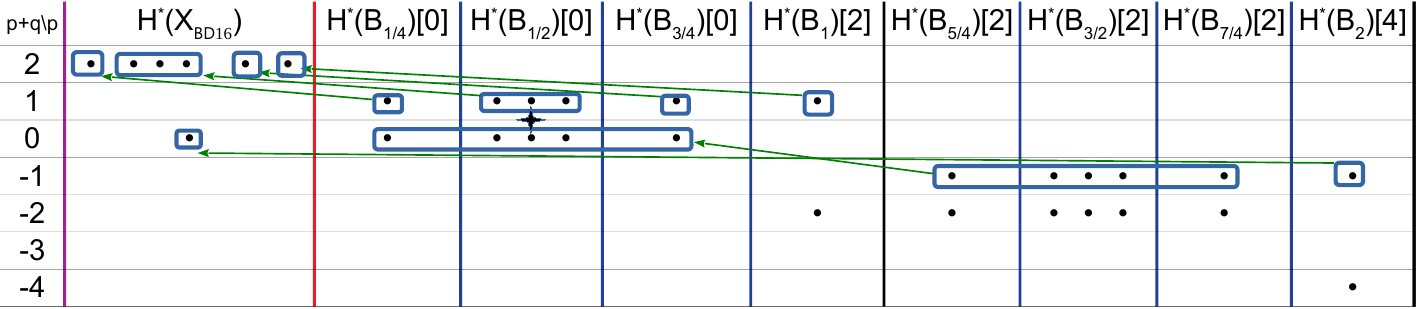}
			\caption{Spectral sequence for $D_6$ ($k=2$, $n=4$).}
			\label{D6_sp_seq}
		}
	\end{figure}
\end{ex}

\begin{ex} \textbf{$\Gamma=BD_{8k+4}$ ($D_{2k+3}$-singularity), $k\geq 1.$}\label{ExampleD_nSingOdd} (See \cref{D5_sp_seq})\vspace{1mm}

	For $n=2k+1$, (\ref{SquareRootStandard}) is
	$t\cdot X=t^{2k+2} X$, $t\cdot Y=t^{2k+1} Y$, $t\cdot Z=t^2 Z.$
	There are three torsion lines 
	$$\a=(0,Y,0),\ \b_1=(2i Z^{k+1},0,Z),\ \b_2=(-2i Z^{k+1},0,Z),$$ 
	with isotropies $\Z/(2k+1),\Z/2,\Z/2$. Again $B_{p,\c}\cong S^1$ with zero shifts. The time $\frac{1}{2}$-column is new with \emph{two} {\MB } submanifolds, so the filtration %
 at time $\frac{1}{2}$ jumps by rank 2 in degree 2. The filtration is
	\begin{equation*}
		0\subset \k_{2} \subset \dots \subset \k_2^{k} \subset \k_2^{k+2} \subset \k_2^{k+3} \subset \dots \subset \k_2^{2k+3} \subset H^*(X_{\Gamma}).
	\end{equation*}
	
	\begin{figure}[H] %
		\centering
		{
			\includegraphics[scale=0.6]{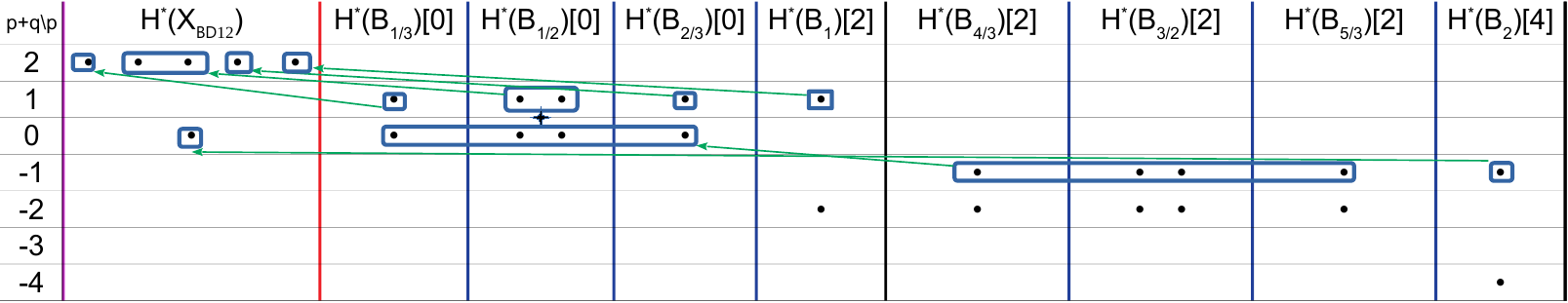}
			\caption{Spectral sequence for $D_5$ ($k=1$, $n=3$).}
			\label{D5_sp_seq}
		}
	\end{figure}
\end{ex}

\begin{ex}\textbf{$D_n$ via the attraction graph.} \label{Dn via extended attraction graph}
	It soon becomes desirable to obtain descriptions of the {\MB } manifolds that avoid the coordinate presentations like (\ref{ADEsingularities}), instead using the attraction graph (\cref{Subsection Attraction graphs}).
We illustrate this for $D_n$. 
 The core consists of a Dynkin tree $D_n$, each vertex corresponds to a sphere, and each edge corresponds to an intersection of spheres which in turn is a fixed point of the action,
 and the sphere $C$ at the unique trivalent vertex is fixed (as it carries three fixed points).
At each fixed point, $\omega_{\C}$-duality implies that the weight decomposition is $H_k\oplus H_{2-k}$, each summand having $\dim_{\C}=1$.
If a sphere in the core is fixed, it has a weight $k=0$, thus $H_0\oplus H_2$ and the {\MB } index is $0$. By \eqref{EqnFrankel} there is only one fixed sphere, $C$, and the other spheres must carry exactly two fixed points (by considering the rank in \eqref{EqnFrankel}). 
By \cref{CorollaryGradientTrajBecomeSpheres}, it is easy to write out the weight decompositions at all fixed points, and the attraction graph is a $D_n$-Dynkin graph.
It is preferable to use the square-root of the action, %
which exists as there is a fixed component $C$ of maximal ($\om_\C$-Lagrangian) dimension in the core. 
This is based on the analytic continuation of the square-root that exists on the torsion bundle 
$\mathcal{H}_2\fun C,$
see \cite[Lem.5.4., Prop.5.5]{vzivanovic2022exact}.
The square-root is a weight-1 {\contracting } action, so 
$\omega_{\C}$-duality gives $H_k\oplus H_{1-k}$ decompositions: $H_0\oplus H_1$ at $C$, e.g.\;at intersections points $p_1,p_2,p_3$ with the adjacent spheres $C_1,C_2,C_3$; at the second fixed points $q_1,q_2,q_3$ of $C_1,C_2,C_3$ we get $H_{-1}\oplus H_2$ by \cref{CorollaryGradientTrajBecomeSpheres} ($C_i$ defines a $\C^*$-flow from $p_i$ to $q_i$).
 For $D_4$, the $q_i$ are leaf nodes in the attraction graph, and their $H_2$ defines a $\Z/2$-torsion line bundle $\mathcal{H}_2$, yielding {\MB } manifolds $S^1$ when intersecting with slices.
 For $D_n$ with $n>4$, the same applies to two of those $q_i$. The third corresponds to where additional edges arise in the $D_n$ Dynkin diagram: they correspond to fixed points with $H_{-2}\oplus H_3$, $H_{-3}\oplus H_4$, $\ldots$, $H_{-(n-3)}\oplus H_{n-2},$ yielding a $\Z/(n-2)$-torsion line bundle at the leaf (so again $S^1$ {\MB } manifolds). So we recover Examples \ref{ExampleD_nSingEven} and \ref{ExampleD_nSingOdd} without needing (\ref{ADEsingularities}).
\end{ex}
	\begin{figure}[H]%
		\centering
		{
			\includegraphics[scale=0.6]{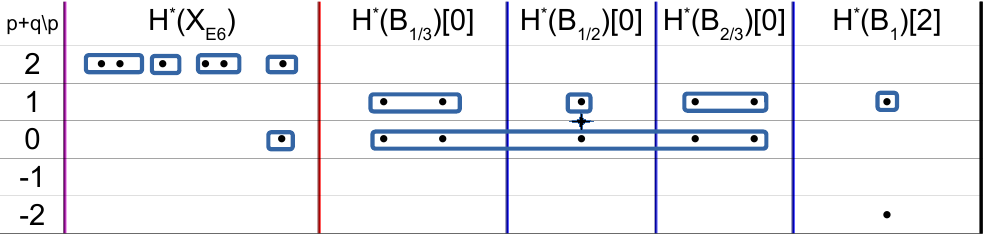}\\
   \includegraphics[scale=0.6]{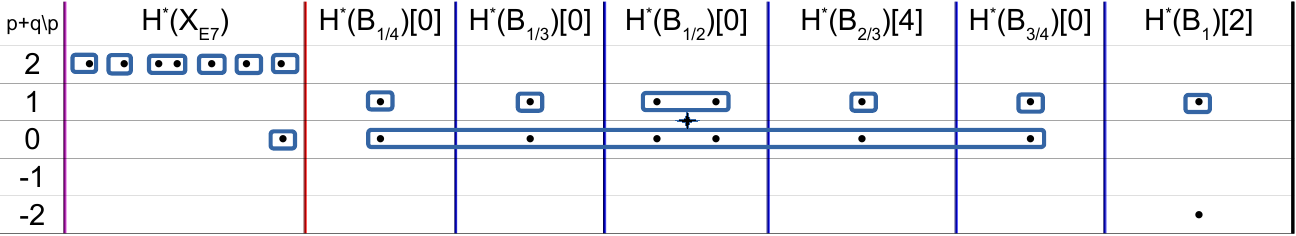}\\
   \includegraphics[scale=0.60]{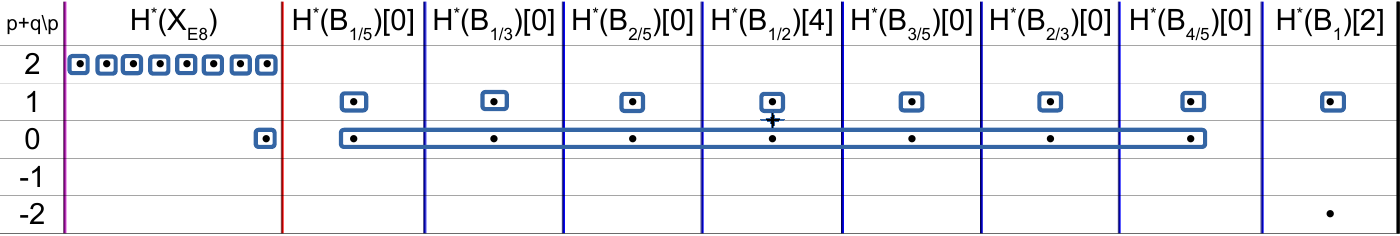}
			\caption{Spectral sequence for $E_6$, $E_7$, $E_8$ (omitting edge-differentials 
   for 
   readability)}
			\label{E6E7E8}\label{E6_sp_seq}\label{E7_sp_seq}\label{E8_sp_seq}
		}
	\end{figure}
\begin{ex}\textbf{$E_6$, $E_7$, $E_8$ Singularities.} \label{Example E_678}(See Figure %
\ref{E6E7E8})

We use the method from \cref{Dn via extended attraction graph} (using invariant polynomials is too complicated). 
The core for $E_6$ is the $E_6$-Dynkin tree of spheres; the sphere $C$ corresponding to the trivalent node is fixed. 
Thus, as in the previous example, %
the standard action has a square-root: a weight-1 {\contracting } action with weight decomposition $H_0\oplus H_1$ at $C$. 
The attraction graph is an $E_6$ Dynkin diagram.
We get a $\Z/2$-torsion and two $\Z/3$-torsion line bundles over fixed points, corresponding to the 
leaves in \cref{E_6attractionGraph}. %
	\begin{figure}[H]%
		\centering
		{
		\includegraphics[scale=0.20]{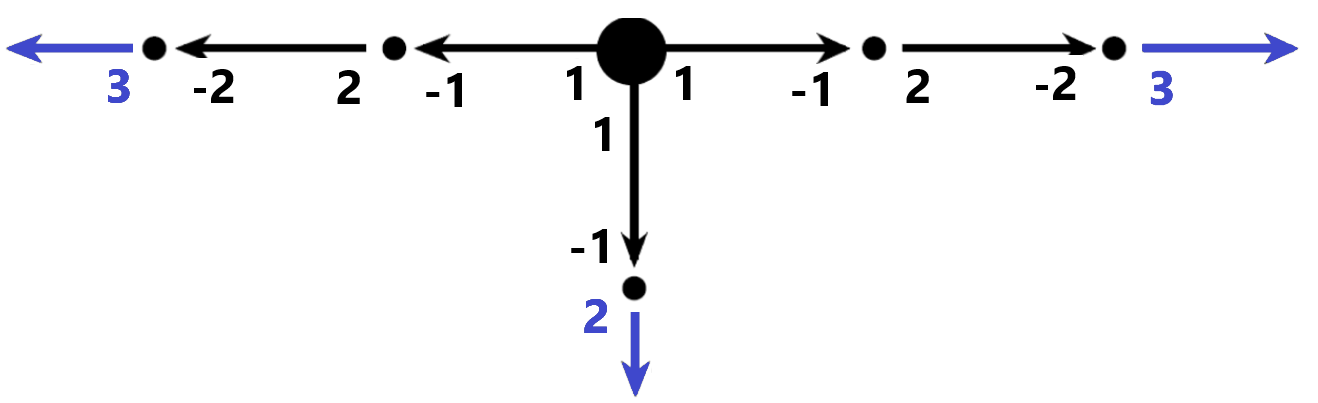}
			\caption{$E_6$-singularity: extended attraction graph of the minimal resolution.}
			\label{E_6attractionGraph}
		}
	\end{figure}
Thus, the {\MB } submanifolds are circles, and their shifts are calculated in the usual way from the weight decompositions $H_{-1}\oplus H_2$, $H_{-2}\oplus H_3$, $H_{-2}\oplus H_3$ of their converging points. 

The three torsion line bundles over fixed points at leaves of the attraction graph for $E_6$ were $\Z/2$, $\Z/3$, $\Z/3$; for $E_7$ this becomes $\Z/2$, $\Z/3$, $\Z/4$; for $E_8$: $\Z/2,\Z/3,$ $\Z/5.$ The filtrations are:
\begin{align*}\small
\begin{split}
E_6: & \;0\subset \k_2^2 \subset \k_2^3 \subset \k_2^5 \subset \k_2^6
	\subset \k_0\oplus\k_2^6=H^*(X_{\Gamma}).\\
E_7: & \;0\subset \k_2 \subset \k_2^2 \subset \k_2^4 \subset \k_2^5
	\subset \k_2^6 \subset \k_2^7
	\subset \k_0\oplus\k_2^7=H^*(X_{\Gamma}).\\
E_8: & \;0\subset \k_2 \subset \k_2^2 \subset \k_2^3 \subset \dots
	\subset \k_2^7 \subset \k_2^8
	\subset \k_0\oplus\k_2^8=H^*(X_{\Gamma}).
 \end{split}
\end{align*}
\end{ex}

\subsection{Slodowy varieties of type A}\label{Examples Slodowy varieties}
\renewcommand{\RZ}[1]{\mathcal{R}_{\Z/{#1}}}
We will now consider resolutions of ordinary Slodowy varieties of type A, with the Kazhdan $\C^*$-action.
Recall that the nilpotent cone $\mathcal{N}$ of the Lie algebra $\sl_n$ is the variety of all nilpotent matrices, and it has an associated (ordinary) Springer resolution,
\begin{equation}\label{OrdinarySpringer}
	\nu: \widetilde{\NN}\fun \NN, \ \{(e,\f=(0=F_0 \subset F_1 \subset \dots \subset F_{n-1} \subset F_n=\C^n)) \mid  e F_{i} \subset F_{i-1} \}\mapsto e,
\end{equation}
a special case of the map \eqref{partialSpringer}, when $p=(1,\dots,1),$ as then $\NN = \ol{\O_n}.$
Recall $\NNN \iso T^*\B,$ where $\B=\B_{1\dots1}$ is the ordinary flag variety in $\C^n.$ %
The fibres of the map \eqref{OrdinarySpringer} are called \textbf{Springer fibres}. It is a standard fact that
they are equidimensional projective varieties, whose irreducible components are labelled bijectively by Standard Young Tableaux\footnote{Young diagrams filled with numbers that strictly increase row-wise and column-wise.}
of the shape that corresponds to the Jordan partition of the nilpotent matrix, \cite{Spa76}. %
Given a nilpotent matrix $e \in \NN,$  there is an $\mathfrak{sl}_2$-triple $(e,f,h)$ of nilpotent matrices. Choosing such a triple, one constructs the {\bf Slodowy slice}
$$S_e=\{x\in \sl_n \mid [x-e,f]=0\}= e + \text{ker} (\ad f).$$
The \textbf{Slodowy variety} and its resolution are defined as
	$\mathcal{S}_{e}:=S_e\cap \NN,$ $\SSS_{e}:=\nu^{-1}(\mathcal{S}_{e}).$
On the Slodowy slice and Slodowy variety, the \textbf{Kazhdan action} is $t \cdot x = t^2 \Ad(t^{-h}) x.$
Its lift
to $\SSS_{e}$ is given by
\begin{equation}\label{KazhdanActionOnResolutionExamples}
	t \cdot (x, \f) = (t^2 \mathrm{Ad}(t^{-h}) x, t^{-h} \f).
\end{equation}
The map $\nu:\SSS_{e}\fun \SS_e$ is a \textbf{weight-2 CSR} with that action. Its core is the Springer fibre $\B^e=\nu^{-1}(e).$
Picking a Jordan basis $(v_i)$ for the triple $(e,f,h),$ that action on a matrix $x=(X_{ij})$ becomes 
\begin{equation}\label{weightOfKAzdhanActionExamples}
	t\cdot (X_{ij})=(t^{2+h_j-h_i}X_{ij}),
\end{equation}
whereas the action on flags is induced by the action on $\C^n$ given by 
\begin{equation}\label{KazdhanActionOnBasisExamples}
	t\cdot v_i = t^{-h_i}v_i,
\end{equation}
where $h=\diag(h_1,\dots,h_n)$ in the basis $v_i.$ %
Given a partition $\lambda=(\lambda_1,\dots,\lambda_k),$ we will use the nilpotent element $e_\lambda:=\diag(J_{\lambda_1},\dots, J_{\lambda_k})$ 
given by the standard Jordan canonical form for the partition $\lambda$ (so $\lambda_i\times \lambda_i$ Jordan blocks with eigenvalue zero). There is a choice of an $\sl_2$-triple $(e_\lambda,f_\lambda,h_\lambda)$ such that they have the same Jordan basis and $h_\lambda$ becomes diagonal in this basis. One can also explicitly compute the corresponding Slodowy slice $S_\lambda:=S_{e_\lambda}.$ For the details of this, see \cite[Sec.5.2.1-5.2.2]{FZ20},\footnote{There is a typo in Sec.5.2.1 of the thesis: the diagonal entries $h_i^k$ of $h_k$ are $h_0^k,\ldots,h_{k-1}^k$, so $i$ starts with $i=0$.}  in particular \cite[Prop.5.2.14.]{FZ20} for the description for $S_\lambda.$ 
One additional observation that we prove here is 
\begin{lm}\label{Lemma Kazhdan = Morse when lambda equal to odd and even}
The fixed points of the Kazhdan action on $\SSS_\lambda$ are isolated, or equivalently, its corresponding $S^1$-moment map is Morse if
$\lambda=(m,n),$ and $m \neq n \ \mathrm{  (mod \ 2)}.$\footnote{The ``only if'' direction is also true, however, we do not prove it here for brevity. 
Briefly, %
if for some $i \neq j$ the $\lambda_i$, $\lambda_j$ are of the same parity, one can explicitly (in terms of flags) 
construct a Kazhdan-fixed $\P^1 \subset \B_\lambda.$} %
\end{lm}
\begin{proof}
From the description of $h_\lambda$ given in
\cite[p.112]{FZ20},
it is clear that if $\lambda=(m,n),$ $m \neq n \ \mathrm{  (mod \ 2)},$ then $h_\lambda$ has distinct entries
(explicitly: $h_\lambda=(m-1,m-3,\dots,-(m-1), n-1,n-3,\dots,-(n-1))$).
Thus there is no Kazhdan-invariant subspace on $\C^n,$ and all fixed flags on $\B$ are isolated. 
Now recall that fixed points in $\SSS_\lambda$ 
are contained in its core, which is $\B^\lambda \subset \B.$
\end{proof}

The basic algorithm we use is that we first find the torsion points $\{\TZ{m}\}_{m\geq 2}$ in the Slodowy slice $S_\lambda,$ then we restrict to the nilpotent cone $\NN$ in order to get the torsion points $$\PZ{m}:=\TZ{m}\cap \NN$$ in the Slodowy variety $\mathcal{S}_{\lambda}.$ Then, using the Springer resolution $\nu$, we find the corresponding torsion points in $\SSS_{\lambda}=\nu^{-1}(\mathcal{S}_{\lambda}),$ denoted by $\{\RZ{m}\}_{m\geq 2}.$ As the fibres of the resolution $\nu$ depend on the nilpotent orbit $\mathcal{O}_{\gamma}$ where they are taken, we will keep track of $\gamma$ in the intersections $$\PZ{m,\gamma}:=\TZ{m}\cap\mathcal{O}_{\gamma}, \ \ \PZ{m}= \sqcup_\gamma \PZ{m,\gamma}.$$ 

When the Kazhdan action is even, we take the square-root to avoid all points being $\Z/2$-torsion.
We will use the notation $B_{T_p}:=B_{p}$ for {\MB } submanifolds (same labels as in the spectral sequence).

\begin{ex} \textbf{Slodowy variety $\mathcal{S}_{22}.$} %
	For $\lambda=(2,2)$, $n=4,$ from \cite[Prop.5.2.14.]{FZ20} we
	obtain the Slodowy slice $S_{22}$ in $\mathfrak{sl}_4$ at the nilpotent element $e_{22}=\diag(J_2,J_2),$ with $h_{22}=\diag(1,-1,1,-1):$ 
	\[ \small
	S_{22}=\left\{ \begin{bmatrix}
		a&1&c&0\\
		m&a&p&c\\
		e&0& -a &1\\
		x&e&w&-a
	\end{bmatrix}  \middle| \ a,c,e,m,p,x,w\in \C \right\}.
	\]
	As $t^{h_{22}} =diag(t,  t^{-1},  t,  t^{-1})$, by formula (\ref{weightOfKAzdhanActionExamples}) the Kazhdan action on $S_{22}$ is 
	\begin{equation}\small
		\label{KazhdanOnS22}
		t\cdot \begin{bmatrix}
			a&1&c&0\\
			m&a&p&c\\
			e&0& -a &1\\
			x&e&w&-a
		\end{bmatrix}=  \begin{bmatrix}
			t^2a&1&t^2c&0\\
			t^4m&t^2a&t^4p&t^2c\\
			t^2e&0& -t^2a &1\\
			t^4x&t^2e&t^4w&-t^2a
		\end{bmatrix}.
	\end{equation}
Thus, we see that the action is even, so we take a square root of it. Then, we only have powers $t^1$ and $t^2$ of $t$ in equation (\ref{KazhdanOnS22}), hence we get only $\Z/2$-torsion points in $S_{22},$ 	
	\[ \small \TZ{2}=
	\left\{ \begin{bmatrix}
		0&1&0&0\\
		m&0&p&0\\
		0&0& 0 &1\\
		x&0&w&0
	\end{bmatrix}  \middle| m,p,x\in \C \right\}.
	\]
Intersecting this with the nilpotent cone $\NN=\{A\mid A^4=0\},$ after some computations one gets the set of $\Z/2$-torsion points in the Slodowy variety $\mathcal{S}_{22}:$ 
	\begin{equation} \small \label{TorsionS22}
		\PZ{2}=
		\left\{ \begin{bmatrix}
			0&1&0&0\\
			m&0&p&0\\
			0&0& 0 &1\\
			x&0&w&0
		\end{bmatrix}  \middle| \ m^2+px=0, \ m,p,x\in \C \right\} \iso \C^2/(\Z/2).
	\end{equation} 		
One can check that these matrices, except for the zero matrix, all lie in the regular orbit $\mathcal{O}_{4},$ hence are bijectively lifted to $\SSS_{22}$ under the Springer resolution. Denoting by $A_{mpx}$ the matrix in  equation (\ref{TorsionS22}), for $(m,p,x)\neq (0,0,0),$ its lift is
	the Springer fibre $(A_{mpx},\f_{mpx}).$ Thus the flag {\small $$\f_{mpx}=(0\subset F_1 \subset F_2 \subset F_3 \subset \C^4)$$} is defined by $A_{mpx}F_i \subset F_{i-1}.$ Let $(v_1,v_2,v_3,v_4)$ be the basis in which these matrices are given. Then
	{\small \begin{align*}
		\f_{mpx}&=(0 \subset \langle p v_1 - m v_3\rangle \subset\langle  p v_1 - m v_3 , p v_2 - m v_4\rangle \subset\langle  v_1,v_3,pv_2-mv_4 \rangle\subset  \C^4), \text{when }(p,m)\neq(0,0),\\
		\f_{mpx}&=(0 \subset\langle m v_1  + x v_3 \rangle\subset\langle m v_1 + x v_3 , m v_2 + x v_4\rangle \subset\langle v_1,v_3, m v_2 + x v_4\rangle \subset \C^4),  \text{when }(m,x)\neq(0,0).
	\end{align*}}
	
	By \eqref{KazdhanActionOnBasisExamples}, the Kazhdan action on $\C^n$ is $t \cdot v=(t^{-1}v_1,t v_2,t^{-1}v_3,t v_4),$ so the induced action fixes the the flag $\f_{mpx}$. Thus, the fixed locus to which $\RZ{2}$ converges (when $t\fun 0$) is a $\P^1$ in the core:
	{\small $$\F_{mpx}:=(0 \subset\langle \a v_1+\b v_3 \rangle\subset\langle \a v_1 + \b v_3, \a v_2 + \b v_4 \rangle\subset\langle v_1,v_3,\a v_2 + \b v_4 \rangle\subset \C^4 \mid [\a:\b]\in \P^1).$$}
	Thus, $\RZ{2}$ has a single connected component, which is a torsion bundle $\RZ{2}=\mathcal{H}_2\fun \F_{mpx}.$ Being a resolution of the cone $\PZ{2}\iso \C^2/(\Z/2),$ it is a $A_1$ singularity resolution, hence $\mathcal{H}_2 \iso T^*\C P^1.$ So, the cohomology of its hypersurface 
	$B_{1/2}=\mathcal{S}_{1/2} \cap \mathcal{H}_2$ is
	isomorphic to the cohomology of the unit bundle $S(T^*\C P^1)\iso \R P^3,$ by  \cref{CohomologicalIndependenceOnHypersurface}.  As we work over the field $\k$ of characteristic zero,
	$$H^*(B_{1/2})\iso H^*(\R P^3)=\k_0\oplus \k_{3}.$$
	The cohomology of $\SSS_{22}$ is isomorphic to the cohomology of its core, the Springer fibre $\B^{22}.$ Vector space generators of $H^*(\B^{\lambda})$ are labelled by row-standard tableaux of shape $\lambda;$ the degrees of generators can be easily computed by tableaux combinatorics
	\cite[Sec.1.3]{Fr09a}. %
	Using that prescribed method, we get:
	$$H^*(\SSS_{22})\iso H^*(\B^{22})\iso \k_0\oplus \k_2^3 \oplus \k_4^2.$$
	Thus, from \cref{CalculationOfTheSliceCohomology} and \cref{Mid Coh of Hypersurface Vanish} one obtains the cohomology of the slice. 
	We now find the degree shifts.	
	As the weight decomposition of $\F_{mpx}$ has $H_0\oplus H_2$ as a summand, by $\om_\C$-duality it also has $H_{1-2}=H_{-1}$ and $H_{1-0}=H_1$ (recall we square-rooted the action so $\om_\C$ has weight 1). Thus we have
	$T_{\F_{mpx}}\SSS_{22}=H_{-1} \oplus H_0\oplus H_1 \oplus H_2,$ and then by (\ref{CZindecesOfMorseBottSumbanifolds}) we find
	$\mu(B_{1/2})=0.$ 
 By \cref{Killing the unit in sp seq CSR}\eqref{killing the unit wt1 case}, 
the unit is killed by the top class of the time-2 column.
 We obtain \cref{S22_sp_seq}, where we crop 
            at the mid-dimensional cohomology of the time-$2$ slice.\footnote{since the rest is known by Poincaré duality.}
            Thus, the rank-wise filtration on $H^*(\SSS_{22})$ is 
	\begin{equation}\label{KazhdanFiltrS22Longer}
	    0\subset \k_4 \subset \k_2^3\oplus  \k_4^2  \subset \k_0 \oplus \k_2^3\oplus  \k_4^2 = H^*(\SSS_{22}).
	\end{equation}
	\begin{figure}[H]%
		\centering
		{
			\includegraphics[scale=0.85]{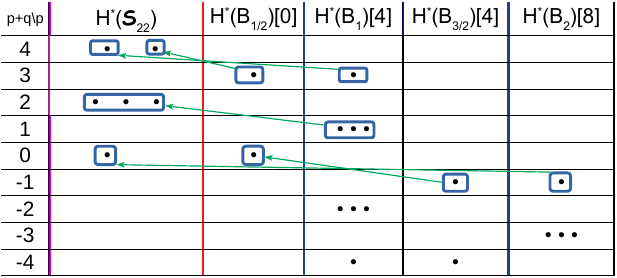}
			\caption{Spectral sequence for $\mathcal{S}_{22}$}
			\label{S22_sp_seq}
		}
	\end{figure}
Using \eqref{Cor intro about Fmin surviving}, $\FF_{1/2}=H^2(\F_{mpx})[2]$ as 
we have only two fixed components by \cref{CSRproperties}\eqref{Thm_CSR_weight_1_number_of_Fa_equals_Htop}, the other is 
$\F_{min}$.
In the {\AB} decomposition 
$H^*(\SSS_{22})=H^*(\F_{\min})\oplus H^*(\F_{mpx})[-2]$,
the class $H^2(\F_{mpx})[-2]=[\overline{\L_{mpx}}]$ corresponds to the irreducible 
core component $\overline{\L_{mpx}}$: the closure of the downward flow from
$\F_{mpx}$. %
In Springer-theoretic notation, this component is labelled $\overline{\L_{mpx}}=\mathcal{K}_{T}$ by the standard Young tableaux 
$T= \begin{ytableau} 1 & 2 \\ 3 & 4 \end{ytableau}$,
and $\F_{min}=\mathcal{K}_{T'}$ is  $T'=\begin{ytableau} 1 & 3 \\ 2 & 4 \end{ytableau}$.\footnote{See the labelling definition in \cite{Spa76} and the general description of $\F_{min}$ in \cite[Prop.5.3.12]{FZ20}. Explicitly, %
$\F_{\min}=(0 \subset\langle \a v_1+\b v_3 \rangle\subset\langle v_1,v_3 \rangle \subset \langle v_1,v_3,\gamma v_2 + \delta v_4 \rangle\subset \C^4 \mid [\a:\b],[\gamma:\delta]\in \P^1)$; the recipe from \cite{Spa76} yields the label $T'$.}
We get the complete description: 
\begin{equation}\label{KazhdanFiltrS22Longer_complete}
	    0\subset \FF_{1/2}=[\mathcal{K}_{T}] \subset \FF_{1}=\k_2^3\oplus  \k_4^2  \subset \FF_{2}=\k_0 \oplus \k_2^3\oplus  \k_4^2 = H^*(\SSS_{22}).
\end{equation}
\end{ex}
		
\begin{ex}\label{Example Slodowy variety S32}\textbf{Slodowy variety $\mathcal{S}_{32}.$}		
For $\lambda=(3,2)$, $n=5,$ as before one computes the
slice $S_{32}:$ 
			\[ \small
			S_{32}=\left\{ \begin{bmatrix}
				2a&1&0&0&0\\
				f&2a&1&j&0\\
				l&f&2a&p&2j\\
				2y&0&0&-3a&1\\
				w&y&0&x&-3a
			\end{bmatrix}  \middle| \ a,f,j,l,p,x,y,w\in \C \right\}.
			\]
			and $h_{32}=\diag(2,0,-2,1,-1).$ 
			As $t^{h_{32}} =\diag(t^2,t^{0},t^{-2},t^{1},t^{-1}),$ the Kazhdan action on $S_{32}$ by (\ref{weightOfKAzdhanActionExamples}) is
			$$\small t \cdot \begin{bmatrix}
				2a&1&0&0&0\\
				f&2a&1&j&0\\
				l&f&2a&p&2j\\
				2y&0&0&-3a&1\\
				w&y&0&x&-3a
			\end{bmatrix}= \begin{bmatrix}
				2t^2a&1&0&0&0\\
				t^4f&2t^2a&1&t^3j&0\\
				t^6l&t^4f&2t^2a&t^5p&2t^3j\\
				2t^3y&0&0&-3t^2a&1\\
				t^5w&t^3y&0&t^4x&-3t^2a
			\end{bmatrix}.$$
			Thus, we see that $S_{32}$ has $\Z/k$-torsion points for all $k=2,3,4,5,6.$ Restricting to
			$\SS_{32}=\{A\in S_{32} \mid A^5=0 \},$ 
   some matrix calculations
    by computer show that only $\Z/3$ and $\Z/5$-torsion points remain on it. Explicitly (where in lower indices we label the coordinates used on $\C^2$ and $\C$):\vspace{-3mm}
\begin{multicols}{2} \small
\noindent
  \begin{equation}\small
    \label{Z3IsotropiesS32}
				\PZ{3}=\left\{A_{jy}:=
				\begin{bmatrix}
					0&1&0&0&0\\
					0&0&1&j&0\\
					-5jy&0&0&0&2j\\
					2y&0&0&0&1\\
					0&y&0&0&0
				\end{bmatrix}  \middle|\ j,y\in \C \right\} \iso \C^2_{j,y},
 \end{equation} %
 \begin{equation}\small
    \label{Z5IsotropiesS32}
				\PZ{5}=\left\{A_{pw}:=
				\begin{bmatrix}
					0&1&0&0&0\\
					0&0&1&0&0\\
					0&0&0&p&0\\
					0&0&0&0&1\\
					w&0&0&0&0 
				\end{bmatrix} \middle|\ pw=0 \right\} \iso \C_{p}\cup \C_{w},
   \end{equation}%
\end{multicols}%
\noindent{\textbf{$\Z/5$ torsion points:}}
			Apart from the zero matrix, $\PZ{5}$ lies in the regular orbit $\mathcal{O}_5.$ Therefore it is bijectively lifted to 
			$\SSS_{32}$ under the Springer resolution, to pairs
			$(A_{p0},\f_p)$ and $(A_{0w},\f_w)$ where 
			{\small $$\f_p:=(0\subset \langle v_1  \rangle \subset \langle v_1,v_2 \rangle \subset \langle v_1,v_2,v_3 \rangle \subset \langle v_1,v_2,v_3,v_4 \rangle \subset  \C^5)$$
			$$\f_w:=(0 \subset \langle v_4 \rangle \subset \langle v_4,v_5 \rangle \subset \langle v_1,v_4,v_5 \rangle \subset \langle v_1,v_2,v_4,v_5 \rangle \subset  \C^5)$$}
			as one can easily compute from the Springer fibre conditions $A_{0w}\f_w \subset \f_w, \ A_{p0}\f_p \subset \f_p.$
			Here, $(v_1,v_2,v_3,v_4)$ is the basis in which the matrices are given.
			In the limit $t\fun0$ of the Kazhdan action, 
			$$\lim_{t\fun 0} t\cdot (A_{p0},\f_p)=(e_{32},\f_p), \ \ \ \lim_{t\fun 0} t\cdot (A_{0w},\f_w)=(e_{32},\f_w)$$
			yielding two $\Z/5$-torsion lines converging to the points $\f_p$, $\f_w$ in the Springer fibre $\B^{32}.$ Thus $$B_{1/5}=B_{1/5,1} \sqcup B_{1/5,2}\iso S^1 \sqcup S^1$$
            are their {\MB } submanifolds.
			We cannot compute the shifts $\mu(B_{1/5, i})$ yet, as we lack information on weight decompositions at
			$\FF_p$, $\FF_w$ in the core $\B^{32}.$ We do this after considering $\Z/3$-torsion points.\\[2mm]
\noindent{\textbf{$\Z/3$ torsion points:}}
   By checking for which $j,y\in\C$ we have $A_{jy}^4=0,$ we find that in $\PZ{3}$ the two lines $j=0$ and $y=0$ lie on the subregular orbit $\O_{41},$ whereas the rest lie in $\O_5.$
			Thus, for $j,y\neq 0$ the fibre above $A_{jy}$ is $(A_{jy},\f_{jy})$ where, after some computation, 
			$$\f_{jy}=(0 \subset \langle j v_3-v_4 \rangle \subset \langle jv_3-v_4, v_1 + y v_5 \rangle \subset \langle jv_3-v_4, v_1 + y v_5, v_2 
			\rangle \subset \langle v_2,v_3,v_4, v_1+y v_5 \rangle \subset \C^5).$$ This flag also makes sense for $j=0$ or $y=0,$ and due to continuity, we get a slice $$S_{jy}=\{(A_{jy},\f_{jy})\mid j,y\in\C\}\iso \C^2.$$
			
			By (\ref{KazdhanActionOnBasisExamples}), the Kazhdan action acts on $\C^n$ by 
			$t \cdot(v_1,v_2,v_3,v_4,v_5)=(t^{-2}v_1, v_2,t^{2}v_3,t^{-1} v_4, t v_5),$ 
			thus one immediately gets
			$\displaystyle \lim_{t\fun 0} t\cdot (A_{jy},\f_{jy})=(e_{32},\f_{big}),$ where 
			 $$\small \f_{big}:=(0 \subset  \langle v_4 \rangle \subset \langle v_1,v_4 \rangle \subset \langle v_1,v_2,v_4 \rangle \subset \langle v_1,v_2,v_3,v_4 \rangle \subset  \C^5).$$
			Thus the whole slice $S_{jy}$ converges to the flag $\f_{big},$ yielding a trivial rank 2 torsion bundle $\mathcal{H}_3 \fun \f_{big}.$ The corresponding {\MB } submanifold $B_{1/3,big}$ is cohomologically isomorphic to the sphere bundle 
			\begin{equation}
				H^*(B_{1/3,big})\iso H^*(S(\mathcal{H}_3))\iso H^*(S^3)=\k_0\oplus \k_{3},
			\end{equation}
			and knowing the weight decomposition $H_3 \oplus H_3 \oplus H_{-1} \oplus H_{-1}$ of $\f_{big}$ we also compute $\mu(B_{1/3,big})=0.$
			
			Now let us consider the $\Z/3$-torsion points in the fibre over $A_{j0},$ where $j\neq 0.$ The Springer fibre over $A_{j0}$ is a Dynkin $A_4$-tree of spheres $\P_1^1 \cup \P_2^1 \cup \P_3^1 \cup \P_4^1$  that intersect transversally. Standard computations from Springer theory yield their explicit flag description:\footnote{In short: One modifies the base $v_i$ such that it becomes a Jordan base for $A_{j0}$ and then one uses the description of the Springer fibre of the standard matrix $e_{41}.$}
			{\small \begin{align*} 
				\P_1^1[\a:\b]=& (0 \subset  \langle \a v_1 + \b (\tfrac{1}{j} v_4 - v_3) \rangle \subset \langle v_1, \tfrac{1}{j} v_4 - v_3 \rangle \subset \langle v_1, v_2, \tfrac{1}{j} v_4 - v_3 \rangle \subset \langle v_1,v_2,v_3,v_4 \rangle \subset  \C^5 )\\
				\P_2^1[\a:\b]=& (0  \subset \langle v_1 \rangle \subset \langle v_1, \a v_2 + \b (\tfrac{1}{j} v_4 - v_3) \rangle \subset \langle v_1,v_2, \tfrac{1}{j} v_4 - v_3 \rangle \subset \langle v_1,v_2,v_3,v_4 \rangle \subset  \C^5 )\\
				\P_3^1[\a:\b]=& (0  \subset \langle v_1 \rangle \subset \langle v_1, v_2 \rangle \subset \langle v_1,v_2, \a v_3 + \b \tfrac{1}{j} v_4  \rangle \subset \langle v_1,v_2,v_3,v_4 \rangle \subset  \C^5 )\\
				\P_4^1[\a:\b]=& (0  \subset \langle v_1 \rangle \subset \langle v_1, v_2 \rangle \subset \langle v_1,v_2, 2 v_3 + \tfrac{1}{j} v_4  \rangle
                \subset \langle 
				v_1,v_2,2v_3 + \tfrac{1}{j}v_4, \a v_3 + \b \tfrac{1}{j}  v_5 \rangle \subset  \C^5 )
			\end{align*}}
			The Kazhdan action $t \cdot v=(t^{-2}v_1, v_2,t^{2}v_3,t^{-1} v_4, t v_5),$  for a primitive third root of unity $\eps$, is
			$$\eps \cdot(v_1,v_2,v_3,v_4,v_5)=(\eps v_1, v_2, \eps^{-1} v_3,\eps^{-1} v_4, \eps v_5).$$ Notice that the $\Z/3$-torsion points in the fibre
			are exactly the flags fixed by the $\eps$ action. Thus, we get that the $\Z/3$-torsion points in the Springer fibre above $A_{j0}$ are 
			$$\P_1^1[0:1],\ \ \P_1^1[1:0],\ \  \P_3^1[\a:\b],\  \ \P_4^1[0:1],$$ so three points and a sphere.
			Note the point
			$\P_1^1[0:1]=\f_{j0}$
			is already contained in the slice $S_{jy}.$ %
			One computes directly that
			$\displaystyle \lim_{t\fun 0} t\cdot (A_{j0},\P_1^1[1:0])=(e_{32},\f_j^3),$ and $\displaystyle \lim_{t\fun 0} t\cdot (A_{j0},\P_4^1[0:1])=(e_{32},\f_j^1),$ where 
			{\small $$\f_j^3:=(0  \subset \langle v_1 \rangle \subset \langle v_1,v_4 \rangle \subset \langle v_1,v_2,v_4 \rangle \subset \langle v_1,v_2,v_3,v_4 \rangle \subset \C^5)$$
			$$ \f_j^1:=(0  \subset \langle v_1 \rangle \subset \langle v_1,v_2 \rangle \subset \langle v_1,v_2,v_4 \rangle \subset \langle v_1,v_2,v_4,v_5 \rangle \subset  \C^5).$$}
			
			So these two yield $\Z/3$-torsion line bundles over the points $\f_j^3$, $\f_j^1.$ Their {\MB } submanifolds
			$$B_{1/3,j^3} \iso B_{1/3,j^1}\iso S^1$$
             are therefore circles.
			The convergence (under Kazhdan action and $t\fun 0$) of the spheres
				\begin{equation}\small\label{spheresbreakingexample}
					\P_3^1[\a:\b]_j=(0  \subset \langle v_1 \rangle \subset \langle v_1, v_2 \rangle \subset \langle v_1,v_2, \a v_3 + \b \tfrac{1}{j} v_4  \rangle \subset \langle v_1,v_2,v_3,v_4 \rangle \subset  \C^5 \mid [\a:\b]\in \P^1)
				\end{equation}
			is a bit more involved. Namely, (\ref{spheresbreakingexample}) converges to
			{\small\begin{align*}
				\f_p=&(0  \subset  \langle v_1  \rangle \subset  \langle v_1,v_2 \rangle  \subset \langle  v_1,v_2,v_3  \rangle \subset \langle  v_1,v_2,v_3,v_4 \rangle  \subset  \C^5), \text{when }\b=0\\
				\f_j':=&(0  \subset \langle v_1 \rangle \subset \langle v_1,v_2 \rangle \subset \langle v_1,v_2,v_4 \rangle \subset \langle v_1,v_2,v_3,v_4 \rangle \subset  \C^5), \text{otherwise.}
			\end{align*}}
			Notice also that there is a $\Z/3$-torsion sphere in the core between these two points:
			$$C_{jp}=\{(0  \subset \langle v_1 \rangle \subset \langle v_1,v_2 \rangle \subset \langle v_1,v_2,\a v_3 + \b v_4 \rangle \subset \langle v_1,v_2,v_3,v_4 \rangle \subset  \C^5)\mid [\a:\b]\in\P^1 \}.$$
			Thus, altogether with $\P_3^1[\a:\b]$ we get a $\Z/3$-torsion line bundle over the (non-fixed) sphere $C_{jp}:$ 
    $$\{(A_{j0}, F) \mid j\in\C \} \mapsto F\in C_{jp}.$$
			This bundle is trivial, and by Lemma \ref{CohomologicalIndependenceOnHypersurface} the corresponding {\MB } submanifold $B_{1/3,jp}$ is cohomologically isomorphic to its 
			sphere bundle, 
			\begin{equation}
				H^*(B_{1/3,jp})\iso H^*(S^2\times S^1) \iso \k_0\oplus \k_{1} \oplus \k_{2} \oplus \k_{3}.
			\end{equation}
			The weight-decomposition of $\f_j'$ is $H_3 \oplus H_3 \oplus H_{-1} \oplus H_{-1},$ %
			thus by formula (\ref{CZindecesOfMorseBottSumbanifolds}) one computes the grading $\mu(B_{1/3,jp})=0.$
			Now we compute the gradings of the {\MB } submanifolds $B_{1/3,j^3}$ and $B_{1/3,j^1}.$ 
			We claim that the weight-decompositions of their corresponding fixed loci $\f_j^3$ and $\f_j^1$ are both 
			\begin{equation}\label{weightdecompExampl}
				H_3 \oplus H_{-1} \oplus H_1 \oplus H_1,
			\end{equation}
			by the following \textbf{accumulative} argument. Suppose there is a weight-space $H_2,$ by contradiction. It yields a $\C^*$-flowline in the core (as there are no $\Z/2$-torsion points outside of the core)
			converging to another fixed set with weight-space $H_{-2}$
   (\cref{CorollaryGradientTrajBecomeSpheres}).
   By $\om_\C$-duality, that fixed set has an $H_4$ weight-space, which yields another $\C^*$-orbit etc. As the number of components of the fixed locus is finite, the process must stop yielding a contradiction. A similar argument applies, if there were a weight-space $H_{k\geq 4}$ in $\f_j^3$ and $\f_j^1.$ Now assume  $H_3$ is 2-dimensional; then there are more outer\footnote{I.e., points that are outside of the core.} $\Z/3$-torsion points converging to $\f_j^3$ and $\f_j^1,$ again a contradiction. %
			Thus (\ref{weightdecompExampl}) holds, and by \eqref{CZindecesOfMorseBottSumbanifolds} we get 
			\begin{equation*} %
				\mu(B_{1/3,j^3})=\mu(B_{1/3,j^1})=0.
			\end{equation*}
			Completely analogously, considering $A_{0y}$ and its fibres we get (cohomologically) the same {\MB } submanifolds together with the same gradings. Also, we get analogous fixed points $\f_{y}^3,\f_{y}^1,\f_{y}'.$\footnote{Explicitly: 
   $\f_y^3=(0  \subset \langle v_4 \rangle \subset \langle v_1,v_4 \rangle \subset \langle v_1,v_2,v_4 \rangle \subset \langle v_1,v_2,v_4,v_5 \rangle \subset \C^5)$,
   $\f_{y}^1=(0  \subset \langle v_1 \rangle \subset \langle v_1,v_4 \rangle \subset \langle v_1,v_4,v_5 \rangle \subset \langle v_1,v_2,v_4,v_5 \rangle \subset \C^5)$,
   $\f_{y}'=(0  \subset \langle v_4 \rangle \subset \langle v_1,v_4 \rangle \subset \langle v_1,v_4,v_5 \rangle \subset \langle v_1,v_2,v_4,v_5 \rangle \subset \C^5)$.
   }
			
			Similarly to the last example, by counting the row-standard Young tableaux of shape $(3,2)$ with grading defined in \cite[Sec.1.3]{Fr09a} %
			we get the ordinary cohomology
			$$H^*(\SSS_{32})\iso H^*(\B^{32})\iso \k_0\oplus \k_2^4 \oplus \k_4^5.$$
			Observe Figure \ref{S32_sp_seq}. 
   At time-$\frac{2}{3}$ column, the four copies of $S^1$ have shift 2 whereas $S^3$ and the two $S^2\times S^1$ have shift 4, hence we use (as before) the asterisk notation %
   to denote mixed shifts.
By \cref{Killing the unit in sp seq CSR}\eqref{killing the unit wt2 Fmin=pt case}, 
$1\in H^0(\SSS_{32})$ is killed by the top class of the time-1 column.   
   The %
   filtration is %
\begin{equation}\label{Filtration on S_32 via sp seq}
    0\ \subset\ \FF_{1/5}=\k_4^2\ \subset\ \FF_{1/3}=\k_2^r\oplus  \k_4^5 
    \subset  \FF_{2/5}=\k_2^4\oplus  \k_4^5  
    \subset\ \FF_1=\k_0 \oplus \k_2^4\oplus  \k_4^5 = H^*(\SSS_{32}),
\end{equation}			
where $r\in\{2,3,4\}$ remains unknown, due to potential vertical differentials in column $B_{1/3}$.
   \begin{figure}[H]%
				\centering
				{
					\includegraphics[scale=0.65]{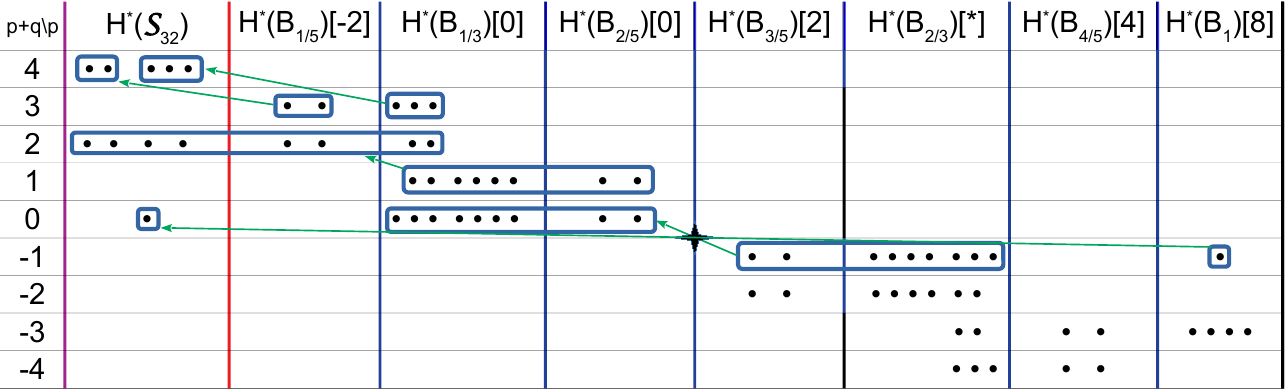}
					\caption{Spectral sequence for $\mathcal{S}_{32}$}
					\label{S32_sp_seq}
				}
			\end{figure}
			
  \begin{figure}[h]%
				\centering
				{
					\includegraphics[scale=0.65]{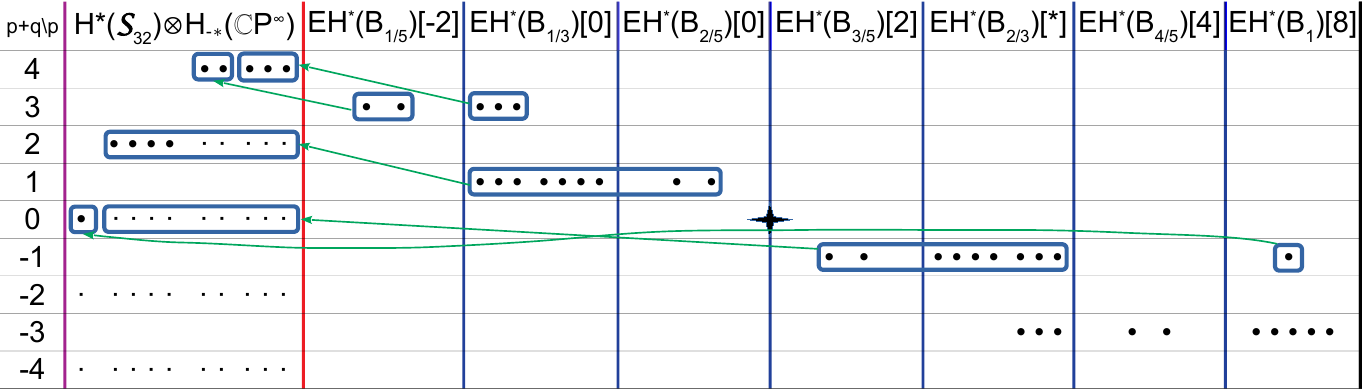}
					\caption{$S^1$-Equivariant spectral sequence for $\mathcal{S}_{32}$}
					\label{S32_sp_seq_eq}
				}
			\end{figure}
\begin{ex}\label{Example Slodowy variety S32 equivariant}\textbf{Slodowy variety $\mathcal{S}_{32},$ using the $S^1$-equivariant spectral sequence}.

\cref{S32_sp_seq_eq} shows the $S^1$-equivariant spectral sequence for $ESH^*(Y,\Fi)$. The $0$-th column $H^*(Y;\k)\otimes_{\k} \mathbb{F}$ consists of copies of the $\k [\![u]\!]$-module $\mathbb{F}=\k(\!(u)\!)/u\k [\![u]\!]\cong H_{-*}(\C\P^{\infty})$ mentioned in \cref{Subsection intro S1 spectral seq}: the larger dots correspond to the $u^0$-part, and the smaller dots correspond to the $u^{-j}$-copies of that generator for $j\geq 1$. 
This illustrates \cref{Cor eq sp seq collapses intro}, and we see explicitly how \cref{Cor big S1 equi formula intro} arises.

The edge-differentials preserve the $u$-action, but unfortunately the $u$-action has been mostly forgotten once we pass to the $E_1$-page. 
The first three dots in degree $1$ for the $B_{1/3}$-column have a non-trivial vertical $u$-action (see \cref{Theorem intro spectral seq}), so they must hit $u^{<0}$-classes in the $0$-th column. However, the other four dots in degree $1$  for the $B_{1/3}$-column may have non-trivial $u$-action hitting the $B_{1/5}$-column, so in principle we only know that at least two of the first four dots in the $0$-th column in degree 2 (which are $u^0$-terms) get killed. This ambiguity is precisely the unknown value $r\in \{2,3,4\}$ in \eqref{Filtration on S_32 via sp seq}.
\end{ex}

\begin{rmk}
\cref{CorSpectralSeqAgree} and
\eqref{pureHam} implies information about $H^*(\F_\a)$ from the spectral sequence. The $0$-th column in \cref{S32_sp_seq} has total rank  $10$: this total rank must persist as we add new columns (taking into account all green arrows). By \eqref{pureHam}, the rank distributes over degrees that shift in complicated ways as we increase the slope. These shifts reveal information about individual $H^*(\F_\a)$. 
In Figure \ref{S32_sp_seq}, consider all columns up to $B_{2/5}$: the rank $10$ is all in degree $0$, so \eqref{pureHam} implies each $H^*(\F_\a)$ is supported in just one degree so $\F_\a$ must be a point. So $\F={\SSS_{32}}^{\C^*}=\{10 \textrm{ points}\}$ (cf.\,\cref{Lemma Kazhdan = Morse when lambda equal to odd and even}).
\end{rmk}
\end{ex}
\begin{figure}[H]%
		\centering
		{
			\includegraphics[scale=0.779]{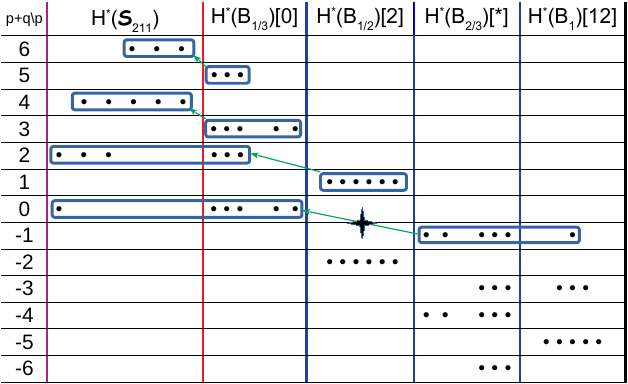}\; \includegraphics[scale=0.779]{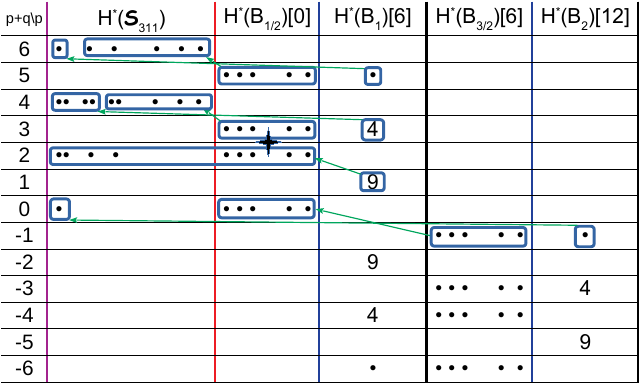}
			\caption{Spectral sequences for $\mathcal{S}_{211}$ and $\mathcal{S}_{311}$}
			\label{S211_sp_seq}\label{S311_sp_seq}
		}
	\end{figure}
	
\begin{figure}[H]%
		\centering
		{
			\includegraphics[scale=0.65]{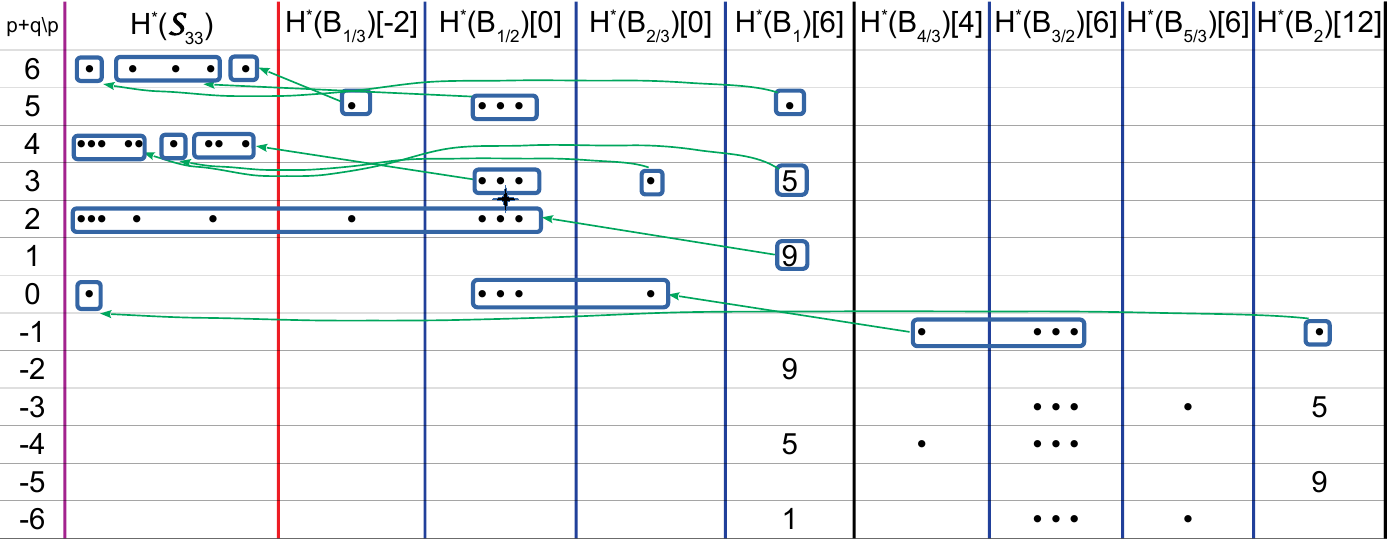}
			\caption{Spectral sequence for $\mathcal{S}_{33}$}
			\label{S33_sp_seq}
		}
	\end{figure}
	
\begin{figure}[H]%
		\centering
		{
			\includegraphics[scale=0.65]{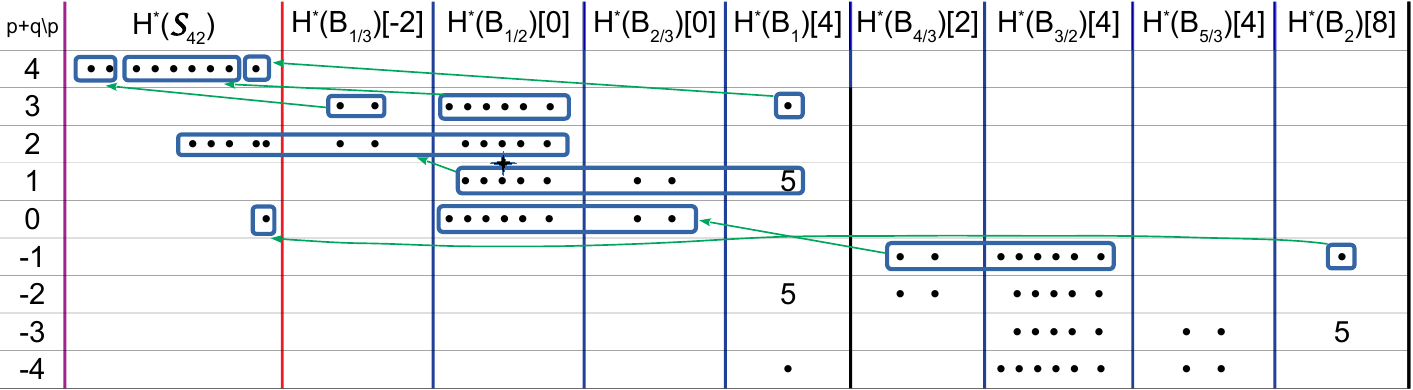}
			\caption{Spectral sequence for $\mathcal{S}_{42}$}
			\label{S42_sp_seq}
		}
	\end{figure}
\begin{ex}  \textbf{Slodowy varieties $\mathcal{S}_{211},\mathcal{S}_{311},\mathcal{S}_{33},\mathcal{S}_{42}.$}

	We state the spectral sequences for 
	$\mathcal{S}_{211}$ and $\mathcal{S}_{311}$, 
    $\mathcal{S}_{33}$, $\mathcal{S}_{42}$ in Figures \ref{S211_sp_seq}--\ref{S42_sp_seq},
	obtained analogously to previous examples. 
 We abbreviate cohomology ranks of the $B_1,B_2$ columns in some entries by a number rather than by dots. One can read-off filtrations on ordinary cohomology rank-wise, as before.
\end{ex}

\subsection{Moduli of Higgs bundles}
The last class of examples that we consider are moduli spaces of Higgs bundles. 
They %
come 
with a {\hk} structure $(\MM,I,J,K,g)$
and a canonical $I$-holomorphic $\C^*$-action $\Fi$
that acts with weight-1 on the holomorphic symplectic form 
\begin{equation}\label{weight-1-condition}
   t\cdot \om_\C=t \om_\C, \ \om_\C:=\om_I+ i \om_J 
\end{equation}
and the $S^1$-part preserves the {\Kh} form $\om_I.$
There is a proper holomorphic map, the Hitchin fibration,
$$\Psi: \MM\to B\iso \C^{\frac{1}{2} \dim_\C\MM},$$
equivariant w.r.t.\;a certain linear %
action on the base.
$(\MM,\om_I)$ is a symplectic $\C^*$-manifold globally defined over the convex base $B$ via $\Psi.$
Given $G\in \{GL(n,\C),\ SL(n,\C)\},$ $d \in \N_0$ coprime to $n,$ 
consider $\MM_G(d,g),$ %
the moduli of
$G$-Higgs bundles of degree $d$ over a Riemann surface %
of genus $g \geq 2.$ 
\begin{prop}\cite{RZ1}\label{Higgs_moduli_Q_Fi=0}
    $Q_\Fi=0$ for $\MM_G(d,g),$ %
    so $\FF_1^{\Fi}=H^*(\MM_G(d,g)).$ 
\end{prop}

\begin{ex} \textbf{$G=GL(1,\C),$ $\MM_G(d,g) \iso T^* T^{2g}.$} \label{Example_HiggsGL(1)}

For $G=GL(1,\C),$ 
the moduli space is the cotangent bundle of the torus $T^{2g}$ (for $SL(1,\C)$ it is a point), with the standard weight-1 $\C^*$-action on fibres.
The Maslov index is $\mu= \frac{1}{2} \dim_\C \MM =g,$ 
and the columns in the spectral sequence (after the 0-column) %
are cohomologies of the sphere bundle 
$H^*(B_k)=H^*(S^*T^{2g})$ shifted by 
$2g, 4g, 6g,\dots.$
Due to the Gysin sequence and $\chi(T^{2g})=0,$ we get 
$H^*(B_k)\iso H^*(T^{2g}) \otimes H^*(S^{2g-1}).$
By \cref{Higgs_moduli_Q_Fi=0}, the filtration is trivial, 
$$\FF^{\Fi}_0=0 \subset H^*(\MM_{GL(1,\C)}(d,g))=\FF^{\Fi}_1.$$ Interestingly, one \textit{cannot} conclude this just by looking at the spectral sequence itself
(see \cref{Higgs_GL_SL}, left).
The unit $1\in H^0(\MM)$ could in principle be killed by $H^{top}(B_2),$ but the triviality of the intersection form
(used in the proof of \cref{Higgs_moduli_Q_Fi=0}) 
prevents this.
Thus the unit gets killed by a 
mid-degree class from the time-1 column (lower yellow box in \cref{Higgs_GL_SL}).
This is very different from CSRs,
whose intersection form is non-degenerate, and mid-degrees in the time-1 column vanish, so differentials from $H^{top}(B_2)$ can hit the unit
$1\in H^0(\MM)$ (e.g.\;\cref{SS_S11}).
\end{ex}

\begin{figure}[ht]
		\centering
		{
                \includegraphics[scale=0.9]{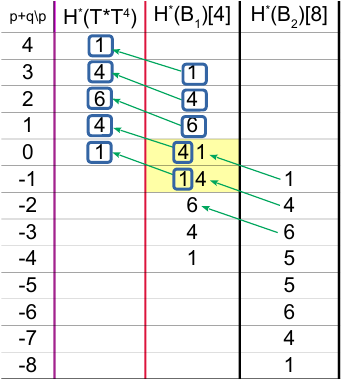}\; \includegraphics[scale=0.9]{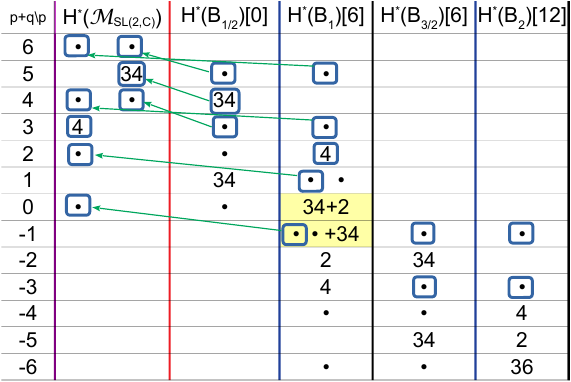}
			\caption{Spectral sequence for $GL(1,\C)$ and $SL(2,\C)$ Higgs moduli over $g=2$}
			\label{Higgs_GL_SL}
		}
\end{figure}
\begin{ex}\textbf{$G=SL(2,\C),$ $g=2.$}

The topology of Higgs moduli for $G=SL(2,\C)$ are
well-understood. The $\C^*$-fixed locus is completely described \cite[Prop.7.1]{Hi87}. For $g=2$
we have two fixed components,
$\F_0=\F_{\min}$ and $\F_1$,
$$H^*(\F_0)\cong \k_0 \oplus \k_2 \oplus \k_3^4 \oplus \k_4 \oplus \k_6, \ \ \ \ H^*(\F_1)\cong\k_0 \oplus \k_1^{34} \oplus \k_2,$$
with {\MB} indices $\mu_0=0$, $\mu_1=4.$ %
The attraction graph is quite simple: by 
$\om_\C$-duality, there are only $H_0$ and $H_1$ weight spaces at $\F_0,$ so 
there is a weight-1 flow from $\F_0$ to $\F_1,$ creating a $H_{-1}$ weight space
at $\F_1.$
Due to the {\MB }
index, $H_{-1}$ is 2-dimensional.
By $\om_\C$-duality, %
there is 2-dimensional $H_2,$
creating a torsion bundle $\mathcal{H}_2 \fun \F_1$ of rank 2. %
Thus, the associated {\MB } manifold has
$$H^*(B_{1/2})\iso H^*(S\mathcal{H}_2)\iso H^*(S^3)\otimes H^*(\F_1) \cong \k_0 \oplus \k_1^{34} \oplus \k_2 \oplus \k_3 \oplus \k_4^{34} \oplus \k_5.$$ 
The time-1 column is computed via \cref{CalculationOfTheSliceCohomology} and the fact that the intersection form is trivial, by Hausel \cite{hausel_1998}.
The spectral sequence (see \cref{Higgs_GL_SL}, right) yields the filtration
$$0 \subset \k_2^r \oplus \k_3^s \oplus \k_4 \oplus \k_5^{34} \oplus \k_6  \subset H^*(\MM_{SL(2,\C)}(d,2))$$
where $(r,s) \in \{0,1\} \times \{0,1\}$ remains unknown.
Notice that the minimal component $\F_0$ has only weights 0 and 1, due to weight-1 condition \eqref{weight-1-condition} and $\om_\C$-duality \eqref{NonDegenPairingByOmC}.\footnote{Although we have given it as a CSR property, this holds in any holomorphic symplectic setup with weight-1 action.} However, due to $H^{odd}(\F_0) \neq 0,$ we cannot use
\cref{Cor intro about Fmin surviving}, which would have given us $\FF^{\Fi}_{1/2} \subset H^*(\F_1)[-4],$ hence $r=s=0.$ 
\end{ex}
\begin{ex} \textbf{Parabolic Higgs bundles of $\dim_\C=2.$}

One can also consider moduli of so-called \textit{parabolic} Higgs bundles, 
where we add marked points to the Riemann surface where poles of the Higgs field of certain types are allowed.
All parabolic Higgs moduli in the lowest complex dimension $2$ are given as crepant resolutions 
\begin{equation}\label{Resolution parabolic Higgs moduli Toy Example}
    \pi: \MM_\Gamma \fun (T^*E)/\Gamma,
\end{equation}
where $E$ is an elliptic curve, and
$\Gamma$ is a finite group of automorphisms of $E$, 
\cite{groechenig2014hilbert,zhang2017multiplicativity}.
Composing $\pi$ and the projection to $\C$ coming from the isomorphism 
 $T^*E \iso E \times \C$, 
 yields the Hitchin fibration \begin{equation}\label{Hitchin fibration for toy examples}
  \Psi: \MM_\Gamma \fun \C.  
\end{equation}
The core of $\MM_\Gamma$ is the central fibre $\Psi^{-1}(0)$ whose irreducible components are curves intersecting according to affine Dynkin graphs
$Q_\Gamma=\widetilde{A_0},\ \widetilde{D_4},\ \widetilde{E_6},\ \widetilde{E_7},\ \widetilde{E_8},$
when the group $\Gamma = {0},\ \Z/2,\ \Z/3,\ \Z/4,\ \Z/6,$ respectively.
In particular, the case $\Gamma=\{0\}$ gives us $\MM_\Gamma = T^*E;$ in other cases all curves are $\CP^1$s. 

Similarly to Examples \ref{Dn via extended attraction graph}, \ref{Example E_678}, we %
get the spectral sequence via the attraction graph, which for the $\Gamma=\Z/6$ case are given in Figures \ref{E_8_affine_attraction_graph} and
\ref{E_8_affine_sp_seq}.
The $\CP^1$ corresponding to 
the (3+)-valent 
vertex of the graph
is fixed,\footnote{Being a sphere with $(3+)$-fixed points.}
thus it is precisely the minimal component $\F_{\min}$ \cite{FZ20}
of the \CC-action.
  \begin{figure}[H]%
				\centering
				{
					\includegraphics[scale=0.20]{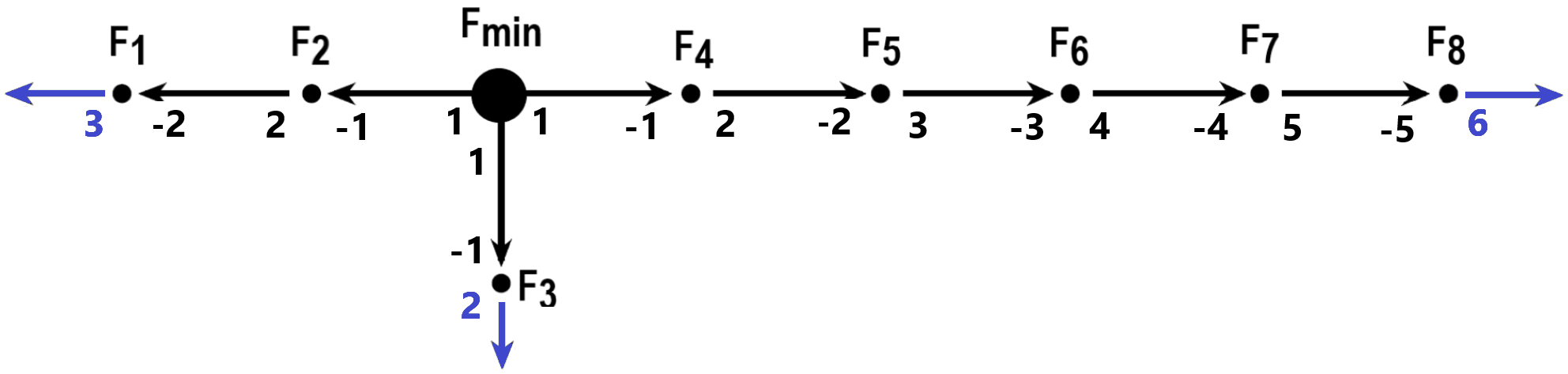}
					\caption{Extended attraction graph for $\MM_{\Z/6}$} %
					\label{E_8_affine_attraction_graph}
				}
\end{figure}

\begin{figure}[H]%
				\centering
				{
					\includegraphics[scale=0.6]{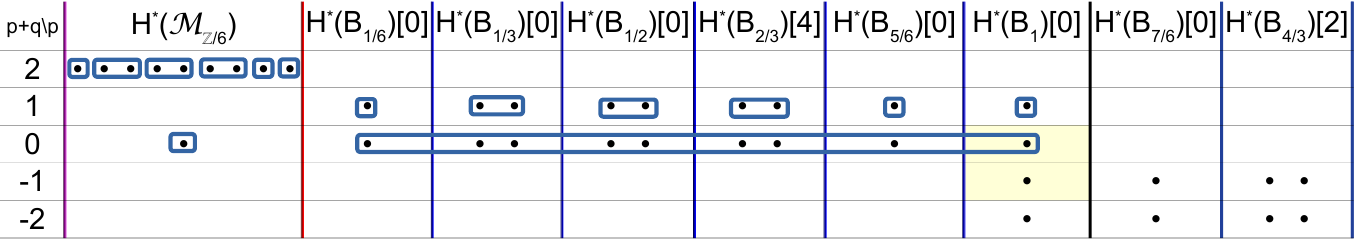}
					\caption{Spectral sequence for $\MM_{\Z/6}$   ($Q=\widetilde{E_8}$), omitting edge-differentials}
					\label{E_8_affine_sp_seq}
				}
\end{figure}
An interesting new phenomenon
is that 
$\FF^{\Fi}_1 \neq H^*(\MM_\Gamma),$\footnote{Compare with \cref{Higgs_moduli_Q_Fi=0}.} 
whereas the spectral sequence alone cannot prove it: the
mid-cohomologies 
of the time-1 column (lower light-yellow box in \cref{E_8_affine_sp_seq}) could potentially kill the unit.
The first observation is due to 
$\F_{\min}\iso \P^1$ having non-zero Euler characteristic and the result \cite[Cor.7.28]{RZ1} from our previous paper.\footnote{That statement is about CSRs but the essential condition is that of a holomorphic symplectic manifold with a weight-1 holomorphic \CC-action on $\om_\C,$ which makes $\F_{\min}$ a Lagrangian submanifold.}
Each mid-cohomology of the time-1 column has rank 1 due to \cref{CalculationOfTheSliceCohomology} (and \cref{Equation intersection pairing} in it), 
as 
$H^1(\MM_\Gamma)=0$ and the intersection form matrix on $H_2(\MM_\Gamma) \iso H_2(\Psi^{-1}(0))$ equals minus the Cartan matrix\footnote{The Cartan matrix of a graph $Q$ is $C_Q=2I-A_{Q}$, where $A_{Q}$ is the adjacency matrix of $Q$.} of an affine Dynkin graph $Q_\Gamma$, thus
has 1-dimensional kernel.  %
From the spectral sequences, we obtain %
$\FF^{\Fi}_\lambda$ for $\MM_\Gamma:$
\begin{align}\label{Filtrations_Toy_examples_Higgs}\small
\begin{split}
\widetilde{A_0}: & \;0\subset \k_0 \oplus \k_1^2 \oplus \k_2=H^*(\MM_{\Gamma}).\\
\widetilde{D_4}: & \;0\subset \k_2^4 \subset \k_2^5 \subset \k_0\oplus\k_2^5=H^*(\MM_{\Gamma}).\\
\widetilde{E_6}: & \;0\subset \k_2^3 \subset \k_2^6 \subset \k_2^7 \subset \k_0\oplus\k_2^7=H^*(\MM_{\Gamma}).\\
\widetilde{E_7}: & \;0\subset \k_2^2 \subset \k_2^5 \subset \k_2^7 \subset \k_2^8
	\subset \k_0\oplus\k_2^8=H^*(\MM_{\Gamma}).\\
\widetilde{E_8}: & \;0\subset \k_2 \subset \k_2^3 \subset \k_2^5
	\subset \k_2^7 \subset \k_2^8 \subset \k_2^9
	\subset \k_0\oplus\k_2^9=H^*(\MM_{\Gamma}).
 \end{split}
\end{align}
Interestingly, the filtrations obtained on the $H^2(\MM_\Gamma)$ 
have representation-theoretic meaning: we thank Alexandre Minets for pointing it out. 
Consider the root system of 
the affine Dynkin graph $Q:=Q_\Gamma$,
$$R_Q=\{ \a_i \mid i\in Q^0\} \sqcup \la \Z-{0} \ra \delta,$$
where $Q^0$ are the vertices of $Q$.
It contains simple roots $\a_i$, labelled by the vertices, and 
multiples 
of
\begin{equation}\label{imaginary root}
\delta=\sum_{i\in {Q^0}} n_i \a_i, 
\end{equation}
 the \textit{imaginary root} that generates the kernel of the corresponding Cartan matrix $C_Q.$ 
An example of these numbers $n_i$ for the graph $\widetilde{E_8}$ is given in \cref{E_8_imaginary_root}, and the others can be found in \cite[Fig.1.3]{Kir16}.
  \begin{figure}[H]%
				\centering
				{
					\includegraphics[scale=0.15]{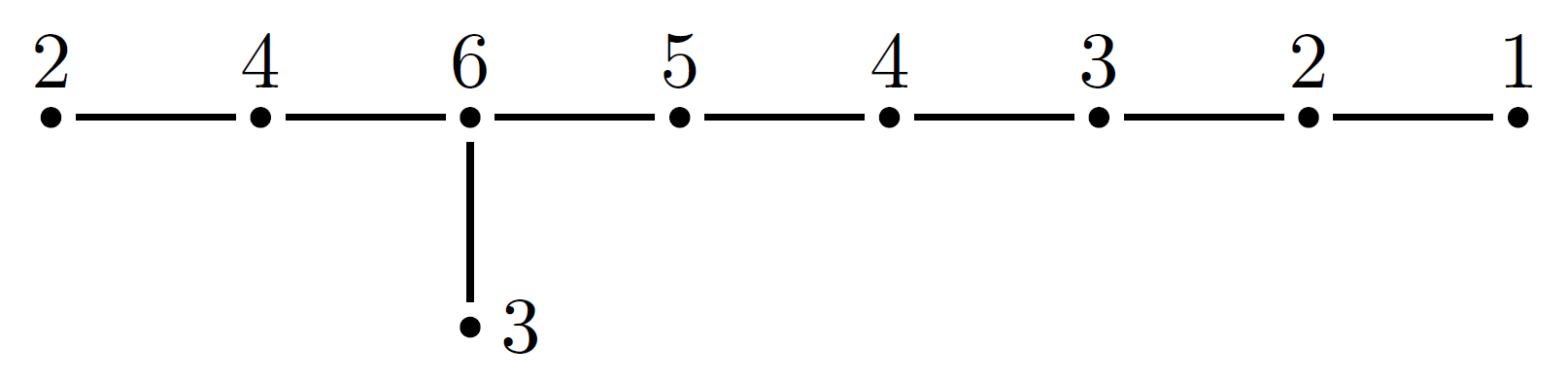}
					\caption{$\widetilde{E_8}$ with imaginary root labelling}
					\label{E_8_imaginary_root}
				}
			\end{figure}
The irreducible components of the core of $\MM_\Gamma$ form the base for its top-cohomology $H^2(\Core(\MM_\Gamma))\iso H^2(\MM_\Gamma),$
and are labeled by $Q^0,$ thus we have
$H^2(\MM_\Gamma)=\k \la Q^0\ra.$ 
By direct comparison of \eqref{Filtrations_Toy_examples_Higgs} with the numbers $n_i$ from \eqref{imaginary root} we get (where $|V|$ denotes the dimension over $\k$):
\begin{prop} Denoting by $\FF_k$ the $k$-th step of the filtration $\FF^{\Fi}_{\lambda}$ of $H^2(\MM_\Gamma),$ we have
\par  
        \vspace{\abovedisplayshortskip}
\hfill $\displaystyle |\FF_k| = \#\{i\in Q^0 \mid n_i \leq k\}.$ \qed
\end{prop}

Let us now compare our filtration with the perverse filtration for $\Psi,$ which 
is proved in \cite{zhang2017multiplicativity} to be equal to the weight filtration of the associated character
variety $\MM_{\Gamma}',$ confirming the famous ``$P=W$'' conjecture of \cite{de2012topology} in this case.
Since the two filtrations are of different type (the perverse filtration has the unit in the first filtered level, whereas in ours it is at the last level), 
it makes sense to compare them only degree-wise, which here amounts to considering %
$H^2(\MM_\Gamma)$:
\begin{prop}
    The filtration $\FF^{\Fi}_{\lambda}(H^2(\MM_\Gamma))$ refines the perverse, thus $P=W$, filtration.
\end{prop}
\begin{proof}
In the trivial example $T^*E,$
the filtrations are equal, %
thus let us consider $\Gamma \neq \{0\}.$ 
The perverse 
filtration on $H^2(\MM_\Gamma)$
due to 
\cite[Prop.5.4]{zhang2017multiplicativity}\footnote{see also the
proof of \cite[Thm.5.9]{zhang2017multiplicativity}.}
has two steps: $$(P_0 H^2(\MM_\Gamma)=0 ) \subset (P_1 H^2(\MM_\Gamma)=\la [E_i] \mid i \ra) \subset (P_2 H^2(\MM_\Gamma)= H^2(\MM_\Gamma))$$ 
where $E_i$ are the exceptional divisors of \eqref{Resolution parabolic Higgs moduli Toy Example}, which are the components of the core
$\Psi^{-1}(0)$ except
the minimal component $\F_{\min}$ corresponding to 
the (3+)-valent vertex $c$ of Dynkin graph $Q.$ 
As $\F_{min}$ is fixed, its weight decomposition is $H_0\oplus H_1$, due to $\om_\C$-duality \eqref{NonDegenPairingByOmC} for the weight-1 action \eqref{weight-1-condition}.
Thus, by \cref{Cor intro about Fmin surviving}, the penultimate filtered level is 
\begin{equation}\label{penultimate level, Higgs example}
\FF^{\Fi}_{1^-}(H^2(\MM_\Gamma))=\oplus_i H^0(\F_i)[-2],
\end{equation}
where $\F_i$ are the other fixed points, 
labelled bijectively by $i \in Q^0\setminus \{c\}$
via their $(t\fun \infty)$-flow: $E_i=\ol{\L_i},$ where $\L_i:=\{p \mid \lim_{t\fun \infty } t\cdot p = \F_i\}.$
Since the classes in the {\AB } decomposition \eqref{penultimate level, Higgs example}
correspond to the classes $[E_i]$ of their downward flows, %
we conclude that 
$\FF^{\Fi}_{1^-}(H^2(\MM_\Gamma)=\la [E_i] \ra =P_1 H^2(\MM_\Gamma).$ %
\end{proof}
\end{ex}

\appendix
\section{{\MBF } complex and cascades}\label{AppendixCascades}
\subsection{Cascades, drops and ledges: the {\MBF } complex}\label{Subsection cascades drops and ledges}

 We refer the reader to {\BO } and Frauenfelder \cite{Bourgeois-Oancea,Bourgeois-Oancea2,Frauenfelder-cascades} for the original implementations of this construction, and also to the applications thereof in \cite[App.B]{KwonvanKoert} and \cite[App.E]{McLR18}.

The arguments in this section are quite general, so $(Y,\omega)$ will denote a symplectic manifold for which Hamiltonian Floer cohomology and continuation maps can be defined for a suitable class of ``admissible'' Hamiltonians. For example, closed symplectic manifolds, or Liouville manifolds using Hamiltonians radial at infinity, or symplectic $\C^*$-manifolds $Y$ over a convex base using admissible Hamiltonians as in \cite{RZ1}. We fix an $\omega$-compatible almost complex structure $\J$ on $Y$, and we use the Riemannian metric $\omega(\cdot,\J \cdot)$ when defining gradients and Morse flowlines.

Call $H:Y \to \R$ a \textbf{{\MBF }} Hamiltonian if it is an admissible Hamiltonian whose $1$-orbits consist of finitely many closed connected submanifolds $B_i\subset Y$ (these parametrise the initial points of the $1$-orbits) and $B_i$ is {\bf {\MB }} for $H$, which is the condition from \cref{TorsionAreMorseBottSubmanifolds}:
\begin{equation}\label{Equation {\MB }}
T_y B_i = \mathrm{Ker}\,(d\varphi^1_H - \mathrm{id})|_y \textrm{ for all }y\in B_i.
\end{equation}
Equivalently, the intersection of the graph of $\varphi^1_H$ in $(Y\times Y,\omega \oplus -\omega)$ with the diagonal is a clean intersection (of two Lagrangian submanifolds) and equals the diagonal of $B_i$.
One can also allow $H$ to be time-dependent. Given a {\MBF } Hamiltonian $H$, we choose auxiliary Morse functions $f_i:B_i \to \R$. 
Simplify notation further by denoting 
$B=\sqcup_i B_i=\mathrm{Crit}(H)$, and $f:B \to \R$ the Morse function defined by the $f_i$.
As in \cref{SubsectionHam1-orbitsOfHlambda}, a point $x_0\in B$ corresponds to the initial point $x(0)=x_0$ of a $1$-orbit $x(t)$; we call it the \textbf{associated $1$-orbit}. If $x_0\in \mathrm{Crit}(f)$ is a critical point, we call $x(t)\subset Y$ a \textbf{critical $1$-orbit}. Thus the critical $1$-orbits are an isolated subset of the space of $1$-orbits of $H$. We will abusively blur the distinction between a critical point and its associated critical $1$-orbit.

The {\MB } Floer complex $BCF^*(H)$ is freely generated over the Novikov field $\k$  by the critical $1$-orbits, equivalently (by blurring the above distinction) the generators are the subset $\mathrm{Crit}(f)\subset B$ of critical points. 
As usual, there are various possible choices of $\k$ depending on the setup (\cref{Rmk technical symplectic assumptions on Y}); this is not important in our discussion.
The differential on $BCF^*(H)$ is a count of isolated cascades, which we define below, yielding a cohomology group $BHF^*(H)$.
When we say ``count'', we always mean an oriented count, with weights depending on the choice of $\k$. Often $\k$ involves a variable $T$ to encode energy, and the ``weight'' means $T^{E(u)}$ where $E(u)$ is the energy of the cascade in \cref{Definition energy of a cascade}.

The discussion of orientation signs is rather subtle \cite[Prop.3.9]{Bourgeois-Oancea2}. {\BO } \cite[Sec.4.4]{Bourgeois-Oancea2} discussed these in detail when the {\MB } manifolds are circles. 
{\KvK } \cite[Lem.B.7]{KwonvanKoert} studied orientations in detail for the local Floer cohomology of a class of {\MB } manifolds satisfying a certain symplectic triviality condition. We suggest a general construction of orientation signs in \cref{Section {\MBF } theory: orientations} assuming only that $c_1(Y)=0$, building on Appendices \ref{AppendixCascades2}--\ref{Appendix {\MBF } theory: perturbations}.

In general, only when $c_1(Y)=0$ there is a well-defined $\Z$-grading on $BCF^*(H)$ (rather than just a $\Z/2$ grading), with $T$ lying in degree zero.
The grading of the critical $1$-orbits uses the conventions of \cite[App.C]{McLR18} explained in \cref{GradingConventions}:
for a critical $1$-orbit $x$ associated to $x_0\in \mathrm{Crit}(f_i)\subset B_i$,
\begin{equation}\label{Equation CZ shift MorseBott}
|x| := \mu_{f_i}(x_0) + \mu_{H}(B_i)
= \mu_{f_i}(x_0) + \dim_\C\, Y - \tfrac{1}{2}\dim_{\R} B_i - RS(B_i, H),
\end{equation}
where $\mu_{f_i}$ is the Morse index, and $RS(B_i, H)=RS(x,H)$ is the {\RS } index of $x$ (which only depends on the connected {\MB } manifold $B_i$). The motivation for this grading follows from the proof of \cref{Thm Floer traj near Bi new perspective}.

A \textbf{ledge} is any flowline of $-\nabla f$ inside $B$ (using the Riemannian metric on $B$ restricted from $Y$). We say the ledge is \textbf{modelled} on an interval $\mathrm{Int}\subset \R\cup \{\pm \infty\}$ if the flowline is a map $\mathrm{Int} \to B$ (which will land in some component $B_i\subset B$). We say that a Floer solution is {\bf non-trivial} if it has strictly positive energy, $E(u)=\int \|\partial_s u\|^2\,ds\wedge dt>0$. Trivial Floer trajectories have zero energy, so $\partial_t u - X_{H}=0$, so $t\mapsto u(s,t)$ is an $s$-independent $1$-orbit; we will prohibit these inside cascades.
A \textbf{drop} is any non-trivial Floer trajectory\footnote{Our conventions: $\partial_s u + I(\partial_t u - X_H)=0$, and $\omega(\cdot,X_H)=dH$.} $u:\R\times S^1 \to Y$ for $u$ asymptotic to any (possibly non-critical) associated $1$-orbits $u(-\infty,\cdot)=c$ and $u(+\infty,\cdot)=b$. We call $c$ the \textbf{crest} and $b$ the \textbf{base} of the drop. 
A \textbf{cascade} consists of an alternating succession of ledges and drops, as in the picture below: the ledges are the bold lines; the cylinders are the drops; the shaded planes denote the {\MB } manifold $B=\mathrm{Crit}(H)$. 
\\[2mm]
\strut\hspace{30ex}\scalebox{1.5}{\begin{picture}(0,0)%
\includegraphics{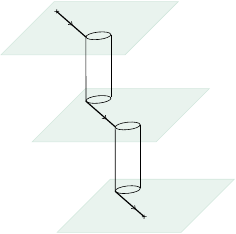}%
\end{picture}%
\setlength{\unitlength}{2486sp}%
\begingroup\makeatletter\ifx\SetFigFont\undefined%
\gdef\SetFigFont#1#2#3#4#5{%
  \reset@font\fontsize{#1}{#2pt}%
  \fontfamily{#3}\fontseries{#4}\fontshape{#5}%
  \selectfont}%
\fi\endgroup%
\begin{picture}(2986,2961)(486,868)
\put(1917,3362){\makebox(0,0)[lb]{\smash{{\SetFigFont{7}{8.4}{\familydefault}{\mddefault}{\updefault}$c_1$}}}}
\put(1930,2522){\makebox(0,0)[lb]{\smash{{\SetFigFont{7}{8.4}{\familydefault}{\mddefault}{\updefault}$b_1$}}}}
\put(2296,2242){\makebox(0,0)[lb]{\smash{{\SetFigFont{7}{8.4}{\familydefault}{\mddefault}{\updefault}$c_2$}}}}
\put(2309,1388){\makebox(0,0)[lb]{\smash{{\SetFigFont{7}{8.4}{\familydefault}{\mddefault}{\updefault}$b_2$}}}}
\put(1264,3682){\makebox(0,0)[lb]{\smash{{\SetFigFont{7}{8.4}{\familydefault}{\mddefault}{\updefault}$x_-$}}}}
\put(2371,1002){\makebox(0,0)[lb]{\smash{{\SetFigFont{7}{8.4}{\familydefault}{\mddefault}{\updefault}$x_+$}}}}
\end{picture}%
}
\\[2mm]
A cascade that contributes to the coefficient of a critical $1$-orbit $x_-$ in the differential $d(x_+)$ therefore consists of the data of a finite sequence of points
$$c_1,\ b_1,\ c_2,\ b_2,\ \ldots,\ b_k\in B=\mathrm{Crit}(H),$$
together with drops with crest $c_i$ and base $b_i$ (using associated $1$-orbits), admitting ledges in between. 
Here $c_i,b_i$ are non-critical points,\footnote{The limiting case when $c_i$ or $b_i$ is critical corresponds to broken cascade solutions arising when compatifying the moduli spaces of non-isolated cascades.} except in the special cases that occur if $x_-=c_1$ or $b_k=x_+$ (which involve a constant ledge solution at $x_-$ or $x_+$ respectively).

The in-between ledges involve: a ledge from $x_-$ to $c_1$ modelled on $(-\infty,0]$, ledges joining $b_i$ to $c_{i+1}$ modelled on $[0,T_i]$ for some finite positive%
\footnote{The limiting cases $T_i=0$ or $T_i=\infty$ correspond to broken cascade solutions arising when compatifying the moduli spaces of non-isolated cascades. When $T_i=\infty$ either the crest or the base (or both) have become critical, since each $B_i\subset B$ is compact.
The lengths $T_i>0$ are free parameters, but they ``rigidify'' for isolated cascades.} 
length $T_i>0$, and a ledge from $b_k$ to $x_+$ modelled on $[0,+\infty)$.
The exception to this rule are the \textbf{simple cascades} consisting of just one ledge from $x_-$ to $x_+$ modelled on $(-\infty,+\infty)$. For simple cascades both $x_{\pm}$ belong to the same component $B_j\subset B$, and such ledges are given by the moduli space of Morse trajectories in $B_j$ joining $x_{\pm}$. Thus the simple cascade contributions to the differential coincide with the differential for the Morse complex of $(B,f)$.

The ledges are uniquely determined by the points $x_{\pm},c_i,b_i$ by ODE theory (except for the simple cascades which we saw arise in moduli spaces), whereas the drops can arise in moduli spaces. This fact can be used to express the moduli space of cascades as a complicated union of fibre products of moduli spaces of drops. Here, when we speak of moduli spaces, it is
always understood that we quotient by automorphisms: there is an $\R$-reparametrisation action on simple cascades, and there is an $\R$-reparametrisation action for each drop.
Transversality for the moduli spaces of cascades requires a generic perturbation of the almost complex structure to ensure that moduli spaces of drops are cut out transversely, and requires generic choices of $f_i$ so that the $-\nabla f_i$ flow on each $B_i$ is {\MS } as well as ensuring transversality for the evaluation maps involved in the fibre products.
Orientation signs will be discussed in \cref{Section {\MBF } theory: orientations}, specifically in \cref{Subsection Orientation signs for {\MBF } theory}. For now, we warn the reader that the orientation signs are not simply obtained by a fibre product construction, using Morse-theoretic orientation signs on ledges. It is necessary to use Floer-theoretic orientation signs on the ledges, which corresponds to having a local system of coefficients in play for the ledges (see \cref{Subsection Orientation signs for {\MBF } theory})

The proof that the differential squares to zero is a standard Floer theory argument. Namely, 
in a $1$-family $\mathcal{M}$ of cascade solutions, there are three types of breaking that describe the boundary of the compactification of $\mathcal{M}$: (1) Morse trajectory breaking, when a ledge breaks at a critical point because one of the parameters $T\to +\infty$; 
(2) Floer trajectory breaking, when a $1$-family of drops converge to a once-broken Floer trajectory; and (3) one of the ledge parameters $T\to 0$.
Notice that the broken configurations arising in (2) and (3) both involve a broken cascade where two drops are joined at a critical $1$-orbit.
The counts of types (2) and (3) cancel out, whereas type  (1) is counted by $d^2(x_+)$, so $d^2(x_+)=0$ as it is the oriented count of boundaries of a compact $1$-manifold.

\begin{rmk}[Hamiltonian Perturbations]\label{Rmk perturbing MorseBott H in specific way}
The above complex is supposed to model the Floer complex of an infinitesimal perturbation $\widetilde{H}$ of $H$, as follows.
We first perform a perturbation
$$
\widetilde{H}=H+\epsilon \sum \rho_i f_i
$$
where $\epsilon>0$ is a small constant, $\rho_i$ is a cut-off function supported in a tubular neighbourhood of $B_i$, and $f_i: B_i \to \R$ is extended to that neighbourhood to be constant in the normal fibre directions. 
For small $\epsilon>0$, this ensures that the autonomous Hamiltonian $\widetilde{H}$ has transversally%
\footnote{This is precisely the {\MB } condition of \cref{Equation {\MB }} when the {\MB } manifolds are circles.
} 
non-degenerate orbits, it is a {\MBF } Hamiltonian whose {\MB } manifolds are circles (non-constant $1$-orbits) or points (constant $1$-orbits). To get rid of the time-translation symmetry of non-constant $1$-orbits, a second perturbation is necessary, which is time-dependent near those orbits. Such perturbations were first studied by {\CFHW } \cite[Prop.2.2]{CFHW}, and we also refer to adaptations in more general settings by Oancea \cite[Sec.3.3]{OanceaEnsaios} and {\KvK } \cite[App.B]{KwonvanKoert}.
Explicitly, for each orbit $Y\supset \mathcal{O}\cong S^1 = \R/\Z$ parametrised by $\sigma$, 
one picks a Morse function $f_{\mathcal{O}}:\mathcal{O} \to \R$ with exactly two critical points, one extends $f_{\mathcal{O}}$ trivially in normal directions,
one introduces a bump function $\rho_{\mathcal{O}}$ supported near $\mathcal{O}$, and finally one adds $\epsilon \rho_{\mathcal{O}}\cdot f_{\mathcal{O}}(\sigma - kt)$ in the definition of $\widetilde{H}$, where $1/k$ is the minimal period of the orbit $\mathcal{O}$ (so the orbit is a $k$-fold cover, $\sigma \mapsto k\sigma$). For small $\epsilon>0$, the $1$-orbits of $\widetilde{H}$ are precisely the critical $1$-orbits.
With these choices, one expects that for sufficiently small $\epsilon>0$ the Floer trajectories for $\widetilde{H}$ are arbitrarily close to cascades (after taking associated $1$-orbits for the ledges to get cylinders).
{\BO } \cite{Bourgeois-Oancea,Bourgeois-Oancea2} proved this rigorously in the setup of a Liouville manifold $Y$ with an autonomous Hamiltonian radial at infinity whose $1$-orbits are transversally non-degenerate (i.e.\,the {\MB } manifolds are circles or points). For a sufficiently small time-dependent perturbation as above, Bourgeois and Oancea show that 
the Floer complex for the perturbed Hamiltonian agrees (at the level of complexes) with the {\MB} Floer complex.
A similar proof is expected to hold for Liouville manifolds with any {\MB } manifolds (not necessarily circles), however as far as we know this has not appeared in the literature. For the manifolds $Y$ considered in this paper, the absence of a priori energy estimates would force us to use energy-bracketed Floer complexes (\cref{Section Transversality for Floer solutions}), but with that proviso, the proof should be the same as in the Liouville case. A trickier aspect are the limits in \cref{Section Transversality for Floer solutions}: one needs to show that chain level identifications can be made compatibly with continuation maps, for small perturbations. Due to these complications, we decided to work at the cohomology level in \cref{Subsection {\MB } Floer cohomology agrees with perturbed Floer cohomology}.\\
Finally, we remark that for Morse cohomology the discussion by Frauenfelder \cite[Appendix]{Frauenfelder-cascades} works with arbitrary {\MB } submanifolds, and Banyaga--Hurtubise \cite{BH13} proved that the {\MB} complex agrees with the perturbed Morse complex for suitable small perturbations.
\end{rmk}

\subsection{{\MB } Floer cohomology agrees with perturbed Floer cohomology}\label{Subsection {\MB } Floer cohomology agrees with perturbed Floer cohomology}

Denote by $\widetilde{H}$ any generic perturbation of $H$ within the class of admissible Hamiltonians so that the $1$-orbits become non-degenerate, in particular isolated (as we work at the cohomology level, it does not need to be of the type described in 
\cref{Rmk perturbing MorseBott H in specific way}). %
If there are non-constant $1$-orbits then $\widetilde{H}$ will have to be time-dependent near those $1$-orbits. For the sake of clarity, we will suppress the (straightforward but messy) discussion of energy-bracketing as in \cref{Section Transversality for Floer solutions} needed when $Y$ is not Liouville.
The results in this section rely on constructing orientation signs for {\MBF } theory as in \cref{Subsection Orientation signs for {\MBF } theory}.

\begin{cor}\label{Lemma MorseBottFloer is perturbed Floer}
There is an isomorphism $BHF^*(H) \cong HF^*(\widetilde{H})$ between {\MB } Floer cohomology and perturbed Floer cohomology. Such isomorphisms are compatible with continuation maps.
\end{cor}
\begin{proof}
Note that $BCF^*(\widetilde{H})=CF^*(\widetilde{H})$ since by definition the only possible isolated cascade solutions for $\widetilde{H}$ consist of a single drop (so $x_-=c_1$ and $b_1=x_+$ above, with constant ledges at $x_-$ and $x_+$). So the claim follows from \cref{Thm Morse Bott Floer iso} below.
\end{proof}

For sake of clarity in the phrasing of the next Theorem, when $Y$ is non-compact we assume there is a notion of ``slope at infinity'' for admissible Hamiltonians, so that if slope($\widetilde{H}$)$\,\geq\,$slope($H$) there is a notion of ``admissible homotopy'' $H_s$ from $\widetilde{H}$ to $H$ and Floer continuation solutions satisfy a suitable maximum principle at infinity. If slope($\widetilde{H}$)$\,=\,$slope($H$) we have admissible homotopies in both directions, and we assume there is a generic $1$-parameter family of admissible homotopies from their concatenation to the constant homotopy (and each homotopy in the family satisfies the maximum principle at infinity). These assumptions are well-known to hold for Liouville manifolds for Hamiltonians radial at infinity using monotone homotopies; in this paper in 
\cref{Subsection Dependence of the filtration on phi} we saw the assumptions hold for symplectic $\C^*$-manifolds over a convex base using admissible Hamiltonians and admissible homotopies. For closed symplectic manifolds $Y$ these assumptions are not necessary (let slope$\,=0$ for all Hamiltonians).

\begin{thm}\label{Thm Morse Bott Floer iso}
Let $H,\widetilde{H}$ be any two {\MBF } Hamiltonians  with slope($\widetilde{H}$)$\,\geq\,$slope($H$). Then one can construct continuation maps
$$
BCF^*(H) \to BCF^*(\widetilde{H})
$$
which define the same map $BHF^*(H) \to BHF^*(\widetilde{H})$ on cohomology, satisfying the following.
\begin{enumerate}
\item Identity continuation map: if $\widetilde{H}=H$ and $H_s=H$ is constant, and we use the same auxiliary data $f=\widetilde{f}$ on $B=\mathrm{Crit}(H)$, then the continuation map is the identity.
\item Compatibility of continuation maps: composing $BHF^*(H)\to BHF^*(\widetilde{H}) \to BHF^*(\doublewidetilde{H})$ yields the continuation map $BHF^*(H) \to BHF^*(\doublewidetilde{H})$.
\item If slope($\widetilde{H}$)$\,=\,$slope($H$), the two continuation maps $BHF^*(H) \rightleftarrows BHF^*(\widetilde{H})$ are inverse to each other, so $BHF^*(H) \to BHF^*(\widetilde{H})$ is an isomorphism compatible with continuations maps.
\end{enumerate}
\end{thm}
\begin{proof}
For the analogous claims for {\MB } cohomology we refer to Frauenfelder \cite[(A.136), (A.137)]{Frauenfelder-cascades}, to whom we are indebted as the following argument builds upon those ideas. We should however mention that our approach appears to differ from \cite[(A.134)]{Frauenfelder-cascades}: that approach would be analogous to prohibiting an arbitrary number of drops in between two continuation cylinders, whereas we will see that this needs to be allowed as it naturally arises from composing two continuation maps. 

We are given $H,\widetilde{H}$ and auxiliary Morse functions: $f: B=\mathrm{Crit}(H) \to \R$ and $\widetilde{f}: \widetilde{B}=\mathrm{Crit}(\widetilde{H})\to \R$. Let $H_s$ be any admissible homotopy with $H_s=\widetilde{H}$ for $s\ll 0$ and $H_s=H$ for $s\gg 0$. 
We need to define the continuation cascades counted by the continuation map. 
A \textbf{continuation drop} for $H_s$ is any Floer continuation solution $u:\R\times S^1 \to Y$ for $H_s$, asymptotic to any (possibly non-critical) associated $1$-orbits $u(-\infty,\cdot)=c$ and $u(+\infty,\cdot)=b$ (in this setting, we do not prohibit trivial Floer continuation solutions, i.e.\,those which are $s$-independent). Here the crest $c$ is a $1$-orbit of $\widetilde{H}$ and the base $b$ is a $1$-orbit of $H$.
A \textbf{continuation cascade} consists of three parts:%
\footnote{More precisely, a continuation cascade that contributes to the coefficient of a critical $1$-orbit $x_-$ in the differential $d(x_+)$ will consist of the data of a finite sequence of points
$$c_{-\widetilde{k}},\ b_{-\widetilde{k}},\ c_{-\widetilde{k}+1},\ b_{-\widetilde{k}+1},\ \ldots,\ c_0 \in \mathrm{Crit}(\widetilde{H}) \qquad \textrm{ and } \qquad b_0,\ c_1,\ b_1, \ c_2,\ b_2,\ \ldots,\ b_{k}\in \mathrm{Crit}(H)$$
together with drops for $\widetilde{H}$ with crest $c_{-j}$ and base $b_{-j}$ (where $j\leq 0$); a continuation drop for $H_s$ with crest $c_0$ and base $b_0$; and drops for $H$ with crest $c_i$ and base $b_i$  (where $i\geq 0$), admitting ledges in between. The ledge from $x_-$ to $c_{-\widetilde{k}}$ and the ledges from $b_{-j}$ to $c_{-j+1}$ are in $\mathrm{Crit}(\widetilde{H})$ and use $-\nabla \widetilde{f}$; whereas the ledges from $b_{i}$ to $c_{i+1}$ and the ledge from $b_k$ to $x_+$ are in $\mathrm{Crit}(H)$ and use $-\nabla f$. As usual, only the ledges involving $x_{\pm}$ are modelled on semi-infinite intervals, the others use finite positive length intervals.
As usual ``counting solutions'' means an oriented weighted count, with weight $T^{E(u)}$ where $E(u)$ is the sum of the energies of the drops except for the continuation drop we use the topological energy $E_0$ \cite[Sec.3.2 and 3.3]{R10} (see also \cref{Definition monotone hpy}).
}
\begin{enumerate}
\item an alternating succession of ledges for $\widetilde{f}$ and drops for $\widetilde{H}$, ending in a ledge;
\item a continuation drop for $H_s$; and
\item an alternating succession of ledges for $f$ and drops for $H$, ending in a ledge.
\end{enumerate}

Note that (1) are the cascades one sees at the negative ends of cascades used in the definition of $BCF^*(\widetilde{H})$, and (3) are the cascades on the positive ends for $BCF^*(H)$. This ensures we can run the standard Floer theory argument that a continuation map is a chain map. Namely, the boundary of $1$-dimensional moduli spaces of such solutions consists of solutions that are broken at one of the two ends: one of the ledges breaks at a critical point (Morse trajectory breaking).\footnote{As in the discussion of $d^2(x_+)=0$, we can ignore the case when a drop breaks (Floer trajectory breaking) as these counts will cancel with the count of boundary configurations arising when a ledge parameter converges to zero. The same argument applies to the case when a continuation drop breaks at a $1$-orbit giving rise to a (non-continuation) drop and a continuation drop (Floer continuation solution breaking).} Such broken configurations are counted by the two possible ways of composing the differential with the continuation map.

When considering moduli spaces, we quotient by automorphisms as before, so modulo $\R$-actions on all drops (but not on the continuation drop).
Transversality for the moduli spaces of isolated continuation cascades requires that for the continuation drop one uses a generic small domain-dependent perturbation of the almost complex structure (recall that for the spaces $Y$ considered in this paper, where a priori energy estimates are not available, we must run the energy-bracketing procedure of \cref{Section Transversality for Floer solutions} to achieve transversality). 

The count of isolated continuation cascades defines the required map $BCF^*(H) \to BCF^*(\widetilde{H})$. The choice of $H_s$ means we have in fact several such maps, but all these maps agree on cohomology by a standard Floer theory argument. Namely, one considers the parameterised moduli space of continuation cascades for $H_{s,\lambda}$ for each parameter value $\lambda$, where $H_{s,\lambda}$ is a family of homotopies that interpolate two given choices of a homotopy $H_s$ as above. The count of the boundaries of such $1$-dimensional parameterised moduli spaces shows that the two continuation maps are chain homotopic (we will use a similar argument again below, providing more details). 

We now prove Claim (2) of the Theorem. The proof is similar to the standard Floer theory argument, which shows that the continuation map of a concatenation of two homotopies is chain-homotopic to the composition of two continuation maps for the two homotopies. 
We consider a more general type of continuation cascade, which involves two continuation drops. More precisely, they  consist of five parts: the first, third and fifth parts are alternations of ledges/drops (respectively for the three pairs $(H,f)$, $(\widetilde{H},\widetilde{f})$ and $(\doublewidetilde{H},\doublewidetilde{f})$);
whereas the second and fourth parts involve continuation drops for the two respective homotopies $H_s$ and $\widetilde{H}_s$.
Call {\bf $2$-continuation cascades} these solutions.
Suppose we have a $1$-family of $2$-continuation cascades, and consider how these solutions might break. When a ledge in the third part breaks, we see the broken solutions that contribute to the composition $BCF^*(H) \to BCF^*(\widetilde{H}) \to BCF^*(\doublewidetilde{H})$.
When a ledge in the first or fifth part breaks, we obtain a (non-continuation) cascade of index $1$ (counted by the {\MB } Floer differential) glued to a $2$-continuation cascade of index $-1$ (the so-called ``rogue solutions'' because they normally would not occur in view of the negative index, but can arise in $1$-families). The count of rogue solutions defines a chain homotopy.
As in a previous footnote, we may ignore drop-breaking, as these will cancel with cases when a ledge parameter shrinks to zero. This accounts for cases where the ledge-parameter shrinks to zero, for ledges not involving the continuation cylinders. A ledge that lies between a continuation cylinder and a drop may have its ledge-parameter shrink to zero: these contributions cancel with the contributions coming from the breaking of a Floer continuation cylinder (the standard breaking of a Floer continuation cylinder, which breaks off a Floer trajectory at one end). There is one final ledge-shrinking case: when the third part consists of just one ledge which shrinks to zero, so the ledge lies between the two continuation cylinders. In that case, we apply 
the standard glueing argument for Floer continuation cylinders: the two continuation cylinders glue to give a continuation cylinder for a concatenation of the two homotopies $H_s$ and $\widetilde{H}_s$.

Claim (3) of the Theorem follows from (1) and (2) by using that the continuation map on cohomology is independent of the choice of homotopy (the concatenation of the two homotopies involved in (3) can thus be replaced by a constant homotopy, and then we use (1)). Claim (1) relies on a standard Floer theory argument: the $s$-independence of $H_s$ means there is a free $\R$-action on the continuation drop (in this setup transversality can be achieved using an $s$-independent perturbation of the almost complex structure), so these must be constant if the continuation cascade is isolated. This means that the two ledges adjacent to the constant continuation drop will glue. If either of those ledges was modelled on a finite interval with a free parameter $T$, then the continuation cascade solution would not be isolated: it would arise in a family as one could vary $T$. The final case is when the adjacent ledges are modelled on semi-infinite intervals, i.e.\,they are the ledges joined to $x_{\pm}$. In that case, glueing those two ledges (since the continuation drop is constant) gives rise to $-\nabla f$ flowlines $(-\infty,+\infty)\to B$: these have a free $\R$-translation action, so they are only isolated if they are constant. Thus the continuation endomorphism we get on $BCH^*(H)$ is the identity map, proving Claim (1).
\end{proof}
\begin{rmk}[The choice of auxiliary Morse functions]
Of course $BCF^*(H)=BCF^*(H;f)$ depends on the choice of auxiliary Morse function $f:\mathrm{Crit}(H)\to \R$. A small perturbation of $f$ would yield a natural bijection of generators, but would affect the Morse flows for the ledges. In general, however, different Morse functions $f,\widetilde{f}:\mathrm{Crit}(H)\to \R$ will have unrelated generators (critical points). Nonetheless, there is an isomorphism on cohomology $BHF^*(H;f)\to BHF^*(H;\widetilde{f})$ given by the continuation map.
Indeed, we can view the variation of auxiliary Morse functions as a special case of \cref{Thm Morse Bott Floer iso}(3): the continuation map $BHF^*(H,f)\to BHF^*(\widetilde{H},\widetilde{f})$ is an isomorphism when $\widetilde{H}=H$ (so equal slopes).
\end{rmk}
\section{Morse--Bott--Floer theory: cascades near a Morse--Bott manifold}\label{AppendixCascades2}
\subsection{Low-energy cascades and the self-drop exclusion theorem}
\label{Subsection Low-energy cascades and the self-drop exclusion theorem}

A drop $u$ is a {\bf self-drop} if the crest $c$ and base $b$ lie in the same {\MB } submanifold $B_i$. For ordinary {\MB } cohomology of a {\MBF } function $H$ that is constant on each $B_i$, these would not arise: the energy $E(u) = \int \|\partial_s u\|^2\,ds\wedge dt = H(c)-H(b)$ would be zero since $H$ is constant on $B_i$, and trivial drops are prohibited in a cascade. We now prove an analogous result in the {\MB } Floer setting. 
For Floer trajectories, such an a priori estimate will fail in general, as the Floer action functional is multi-valued.

\begin{de}\label{Definition energy of a cascade}
The {\bf energy} of a cascade is the sum of the energies of its ledges and drops.
Any ledge in $B_i$ from $b$ to $c$ has an a priori energy estimate%
\footnote{
Technical clarification: In the proof of \cref{Theorem self-drop exclusion theorem}, we construct action values $\mathcal{A}(x,u_x)$ for $1$-orbits $x$ with a choice of filling $u_x$, and one can build $u_x$ so that $\mathcal{A}(x,u_x)$ takes the same value for all $x$ in a given $B_i$. This would make it seem that a ledge $v$ from $b$ to $c$ should have zero action difference, so zero energy. This apparent contradiction occurs because the energy is concealed in the choice of filling: the induced filling $u_c = u_b \# v$ for $c$ is what arises algebraically as $T^{E(v)} c$.
}
$E(u)=\epsilon(f_i(b(0))-f_i(c(0)))$. The {\MB } Floer differential counts isolated cascades with weight $T^{E}$ where $E$ is the energy of the cascade. 
\end{de}

\begin{de}\label{Definition simple Bott mfds atoroidal Mi}
The {\MB } submanifold $B_i$ is called {\bf simple} if it admits a neighbourhood $B_i\subset N_i\subset Y$ such that all cascades in $N_i$ are simple cascades, i.e.\,Morse trajectories in $B_i$. For example, this holds if all drops in $N_i$ are trivial.

Let $b_i\in [S^1,B_i]$ be the free homotopy class of $1$-orbits in $B_i$.
Then $\omega$ is {\bf orbit-atoroidal} on $B_i$ if
$$
\textstyle
\int v^*\omega = 0 \textrm{ for all smooth }v: S^1 \times S^1 \to B_i \textrm{ with }[v|_{1\times S^1}]=b_i.
$$
\end{de}

\begin{rmk}\label{Remark orbit atoroidal examples}
If $\omega$ is {\bf exact} on $B_i$ (e.g.\,if $H^2(B_i;\R)=0$), then it is orbit-atoroidal by Stokes's theorem.
So the Theorem applies to any Liouville manifold $Y$ as $\omega$ is globally exact. For $B_i\subset Y^{\mathrm{out}}$ it also applies to symplectic manifolds $Y$ that are convex at infinity, as $\omega$ is exact on $Y^{\mathrm{out}}$.

If $\omega$ is orbit-atoroidal then\footnote{one can always attach a sphere at an image point of $v$, forcing $\omega$ to integrate to zero on the sphere.} it is {\bf aspherical}, i.e.\,$[\omega]$ vanishes on $\pi_2(B_i)$.
If a (hence any) $1$-orbit of $B_i$ is contractible inside $B_i$, then being orbit-atoroidal is equivalent\footnote{if $u_x:D \to B_i$ is a null-homotopy (a filling disc for a $1$-orbit $x$), and we view $v$ as a map from a cylinder to and from the same $1$-orbit $x$, then the concatenation $u_x\# v \# (-u_x)$ is a sphere and the integral of $\omega$ equals $\int v^*\omega$ (since $-u_x$ is the reversed null-homotopy, so the integrals of $\omega$ on $u_x,-u_x$ cancel out).} to being aspherical.

We can strengthen the orbit-atoroidal condition by allowing maps $v: S^1 \times S^1 \to S_i$ landing in a subset $S_i\subset Y$ that is possibly larger than $B_i$ (and keeping the same condition for $v|_{1\times S^1}$). Call this {\bf orbit-atoroidal on $S_i$}. If $S_i$ contains a null-homotopy of a (hence any) $1$-orbit of $B_i$, then being orbit-atoroidal on $S_i$ is equivalent to being aspherical on $S_i$.

For symplectic $\C^*$-manifolds $Y$, a subset $S_i$ that contains such a null-homotopy (by \cref{LemmaCanonicalFillingDiscs}) is:
$$
S_i:=\overline{\cup_{t\in (0,1]} \varphi_t(B_i)}=\cup \{-\nabla H\textrm{ flowlines from }B_i\textrm{ including convergence points}\},
$$
and if $\pi_2(S_i)=0$ then $\omega$ is orbit-atoroidal for $B_i$.
In general however the presence of a 
non-constant holomorphic sphere in $B_i$ will obstruct the orbit-atoroidal condition.\footnote{attaching such a sphere to a filling (defined in the proof of \cref{Theorem self-drop exclusion theorem}) would imply that the Floer action functional is multi-valued.}
\end{rmk}

\begin{thm}\label{Theorem self-drop exclusion theorem}
If $\omega$ is orbit-atoroidal on $B_i$ then all drops near $B_i$ are trivial and $B_i$ is simple.

More generally: let $C\subset Y$ be any subset such that $C\cap B_i$ is a path-connected subset of $1$-orbits, $C$ does not intersect other $B_j$, and $\omega$ is orbit-atoroidal on $C$. Then any drop in $C$ is trivial.
\end{thm}
\begin{proof}
Recall that there is a definition of the Floer action functional for $H$,
\begin{equation}\label{Equation Floer action w filling}
\textstyle\mathcal{A}(x,u_x) = -\int u_x^*\omega + \int H(x(t))\, dt,
\end{equation}
which depends on a choice of ``filling'' $u_x$, which we define below, and which is related to the energy of a Floer trajectory $u$ from $x$ to $y$ by
\begin{equation}\label{Equation energy difference of actions}
E(u)=\mathcal{A}(x,u_x)-\mathcal{A}(y,u_x\# u),
\end{equation}
where $\#$ denotes concatenation. 
We claim that the orbit-atoroidal condition ensures that $\mathcal{A}(x,u_x)$ does not depend on the choice of $u_x$. 
More generally, if $\mathcal{A}$ is independent of the filling, we will prove that $\mathcal{A}$ only depends on $B_i$, not on the choice of $x\in B_i$. It follows that self-drops $u$ have energy $E(u)=0$, so they are trivial and are not allowed in a cascade (see \cref{AppendixCascades}).

We begin with an observation: given $1$-orbits $x,y$ in $B_i$ then, since $B_i$ is connected, we can pick a path\footnote{In terms of initial points, pick a path from $x(0)$ to $y(0)$, then consider the path of associated $1$-orbits in $B_i$.} $s\mapsto v(s,t)$ of $1$-orbits in $B_i$ from $x$ to $y$. Being $1$-orbits, they satisfy $\partial_t v = X_{H}$, so
$$
\textstyle
\int v^*\omega = \int v(\partial_s v, \partial_t v)\, ds \wedge dt =
\int dH(\partial_s v)\, ds \wedge dt =
\int \partial_s (H(v))\, ds \wedge dt =
\int (H(y(t)) - H(x(t)))\, dt. %
$$
This also shows that all $1$-orbits of $B_i$ lie in the same free homotopy class of loops in $B_i$. Let $\gamma: S^1 \to B_i$ be any loop representing this class.
A {\bf filling} of a $1$-orbit $x$ in $B_i$ is any smooth free homotopy $s\mapsto u_x(s,t)$ in $B_i$ from $\gamma$ to $x$. In \eqref{Equation energy difference of actions} the homotopy $u_x\# u$ need not lie in $B_i$, so it is not strictly a filling in this sense. However, if we assume that $u$ lies in the neighbourhood of $B_i$ where $\omega$ is orbit-atoroidal, then we can join the $1$-orbits $x,y$ at the ends of $u$ via a path $v$ of $1$-orbits in $B_i$ using the previous observation. The orbit-atoroidal assumption implies\footnote{concatenate $u$ with the reverse $-v$ of $v$ to obtain a torus, then integrate $\omega$ on it.} that $\int u^*\omega=\int v^*\omega$, so for the purpose of defining $\mathcal{A}(x,u_x\# u)$ we may replace $u$ by $v$, and $u_x\# v$ lies in $B_i$ so it is a filling in the strict sense above.
Note also that $\mathcal{A}$ is independent of the choice of $u_x$, as we obtain a torus by concatenating $u_x$ with the reverse $-\widetilde{u}_x$ of another choice (concatenating both at $\gamma$ and at $x$), so $\int u_x^*\omega=\int \widetilde{u}_x^*\omega$.

Independence of the choice of fillings also implies that $\mathcal{A}(x)$ is independent of the $1$-orbit $x$:
$$
\textstyle
\mathcal{A}(y,u_x\# u) = \mathcal{A}(y,u_x\# v) = \mathcal{A}(x,u_x)- \int v^*\omega -\int H(x(t))\,dt +\int H(y(t))\,dt = \mathcal{A}(x,u_x). 
$$
The final claim includes the assumptions needed above: the assumption on $B_j$ ensures drops in $C$ have ends on $B_i$; the path-connectedness assumption ensures the above observation can build $v$ inside $C\cap B_i$ when $x,y\subset C\cap B_i$; the orbit-atoroidal assumption ensures the $\int u^*\omega = \int v^*\omega$ trick when $u\subset C$.
\end{proof}

\begin{prop}\label{Lemma small energy implies cascades are near Bott mfds}
There is a constant $E_0>0$ such that 
any Floer trajectory $u$ of $H$ with ends in $\cup B_i$ and with energy $E(u)\leq E_0$ must be trivial.
\end{prop}
\begin{proof}
We first prove that drops for $H$ with small energy must have crest and base in the same $B_i$.
If $E_0$ did not exist, then there would be a sequence $u_n$ of drops for $H$ with energies $E(u_n)\to 0$, whose crest $u_n(-\infty,t)$ and base $u_n(+\infty,t)$ lie in different {\MB } submanifolds.
By passing to a subsequence still denoted $u_n$, we may assume that the $u_n$ have crest in $B_i$ and bases in $B_j \neq B_i$, for fixed $i,j$ (using that there are finitely many {\MB } submanifolds, and that they are closed submanifolds).
By Gromov compactness, after passing to a subsequence still denoted $u_n$, the $u_n$ would converge to a (possibly broken) Floer trajectory $u_{\infty}$ for $H$ with zero energy, with negative end in $B_i$ and positive end in $B_j$. As it has zero energy, no bubbling has occurred, and $u_{\infty}$ is $s$-independent, so it is equal to a $1$-orbit. This contradicts that the ends lie in different {\MB } submanifolds $B_i\neq B_j$.

We claim that given a small $\epsilon>0$, there is an $E_0>0$ such that any drop of $H$ of energy $E(u)\leq E_0$ must lie in a tubular neighbourhood of a $1$-orbit of width\footnote{Use the exponential map starting from any $1$-orbit using smooth sections of the normal bundle of norm at most $\epsilon$. There are finitely many $B_i$, and the $B_i$ are compact, so for small $\epsilon$ these define tubular neighbourhoods of the $1$-orbits.} $\epsilon$, and must therefore be constant by  \cref{Theorem self-drop exclusion theorem} 
since $\omega$ is orbit-atoroidal on such a neighbourhood (as it has the homotopy type of a circle or of a point).
This is proved by contradiction: suppose $u_n$ does not lie in an $\epsilon$-neighbourhood of a $1$-orbit. The above Gromov-limit argument shows that a subsequence $u_n$ lies in an $\epsilon$-neighbourhood of a limit $u_{\infty}$ for large $n$. But this zero-energy $u_{\infty}$ is an $s$-independent $1$-orbit, contradiction.
\end{proof}

\subsection{Local Floer cohomology and the Energy spectral sequence}
\label{Subsection Local Floer cohomology and the Energy spectral sequence}

Let $B_S=\cup \{B_i: i\in S\}$ be a sub-collection of the {\MB } submanifolds indexed by $S$. Suppose there are open sets $U,V$ s.t.
\begin{enumerate}
\item
$B_S\subset U\subset \overline{U} \subset V \subset Y;$
\item if $B_i$ intersects $V$ then $i\in S$;
\item any Floer trajectory for $H$ in $V$ asymptotic to $1$-orbits in $B_S$ lies in $U$.
\end{enumerate}
Then one can define the {\bf local {\MB } Floer cohomology}
$$
BHF^*_{\mathrm{loc}}(U;H),
$$
by considering only $1$-orbits in $B_S$ and redefining the {\MB } Floer differential to only count cascades lying entirely in $V$. By assumption (3), cascades in $V$ must in fact lie in $U$, as all drops lie in $U$ and all ledges lie in $B_S\subset U$. By assumption (1), if a $1$-family of cascades contains a cascade in $U$ then\footnote{otherwise by continuity there would exist a cascade in $V$ not lying entirely in $U$, contradicting $(2)$.} all cascades in the family lie in $U$. This ensures that the $1$-dimensional moduli spaces of cascades considered in the proof of $d^2=0$ lie in $U$. Condition (2) ensures that broken cascades in the proof of $d^2=0$ only involve $1$-orbits that are generators of the local complex.

Condition (3) holds for example when there is a well-defined Floer action functional $\mathcal{A}$ independent of fillings in the sense of the proof of \cref{Theorem self-drop exclusion theorem}, and such that $\mathcal{A}$ takes the same value on all $1$-orbits in $B_S$ (if $B_S=B_i$ is just one {\MB } submanifold, then this second condition follows from independence of fillings as in the proof of \cref{Theorem self-drop exclusion theorem}). Just as in the proof of \cref{Theorem self-drop exclusion theorem} it would follow that the drops are trivial, and that all cascades are simple cascades, so (3) holds, thus
\begin{equation}\label{Eqn Bott is standard}
BHF^*_{\mathrm{loc}}(U;H) \cong \oplus_{i\in S} H^{*-\mu_H(B_i)}(B_i;\mathcal{L}_{B_i}),
\end{equation}
where the ordinary cohomology is computed as the Morse cohomology of $f_i$ except that orientation signs are assigned by a local system of coefficients $\mathcal{L}_{B_i}$ (with values in $\{\pm 1\}$) that we will construct in \cref{Subsection Construction of the local system}. The grading shift by $\mu_H(B_i)$ is due to the grading convention in \eqref{Equation CZ shift MorseBott}. 
By \cref{Lemma MorseBottFloer is perturbed Floer},
\begin{equation}\label{Eqn Bott is perturbed}
BHF^*_{\mathrm{loc}}(U;H)\cong HF^*_{\mathrm{loc}}(U;\widetilde{H})
\end{equation}
for a generic small time-dependent perturbation $\widetilde{H}$ of $H$. The local Floer cohomology group for $\widetilde{H}$ is still meaningful because condition (3) will hold also for a small enough perturbation $\widetilde{H}$, by a standard\footnote{consider Floer trajectories $u_n$ for $\widetilde{H}_n$, lying in $V$ and not entirely in $U$, and with ends in $U$. Here we let the perturbation decay, so $\widetilde{H}_n\to H$ as $n\to \infty$. By $\R$-rescaling, we may assume $t\mapsto u_n(0,t)$ intersects $V\setminus U$. For large $n$, the $1$-orbits of $\widetilde{H}_n$ are close to those of $H$, so the ends of $u_n$ are close to $1$-orbits in $B_S$. So the actions of the ends of $u_n$ are almost the same, so the energy $E(u_n)\to 0$. By Gromov compactness, a subsequence of the $u_n$ converges to a Floer solution $u$ for $H$, with $t\mapsto u(0,\cdot)$ intersecting $V\setminus U$. So $u$ is not contained in $B_S$. But $E(u)=\lim E(u_n)=0$, so $u$ is a trivial solution, i.e.\,equals a $1$-orbit, so it is contained in $B_S\subset U$, contradiction.} Gromov compactness argument.
The idea of defining a local Floer cohomology  (in a setting where $B_S$ was a circle) goes back to {\CFHW} \cite[Sec.2]{CFHW}.

Condition (3) also holds if there is a {\bf filtration} $F$ on $1$-orbits for $H$ s.t.
\begin{enumerate}[(i)]
\item $F$ is equal on $1$-orbits in $B_S$;
\item if there exists a Floer trajectory $u$ for $H$ from $c$ to $b$ then $F(c)\geq F(b)$;
\item moreover if $u$ enters $V\setminus U$ then $F(c)>F(b)$.
\end{enumerate}
For example, we built such a filtration in \cref{SubsectionFiltrationFunctional} for the Hamiltonian $H_{\lambda}$, which would apply to the above discussion if we let $B_S$ be the union of the $B_i$ arising with the same period value $c'(H)=p$.
Condition (ii) holds by \cref{H_lambdaIsOneDirected}. When
Floer trajectories exit some neighbourhood $U$ of the region where $c'(H)=p$ they will enter the region where the filtration action becomes strictly negative on Floer trajectories, thus ensuring condition (iii). Note that (ii) and (iii) implies condition (3). In our setting, where $F$ is obtained from a functional on loops whose differential is negative on Floer trajectories for $H$, it is clear that condition (3) continues to hold for small perturbations $\widetilde{H}$, by a Gromov compactness argument. So \eqref{Eqn Bott is perturbed} holds (and is well-defined). We cannot deduce \eqref{Eqn Bott is standard} in this case because in general we cannot rule out the presence of drops in $U$.
We only know the following general result.

\begin{prop}\label{Prop appendix spectral sequence and E1 page}
Assume Floer cohomology is defined with coefficients in a Novikov principal ideal domain $R$ which does not involve negative powers of $T$, such as $\k_{\geq 0}$ or $\k_E$ as in\footnote{in particular, one can then localise in $T$ after taking cohomology. If $T^{-1}$ is part of the Novikov ring, then it is unclear how to get a bounded filtration on generators as in the proof. We remark that $\k_{\geq 0}$ and $\k_E$ are PIDs, as they are Euclidean domains for the Euclidean function given by the smallest real power of $T$ occurring in a non-zero ``series''.} \cref{Subsection Period-filtered Floer cohomology}.
Then there is a spectral sequence $E_r^{pq}$ converging to $BHF^*_{\mathrm{loc}}(U;H)$ such that
$$
E_1^{pq}=\oplus_{i\in S} H^{*-\mu_{H}(B_i)}(B_i;\mathcal{L}_{B_i}),
$$
where $\mathcal{L}_{B_i}$ is a local system constructed in 
\cref{Subsection Construction of the local system}, 
and the grading shift $\mu_{H}(B_i)$ is as in 
\eqref{Equation CZ shift MorseBott}. \cref{Theorem Construction of the local system} mentions settings in which $\mathcal{L}_{B_i}$ can be ignored (e.g.\,symplectic $\C^*$-manifolds). 
\end{prop}
\begin{proof}
There is an ``energy filtration'' on the local Floer complex $C^*$ that computes $BHF^*_{\mathrm{loc}}(U;H)$,
given by letting $F_p:=F^p(C^*)$ be generated over $R$ by $T^{r_p}x$ for all $1$-orbits $x$ in $B_S$, for a suitable increasing sequence of values $r_p\geq 0$ that we now explain.
By \cref{Lemma small energy implies cascades are near Bott mfds}, there is an $E_0>0$ such that all drops arising in cascades counted by the {\MB } Floer differential $d$ on $C^*$
have energy greater than $E_0$.
Note $E_0$ only depends on $H$, not on the choices of $\epsilon$ or $f_i$. So we may assume that the auxiliary data $\epsilon,f_i$ for $i\in S$ is chosen so that 
$$\epsilon \cdot (\max_{B_i} f_i-\min_{B_i})<E_0.$$
In summary, we have ensured that (by \cref{Definition energy of a cascade})  the energies of the simple cascades in $B_i$ are less than $E_0$, whereas the cascades that involve drops have total energy larger than $E_0$.
If we now split up the {\MB } Floer differential on $C^*$ as
$$
d=d_0 + d_+
$$
where $d_0$ counts simple cascades, and $d_+$ the non-simple cascades, then it follows that $d_0$ counts cascades with factors $T^r$ with $r<E_0$, and $d_+$ counts cascades with factors $T^r$ with $r>E_0$.   

We choose $r_p=0$ for $p\leq 0$ (so $F_p=C^*$ for $p\leq 0$), and for integers $p>0$ we pick $r_p\geq 0$ to be an increasing sequence of reals with $r_p+E_0\geq r_{p+1}$ (for example, $r_p=pE_0$). Thus $T^{>E_0}\cdot F_p \subset F_{p+1}$, and therefore $d_+(F_p)\subset F_{p+1}$.
In the spectral sequence associated to the filtration $F_p$ on $(C^*,d)$, the differential on the $E_0$-page $\oplus_p F_p/F_{p+1}$ is the map $[d]$ induced by $d$ on the quotient. Since $d_+(F_p)\subset F_{p+1}$, we have $[d]=[d_0]$. Thus only the Morse differential $d_0$ contributes on the $E_0$-page, and it gives rise to the ordinary cohomology of $B_S$ up to discussing coefficients.
We will abusively write that $H^*(B_S;R)\cong \oplus_{i\in S}H^*(B_i;R)$, although more precisely it is $\oplus H^{*-\mu_H(B_i)}(B_i;\mathcal{L}_{B_i})$: \cref{Subsection Construction of the local system} will explain why the local system arises, which has to do with orientation signs, and the degree shift $\mu_H(B_i)$ arises in view of the grading convention \eqref{Equation CZ shift MorseBott}. 
More precisely, we need to discuss the coefficients in our given Novikov ring $R$: the contribution of the $F_p/F_{p+1}$ summand to the $E_1$-page $H^*(E_0,[d_0])$ is $(T^{r_p}R/T^{r_{p+1}}R) \otimes_{R} H^*(B_S;R)$. Thus, the $E_1$-page is the associated graded $R$-module $\mathrm{gr}\,H^*(B_S;R)$ of $H^*(B_S;R)$ arising from the filtration $F_p(R):=T^{r_p}R$ of $R$. In this setting, there is an obvious identification of $R$-modules between $\mathrm{gr}\,R$ and $R$, and therefore between $\mathrm{gr}\,H^*(B_S;R)$ and $H^*(B_S;R)$. 

The filtration $F_p$ of $C^*$ is not bounded, so the convergence of the spectral sequence is not immediate.
The Morse chain complex of $f_i:B_i \to \R$ over $R$ is a finite $R$-module (with finite basis $\mathrm{Crit}(f_i)$, since $B_i$ is compact). Thus, as $R$ is a PID by assumption, the $E_0$-page $H^*(B_S;R)$ and indeed all $E_p$-pages are finitely generated $R$-modules.\footnote{Finitely generated $R$-modules are Noetherian modules as $R$ is Noetherian, so submodules are again finitely generated.} By the structure theorem for $R$-modules over a PID, 
the differential on the $E_p$-page will have the form $d_p: R^m\oplus S \to R^m \oplus S$ where $S$ is a torsion $R$-module and $m\in \N$ is the total (finite) rank of the $E_p$-page.
Afer tensoring with the fraction field $K:=\mathrm{Frac}(R)$, we have a $K$-linear map $d_p\otimes 1: K^m \to K^m$ so the dimension $m$ of the pages can only decrease finitely many times as we increase $p$. So suppose the rank $m$ has stabilised by page $E_p$. This implies that $d_p\otimes 1=0$, so the original map factors as $d_p: R^m \oplus S \to 0\oplus S \subset R^m \oplus S$.
The torsion part $S$ is a finite sum of $R$-modules of type $R/T^rR$ (as a consequence of the Smith normal form for a lift of $d_p$ to an endomorphism of a free $R$-module).
The final ingredient, is that the differentials $d_p$ arise from counts of cascades that are counted with factors $T^a$ where $a$ lies in a discrete subset of $[0,\infty)\subset \R$ (as a consequence of Gromov compactness: below a given energy bound only finitely many cascades contribute to $d$). Therefore (by the Smith normal form) the value of $r$ in those torsion summands $R/T^rR$ can only drop within a discrete set of values of $[0,\infty)$, so eventually those $r$ values also stabilise for large $p$.
Thus $d_p=0$ for sufficiently large $p$, at which point the spectral sequence has degenerated and $E_p\cong E_{\infty}$ has converged.
\end{proof}

\section{{\MBF } theory: perturbations}\label{Appendix {\MBF } theory: perturbations}
\subsection{Periodic flows in a local model near a {\MB } manifold of orbits}

We now generalise the {\MB } perturbation argument \cite{CFHW,KwonvanKoert}.
The method in \cite{CFHW} to deal with a circle $B$ arising as a {\MB } manifold of $1$-orbits of $H$, was to introduce a Hamiltonian $S^1$-action $p_t$ near $B$ extending the $1$-periodic flow $\varphi_H^t|_B$. They use the reversed $\Delta_t=p_{-t}$ ``unwrapping'' transformation to obtain an equivalent Hamiltonian system where $B$ becomes a {\MB } manifold of critical points. The proof actually only relies on $\Delta_t$ being a loop of Hamiltonian symplectomorphisms.

{\KvK } \cite[Prop.B.4]{KwonvanKoert} generalised \cite{CFHW} to Liouville manifolds, for {\MB } manifolds $B \subset \Sigma_{\tau}$ arising in a level set $\Sigma_{\tau}={h'(R)=\tau}$ of a radial Hamiltonian $H=h(R)$. However, in their argument on p.220 of \cite{KwonvanKoert} the Hamiltonian $K=-\tau R$ which generates their flow $\Delta_t$ is not%
\footnote{unless $B=\Sigma_{\tau}$ is the whole slice, as in the examples of \cite[Sec.3.3]{OanceaEnsaios}. Having $t\in \R$ rather than $t\in S^1=\R/\Z$ is an issue because it unwraps Floer cylinders into Floer strips.
 In our examples, so for symplectic $\C^*$-manifolds, we have a similar situation where {\MB } components $B$ for $H_{\lambda}=c(H)$ arise inside slices $c'(H)=\tau$, but $K=-\tau\cdot H$ is only a circle action on the torsion submanifold for that component (see \cref{SubsectionHam1-orbitsOfHlambda}), which is not necessarily the whole slice.}
  actually $1$-periodic. One really needs a loop of Hamiltonian symplectomorphisms near $B$.

 We will try to remedy this by providing in \cref{Subsection generalisation to all Morse-Bott manifolds} a general argument for how to construct a loop of Hamiltonian symplectomorphisms near $B$.
 For the purposes of this paper, we prove a stronger result for our {\MBF} manifolds, as we always have a very nice local model, as follows.

\subsection{The local model}\label{Subsection local model near a Morse Bott Floer manifold}

Recall $(Y,\omega)$ is a symplectic manifold, $I$ is an $\omega$-compatible almost complex structure.
We consider the following {\bf local model}, working in some open subset $U\subset Y$.

 \begin{lm}\label{Lemma pseudoholo implies vf is still in TB}
Suppose that: 
\begin{enumerate}
\item $B$ is a connected {\MB } manifold of $1$-orbits of a Hamiltonian $H$; 
\item there is a function $K$ defined near $B$ such that the time-$t$ flow $\psi_t$ of $X_K$ is {\ph}; 
\item $B$ lies in a regular level set of $K$; 
\item and $H=c(K)$ for a function $c$ with $c''(K(B))>0$. 
\end{enumerate}
Let $\tau:=c'(K(B))$ and let $X_{\R_+}:=\nabla K$. Then 
\begin{enumerate}
\item $X_{\R_+}$ commutes with $X_K$;
\item there is a connected $I$-holomorphic submanifold $C$ defined near $B$ with $\psi_{\tau}|_C=\mathrm{id}$;
\item $C$ is a {\MB} manifold of $\tau$-orbits of $K$.
\end{enumerate}
\end{lm}
\begin{proof}
As $(\psi_t)_*$ commutes with $I$, it also preserves its gradient:
$$
(\psi_t)_*(\nabla K) = (\psi_t)_*(-IX_K)=
 -I(\psi_t)_*(X_K)=-IX_K=\nabla K.
$$
Locally near $B$ we can therefore define $X_{\R_+}:=\nabla K$.
As $\psi_t$ preserves $X_{\R_+}$, the Lie derivative $\mathcal{L}_{X_K}X_{\R_+}=0$, so $[X_K,X_{\R_+}]=0$.
The same proof as for
\cref{LemmaForTheShear} yields
$$
(\phi_{\tau}^{c(K)})_* X_{\R_{+}}=X_{\R_+}+\tau c''(H)\norm{\nabla K}^2 X_{K},
$$
using $(\psi_t)_*X_{\R_+} = X_{\R_+}$ in the last line of the proof. Integrating $\pm X_{\R_+}$ starting from $B$ we get a submanifold $C\cong (-\varepsilon,+\varepsilon)\times B$ near $B$.
On $B$, $\varphi_H^{t}|_B = \psi_{\tau t}|_B$.
As $B$ consists of $1$-orbits of $H$, it consists of $\tau$-orbits of $K$, 
and as $X_{\R_+}$ commutes with $X_K$, 
$C$ also consists of $\tau$-orbits of $K$.
The {\MB } assumption is $TB=\ker ((\varphi_H^1)_* - \mathrm{id})$.
Thus, $TC=\R X_{\R_+}\oplus TB = \ker ((\psi_{\tau})_*-\mathrm{id})$.
As $I$ commutes with $(\psi_{\tau})_*$ it follows that it commutes with $((\psi_{\tau})_*-\mathrm{id})$, therefore $C$ is $I$-{\ph}.
\end{proof}

\begin{ex}[The local model arising from $S^1$-actions]
The above local model arises naturally in \cref{SubsectionHam1-orbitsOfHlambda}, where $H_{\lambda}$ and $H$ play the roles 
 of the above $H$ and $K$, respectively. In that setup we have:
 
\begin{enumerate}

\item A Hamiltonian $K=K_t: U \to \R$ generating a {\ph } $S^1$-action $(\psi_t)_{t\in \R/\Z}$ on $U$.

\item A choice of integer $m\geq 1$, so the time $1/m$ flow defines a fixed locus $$C:=\mathrm{Fix}(\psi_{1/m})=\{\Z/m\textrm{-torsion points in }U\},$$ 
and we assume $C\subset U$ is connected.

\item We pick a subset $B\subset C$ lying in a regular level set of $K$ (so $B$ lies inside the {\bf slice} $\{K=K(B)\}$). 
For a given ``slope'' $\tau \in \tfrac{1}{m}\Z$, we pick a real-valued function $c$ with $c'(K(B))=\tau$ and $c''(K(B))>0$. Finally define the reparametrised Hamiltonian $H:=c(K): U \to \R$.
\end{enumerate}

$C\subset Y$ is a closed 
submanifold as it is the fixed locus of a compact Lie group action; $C$ is $I$-holomorphic (so symplectic) since $\psi_{1/m}$ is 
{\ph}.%
\footnote{The tangent space $T_xC=\mathrm{Fix}(d_x \psi_{1/m})$ is the subspace of $T_x Y$ fixed under the $\Z/m$-linear action generated by $d_x \psi_{1/m}$, and $d_x \psi_{1/m}$ commutes with $I$.}
$B\subset C$ is a real codimension one subset. Shrinking $U$ if necessary, we may assume $B$ is also connected.
By construction, $H$ is constant on $B$, and any point of $B$ is the initial point of a $1$-orbit of $X_H$ since $X_H|_B=\tau\cdot X_K$.
Since $X_H=c'(K)X_K$, the flows are related by $\varphi_H^t(y) = \psi_{c'(K(y))t}(y)$.
By \cref{Rmk about S1 actions}, and shrinking $U$ if necessary, we can extend the $S^1$-action to a (partially defined) $\C^*$-action on $U$. We get a (partially defined) $\R_+$-action with $X_{\R_+}=\nabla K=-IX_K$.
Then $B\subset C$ is a slice for the $\R_+$-action on $C$, and $H^{-1}(H(B))$ is a slice for the $\R_+$-action on $U$.
Shrinking $U$ if necessary, $H$ is {\MBF } and $B$ is its only {\MB } component of $1$-orbits in $U$, and $C$ is the only {\MB } component of $(1/m)$-periodic orbits for $K$.
\end{ex}

Note that $X_H=c'(K)\,X_K$ is $1$-periodic on $B$, as $X_K$ is $1/m$-periodic on $B$. However, these periodic flows do not extend to periodic flows near $B$. Our trick is to create a new flow by tweaking $H$.

\subsection{The $\omega$-normal bundle}
We need some preliminary remarks about the {\bf $\omega$-normal bundle} $\pi: \nu_C \to C$, whose fibre is $\nu_c=T_cC^{\perp \omega}\subset T_cY$.
It has a symplectic structure $\omega_{vert}(v_1,v_2):=\omega(v_1,v_2)$, a complex structure $I|_{\textrm{fibre}}$, and a
Hermitian metric $\langle v,w \rangle := \omega_{vert}(v,Iw)+i \om_{vert}(v,w)$.
We pick a compatible Hermitian connection, which splits $T_v\nu_C \cong \nu_c \oplus T_{\pi(v)} C$ into vertical and horizontal vectors, allowing us to extend $\omega_{vert}$ and $\langle \cdot, \cdot \rangle$ to $T_v \nu_C$.
We want to build a closed $2$-form $\Omega$ on $\nu_C$ which agrees with $\omega_{vert}$ in the fibre directions, and this is a standard procedure (e.g.\;see \cite[Sec.4]{GinzburgGurel} and \cite[Sec.7.2]{R14}):
define a radial coordinate $r^2:=\langle v,v \rangle$ and
an angular one-form $\theta_v := \tfrac{1}{2\pi r^2} \langle I v, \cdot \rangle$, 
then let\footnote{Here, $v$ is a general point of $\nu_C$, and one must compute the general differential before evaluating at a specific $v$. One can use the formula
$d(\omega_{vert}(v,\cdot))(v_1,v_2)=
v_1\cdot \omega_{vert}(v,v_2) - v_2 \cdot \omega_{vert}(v,v_1)-\omega_{vert}(v,[v_1,v_2])$ with $v\in \nu_C$ general, and $v_1,v_2$ extended to local vector fields, but in the first two terms $v$ is still general when computing directional derivatives.
}
$$\Omega:=d(\pi r^2 \theta)=\tfrac{1}{2} d\langle Iv,\cdot \rangle = \tfrac{1}{2} d (\omega_{vert}(v,\cdot )).$$
This $\Omega$ equals $\omega_{vert}$ on the fibres.\footnote{We can \cite[Lem.2.6.6]{McDuffSalamonSymplecticBook} pick a unitary basis for the fibre $\nu_c\cong \C^n$ so that $\omega_{vert}= \omega_{\C^n} = \sum dx_j \wedge dy_j$, then one finds that $\tfrac{1}{2} d\langle Iv,\cdot \rangle$ gives $\omega_{\C^n}$ on $T\nu_c$ (see a subsequent footnote for a similar calculation).} 
Let $\omega_C$ be the pull-back of $\omega$ to $C$, and let $K_C:=K|_C$.

\begin{lm}\label{Lemma get a periodic flow}
The horizontal lift $\mathcal{K}$ of $X_K$ defines a $1/m$-periodic flow on $\nu_C$. 
\end{lm}
\begin{proof}
On $B$ the flow of $X_K$ is $1/m$-periodic. Pick a unitary trivialisation of $\nu_C$ along a closed orbit (this exists since the unitary group is path-connected).
Observe that the fibre coordinates (whose derivatives span the vertical vectors) remain constant under the flow of $\mathcal{K}$ since $\mathcal{K}$ is horizontal.
\end{proof}

Our key trick is to replace $\Omega$ by its average
$
\textstyle
\Omega_{\nu}:=\int_0^1 \varphi_{t}^*\Omega\, dt = m\int_0^{1/m} \varphi_{t}^*\Omega\, dt,
$
where $\varphi_{t}$ is the flow for time $t$ of $\mathcal{K}$. We then obtain a new two-form\footnote{the last equality uses that $\pi \circ \varphi_t = \pi \circ \varphi_{K_C}^t$, and the Hamiltonian flow $\varphi_{K_C}^t$ preserves $\omega_C$.}
$$
\textstyle
\omega_{\nu}:= \int_{0}^{1}\varphi_{t}^*(\pi^*\omega_C + \widetilde{\Omega})\, dt = \pi^*\omega_C + \Omega_{\nu}.
$$
\begin{lm}\label{Lemma omeganu form description}
The two-form $\omega_{\nu}$ is a symplectic form in a neighbourhood $\mathcal{N}$ of the zero section of $\nu_C$, such that on $T Y|_C$ vectors we have $\omega_{\nu}|_{TC}=\omega_C + \Omega|_C = \omega_C + \omega_{vert}|_C = \omega$.

On $\mathcal{N}$, the vector field $\mathcal{K}$ is a Hamiltonian vector field for $\omega_{\nu}$ and generates a $1/m$-periodic flow, which on $C$ coincides with the flow of $X_{K_C}$.

After shrinking $\mathcal{N}$ if necessary, a neighbourhood of $C\subset U$ (endowed with the symplectic form $\omega$) is symplectomorphic to $(\mathcal{N},\omega_{\nu})$.

\end{lm}
\begin{proof}
The form $\omega_{\nu}$ is closed by construction. That $\omega_C + \Omega|_C = \omega_C + \omega_{vert}|_C = \omega$ on vectors in $TY|_C$ follows by construction.\footnote{that $\Omega|_C=\omega_{vert}$ on $T\nu_C|_C$ follows from the fact that the splitting $(T\nu_C|_C)|_c=\nu_c \oplus T_c C$ at $c\in C$ is natural and independent of the connection, and agrees with the splitting one gets by first taking a local trivialisation of $\nu_C$ near $c\in C$. If we use a local unitary trivialisation (see \cite[Lem.2.6.6]{McDuffSalamonSymplecticBook}) then we can check $\tfrac{1}
{2}d(\omega_{vert}(v,\cdot))|_C=\omega_{vert}$: we have $\nu_c\cong \C^n$ with coordinates $x_j + i y_j$, $\omega_{vert}=\sum dx_j \wedge dy_j$, $v=\sum x_j \partial_{x_j} + y_j \partial_{y_j}$, $\omega_{vert}(v,\cdot)=\sum x_j dy_j - y_j dx_j$, so $\tfrac{1}
{2}d(\omega_{vert}(v,\cdot))=\sum dx_j \wedge dy_j=\omega_{vert}$. 
}
As $\omega$ is symplectic, this implies that $\omega_{\nu}$ is symplectic near the zero section.
As $\omega_{\nu}$ is invariant under the flow $\varphi_t$ of $\mathcal{K}$ by construction, the flow $\varphi_t$ is symplectic. Therefore $i_{\mathcal{K}}\omega_{\nu}$ is a closed form. The flow is exact precisely if that is also an exact form. We may assume $\mathcal{N}$ is a tubular neighbourhood of the zero section $C$ in $\nu_C$, so it deformation retracts to $C$, thus $[i_{\mathcal{K}}\omega_{\nu}]=0 \in H^1(\mathcal{N};\R)$ if and only if the pull-back $[i_{\mathcal{K}}\omega_{\nu}]|_C=0 \in H^1(C;\R)$. As $\mathcal{K}|_C=X_{K_C}$, the pull-back of that form is exact: $\omega_{\nu}(\mathcal{K}|_{C},\cdot)|_{TC}=\omega_C(X_{K_C},\cdot)=-dK_C$.
Thus $\mathcal{K}$ is a Hamiltonian vector field.
The periodicity of the flow is \cref{Lemma get a periodic flow}.
The final claim follows by a  general argument for any $I$-holomorphic submanifold $C\subset Y$.
A tubular neighbourhood of $C\subset Y$ can be identified with a neighbourhood of the zero section of the normal bundle of $C$, and the latter  
can be identified with $\nu_C$ as the fibre of $\nu_C$ is $I$-orthogonal to and complementary to $TC\subset TY$ (using that $C$ is an $I$-holomorphic symplectic submanifold, and $I$ is an $\omega$-compatible almost complex structure).
Thus we have two symplectic forms near $C$, $\omega_{\nu}$ and $\omega$,
which agree on $TC$. A symplectic version of the Weinstein neighbourhood theorem \cite[Thm.3.4.10]{McDuffSalamonSymplecticBook} now implies the final claim.
\end{proof}

\subsection{Construction of the $1$-periodic Hamiltonian flow in the local model}
\begin{lm}
After shrinking $U$ to a smaller neighbourhood of $C\subset Y$,
there is a $1/m$-periodic Hamiltonian flow on $(U,\omega)$ generated by a time-dependent Hamiltonian $k: U \to \R$, satisfying $X_k|_C=X_{K_C}$, $k|_C=K_C$, and $k-K$ is $C^1$-small near $C$.
\end{lm}
\begin{proof}
We shrink $U$ so that the symplectomorphism of the previous result is $\mathcal{N}\cong U$ (in particular this is a connected subset of $Y$).
Via this map, $\mathcal{K}$ defines a Hamiltonian vector field $X_k$ on $(U,\omega)$, where $k:U\to \R$ is determined up to addition of a constant ($\mathcal{N}\cong U$ is connected).
As the isomorphism $\mathcal{N}\cong U$ is the identity map on $C$, its differential maps $\mathcal{K}|_C=X_K|_C\in TC\subset TY|_C$ to $X_k|_C$ by the identity, so $X_k|_C=X_{K}|_C$.
Thus on $TY|_C$ we have $dk|_C=\omega(\cdot,X_k)|_C=\omega(\cdot,X_K)|_C=dK|_C$, so $d(k-K)=0$ on $TY|_C$. The claim follows (we add a constant to $k$ so that $k-K=0$ on $C$).\footnote{We remark that by the defining equation for the Hamiltonian vector fields, both $k,K$ vanish to first order in the normal directions to $C$ in $TY|_C$: they both equal $K_C\circ \pi$ to first order.
} 
\end{proof}

\begin{cor}
\label{Cor Ham circle action near Bott mfd}
There is a Hamiltonian circle action $p_t$ on $U$ with Hamiltonian $P=\tau\cdot k:U \to \R$, where $\tau = c'(K(B))\in \tfrac{1}{m}\Z$, 
and after adding a constant to $P$: $P|_B=H|_B$, $P-H$ is $C^1$-small near $B$,
$$p_t|_B=\varphi_H^t|_B,\quad dp_t|_{TB}=d\varphi_H^t|_{TB}, \quad dp_t|_{TC}=d\psi_{\tau t}|_{TC}, \quad dp_t^{-1} \circ I \circ dp_t|_{TC}=I|_{TC}.$$
\end{cor}
\begin{proof}
The claim about $P=\tau\cdot k$
follows immediately from the fact that $H=c(K)$ and $c'(K(B))=\tau\in \tfrac{1}{m}\Z$. The vector field $X_P=\tau X_k$ is time-independent (since via the symplectomorphism it corresponds to the time-independent $\tau \mathcal{K}$). This time-independence ensures the group property of the $1$-periodic flow $p_t$, so $p_t$ is a circle action.
As $B$ is connected, and $X_P|_B=X_H|_B$, we may add a constant to $P$ so that $(P-H)|_B=0$. As $d(k-K)|_{TC}=0$, $d(P-c(K))|_B=(\tau\, dk-c'(K)dK)|_B=\tau(dk-dK)|_B=0$, so $P-H$ is $C^1$-small near $B$.
The other claims are immediate, in particular $dp_t|_{TC}=d\psi_{\tau t}|_{TC}$, $I\cdot TC\subset TC$ and pseudoholomorphicity of $\psi_t$  imply the claim $I \circ dp_t|_{TC}=dp_t \circ I|_{TC}.$
\end{proof}

\subsection{Hamiltonian perturbations}\label{Subsection A comment about Hamiltonian perturbations}
We first recall some machinery due to Seidel \cite{Sei97}, following the notation from \cite[Sec.3.2]{R14}.
Let $P,G: Y \to \R$ be time-dependent Hamiltonians.
Let $p=p_t$ be the flow of $X_P$ for time $t\in \R$. The {\bf pull-backs} of $G$ and $I$ are
$$
p^*G := G_t\circ p_t - P_t \circ p_t, \qquad
p^*I := dp_t^{-1} \circ I \circ dp_t,
$$
which depend on $t\in \R$. 
In this notation, $p$ is {\ph } if and only if $p^*I=I$.
If $P$ generates an $S^1$-action (as in the example of symplectic $\C^*$-manifolds), then $t\in S^1 = \R/\Z$.
When $P$ is time-independent, we have $P\circ p_t=P$.
On Hamiltonian vector fields,  
$$
X_{p^*G} = dp^{-1} \cdot (X_G - X_{P}).
$$
For the Floer action $1$-form $d_x\mathcal{A}_G\cdot \xi := - \int \omega(\xi,\partial_t x - X_G)\,dt$ where $x:[0,1]\to Y$ and $\xi\in x^*TY$,
 $$p^*d\mathcal{A}_G :=d\mathcal{A}_G\circ dp = d\mathcal{A}_{p^*G}.$$ 
Suppose $P$ generates a $1$-periodic flow.\footnote{The letter $P$ emphasises that one needs a periodic flow $p$, otherwise $p_t\circ x$ need not be a loop: one would get a correspondence between Hamiltonian Floer theory for $G$ and fixed point Floer theory for the (Hamiltonian) symplectomorphism $\varphi_{p^*G}^1$. The latter involves counting Floer strips whose end-points are related by $\varphi_{p^*G}^1$, rather than Floer cylinders. If $P$ is time-independent, one can relate $HF^*(\varphi_{p^*G}^1)$ to the Hamiltonian Floer cohomology $HF^*(p^*G)$ by a continuation isomorphism (see \cite{DostoglouSalamon})
but we are avoiding continuation maps as we describe the chain complex.}
Then 
we get a 1-to-1 correspondence between $1$-orbits of $p^*G$ and $1$-orbits of $G$, via $x(t) \mapsto p_t\circ x(t)$, and non-degeneracy of one implies non-degeneracy of the other.
Under the transformation $x\mapsto p\circ x$ the grading changes by a shift, $|p\circ x| = |x|-2\,\textrm{Maslov}(p)$ for a certain Maslov index associated to $p$.
There is also a 1-to-1 correspondence between Floer trajectories for $(p^*G,p^*I)$ and those for $(G,I)$, via $v(s,t) \mapsto p_t\circ v(s,t)$, and transversality for one pair implies transversality for the other \cite[Lem.4.3]{Sei97}. Under that correspondence, energy is preserved.\footnote{This follows by rewriting the energy of a Floer trajectory $v$ of $(G,I)$ as $E(v)=\int \|\partial_s v\|^2\,ds\wedge dt = -\int d_v\mathcal{A}_{G}(\partial_s v)\, ds$. For $(p^*G,p^*I)$ we use the time-dependent Riemannian metric $\omega(\cdot,p^*I \cdot)$ to compute the energy.}

Returning to our local model via \cref{Cor Ham circle action near Bott mfd}, consider a \textit{time-dependent} perturbation\footnote{In view of the periodicity of $p$, we may take $t\in S^1=\R/\Z$.} of $H$ on $U$,
$$
\widetilde{H}= \widetilde{H}_t := H+\epsilon f_{\pi}\circ p_t^{-1}
$$
where $\epsilon>0$ is a small constant, $\rho$ is a cut-off function supported in a tubular neighbourhood $\mathcal{N}$ of $B\subset U$,
$f:B\to \R$ is a Morse function which we extend to a function $f_{\pi}:\mathcal{N} \to \R$ which is constant in normal directions to $B$. Explicitly, 
identify $\mathcal{N}$ with a neighbourhood of the zero section of $\nu_C$, we project with $\pi: \nu_C\to C$ and then we use the $\R_+$-action to project from $C$ to $B$ near $B$. Let
$$p^*\widetilde{H}=p_t^*H + \epsilon f_{\pi} = 
(H - P_t)\circ p_{t}+\epsilon f_{\pi}. 
$$
By construction,
 \;$p^*\widetilde{H}|_B=\epsilon f_{\pi}$,\; $X_{p^*\widetilde{H}}|_B=\epsilon X_{f_{\pi}}|_B$,\;
$dp_t|_{TB}=d\varphi_H^t|_{TB}$,\; and\; $p^*I|_{TC}=I|_{TC}$. 

Let $\nabla^B f$ denote the gradient of $f: B \to \R$ with respect to the restricted Riemannian metric $\omega(\cdot,I\cdot)|_{TB}$, and $\nabla f_{\pi}$ the gradient of $f_{\pi}:U \to \R$ with respect to $\omega(\cdot,I\cdot)$.
\begin{lm}\label{Lemma gradient trajectories are Floer solns}
$\nabla f_{\pi}|_B = \nabla^B f \in TB$. If $v(s,t)$ is a Floer solution of $(p^*I,p^*\widetilde{H})$ and $\mathrm{Image}(v)\subset B$, then its Floer equation 
is $\partial_s v + I\partial_t v=-\epsilon \nabla^B f$.
\end{lm}
\begin{proof}
As $f_{\pi}=f\circ \pi$ is constant in the normal directions to $B$, $\nabla f_{\pi}$ is perpendicular to the normal fibres to $B$,
so $\nabla f_{\pi}|_B\in TB$. 
So $\nabla f_{\pi}|_B$ is determined by the restricted metric $g|_B=\omega(\cdot,I\cdot)|_{TB}.$ By definition, $g|_B(\cdot,\nabla f_{\pi})=df_\pi|_{TB}=df=g|_B(\cdot,\nabla^B f).$ Thus $\nabla f_{\pi}|_B=\nabla^B f$.
Also $$X_{p^*\widetilde{H}}|_B=\epsilon X_{f_{\pi}}|_B=\epsilon I \nabla f_{\pi}|_B=\epsilon I \nabla^B f|_B\in TC.$$ 
The Floer equation for $(p^*I,p^*\widetilde{H})$ is $\partial_s v + (p^*I) (\partial_t v-X_{p^*\widetilde{H}})=0$.
Recall $(p^*I)|_{TC}=I$. As $v\subset B$, $\partial_t v\in TB$ so $(p^*I)\partial_t v= I \partial_t v$.
Finally, $(p^*I)X_{p^*\widetilde{H}}|_B=-\epsilon\nabla^B f$.
\end{proof}

\begin{de}\label{Definition expected trajectory}
A Floer trajectory $u(s,t)=p_t(v(s))$ of $(I,\widetilde{H})$ 
corresponding to a $-\epsilon \nabla f$-trajectory $v=v(s)$ in $B$ is 
an {\bf expected trajectory}.\footnote{Explicitly: $\partial_t u = X_H$ and the Floer equation is $\partial_s u =  -\epsilon \nabla (f_{\pi}\circ p_t^{-1})=-\epsilon \,dp_t \nabla^B f$.} It consists of the associated $1$-orbits in $B$ for points along the $-\epsilon \nabla f$-trajectory $u(s,0)=v(s)$.
As the energies of $v,u$ agree, they have energy $\epsilon(f(x_-(0))-f(x_+(0)))$.
Define the {\bf relative energy} of a Floer solution $u(s,t)$ for $(\widetilde{H},I)$ by\footnote{the last equality uses that $u$ satisfies the Floer equation for $\widetilde{H}$.}
$$
\textstyle
E_{\mathrm{rel}}(u)=E(u) - \int (\widetilde{H}_t(x_-(t)) - \widetilde{H}_t(x_+(t)))\, dt = \int u^*\omega.
$$
 As $H$ is constant on $B$, and $x_{\pm}$ are $1$-orbits of $H$ in $B$, that 
$\widetilde{H}$-integral equals $\epsilon(f(x_-(0))-f(x_+(0)))$.
For expected trajectories the latter also equals $E(u)$, so 
$E_{rel}(u)=0$.
\end{de}

\begin{thm}\label{Thm Floer traj near Bi new perspective}
There is a $1:1$ correspondence
$$
\{1\textrm{-orbits of } p^*\widetilde{H}\} \stackrel{1:1}{\longleftrightarrow}
\{1\textrm{-orbits of } \widetilde{H}\},\; x(t) \mapsto p_t\circ x(t).
$$
For small $\epsilon>0$, and shrinking $U$ if necessary, the $1$-orbits of $p^*\widetilde{H}$ are precisely the critical points of $f:B\to \R$ and they are non-degenerate $1$-orbits. 
The $1$-orbits of $\widetilde{H}$ are the critical $1$-orbits, they are non-degenerate, and their grading is
\begin{equation}\label{Equation grading MorseBott appendix theorem}
|x|=\mu_f(x_0) - RS(B,\tau K) +\codim_{\C}\, C,
\end{equation}
where $RS(B,\tau K)$ is the {\RS } index contribution for the flow of $\tau K$, where $\tau:=c'(K(B))$.

Similarly, we obtain a correspondence for Floer trajectories in $U$,
$$
\{\textrm{Floer trajectories of } (p^*\widetilde{H},p^*I)\} \stackrel{1:1}{\longleftrightarrow}
\{\textrm{Floer trajectories of } (\widetilde{H},I)\} ,\; v(s,t) \mapsto u(s,t):=p_t\circ v(s,t).
$$
Under this correspondence, $-\epsilon \nabla f$-trajectories $v(s,t)=v(s)$ in $B$
correspond to the expected trajectories.

If one can exclude sphere bubbling in $Y$ near $B$ (see \cref{Definition simple Bott mfds atoroidal Mi}) then for small $\epsilon>0$, and shrinking $U$ if necessary, all Floer trajectories in $U$ are expected trajectories.
Otherwise, this holds at least for Floer trajectories of sufficiently small relative energy.
\end{thm}
\begin{proof}
The $1$-orbits of $H$ are in $1:1$ correspondence with the $1$-orbits of $p_{t}^*H = (H - P)\circ p_t$ via $x \mapsto p_{-t}\circ x(t)$. The $1$-orbits of $H$ in $U$ are labelled by their initial point in $B$, and they are precisely the $1$-orbits of $p_t$ since $X_P|_B=X_H|_B$, so $p_{-t}\circ x(t)$ is constant. Thus the $1$-orbits of $p_{t}^*H$ on $U$ are constant loops at points in $B$, in other words the critical points $B=\mathrm{Crit}(p_{t}^*H)$.

Notice that we are now in the analogous setup of the proof of \cite[Prop.2.2]{CFHW}.\footnote{their notation is related to ours as follows: $\Delta_t=p_{-t}$, $K=-P$, $H_{\delta}=\widetilde{H}$, $h= f_{\pi}\circ p_{-t}$, $\delta=\epsilon$, $\hat{H}_{\delta}=p_{t}^*\widetilde{H}$.}
By a standard compactness argument \cite[Lem.2.1(1)]{CFHW}, given a small neighbourhood $U$ of $B$ inside $\mathcal{N}$,
for sufficiently small $\epsilon>0$ the $1$-orbits of $p^*\widetilde{H}=p_t^*H + \epsilon f_{\pi}$ in $\mathcal{N}$ are trapped in $U$ and $W^{1,2}$-close to the $1$-orbits of $p_t^*H$. As the latter are constant, we can now reduce to working in a Darboux chart around one such point. Thus we may run the calculation of \cite[Prop.2.2, Step 2]{CFHW}:\footnote{We will make a comment about the almost complex structure at the end of the proof.} %
for sufficiently small $\epsilon>0$ and provided the chart is small enough, the $1$-orbits of $p_t^*H + \epsilon f_{\pi}$ are constant orbits at the critical points of $f$. As $B$ is compact, we may assume we chose $\epsilon$ and a small enough tubular neighbourhood $\mathcal{N}$ of $B$, so that the $1$-orbits of $p_t^*H + \epsilon f_{\pi}$ are the critical points of $f$.

The computation of the grading of such a critical $1$-orbit is carried out as in \cite[Sec.3.3]{OanceaEnsaios} (which is based on \cite[Lem.2.2, Step 3.]{CFHW}), the main ideas are summarised in the proof of \cref{CZindecesOfMorseBottSubmanifolds}. 
That computation used a splitting of the tangent space arising from a contact distribution $\xi$; in our case the analogous splitting is \eqref{splittingTangentSpaceContact}:\;it uses $\xi:=(\C\cdot X_H)^{\perp}\subset TY$ (we replace $H_{\lambda}$ by $H$ in our current notation). The analogue of the $2\times 2$ shear matrix in the differential $d\varphi_t(z)$ of the flow in \cite[Sec.3.3]{OanceaEnsaios} in our case is \cref{LemmaForTheShear} (together with $(\phi_t^{H})_*X_{H}=X_{H}$). This shear contributes $\tfrac{1}{2}$ to the {\RS } index (using that $c''(K)>0$ along $B$), so the {\RS } index satisfies 
$$RS(x,H) = \tfrac{1}{2}+RS(B,H,\xi),$$
where $RS(B,H,\xi)$ only counts contributions from the differential of the flow of $X_H$ in $\xi$ directions along the critical $1$-orbit $x\subset B$ (and where by invariance properties of the RS index, $x$ can be replaced by any $1$-orbit in $B$, so we wrote $RS(B,\cdots)$). In the $\xi$ directions, that flow agrees with the flow of $\tau\cdot X_K$ where $\tau:=c'(K(B))$ (recall $H=c(K)$), so $RS(B,H,\xi)=RS(B,\tau K,\xi)$. As the differential of the flow of $\tau K$ is the identity on the summand $\xi^{\perp}=\C\cdot X_H$ which previously gave rise to a shear, we can just write $RS(B,\tau K)$.
The computation \cite[Sec.3.3]{OanceaEnsaios} (equivalently \cite[Lem.2.2, Step 3.]{CFHW}) also shows that adding a Morse function term $\epsilon f_{\pi}$ to a Hamiltonian will yield an additional contribution to the {\RS } index that is half the signature of the Hessian of $f:B\to \R$ at the critical point $x_0\in \mathrm{Crit}(f)$ corresponding to the critical $1$-orbit. By definition, the signature equals $\dim_{\R} B - 2\mu_f(x_0)$, where $\mu_f(x_0)$ is the Morse index. 
Although gradings change by a Maslov index shift when applying $p$, in our case we only care about how the relative grading of a critical $1$-orbit $x$ changes when passing from an unperturbed Hamiltonian to a Hamiltonian perturbed by an auxiliary Morse function, so it suffices that we justified the shift by the Morse index $\mu_f(x_0)$ after applying $p$.
In summary, following the grading conventions of \cite[App.C]{McLR18} and \cite[App.B]{BeRit20}, we proved:
\begin{equation}\label{Grading equation in appendix}
\begin{split}
|x| & := \dim_{\C} Y - RS(x,\widetilde{H})
\\
& = \dim_{\C} Y -
(RS(x,H)+ \tfrac{1}{2}\dim_{\R} B - \mu_f(x_0))
\\
& = \dim_{\C} Y -
(RS(x,\tau K,\xi)+ \tfrac{1}{2}+ \tfrac{1}{2}\dim_{\R} B - \mu_f(x_0))
\\
& =\mu_f(x_0) - RS(x,\tau K,\xi) + \dim_{\C} Y - \tfrac{1}{2}\dim_{\R} B - \tfrac{1}{2},
\end{split}
\end{equation}
or succinctly: $\mu_f(x_0) - RS(x,\tau K) +\codim_{\C}\, C$.

We now wish to run the analogue of the computation \cite[Lem.2.2, Step 4.]{CFHW} for Floer trajectories, and the key is to reduce the problem to a neighbourhood whose homotopy type is $S^1$ on which we have a symplectic trivialisation -- then the cited computation will apply.

First recall a standard result\footnote{One way to see this is to use the monotonicity lemma to show that a small energy sphere must be trapped in a small Darboux chart, and then use exactness of $\omega$ on the chart to prove the sphere is constant by Stokes' theorem.} in Floer theory, that 
there are no non-constant {\ph } spheres of arbitrarily low energy in a compact symplectic manifold (e.g.\,\cite[Lem.3.2]{Sa97}). 

Let $E_2$ be the minimal energy of a {\ph } sphere near $B$. Let $E_0$ be as in \cref{Lemma small energy implies cascades are near Bott mfds}. Then
pick $E_1>0$ and $\epsilon>0$ small enough, so that $$E_1+\epsilon(\max f - \min f)< \min(E_0,E_2).$$
Call {\bf small-drop} any Floer trajectory for $\widetilde{H}$ with relative energy $E_{\mathrm{rel}}\leq E_1$.
For sufficiently small $\epsilon>0$, we may assume that small drops $u$ are very close to $B_i$. Indeed once we fix the size of a compact subset of the infinite cylinder $\R \times S^1$ (up to $s$-translation) a small drop restricted to this subset will land in a small neighbourhood of a $1$-orbit of $H$ in $B_i$. This follows from a Gromov-compactness argument: a subsequence of drops $u_n$ for $\epsilon=\epsilon_n \to 0$ will converge in $C^{\infty}_{loc}$ to a drop for $H$, and the latter is trivial by \cref{Lemma small energy implies cascades are near Bott mfds}. Here note that we avoided bubbling phenomena for Gromov-limits of small drops, in view of the minimal sphere-energy $E_0$ condition.

So far we have proved that small-drops $u$ have Gromov-limits in $B$ as $\epsilon\to 0$. Via the unwrapping $v(s,t)=p_t^{-1}(u(s,t))$, we deduce that drops of $(p^*\widetilde{H},p^*I)$ of small relative energy also have Gromov-limits in $B$ (since $p_t$ preserves the set $B$).

We claim that for a sufficiently small $\epsilon>0$, $E_1>0$, and a sufficiently small neighbourhood $U$ of $B$, the only drops of $(p^*\widetilde{H},p^*I)$ in $U$ are the expected trajectories. It will be helpful in the argument to assume that $U,E_1$ are small in relation to $\epsilon>0$, so let us say\footnote{using that $|X_{p_t^*H}||_B=0$ so $|X_{p_t^*H}|$ is arbitrarily small on $U$ by making $U$ a sufficiently small neighbourhood of $B$.}
\begin{equation}\label{Equation relation between n dependent data}
  \sup_U \epsilon^{-2} |X_{p_t^*H}| \leq 1 \qquad \qquad \epsilon^{-1} E_1 \leq 1.
\end{equation}

We argue by contradiction. Suppose $u_n$ are small-drops for $n$-dependent data $\epsilon=\epsilon_n \to 0$, $E_1=E_{1,n}\to 0$, and neighbourhoods $U=U_n$ approaching $B$, satisfying \eqref{Equation relation between n dependent data}.
Let $v_n=p_{-t}(u_n)$ be the unwrapping of $u_n$. 
Pick positive integers $\delta_n\to \infty$ such that $\delta_n \epsilon_n \to 1$, in particular also $\sup_{U_n}\delta_n |X_{p_t^*H}| \to 0$ by \eqref{Equation relation between n dependent data}.
The $\delta_n$-fold covers
$$
w_n: \R \times [0,1] \to Y, \; w_n(s,t):=v_n(s\delta_n,t\delta_n)
$$
satisfy the Floer equation $\partial_s w_n + (p^*I)(\partial_t w_n-\delta_n X_{p_t^*H}) = -\delta_n \epsilon_n \nabla f_{\pi}$.
The key observation is that
\begin{equation}\label{Equation nabla fpi goes to nabla f}
-\delta_n \epsilon_n \nabla f_{\pi}\to -\nabla^B f,
\end{equation}
by \cref{Lemma gradient trajectories are Floer solns}, and $\delta_n X_{p_t^*H}\to 0$, both uniformly on the image of $w_n$ in $U_n$, as $n\to \infty$. 
Their energy is bounded since\footnote{we are using that the $L^2$-norm of $u_n(s,t)$ is invariant under reparametrisation, and that the rescaling $\R\times S^1 \to \R \times S^1$, $(s,t)\mapsto (s\delta_n,t\delta_n)$ is bijective in the $\R$-direction and a $\delta_n$-fold cover in the $S^1$-direction.}
$$
E(w_n) %
= \delta_n E(u_n) = \delta_n E_{rel}(u_n) 
+ \delta_n \epsilon_n (f(x_{n,-}(0))-f(x_{n,+}(0))),
$$
and using that $\delta_n E_{rel}(u_n) \leq \delta_n E_{1,n}$ and $\delta_n \epsilon_n(\max f - \min f)$ are both bounded.
So a subsequence of the $w_n$ has a (possibly broken but without bubbling) Gromov limit $w_{\infty}$ lying in $B$. As $w_n(s,t)$ is invariant under $t\mapsto t+\delta_n^{-1}$, it follows that $w_{\infty}$ has the same property for a subsequence of values $n$ that grow to infinity (and $\delta_n^{-1}\to 0$). By continuity in $t$, it follows that $w_{\infty}$ is time-independent. Thus $w_{\infty}$ is a $-\nabla f$ trajectory in $B$. Restricting $n$ to the above subsequence for large $n$, $v_n$ is close to the $-\epsilon_n\nabla f$ trajectory $w_{\infty}(s\delta_n^{-1})$, and $u_n(s,t)=p_tv_n(s,t)$ is close to the expected trajectory associated to $w_{\infty}(s\delta_n^{-1})$.
In particular, we may assume that $v_n$ lies in a tubular neighbourhood of such a negative gradient trajectory, so the neighbourhood has the homotopy type of a line. Applying $p_t$ we deduce that the $u_n(s,t)=p_tv_n(s,t)$ lie in a neighbourhood whose homotopy type is that of a circle. So far, this proves that for small $\epsilon,U,E_1$ (in the sense above) a drop for $\widetilde{H}$ must lie in a neighbourhood of an expected trajectory and the neighbourhood can be assumed to have the homotopy type of a circle.

Any symplectic vector bundle on a neighbourhood with this homotopy type is trivial.\footnote{since the symplectic group deformation retracts to the unitary group, and the latter is path-connected.}
Now we can apply the same calculation as in the proof of \cite[Lem.2.2, Step 4.]{CFHW} (we comment on the almost complex structure later). It shows that
after applying $p_t$, for small enough $\epsilon>0$, the only Floer solutions of $(p^*\widetilde{H},p^*I)$ in a small neighbourhood of a negative gradient trajectory of $\epsilon f$ are also negative gradient trajectories of $\epsilon f$ (lying in $B$). Thus, before applying $p_t$, the Floer trajectories for $(\widetilde{H},I)$ in a neighbourhood of an expected trajectory are also expected trajectories, as required. %

Finally, we make a technical comment about the almost complex structure in the proof of \cite[Prop.2.2]{CFHW}. In their argument, after changes of coordinates, they use a standard complex structure $J_0$. The actual $J$ of the proposition \cite[Prop.2.2]{CFHW} is the time-dependent almost complex structure induced by $J_0$ after transforming back to the original coordinates. By \cite[Lem.2.1]{CFHW} the local Floer homology groups are invariant under changing the almost complex structure in this way.
We instead apply their Steps 2 and 4 by replacing the standard almost complex structure $J_0$ by the time-dependent almost complex structure $I_t:=p^*I$.
Their space $N$ of constant solutions corresponds to our {\MB } manifold $B$.
Their operator $F$ becomes $F(x)=-I_t(\partial_t x - X_{p^*H})$, and their gradient function ``$f(x)(t)=h'(x(t))$'' is $I_t X_{f_{\pi}}$ in our setting (in particular, $I_t X_{f_{\pi}}|_B=-\nabla^B f$).
Their Step 2 then proves that the only $1$-orbits of $p^*\widetilde{H}$ are constant orbits at critical points of $f$. 
In Step 4, the same 
Fredholm theory still applies in the time-dependent case (see \cite{Salamon-Zehnder}). In our final argument above we work in a tubular neighbourhood of a given $-\nabla f$ trajectory $w_{\infty}$ in $B$, so we can use $\R^{2n}$ to parametrise it (this plays the same role as the $\R^{2n}$ in the proof of \cite[Prop.2.2, Step.2]{CFHW}, in particular
the same Hilbert spaces $E,G$ as in 
their proof can be used). 
The time-dependence does not affect the argument, in particular does not affect the validity of their usage of \cite[Prop.3.14]{RS95} which is a very general. In particular, the assumption of \cite[Prop.3.14]{RS95} holds because in our case $\|\dot{A}(t)\|$ is small when $\epsilon$ is small -- as explained in \cite[bottom of p.38]{CFHW} with $A(t)$ being their $DF(x_{a(s)})|_{W\cap E}$ (with their $s$ being the $t$ of $A(t)$, and their $\delta$ being our $\epsilon$) -- and $\|A(t)^{-1}\|$ is bounded independently of $\epsilon$ by $\max_{x_a\in B}\|DF(x_{a})|_{W\cap E}^{-1}\|$, using that $B$ is compact.
\end{proof}
\section{{\MBF } theory: orientations}\label{Section {\MBF } theory: orientations}
\subsection{Overview of the orientation issue}
\label{Appendix orientations Overview of the orientation issue}
In \cref{Thm Floer traj near Bi new perspective} we established that Floer trajectories of a specific small perturbation $\widetilde{H}$ of $H$, near the {\MB } manifold $B$ and for small $\epsilon>0$, are in $1:1$ correspondence with $-\epsilon\nabla f$ trajectories in $B$. However, the orientation sign used in the Floer complex to count rigid Floer trajectories need not agree with orientation signs used in the Morse complex (e.g.\,the phenomenon of ``bad orbits'' \cite[Lem.4.29]{Bourgeois-Oancea2}). Subject to conditions on the {\MB } manifold $B$ that we discuss below, {\KvK } \cite[pp.220-225]{KwonvanKoert} carried out a detailed construction of a local system of coefficients $\mathcal{L}_B$ on $B$ (with values in  the orthogonal group $O(1)=\{\pm 1\}$) that one must use for the Morse complex $MC^*(B,\epsilon f;\mathcal{L}_B)$, in order for this complex to agree with the local (and small relative energy) Floer complex $CF^*_{loc}(U,\widetilde{H};I)$ of a small neighbourhood $U$ of $B\subset Y$.

What we just stated is slightly different from what is stated in {\KvK } \cite[p.220]{KwonvanKoert}, so we briefly explain our understanding of that argument.\footnote{On p.220 of \cite{KwonvanKoert}, it talks about rigid gradient flow lines of $\widetilde{K}$ (our $p^*\widetilde{H}$) solving
 $\partial_s u + J \partial_t u = -\nabla \widetilde{K}$, but 
this is not quite meaningful since $\widetilde{K}$ is \textit{time-dependent}. For the same reason, one cannot talk about Morse homology for $\widetilde{K}$.} %
For the Morse complex, one uses the time-independent metric $g:=\omega(\cdot,I\cdot)$, not the time-dependent one induced by $p^*I$, and we may assume that $\omega(\cdot,I\cdot)|_B$ is {\MS } for $f:B\to \R$ by perturbing $f$ if necessary. We are trying to prove that Morse cohomology $MH^*(B,\epsilon f,g,\mathcal{L}_B)\cong HF^*_{loc}(U,p^*\widetilde{H},p^*I)$ (the latter here is understood to mean we only consider small relative energy solutions).
By making $\epsilon>0$ small, the Hamiltonian $\epsilon f$ will be sufficiently $C^2$-small to force the (small energy) Floer complex for $(U,\epsilon f_{\pi},I)$ to coincide with the Morse complex for $(B,\epsilon f,g)$ by Salamon--Zehnder \cite{Salamon-Zehnder}. This in turn coincides, by \cref{Thm Floer traj near Bi new perspective}, to the complex $CF^*_{loc}(U,p^*\widetilde{H},p^*I)$ apart from orientation signs. We get an identification of complexes 
\begin{equation}\label{Eq iso of CF cxes appendix}
CF^*_{loc}(U,p^*\widetilde{H},p^*I)\equiv CF^*_{loc}(U,\epsilon f_{\pi},I,\mathcal{L})
\end{equation}
provided one uses a local system $\mathcal{L}$ to correct the orientation signs, so that on the right the same signs are used in solution counts as those prescribed on the left. 

The local system constructed in \cite[pp.221-223]{KwonvanKoert} only involves $H$, rather than the perturbation $\widetilde{H}$, even though in \cite[p.220]{KwonvanKoert} it is described as being induced by the coherent system of orientations $\mathcal{O}$ for $\widetilde{H}$, so we will try to explain our understanding of this step. Once one fixes a system of coherent orientations for Floer theory (e.g.\,see \cite[App.B.11]{R13}), then orientation signs are determined on any Floer complex by considering the asymptotic operators associated to Floer's PDE equation over the $1$-orbits. 
The coherent system induces orientation signs for $CF^*_{loc}(U,\widetilde{H},I)$, and via $p^*$ these induce the signs to use on the left in \eqref{Eq iso of CF cxes appendix}, which in turn determines $\mathcal{L}$. The orientation sign of a Floer trajectory is essentially obtained by abstractly capping off the asymptotic operators at the two ends, to obtain a Fredholm operator defined over $\C P^1$, and comparing its orientation with the canonical orientation of a Cauchy--Riemann operator over $\C P^1$ (more precisely, see \cite[App.B.5]{R13}). We essentially want to compare the orientations induced by using $(\widetilde{H},I)$ versus $(H,I)$. More accurately, $H$ is only {\MBF } so a {\bf stabilisation} is needed to get a Fredholm operator \cite[Lem.B.5]{KwonvanKoert}.
We call it stabilisation because it adds a summand $\R^{\dim_{\R} B}$ to the domain of the Fredholm operator, denoted $\mathcal{T}_{\gamma} \Sigma\cong \R^{\dim \Sigma}$ in \cite[pp.221-222]{KwonvanKoert}, and arising as $\R^1$ in various notations $\underline{V}$, $\overline{V}$ or $V_{\pm}$ in {\BO } \cite[pp.155-156]{Bourgeois-Oancea2} where $B\cong S^1$. 

The key observation is that, after the abstract capping off, one can homotope the Fredholm operators obtained for $(\widetilde{H},I)$ to the stabilised ones obtained for $(H,I)$, by {\homotopying} $\epsilon$ to zero while switching on the stabilisation. Since they are homotopic, they induce the same orientation, so they will either both agree or both disagree with the canonical orientation of a Cauchy--Riemann operator. Thus they induce the same sign for a given expected Floer trajectory.

Finally, the local system $\mathcal{L}_B$ is induced by identifying $$CF^*_{loc}(U,\epsilon f_{\pi},I,\mathcal{L})\equiv MC^*(B,\epsilon f,g,\mathcal{L}_B),$$ and arises from studying the stabilised asymptotic operators for $(H,I)$ associated to expected Floer trajectories  \cite[pp.223-225]{KwonvanKoert}.
They worked in the setup of a Liouville manifold $Y$, with $c_1(Y)=0$, with $B$ arising inside a level set of a radial Hamiltonian, so that they could construct $\Delta_t=p_{-t}$. In \cref{Subsection local model near a Morse Bott Floer manifold} we constructed $p_t$ without any conditions on $B,Y$.
To discuss orientations, {\KvK } also needed a restrictive assumption called {\bf symplectic triviality}\footnote{Symplectic triviality of $TY|_B$, but also that this trivialisation over each $1$-orbit $\gamma$ in $B$ extends to a trivialisation over some filling disc for $\gamma$ living in $Y$ (the latter extension problem involves certain obstructions in $\pi_1(U(n))\cong \Z$).} \cite[p.218]{KwonvanKoert} that held for the spaces they were considering, but unfortunately does not usually hold for the spaces in our paper. 
We will prove that this assumption is \textit{not needed}, and that $c_1(Y)=0$ suffices.
\subsection{Admissible trivialisations}
From now on assume our symplectic manifold $(Y,\omega)$ satisfies $$c_1(Y):=c_1(TY;I)=0.$$
The real symplectic group deformation retracts to the unitary group, so investigating the triviality of the symplectic vector bundle $TY$ over subsets is essentially equivalent to studying the triviality of the bundle viewed as a complex vector bundle for the $\omega$-compatible almost complex structure $I$. 
Let $\det\, TY=\Lambda_\C^{top}(TY),$ %
and recall that $c_1(TY)= -c_1(T^*Y)=-c_1(\det\, T^*Y)$.
As $c_1(Y)=0$, we can pick a global non-vanishing smooth section $\Omega$ of $\det\, T^*Y$. If $v_1,\ldots,v_n$ are sections that trivialise $TY$ over some subset $X$ of $Y$, we call them {\bf admissible} if $\Omega(v_1 \wedge \ldots \wedge v_n)=1$. Any trivialising sections $v_i$ can be made admissible by rescaling $v_1$ by $1/f$, where $f=\Omega(v_1,\ldots,v_n): X \to \C^*$.
For a loop $\gamma:S^1 \to Y$, $c_1(\gamma^*TY)=\gamma^*c_1(Y)$ (and similarly for $\det$). A basis of sections $v_1,\ldots,v_n$ trivialising $\gamma^*TY$ is admissible if $\Omega|_{\gamma(z)}(v_1(z),\ldots,v_n(z))=1$ for all $z\in S^1$. For a smooth map $u:S \to Y$ from a surface, $c_1(u^*TY)=u^*c_1(Y)$, and we define admissible trivialisations of $u^*TY$ analogously. 

\begin{lm}\label{Lemma symplectic triviality over 2 complex}
    If $X\subset Y$ is an open subset whose homotopy type is that of a $2$-complex (in the sense of CW complexes), then $TY|_X$ is symplectically trivial.
\end{lm} 
\begin{proof}
Recall that the isomorphism class of a complex vector bundle only depends on the base up to homotopy, so we may just consider a complex vector bundle $E$ over a $2$-complex $X$ with $c_1(E)=0$ ($E$ is the pull-back of $TY|_X$ to the $2$-complex via its homotopy equivalence with $X$). 
By obstruction theory, 
the top Chern class $c_r(E)$ of a complex vector bundle $E$ over a CW complex $X$ is the only obstruction to building a non-vanishing smooth section of $E$ on the $2r$-skeleton of $X$.
In our case, since $X$ is equal to its $2$-skeleton, we can split off trivial complex line bundle summands from $E$ until we get $E= L \oplus \varepsilon$ for a complex line bundle $L$, where $\varepsilon$ is a trivial complex subbundle. 
The obstruction for $L$ to have a non-vanishing smooth section is $c_1(L)$. Finally, $c_1(L)=c_1(L\oplus \varepsilon)=c_1(E)=0$.
\end{proof}

\begin{cor}\label{Cor admissible triv of TY near exp traj}
Let $X_0$ be the union of the (finitely many) rigid expected trajectories of $(\widetilde{H},I)$ in a {\MB } manifold $B$ (\cref{Definition expected trajectory}). Then there is a neighbourhood $X$ of $X_0\subset B$ such that the symplectic vector bundle $TY|_X$ is trivial, so it admits admissible trivialising sections $v_1,\ldots,v_n$.

Let $\gamma$ be a $1$-orbit of $H$ in $X$ and suppose there is a smooth filling $u:D^2 \to Y$, $u|_{S^1}=\gamma$. Then an admissible trivialisation of $TY|_X$ induces an admissible trivialisation of $\gamma^*TY$ which extends to an admissible trivialisation of $u^*TY$.
\end{cor}
\begin{proof}
The first claim follows by \cref{Lemma symplectic triviality over 2 complex}, picking $X$ to have the homotopy type of a CW $2$-complex. 
(This can be constructed explicitly: consider the compact embedded graph in $Y$ arising from the images of all rigid $-\nabla f$ trajectories in $B$, picking an open neighbourhood in $Y$ that deformation retracts onto this CW $1$-complex, %
then apply $(p_t)_{t\in S^1}$ to the neighbourhood.)
Recall we can make trivialising sections $v_i$ of $TY|_X$ admissible.
By \cref{Lemma symplectic triviality over 2 complex}, $u^*TY$ admits trivialising sections $w_i$; and make $w_i$ admissible (rescaling $w_1$).
If $A(t)\in GL(n,\C)$ is the change of basis from $v_i$ to $w_i$ over $\gamma(t)$,  $v_1\wedge \cdots \wedge v_n=\det\,A(t)\,w_1\wedge \cdots \wedge w_n \in \det\, \gamma^*TY$. Evaluate $\Omega$:\;$\det\, A(t)=1$ by admissibility. The obstruction to extending $v_i\in C^{\infty}(\gamma^*TY)$ to trivialising sections of $u^*TY$ is\footnote{the columns of $A(t):S^1 \to GL(n,\C)$ express the $w_i$ in the basis $v_i$, and we are trying to extend $A(t)$ to a map $D^2\to GL(n,\C)$ on the disc $D^2$, given $A(t)$ on $S^1=\partial D^2$.} the class $[A(t)]\in \pi_1(GL(n,\C))$. Recall $\pi_1(GL(n,\C))
\cong \pi_1(\C^*)\cong \Z$ via $[A(t)]\mapsto [\det\,A(t)]$. Thus they extend.
\end{proof}
\subsection{Construction of the local system}
\label{Subsection Construction of the local system}
\begin{thm}\label{Theorem Construction of the local system}
   Let $(Y,\omega)$ be any symplectic manifold with $c_1(Y)=0$. Let $H: Y \to \R$ be a Hamiltonian with a closed {\MB } component of $1$-orbits $B\subset Y$.
   Let $\widetilde{H}=H+\epsilon f_{\pi} \circ p_t^{-1}$ be a perturbation of $H$ as in \cref{Thm Floer traj near Bi new perspective}, working with small $\epsilon>0$, in a small neighbourhood $U$ of $B$, considering only small relative energy Floer solutions. We also assume that the Morse function $f: B \to \R$ is generic enough so that the Riemannian metric $g=\omega(\cdot, I\cdot)$ is {\MS } for $f$.
   Then there is a subset $X\subset B$ as in \cref{Cor admissible triv of TY near exp traj}, such that a $\{\pm 1\}$-valued local system $\mathcal{L}_X$ can be constructed on $X$
   as in {\KvK } \cite[pp.220-225]{KwonvanKoert}, so that
   \begin{equation}\label{Equantion CFloc is Morse}
   CF^*_{loc}(U,\widetilde{H}, I)\cong MC^*(B,\epsilon f,g,\mathcal{L}_X),
   \end{equation}
   noting that the Morse complex only requires the local system of coefficients to be defined in a neighbourhood of the rigid $-\epsilon \nabla f$ trajectories.

   The local system $\mathcal{L}_X$ corresponds to a class in $H^1(X;\Z/2)$, e.g.\,it is trivial if $H^1(X;\Z/2)=0$.

Finally, the pseudoholomorphicity of $\psi_t$ implies that $\mathcal{L}_X$ is trivial and can be ignored in \eqref{Equantion CFloc is Morse}.
\end{thm}
\begin{proof}
This follows from a combination of the construction of $p_t$ from \cref{Subsection local model near a Morse Bott Floer manifold} which is used in place of the $\Delta_{-t}$ of \cite[p.220]{KwonvanKoert}; 
using \cref{Thm Floer traj near Bi new perspective} to show that the chain complexes \eqref{Equantion CFloc is Morse} have the same generators and their differentials count the same moduli spaces except possibly for orientation signs;
and using \cref{Cor admissible triv of TY near exp traj} instead of the (ST) assumption in \cite[p.222]{KwonvanKoert} to build the local system (on which we comment further) that takes care of the orientation signs.

Regarding the local system construction: the operators in \cite[p.222]{KwonvanKoert} only depend on a neighbourhood of the expected trajectories, so it is not necessary to work with a neighbourhood of $B$. In \cite[p.222]{KwonvanKoert} a finite good cover $U_a$ of $B$ is used (e.g.\,strongly geodesically convex sets). In our setup it suffices to pick $U_a\subset X$ and covering $X_0$; compactness of $X_0$ yields a finite subcover; then replace $X$ by the union of the subcover (by \cref{Cor admissible triv of TY near exp traj}, $TY|_X$ is still trivial on the smaller set).

For the final claim, we need to adapt the argument of {\KvK } \cite[Lem.B.7]{KwonvanKoert}. For now, we assume as they do that
one (hence all) $1$-orbits in $B$ are contractible in $Y$.
Their argument uses a splitting argument from {\BO } \cite[Proofs of Prop.4.8 and 4.9]{Bourgeois-Oancea2}, which in turn appeared in Dragnev \cite[Lemma 2.3 and equation (3.1)]{Dragnev}.
In their setup, on a conical end of a Liouville manifold or a symplectisation, $TY=\R R \partial_R \oplus \R \mathcal{R} \oplus \xi$ where $R\partial_R$ is the Liouville field, $\mathcal{R}=X_R$ is the Reeb field, $\xi=\ker \alpha$ is the contact structure, and they assume that the almost complex structure $I$ preserves $\xi$ and satisfies the contact type condition $I(R\partial_R)=\mathcal{R}$. They use the split metric $g:=dR^{\otimes 2}+\alpha^{\otimes 2} + d\alpha(\cdot,I\cdot)$ and the Levi-Civita connection $\nabla$ for $g$. The discussion of the orientation sign associated to the asymptotic operator in Floer theory reduces to the computation of $S=I\nabla X_H$ for a radial Hamiltonian $H=c(R)$ \cite[p.225]{KwonvanKoert}.
By Leibniz's rule, $S=-c''(R)\,R\partial_R + c'(R)\widetilde{S}$, where $\widetilde{S}=I \nabla \mathcal{R}$ since $X_H=c'(R)\,\mathcal{R}$.
The whole argument in \cite[p.225]{KwonvanKoert} hinges on showing that the matrix $\widetilde{S}$ splits in the frame $R\partial_R, \mathcal{R}, v_i$ where $v_i$ is any unitary basis for $\xi$ (independent of the coordinate $R$), and this computation follows by Dragnev \cite[Lemma 2.3 and equation (3.1)]{Dragnev}.
We need to carry out this computation in our setup, to complete the proof.

By analogy, we make the following replacements:
\begin{align*}
&R  \;\leadsto\;  K, \qquad \mathcal{R}=X_R   \;\leadsto\;   X_K, \qquad R\partial R=\nabla R   \;\leadsto\;   \nabla K,\\
&\alpha=-R^{-1}\,dR\circ I = - \|\mathcal{R}\|^{-2}\,dR\circ I   \;\leadsto\;   \beta:=-\|X_K\|^{-2}\,dK\circ I, \\ 
&\xi = \ker \alpha \cap \ker dR   \;\leadsto\;   \ker \beta\cap \ker dK = (\R \nabla K \oplus \R X_K)^{\perp} 
\end{align*}
(where $\nabla$ is the gradient for the usual metric $\omega(\cdot,I\cdot)$, and $\|\cdot\|$ is the norm for the usual metric). We split $TY=\R \nabla K \oplus \R X_K \oplus \xi$, where $\xi=(\C\,X_K)^{\perp}$, and let $\pi_{\xi}:TY \to \xi$ denote orthogonal projection to $\xi$.
We define the metric
$$
g:=dK^{\otimes 2}
+ \beta^{\otimes 2}
+ \omega(\pi_{\xi}\cdot, I \pi_{\xi}\cdot).
$$
Let $\nabla$ be the Levi-Civita connection for $g$. We will prove that $\nabla X_K$ (hence $\widetilde{S}:=I \nabla X_K$) splits in the frame $\nabla K$, $X_K$, $v_i$, where $v_i$ is a unitary frame for $\xi$. We will adapt ideas from \cite[Lem.2.3]{Dragnev}.

Note that $\beta(X_K)=1$ so $X_K$ has unit length for $g$.
Recall $B$ lies in a regular level set $\Sigma:=K^{-1}(K_0)$ where $K_0=K(B)$. The local coordinate $K$ has level set isomorphic to $\Sigma$, and these level sets are integral manifolds for $\R\, X_K \oplus \xi$ (noting however that $\xi$ alone is usually not integrable).
So locally near $\Sigma$ we have a product manifold $(K_0-\delta,K_0+\delta) \times \Sigma$ with a split metric $g=dK^{\otimes 2} + g_{\Sigma}$. 
The level sets of $K$ map into each other via the flow of $\|\nabla K\|^{-2}\nabla K$. Recall in the local model we have a partially defined $\C^*$-action, so $X_{S^1}=X_K$ and $X_{\R_+}=\nabla K$ commute: $[X_K,\nabla K]=0$. We claim that
$$
\left[X_K,\|\nabla K\|^{-2}\nabla K\right]=0.
$$
By the Leibniz rule, it suffices to show that $X_K\cdot \|\nabla K\|^{2}=0$. But $\|\nabla K\|^{2}=\omega(\nabla K,I\nabla K)$, and $X_K$ preserves $\omega$, $\nabla K=X_{\R_+}$, and $I$. So $X_K$ preserves $\|\nabla K\|^{2}$, as required.

It follows that $X_K$, viewed in  $(K_0-\delta,K_0+\delta) \times \Sigma$, does not depend on the local coordinate $K$.
Some caution is needed because in general $\xi=X_K^{\perp} \cap T\Sigma$ may depend on $K$, as we do not assume that the flow of $\|\nabla K\|^{-2}\nabla K$ 
preserves the $\omega(\cdot,I\cdot)$-metric which defines the orthogonal space $X_K^{\perp}$. 

As in \cite[Eq.(3.1)]{Dragnev}, the Levi-Civita connection ``splits'': if  
$\tilde{k}_1$ is a vector field on $(K_0-\delta,K_0+\delta)$ independent of the coordinates of $\Sigma$, and $\tilde{\sigma}$ is a vector field on  $\Sigma$ independent of the coordinate $K$, 
$$\nabla_{(k_1,\sigma)}(\tilde{k}_1,\tilde{\sigma})=(\nabla_{k_1}\tilde{k}_1,\overline{\nabla}_{\sigma}\tilde{\sigma}),$$ 
where $\overline{\nabla}$ is the Levi-Civita connection on the level set $\Sigma_K:=\{K\}\times \Sigma$ for the restricted metric $g|_{\Sigma_K}$. 

{\bf Sub-claim:} $\nabla_v X_K\in \xi$
for any vector $v$.\\
\emph{Proof of Sub-claim.}
By the above splitting, 
and using that $X_K$ does not depend on $K$, we have
$\nabla_{\partial/\partial K} X_K=0$ and  $\nabla_v X_K=\overline{\nabla}_v X_K\in T\Sigma$ for any $v\in T\Sigma$.
As $g(X_K,X_K)=1$, we also have $0=v\cdot g(X_K,X_K)=g(2\nabla_v X_K,X_K)$, so $\nabla_v X_K\in \ker \beta$. Thus $\nabla_v X_K\in\ker \beta\cap T\Sigma = \xi$ for $v\in T\Sigma$.
\qedsymbol

To prove that $\nabla X_K$ splits in the frame mentioned above, we now show that $\nabla_{X_K} X_K$ (which is in $\xi$ by the sub-claim) is zero. We need to show that $\nabla_{X_K} X_K$ is orthogonal to any vector $v\in \xi$. 

Extend $v$ to a local vector field in $\xi$. 
As
 $X_K\cdot g(X_K,v)=g(\nabla_{X_K} X_K,v)+g(X_K,\nabla_{X_K} v)$, and $g(X_K,v)=0$,
we want $g(X_K,\nabla_{X_K} v)=0$. %
It suffices to show $\nabla_{X_K} v\in\xi$. By torsion-freeness, the latter equals $\nabla_{v} X_K +[X_K,v]$, and we already have $\nabla_vX_K\in \xi$, so we need $[X_K,v]\in \xi=T\Sigma \cap \ker \beta$.

Let $\sigma_i$ denote a local basis of vector fields for $T\Sigma\subset T((K_0-\delta,K_0+\delta) \times \Sigma)$. So $v=\sum f_i(K)\sigma_i\in \xi$ for some local functions $f_i(K)$ on $\Sigma$ that may also depend on $K$. Then $$[X_K,v]=\sum f_i(K)[X_K,\sigma_i]+\sum (X_K\cdot f_i)(K)\sigma_i,$$ which is in $T\Sigma$ since $[X_K,\sigma_i]\in T\Sigma$ (as $X_K,\sigma_i$ are $K$-independent vector fields on $\Sigma$).
To show $[X_K,v] \in \ker \beta$,
we use a general fact for $1$-forms: $d\beta(X_K,v)=X_K\cdot \beta(v)-v\cdot \beta(X_K) - \beta [X_K,v]$. As $\beta(v)=0$, $\beta(X_K)=1$, it remains to show  $d\beta(X_K,v)=0$. By Cartan formula, and $\beta(X_K)=1,$
$$d\beta(X_K,v)=(i_{X_K} d\beta)(v)=(\mathcal{L}_{X_K}\beta)(v)-d(i_{X_K}\beta)(v)=(\mathcal{L}_{X_K}\beta)(v).$$
To prove $\mathcal{L}_{X_K}\beta=0$ we now show that the flow of $X_K$ preserves $\beta$. We already saw that the flow of $X_K$ preserves $\|\nabla K\|^{2}=\|X_K\|^{2}$, and since it preserves $K,I$ it also preserves 
$\beta=-\|X_K\|^{-2}\,dK\otimes I$.

In conclusion, we have shown that $I\nabla X_K$ has the required form for the orientation-argument of {\KvK} \cite[p.225]{KwonvanKoert} to adapt to our setting.

Finally, we address the case when the $1$-orbits in $B$ are not contractible in $Y$. In this case, one needs to recall how coherent orientations are built, e.g.\,\cite[Appendix B.11 and B.17]{R13}: one uses abstract ``cappings'' of the operators living over $1$-orbits, i.e.\,one considers an abstract standard trivial bundle $\mathrm{triv}\to \C$ over $\C$ carrying a prescribed Fredholm operator, rather than a trivialisation for the geometrically obtained bundle $u_{\gamma}^*TY$ depending on a capping $u_{\gamma}:D^2 \to Y$ of $\gamma\subset B$ in $Y$. The asymptotic operator obtained in the admissible trivialisation of $\gamma^*TY$ is then used over $\{|z|\geq 1\}\subset \C$; one constructs the auxiliary sections $e_{\gamma,i}^a$ needed in the {\MB } setup as in \cite[p.222]{KwonvanKoert} extended over $\C$ by the same cut-off function argument; finally the ``stabilised'' operator \cite[(B5) p.223]{KwonvanKoert} is essentially the Fredholm operator \cite[Lem.B.5]{KwonvanKoert} that one uses on the abstract bundle $\mathrm{triv}\to \C$ to establish orientations by the standard methods \cite[Appendix B.11 and B.17]{R13}. 
\end{proof}

\subsection{Orientation signs for {\MBF } theory}
\label{Subsection Orientation signs for {\MBF } theory}
Any ledge $v:\mathrm{Int}\to B$ of a cascade can be viewed as a (partial) Floer trajectory $u=p_t\circ v:\mathrm{Int}\times S^1 \to B\subset Y$ of $\widetilde{H}$ (see \cref{Definition expected trajectory}). Thus a cascade can be viewed as an alternating succession of Floer cylinders\footnote{Keeping in mind that the Floer cylinders for $\widetilde{H}$ can be finite, semi-infinite, or infinite cylinders.} for $\widetilde{H}$ and for $H$, and the sum of their energies equals the energy of the cascade from \cref{Definition energy of a cascade}.

In \cref{Appendix orientations Overview of the orientation issue} we explained that in the construction of a coherent system of orientations, by a homotopy argument one may replace the Fredholm operators (and asymptotic operators) for $(\widetilde{H},I)$ by the ``stabilised'' versions for $(H,I)$ \cite[Lem.B.5]{KwonvanKoert}. In particular, these stabilised operators for $(H,I)$ on the above succession of Floer cylinders will now glue as the asymptotic operators agree at each crest and base of the cascade, and we can forget about the stabilisation where we glue. In fact, since (after the homotopy argument) we can glue the two operators over $1$-orbits, for this argument the stabilisation is irrelevant, so we do not require $c_1(Y)=0$ here. In summary, for the purposes of deciding orientation signs for {\MBF } theory, a cascade can now be treated as if it were just one cylinder, and one can run the usual axiomatic construction of a system of coherent orientations \cite[App.B]{R13}. Equivalently one can rephrase this using orientation lines \cite[App.B.17]{R13}. As anticipated in \cref{Subsection cascades drops and ledges}, notice that we are using a Floer-theoretic local system of coefficients for the ledges (called $\mathcal{L}$ in \cref{Appendix orientations Overview of the orientation issue}), rather than using Morse-theoretic orientation signs on ledges.

For simple cascades the resulting orientation sign for an expected Floer trajectory $u(s,t)$ in $B$ corresponds to the orientation sign prescribed by $\mathcal{L}$ in \cref{Appendix orientations Overview of the orientation issue}. When $c_1(Y)=0$,
this becomes the local system $\mathcal{L}_X$ of \cref{Theorem Construction of the local system} which prescribes an orientation sign to the $-\epsilon \nabla f$ Morse trajectory $u(s,0)$. So the simple cascades determine a twisted Morse complex, and in view of the pseudoholomorphicity of $\psi_t$ we have that $\mathcal{L}_X$ is trivial by \cref{Theorem Construction of the local system}.

This construction of orientation signs is also consistent with the continuation isomorphisms from \cref{Subsection {\MB } Floer cohomology agrees with perturbed Floer cohomology}. This follows by a homotopy argument analogous to \cref{Appendix orientations Overview of the orientation issue} where we related the operators for the perturbation $(\widetilde{H},I)$ and the stabilised operators for the {\MBF }  $(H,I)$. 
\subsection{Generalisation to all {\MB} manifolds}\label{Subsection generalisation to all Morse-Bott manifolds}

As promised in \cref{Subsection local model near a Morse Bott Floer manifold}, we now prove that there is a loop of Hamiltonian symplectomorphisms $p_t$ near any {\MB } manifold $B$ of $1$-orbits for a Hamiltonian $H: Y \to \R$ on a symplectic manifold $(Y,\omega)$, extending the $1$-periodic flow $\varphi_H^t|_B$. 
Recall in  \cref{Subsection local model near a Morse Bott Floer manifold} we built instead a Hamiltonian circle action $p_t$ which used stronger assumptions (the
local model). The following proof is instead very general and topological.
We thank Mark McLean for suggesting to prove this via homotopy methods by viewing the Lagrangians in $Y\times Y$.

 \begin{lm}\label{Lemma building Hamiltonian loop of symplectos}
 Let $B\subset Y$ be a closed {\MB } manifold of $1$-orbits of a Hamiltonian $H:Y\to \R$. Then there is a loop of Hamiltonian symplectomorphisms $p_t$ in a neighbourhood of $B$ with $p_t|_B=\varphi_H^t|_B$.
 \end{lm}
 \begin{proof}
 The diagonal
 $\Delta = \{(y,y):y\in Y\}$  
 and the graph
 $G_t=\{ (y,\varphi_H^t(y)): y\in Y \}$
 are two Lagrangians in $(Y\times Y,\omega \oplus -\omega)$. The {\MB } property says $\Delta \cap G_1=\Delta_B := \{(b,b):b\in B\}$ is a clean intersection.
 We first prove the result with the simplifying assumption that $(\varphi_H^1)_*|_{B}:TY|_B\to TY|_B$ has no negative real eigenvalues.
 Working close to $B$, $(\varphi_H^1)_*$ has no eigenvalue equal to $-1$, so the tangent space to $G_1$ does not contain a vector of type $(v,-v)\in TY\oplus TY$, thus $TG_1$ has no vector orthogonal to $T\Delta$. Therefore $TG_1$ is transverse to the fibres of $T^*\Delta\to\Delta$ near $\Delta_B$  (where by Weinstein's neighbourhood theorem we are viewing the normal bundle to the Lagrangian $\Delta\subset Y\times Y$ as $T^*\Delta$).
 It follows that $G_1$ is a graph over $\Delta$ near $\Delta_B$. 
 By the Lagrangian neighbourhood theorem, close to $\Delta_B$ we can therefore view the Lagrangian $G_1$ also as a graph in $T^*\Delta$, $$L_1:=\mathrm{gr}(\alpha)=\{(\delta,\alpha|_\delta)):\delta\in \Delta\}\subset T^*\Delta,$$ for a closed $1$-form $\alpha$ on $\Delta$ (defined near $\Delta_B$).
 We can homotope this through graphs $$L_{t}:=\mathrm{gr}(t \alpha)\subset T^*\Delta,\quad \textrm{for } t\in [0,1],$$ in particular $L_{t}$ is Lagrangian and we may view it as a subset of $Y \times Y$ defined near $\Delta_B$.

We want to show that $L_t$ is also a graph in $Y\times Y$ near $\Delta_B$.
So we need to check that the level set $\{y\}\times Y$ intersects $L_{t}\subset Y \times Y$ once transversely for $y$ close to $B$.
To prove this, it suffices to show that $T_{\delta}L_t$ for $\delta\in \Delta_B$ does not have non-zero vectors in $0\oplus TY\subset TY\oplus TY$. In the Weinstein neighbourhood theorem, the fibre of $T_{\delta}^*\Delta$ is identified with the normal directions to $\Delta\subset Y$ at $\delta.$ So tangent vectors to $L_t$ at $\delta$ have the form $(v,v)+t(-w,w)\in TY\oplus TY$ (where $(v,v)=\partial_s|_{s=0}\delta_s$, and $(-w,w)=\partial_s|_{s=0} \alpha|_{\delta_s}$ for a smooth curve $\delta_s \in Y$ with $\delta_0=\delta$). 
The transversality claim fails precisely if $v=tw$ occurs. Suppose the latter occurs. Writing $e:=v-w=(t-1)w$, then $(\varphi_H^1)_*(e)=v+w = (t+1)w=\tfrac{t+1}{t-1}e$ contradicts that $(\varphi_H^1)_*$ does not have a negative real eigenvalue.

 Thus $L_{t} = \{(y,\psi_{t}(y)): y\in Y\}\subset Y \times Y$ for a smooth map $\psi_{t}: Y \to Y$ close to the identity defined near $B$. It satisfies $\psi_{t}|_B=\mathrm{id}$ since $\Delta_B=\{\delta \in \Delta:\alpha|_{\delta}=0\}$.
 As $L_{t}$ is Lagrangian in $(Y\times Y,\omega\oplus -\omega)$, the map $\psi_{t}$ is symplectic. We claim that it is also Hamiltonian. 
 Its generating vector field $v_t(y)=\partial_t (\psi_t(y))$ (which is typically time-dependent) yields a closed form $\omega(v_t,\cdot)$ since $\psi_t$ is symplectic, and we need to show the form is exact.
 We may assume that we work in a tubular neighbourhood of $B$, which deformation retracts to $B$. So the obstruction to exactness is the class $[\omega(v_t,\cdot)]\in H^1(B;\R)$. But $\omega(v_t,\cdot)|_B=0$ since $v_t(b)=\partial_t (\mathrm{id}(b))=0$ for $b\in B$. So $v_t=X_{K_t}$ for some Hamiltonian $K_t: Y \to \R$.
Our problem is now essentially solved by the Lagrangian 
$$F_t :=(\mathrm{id}\times \psi_t^{-1})G_t=\{(y,\mu_t(y)):y\in Y\},\quad \textrm{where }\mu_t:=\psi_t^{-1}\varphi_H^t.$$
Note $\mu_t|_B=\varphi_H^t|_B$ as $\psi_t|_B=\mathrm{id}$; $\mu_0=\mathrm{id}$ as $L_0=\Delta$; and $\mu_1=\mathrm{id}$ since $F_1=(\mathrm{id}\times \psi_1^{-1})L_1=\Delta$. So $\mu_t$ is a loop of symplectomorphisms. It is Hamiltonian as the generating vector field is $\partial_t \mu_t = d\psi_t^{-1} X_H-v_t = X_{H\circ \psi_t} - X_{K_t} = X_{H\circ \psi_t - K_t }$.

Coming back to the simplifying assumption: we claim that we can achieve that assumption after a Hamiltonian isotopy $\varphi_h^1$ supported near $B$, which fixes $B$ (say, by imposing $h=0$ on $B$). Once this is achieved, we replace $\varphi_H^t$ by $\varphi_h^1\circ \varphi_H^t$ in the above argument, to get $p_t$ with $p_t|_B=\varphi_h^1\varphi_H^t|_B=\varphi_H^t|_B$ (as $\varphi_h^1|_B=\mathrm{id}_B$). To build the Hamiltonian isotopy, we first consider the local problem. In a neighbourhood of a point in $B$, a generic small Hamiltonian isotopy $\varphi_h^1$ (with $h|_B=0$) will ensure\footnote{In local coordinates $\R^b\times \R^c$, with $\R^b\times 0$ parametrising $B$, and $\R^c$ the normal coordinates, let $h(x,y)=y_1\cdot f(x,y)$ (using a bump function to make $f=0$ away from the neighbourhood). Then $h(B)=0$, and $X_h = f X_{y_1} + y_1 X_f$. 
Note that $X_{y_1}$ is independent of $f$, so we have complete freedom to pick $f$ to modify the flow of $X_h$ outside $B$. So generic $f$ will cause the negative eigenvalues to vary generically in $\C$, subject only to the symmetries of eigenvalues of Hamiltonian flows. The zero eigenvalues precisely come from the eigenvectors $TB$, and those we leave unchanged.} 
that $(\varphi_h^1\circ \varphi_H^1)_*=(\varphi_h^1)_*\circ (\varphi_H^1)_*$ does not have eigenvalues in $(-\infty,0)\subset \C$. As $B$ is compact, a finite number of such locally supported perturbations will ensure the global statement. Note in our setting that not having negative real eigenvalues is a generic property: the zero eigenvalues precisely come from the eigenvectors $TB$ and those we leave unchanged, and eigenvalues outside of $(-\infty,0]$ will vary continuously in the perturbation parameter.
 \end{proof}

In the absence of the assumptions of \cref{Lemma pseudoholo implies vf is still in TB}, it is unclear how to obtain \cref{Lemma gradient trajectories are Floer solns} and \cref{Thm Floer traj near Bi new perspective}.
There is a work-around, but at the cost of working at the cohomology level: 
$$HF^*_{loc}(p^*\widetilde{H},p^*I)\cong HF^*_{loc}(p^*\widetilde{H},I),$$
via a continuation isomorphism, by choosing a homotopy $I_s$ from $I$ to $p^*I$ within the space of $\omega$-compatible almost complex structures.
This can be done locally near $B$ (so $I_s=p^*I$ away from $B$), by using a cut-off function supported near $B$ applied to the $s$-parameter of a chosen homotopy $I_s$.
Note that we keep $\omega$ unchanged, so the periodic orbits have not changed.
In the proof of \cref{Thm Floer traj near Bi new perspective}, the crucial step \eqref{Equation nabla fpi goes to nabla f} holds automatically, as the Floer equation is
$\partial_s w_n + I (\partial_t w_n - \delta_n X_{p_t^*H})=\delta_n \epsilon_n IX_{f_{\pi}}$ and $IX_{f_{\pi}}|_B=-\nabla f_{\pi}|_B=-\nabla^B f$ (see the proof of \cref{Lemma gradient trajectories are Floer solns}).
The \cite[Lem.2.2, Step.4]{CFHW} argument is then applied to the pair $(p^*\widetilde{H},I)$ rather than $(p^*\widetilde{H},p^*I)$.
The local system from \cref{Equantion CFloc is Morse} can be constructed under the additional assumption $c_1(Y)=0$ (otherwise one would have to work over characteristic $2$). Equation \eqref{Equantion CFloc is Morse}  now only holds at the cohomology level,
$$HF^*_{loc}(p^*\widetilde{H},p^*I)\cong HF^*_{loc}(p^*\widetilde{H},I) \cong MH^*(B,\epsilon f; g, \mathcal{L}_X).\vspace{-0.5cm}$$
\bibliography{FZ}
\bibliographystyle{amsalpha}
\end{document}